\documentclass[leqno, 11pt]{Thesis}

\usepackage{mathtools} 
\usepackage{textcomp} 
\usepackage[T1]{fontenc} 

\usepackage{subfigure}
\usepackage{color}
\usepackage{mathrsfs}
\usepackage{amssymb}

\newfont{\suet}{suet14}                            

\definecolor{darkblue}{rgb}{0,0,.7}
\definecolor{darkred}{rgb}{0.7,0,0}
\definecolor{darkgreen}{rgb}{0,0.5,0.8}
\definecolor{darkdarkgreen}{rgb}{0,0.5,0}

\def\Re{{\rm{Re} } }
\def\Im{{\rm{Im} } }
\def\d{ {\rm{d}}}
\def\meas{ {\rm{meas\ }} }
\def\dens{ {\rm{dens }} }
\def\eps{\varepsilon}

\def\R{\mathbb{R}}
\def\C{\mathbb{C}}
\def\L{\mathcal{L}}
\def\D{\mathbb{D}}
\def\N{\mathbb{N}}
\def\Z{\mathbb{Z}}
\def\Q{\mathbb{Q}}
\def\Sc{\mathcal{S}}

\def\No{\mathcal{N}} 

\def\d{\mbox{\ d}}
\def\ptau{\pmb{\tau}}

\begin{document}

\pagestyle{empty}

\begin{titlepage}
\begin{center}
\medskip
{{\LARGE {\bf Value-distribution of the  }}}\\
\medskip
{{\LARGE {\bf Riemann zeta-function and related functions}}}\\
\medskip
{{\LARGE {\bf near the critical line}}}\\[1ex]

\medskip
\end{center}
\vspace*{2cm}
\begin{center}
Dissertationsschrift zur Erlangung des
naturwissenschaftlichen\\[1ex]
Doktorgrades der Julius-Maximilians-Universit\"at W\"urzburg
\end{center}
\vspace*{1cm}
\begin{figure}[h]
\centering
\includegraphics[width=60mm]{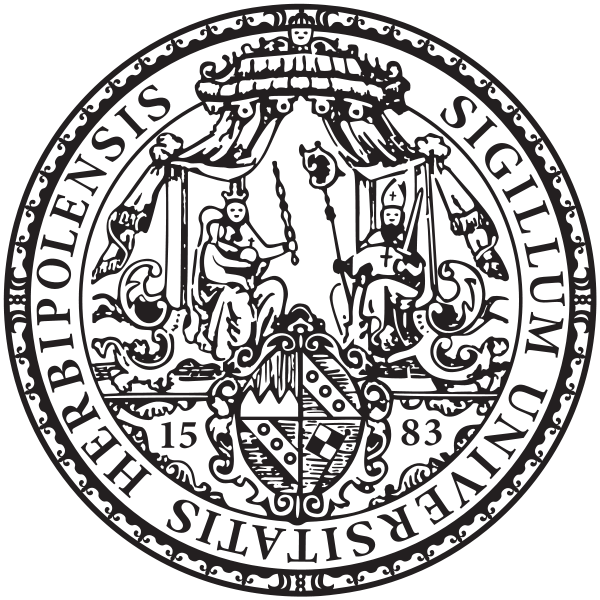}
\end{figure}
\vspace*{1cm}
\begin{center}
vorgelegt von\\[4ex]
{ {\large Thomas Christ}}\\[4ex]
aus\\[4ex]
Ansbach, Deutschland
\end{center}
\vspace*{2cm}
\begin{center}
{W\"urzburg 2013}
\end{center}
\end{titlepage}

$\mbox{ }$
\vspace{18cm}

Eingereicht am 10.12.2013\\[0.5em]
bei der Fakult\"at f\"ur Mathematik und Informatik\\[0.5em]
der Julius-Maximilians-Universit\"at W\"urzburg\\[0.5em]
1. Gutachter: Prof. Dr. J\"orn Steuding\\[0.5em]
2. Gutachter: Prof. Dr. Ramunas Garunk\v{s}tis\\[0.5em]
Tag der Disputation: 22.04.2014

\addtocontents{toc}{\protect\thispagestyle{empty}} 
\tableofcontents

\pagestyle{fancy}
\setcounter{page}{1}

\chapter*{Notations}
\addcontentsline{toc}{chapter}{Notations}
\markboth{Notations}{Notations}
We indicate some of the basic notations that we use in this thesis. 
Usually, we denote a complex variable by $s=\sigma+it$ with real part $\sigma$ and imaginary part $t$. \par
{\bf Set of numbers.}
\begin{longtable}{ll}
$\N$ & $:=\{1,2,3,...\}$, the set of positive integers\\
$\N_0$ & $:=\N\cup\{0\}$, the set of non-negative integers\\
$\mathbb{P}$ & $:=\{2,3,5,...\}$, the set of prime numbers\\
$\Z$ & $:=\{...,-1,0,1,...\}$, the set of integers\\
$\Q$ & the set of rational numbers\\
$\R$ & the set of real numbers\\
$\R^+$ & the set of positive real numbers\\
$\R^+_0$ & the set of non-negative real numbers\\
$\C$ & the set of complex numbers\\
\textcolor{white}{$(x_n)_n$} &\textcolor{white}{ $:=(x_n)_{n\in\N}:=(x_1,x_2,...)$, a sequence of elements $x_n$ from a certain set $X$.}
\end{longtable}
{\bf Subsets of the complex plane.}
\begin{longtable}{ll}
$D_r(s_0)$ & open disc with radius $r>0$ and center $s_0\in\C$\\
$\D $ & $:= D_1(0)$, unit disc\\
$\partial\Omega$ & the boundary of a domain $\Omega\subset\C$\\
$\mathcal{D}$ & $:=\{\sigma+it\in\C\, : \, 0<\sigma<1, \, t>0\}$\\
$\widehat{\C}$ & := $\C\cup\{\infty\}$, the Riemann sphere\\
\textcolor{white}{$(x_n)_n$} &\textcolor{white}{ $:=(x_n)_{n\in\N}:=(x_1,x_2,...)$, a sequence of elements $x_n$ from a certain set $X$.}
\end{longtable}
{\bf Classes of functions.}
\begin{longtable}{ll}
$\mathcal{H}(\Omega)$ & set of functions analytic in a domain $\Omega\subset\C$\\
$\mathcal{M}(\Omega)$ & set of functions meromorphic in a domain $\Omega\subset\C$\\
$\Sc$ & the Selberg class, defined in Section \ref{sec:selbergclass} \\
$\Sc^{\#}$ & the extended Selberg class, defined in Section \ref{sec:selbergclass} \\
$\Sc_R^{\#}$ & a subclass of the extended Selberg class, defined in Section \ref{sec:selbergclass}\\
$\Sc^*$ & a subclass of the Selberg class, defined in Section \ref{sec:selbergclass}\\
$\mathcal{G}$ & an extension of $\mathcal{S}^{\#}$, defined in Section \ref{sec:classG}\\
$\No$ & a class of functions, defined in Section \ref{sec:classN} \\
$\mathscr{H}^2$ & a space of Dirichlet series, defined in Section \ref{sec:DirichletH2} \\
$L_{\pmb{\sigma}}^p(K)$ & the $L^p$ space of a compact group $K$, defined in Section \ref{sec:K}\\
\textcolor{white}{$(x_n)_n$} &\textcolor{white}{ $:=(x_n)_{n\in\N}:=(x_1,x_2,...)$, a sequence of elements $x_n$ from a certain set $X$.}
\end{longtable}
{\bf Some further notations.}
\begin{longtable}{ll}
$\meas A$ & Lebesgue measure of a measurable set $A\subset \R$.\\
$\# B$ & cardinality of a finite subset $B\subset \R$.\\
$\dens^* J$ & upper density of a subset $J\subset \N$, defined in Section \ref{sec:ergodicflow}\\
$\dens_* J$ & lower density of a subset $J\subset \N$, defined in Section \ref{sec:ergodicflow}\\
$(x_n)_n$ & $:=(x_n)_{n\in\N}:=(x_1,x_2,...)$, a sequence of elements $x_n$ from a certain set $X$\\
$n|m$ & $n$ is a divisor of $m$ \\
\end{longtable}

{\bf  Landau's $O$-notation and the Vinogradov symbols.}\par
We use Landau's $O$-notation and the Vinogradov symbols in the following way. Let $f$ and $g$ be real valued functions, which are both defined on a subset of the reals.

\begin{align*}
\begin{minipage}{3cm}$f(x) = O\bigl( g(x) \bigr)$ \\or $f(x)\ll g(x)$, \\ as $x\rightarrow\infty$ \end{minipage} &: \iff \quad   \quad \exists_{C>0} \quad \mbox{s.t.}\quad \limsup_{x\rightarrow\infty} \left|\frac{f(x)}{g(x)}\right| \leq C\\[1em]
\begin{minipage}{3cm}$f(x)=o(g(x))$, \\ as $x\rightarrow\infty$ \end{minipage}
 &: \iff \quad    \limsup_{x\rightarrow\infty} \left|\frac{f(x)}{g(x)}\right| = 0\\[1em]
\begin{minipage}{3cm}$f(x) = \Omega\bigl( g(x) \bigr)$ \\or $f(x)\gg g(x)$,\\ as $x\rightarrow\infty$ \end{minipage} &: \iff \quad  \exists_{C>0} \quad \mbox{s.t.}\quad  \limsup_{x\rightarrow\infty} \left|\frac{f(x)}{g(x)}\right|\geq C\\[1em]
\begin{minipage}{3cm}$f(x)\sim g(x)$, \\ as $x\rightarrow\infty$ \end{minipage}&: \iff \quad   \lim_{x\rightarrow\infty} \left|\frac{f(x)}{g(x)}\right| = 1\\[1em]
\begin{minipage}{3cm}$f(x)\asymp g(x)$, \\ as $x\rightarrow\infty$ \end{minipage}&: \iff \quad   \exists_{A,B>0} \; \; \mbox{s.t.} \;\; \liminf_{x\rightarrow\infty} \left|\frac{f(x)}{g(x)}\right| \geq A \; \; \mbox{ and } \; \;  \limsup_{x\rightarrow\infty} \left|\frac{f(x)}{g(x)}\right| \leq B
\end{align*}
Sometimes we write $O_{\alpha}(\cdot)$, resp. $\ll_{\alpha}$, and $\Omega_{\alpha}(\cdot)$, resp. $\gg_{\alpha}$, to indicate that the implied constants depend on the parameter $\alpha$, respectively.

\chapter*{Acknowledgments}
\addcontentsline{toc}{chapter}{Acknowledgments}
First and foremost, I would like to express my deepest gratitude to my supervisor J\"orn Steuding. I am grateful for his tremendous support and insightful guidance during the last years without which this thesis would not have been completed. I appreciated the friendly and uncomplicated atmosphere in our working group and wish to thank for involving me into the academic and scientific life in such a marvelous way.\par

I would like to give my special thanks to Antanas Laurin\v{c}ikas, Ramunas Garunk\v{s}tis and Justas Kalpokas from Vilnius university for fruitful collaborations and their warm and kind hospitality during my stay in Vilnius in 2011.\par

From April 2011 to September 2013, my work was supported by a scholarship of the Hanns-Seidel-Stiftung funded by the German Federal Ministry of Education and Research (BMBF). I wish to thank the Hanns-Seidel-Stiftung for their ideational and financial support.\par 

I am grateful for the various positions that I could have at the Department of Mathematics in W\"urzburg during my doctorate studies. I would like to thank the members of chair IV and many other people from the department for their friendship, the inspiring discussions and the nice atmosphere at the department.\par

Last but not least, I owe a great debt of gratitude to my family and my dear friends for their enduring support and many unforgettable moments.

\begin{flushright}
W\"urzburg, December 2013\\
Thomas Christ
\end{flushright}

\newpage

\renewcommand{\theequation}{P.\arabic{equation}}
\renewcommand{\thesection}{P.\arabic{section}}
\renewcommand{\thetheorem}{P.\arabic{theorem}}

\chapter*{Introduction and statement of the main results}
\addcontentsline{toc}{chapter}{Introduction and statement of the main results}
\markboth{Introduction and statement of the main result}{Introduction and statement of the main result}

The Riemann zeta-function is a central object in multiplicative number theory. Its value-distribution in the complex plane encodes deep arithmetic properties of the prime numbers. In fact, many important insights into the distribution of the primes were revealed by exploring the analytic behaviour of the Riemann zeta-function.\par 

The value-distribution of the Riemann zeta-function, however, is far from being well-understood and bears many interesting analytic phenomena which are worth to be studied, independently of their arithmetical relevance. A crucial role is assigned to the analytic behaviour of the zeta-function on the so called critical line. The latter forms the background for several open conjectures; for example, the Riemann hypothesis, the Lindel\"of hypothesis and Ramachandra's denseness conjecture.\par

The scope of this thesis is to understand the behaviour of the Riemann zeta-function near and on the critical line in a better way. \par
In Section \ref{sec:riemann} of this introductory chapter, we introduce the Riemann zeta-function and expose the exceptional character of its behaviour on the critical line.  \par
To figure out which basic features of the Riemann zeta-function are responsible for certain phenomena in its value-distribution, it is reasonable to investigate the zeta-function in a broader context. In Section \ref{sec:selbergclass}, we consider the Selberg class, which was introduced by Selberg \cite{selberg:1992} as a promising attempt to gather all Dirichlet series which satisfy similar properties as the Riemann zeta-function.\par

In Section \ref{sec:outline}, we provide an outline of this thesis, state the main results and briefly report on our methods.

\section{The Riemann zeta-function}\label{sec:riemann}

In the following, let $s=\sigma+it$ denote a complex variable with real part $\sigma$ and imaginary part $t$. In the half-plane $\sigma>1$, the Riemann zeta-function is defined by an absolutely convergent Dirichlet series 
$$
\zeta(s):=\sum_{n=1}^{\infty}\frac{1}{n^s}.
$$
Euler revealed an intimate connection of $\zeta(s)$ to the prime numbers. He discovered that $\zeta(s)$ can be rewritten as an infinite product 
$$
\zeta(s)=\prod_{p\in\mathbb{P}}(1-p^{-s})^{-1}, \qquad \sigma>1,
$$
where $\mathbb{P}$ denotes the set of prime numbers.\par

In his seminal paper of 1859, Riemann \cite{riemann:1859} laid the foundations to investigate $\zeta(s)$ as a function of a complex variable $s$. He discovered that $\zeta(s)$ can be continued analytically to the whole complex plane, except for a simple pole at $s=1$ with residue $1$, and satisfies the functional equation
\begin{equation}\label{fct-eq}
\zeta(s)=\Delta(s)\zeta(1-s) \qquad \mbox{with }\qquad \Delta(s)= \pi^{s-\frac{1}{2}} \frac{\Gamma\left(\frac{1-s}{2}\right)}{\Gamma\left(\frac{s}{2}\right)},
\end{equation}
where $\Gamma$ denotes the Gamma-function. Stirling's formula allows to describe the analytic behaviour of the factor $\Delta(s)$ appearing in the functional equation in a rather precise way. As $|t|\rightarrow\infty$, the asymptotic formula 
\begin{equation}\label{Delta}
\Delta(\sigma+it) = \left( \frac{|t|}{2\pi}\right)^{\frac{1}{2}-\sigma-it} \exp\left( i(t+\tfrac{\pi}{4})\right) (1+O(|t|^{-1})) 
\end{equation}
holds uniformly for $\sigma$ from an arbitrary bounded interval. The reflection principle 
$$
\zeta(\overline{s})=\overline{\zeta(s)} \qquad\mbox{for } s\in\C
$$
provides a further functional equation for the Riemann zeta-function. Due to the latter, it is sufficient to study the value-distribution of the zeta-function in the upper half-plane $t\geq 0$.\par
The functional equation \eqref{fct-eq}, together with the reflection principle, evokes a strong symmetry of the Riemann zeta-function with respect to the so called {\it critical line} $\sigma=\frac{1}{2}$. On the latter, the value-distribution of the Riemann zeta-function is exceptional in many ways. \par

{\bf Zeros of the Riemann zeta-function.} The zeta-function has simple zeros at the negative even integers $s=-2n$, $n\in\N$. These zeros are called {\it trivial zeros}. All other zeros lie inside the so called {\it critical strip} $0\leq\sigma\leq1$. We denote these zeros by $\rho=\beta+i\gamma$ and call them {\it non-trivial zeros}. Due to the functional equation and the reflection principle, the non-trivial zeros are symmetrically distributed with respect to the critical line and the real axis. According to the Riemann-von Mangoldt formula, the number $N(T)$ of non-trivial zeros with imaginary part $\gamma\in(0,T]$ is asymptotically given by
$$
N(T)=\frac{T}{2\pi}\log\frac{T}{2\pi e} + O(\log T),
$$
as $T\rightarrow\infty$. The \textbf{\textit{ Riemann hypothesis (RH)}} states that all non-trivial zeros of the Riemann zeta-function lie on the critical line $\sigma=\frac{1}{2}$; or, equivalently, that $\zeta(s)\neq 0$ for $\sigma>\frac{1}{2}$. The Riemann hypothesis is neither proven nor disproven and is considered as a central open problem in number theory. Its arithmetic relevance lies in the impact of the non-trivial zeros on the error term in the prime number theorem. The fact that the Riemann zeta-function is non-vanishing in the half-plane $\sigma\geq 1$ leads to an asymptotic formula for the number $\pi(x)$ of primes $p\in\mathbb{P}$ with $p\leq x$. Building on ideas of Riemann, this was proved by Hadamard \cite{hadamard:1896} and de La Vall\'{e}e-Poussin \cite{vallee:1896}, independently. A zero-free region of the Riemann zeta-function to the left of $\sigma=1$ is needed in order to get an asymptotic formula for $\pi(x)$ with explicit error term. Up to now, the largest known zero-free region is due to Korobov \cite{korobov:1958} and Vinogradov \cite{vinogradov:1958}. They showed independently that, for sufficiently large $|t|$, the Riemann zeta-function has no zeros in the region defined by
$$
\sigma \geq 1 - \frac{A}{(\log |t|)^{\frac{1}{3}} (\log\log|t|)^{\frac{2}{3}}}
$$
with some constant $A>0$. Their result implies that
$$
E(x):=\pi(x) - \int_2^x \frac{\d u}{\log u} \ll x \exp\left(-B \frac{(\log x)^{\frac{3}{5}}}{(\log\log x)^{\frac{1}{5}}} \right) 
$$
with some constant $B>0$. So far, it is not known whether there exists a $\theta\in[\frac{1}{2},1)$ such that the zeta-function has no zeros in the half-plane $\sigma>\theta$. Von Koch \cite{koch:1900} showed that $E(x)\ll x^{\theta+\eps}$ holds, with any $\eps >0$, if and only if the Riemann zeta-function is non-vanishing in $\sigma>\theta$. Thus, in particular, the truth of the Riemann hypothesis would imply that $E(x)\ll x^{\frac{1}{2}+\eps}$.\par

There are some partial results supporting the Riemann hypothesis. Hardy \cite{hardy:1914} showed that there are infinitely many zeros on the critical line. 
His result was improved significantly by Selberg \cite{selberg:1942} who obtained that a positive proportion of all non-trivial zeros can be located on the critical line: let $N^{0}(T)$ denote the number of non-trivial zeros which lie on the critical line and have imaginary part $\gamma\in(0,T]$, then
$$
U:=\liminf_{T\rightarrow\infty}\frac{N^{0}(T)}{N(T)}\geq C
$$
with some (computable but very small) constant $C>0$. Selberg's lower bound for $U$ was improved considerably by Levinson \cite{levinson:1974} who obtained that $U\geq0.3437$. Later, Conrey \cite{conrey:1989} found that $U\geq 0.4088$ and, very recently, Bui, Conrey \& Young \cite{buiconreyyoung:2011} established $U\geq 0.4105$. 
\par 
Besides of measuring the number of zeros on the critical line, there are also attempts to bound the number of possible zeros off the critical line. Let $N(\sigma,T)$ denote the number of non-trivial zeros with real part $\beta>\sigma$ and imaginary part $\gamma\in(0,T]$. Due to a classical result of Selberg \cite{selberg:1946} we know that, uniformly for $\frac{1}{2}\leq \sigma \leq 1$,
\begin{equation}\label{selbergzerodensity}
N(\sigma,T)\ll T^{1-\frac{1}{4}(\sigma-\frac{1}{2})} \log T.
\footnote{For more advanced zero-density estimates for the Riemann zeta-function the reader is referred to Titchmarsh \cite[\S 9]{titchmarsh:1986} and Ivi\'{c} \cite[Chapt. 11]{ivic:1985}.}
\end{equation}

Many computer experiments were done in order to find a counterexample for the Riemann hypothesis. 
However, until now no zero off the critical line was dedected. By using the Odlyzko and Sch\"onhage algorithm, Gourdon \cite{gourdan:2004} located the first $10^{13}$ zeros of the Riemann zeta-function on the critical line.\par

According to the \textit{\textbf{simplicity hypothesis}}, one expects that all zeros of the Riemann zeta-function are simple. Indeed, no multiple zero has been found so far. It is known that at least a positive proportion of all zeros are simple. Let $N^*(T)$ denote the number of simple non-trivial zeros with imaginary part $\gamma\in(0,T]$. Levinson \cite{levinson:1974} proved that
$$
S:=\liminf_{T\rightarrow\infty} \frac{N^*(T)}{N(T)} \geq \tfrac{1}{3}.
$$
Bui, Conrey \& Young \cite{buiconreyyoung:2011} obtained that, unconditionally, $S\geq 0.4058$. Very recently, Bui \& Heath-Brown \cite{buiheathbrown:2013} proved that $S\geq \frac{19}{27}$, under the assumption of the Riemann hypothesis.\footnote{By assuming additionally the truth of the generalized Lindel\"of hypothesis, this was already known to Bui, Conrey \& Young \cite{buiconreyyoung:2011}. Bui \& Heath-Brown \cite{buiheathbrown:2013} succeeded to remove the generalized Lindel\"of hypothesis by making careful use of the generalized Vaughan identity.}\par

Whereas the Riemann hypothesis deals with the horizontal distribution of the non-trivial zeros, there are also many open questions concerning the vertical distribution. Let $(\gamma_n)_n$ denote the sequence of all positive imaginary parts of non-trivial zeros in ascending order. Littlewood \cite{littlewood:1924} showed that the gap between two consecutive ordinates $\gamma_n$, $\gamma_{n+1}$ tends to zero, as $n\rightarrow\infty$. In particular, he obtained that, as $n\rightarrow\infty$,
$$
\gamma_{n+1}-\gamma_n \ll \frac{1}{\log\log\log \gamma_n}.
$$
According to the Riemann-von Mangoldt formula the average spacing between two consecutive ordinates $\gamma_n,\gamma_{n+1}\in(T,2T]$ is given by $\frac{2\pi}{\log T}$, as $T\rightarrow\infty$. The \textbf{\textit{gap conjecture}} predicts that there appear arbitrarily small and arbitrarily large deviations from the average spacing: let
$$
\lambda:= \limsup_{n\rightarrow\infty}  \frac{(\gamma_{n+1}-\gamma_n) \log\gamma_n}{2\pi}
\qquad\mbox{and}\qquad
\mu:= \liminf_{n\rightarrow\infty}  \frac{(\gamma_{n+1}-\gamma_n) \log\gamma_n}{2\pi}.
$$
Then, one expects that $\lambda=\infty$ and $\mu=0$. It was remarked by Selberg \cite{selberg:1946-3} and proved by Fujii \cite{fujii:1975} that $\lambda>1$ and $\mu<1$. These are still the only unconditional bounds for $\lambda$ and $\mu$ which are at our disposal. On the assumption of the Riemann hypothesis, the current records in bounding $\lambda$ and $\mu$ are $\lambda>2.766$, according to Bredberg \cite{bredberg:2011}, and $\mu<0.5154$, according to Feng \& Wu \cite{fengwu:2010}.\footnote{If one is willing to assume additional conjectures, there are better results available. By assuming the generalized Riemann hypothesis, Bui \cite{bui:2011-1} obtained that $\lambda>3.033$. By assuming the Riemann hypothesis and certain moment conjectures originating from random matrix theory, Steuding \& Steuding \cite{steudingsteuding:2007} showed that $\lambda=\infty$, as predicted by the gap conjecture.}\par

Montgomery \cite{montgomery:1973} studied the pair correlation of ordinates $\gamma$, $\gamma'$ of non-trivial zeros. His investigations led him to the conjecture that, for any fixed $0<\alpha < \beta$,
$$
\lim_{T\rightarrow\infty} \frac{1}{N(T)} \# \left\{\gamma,\gamma'\in(0,T] \, : \, \alpha \leq \frac{(\gamma-\gamma')\log T}{2\pi} \leq \beta \right\}
= \int_{\alpha}^{\beta} \left( 1- \left( \frac{\sin\pi u}{\pi u}\right)^2 \right) \d u.
$$
This is known as \textit{\textbf{Montgomery's pair correlation conjecture (PCC)}}. The truth of the PCC implies that $S=1$ and $\mu=0$. Dyson pointed out to Montgomery that eigenvalues of random Hermitian matrices have exactly the same pair correlation function. This observation laid the foundation for many models for the Riemann zeta-function on the critical line by random matrix theory.\par

{\bf $a$-points of the Riemann zeta-function.} Besides the zeros, it is reasonable to study the general distribution of the roots of the equation $\zeta(s)=a$, where $a$ is an arbitrarily fixed complex number. We call these roots $a$-points and denote them by $\rho_a=\beta_a + i\gamma_a$. For sufficiently large $n\in\N$, there is an $a$-point near every trivial zero $s=-2n$. Apart from these $a$-points generated by the trivial zeros, there are only finitely many $a$-points in the half-plane $\sigma\leq 0$. We refer to the $a$-points in $\sigma\leq 0$ as {\it trivial $a$-points} and call all other $a$-points {\it non-trivial $a$-points}. The non-trivial $a$-points can be located in a vertical strip $0\leq \sigma\leq R_a$ with a certain real number $R_a\geq 1$. In analogy to the case $a=0$, Landau \cite{BohrLandauLittlewood:1913} established a Riemann-von Mangoldt-type formula for the number $N_a(T)$ of non-trivial $a$-points with imaginary part $\gamma_a\in(0,T]$: as $T\rightarrow\infty$,
$$
N_a(T)=\frac{T}{2\pi} \log\frac{T}{2\pi e c_a} + O\left( \log T\right)
$$
with $c_a=1$ if $a\neq 1$ and $c_1=2$. Levinson \cite{levinson:1975} proved that all but 
$O(N_a(T) / \log\log T)$ of the $a$-points with imaginary part $\gamma_a\in(T,2T]$ lie in the strip
\begin{equation}\label{levinsonstrip}
\frac{1}{2} - \frac{(\log\log T)^2}{\log T} < \sigma < \frac{1}{2} + \frac{(\log\log T)^2}{\log T}.
\end{equation}
Thus, almost all $a$-points are arbitrarily close to the critical line. Under the assumption of the RH, this phenomenon was already known to Landau \cite{BohrLandauLittlewood:1913}. For $a\neq 0$, Bohr \& Jessen \cite{bohrjessen:1932} showed that the number $N_a(\sigma_1,\sigma_2,T)$ of non-trivial $a$-values which lie inside the strip $\sigma_1<\sigma<\sigma_2$ with arbitrarily chosen $\frac{1}{2}<\sigma_1<\sigma_2<1$ and have imaginary part $\gamma_a \in(0,T]$ is given asymptotically by
$$
N_a(\sigma_1,\sigma_2,T) \sim c T,
$$
as $T\rightarrow\infty$, with a constant $c>0$ that depends on $\sigma_1$, $\sigma_2$ and $a$.
\par

{\bf Voronin's universality theorem.} Building on works of Bohr \cite{bohrcourant:1914,bohrjessen:1930,bohrjessen:1932} and his collaborators, Voronin \cite{voronin:1975} discovered a remarkable universality property of the Riemann zeta-function which states, roughly speaking, that every analytic, non-vanishing function on a compact set with connected complement inside the strip $\frac{1}{2}<\sigma<1$ can be approximated by vertical shifts of the Riemann zeta-function. Voronin's universality theorem was generalized by Bagchi \cite{bagchi:1982}, Reich \cite{reich:1980} and others. In its strongest formulation it can be stated as follows.
\begin{theorem}[Voronin's universality theorem] \label{th:universality}
Let $\mathcal{K}$ be a compact set in the strip $\frac{1}{2}<\sigma <1$ with connected complement. Let $g$ be a continuous, non-vanishing function on $K$, which is analytic in the interior of $K$. Then, for every $\varepsilon>0$,
$$
\liminf_{T\to\infty}\frac{1}{T}\meas\left\{\tau\in(0,T]\,:\,\max_{s\in{\mathcal{K}}}\vert\zeta(s+i\tau)-g(s)\vert<\varepsilon\right\}>0.
$$
\end{theorem}
Here and in the following, $\meas X$ denotes the Lebesgue measure of a measurable set $X\subset \R$. Bagchi \cite{bagchi:1982} discovered that the Riemann hypothesis can be rephrased in terms of universality. The RH is true, if and only if, the zeta-function is {\it recurrent}, i.e., if the zeta-function can approximate itself in the sense of Voronin's universality theorem. The RH is true if and only if, for any compact subset $\mathcal{K}$ of $\frac{1}{2}<\sigma<1$ with connected complement and any $\varepsilon>0$,
$$
\liminf_{T\to\infty}\frac{1}{T}\meas \left\{\tau\in(0,T]\,:\,\max_{s\in { K}}\vert\zeta(s+i\tau)-\zeta(s)\vert<\varepsilon\right\}>0.
$$
As a direct consequence, the universality theorem implies the following denseness statement. For every $\frac{1}{2}<\sigma<1$ and $n\in\N_0$, the set
$$
V_n(\sigma):=\{(\zeta(\sigma+it),\zeta'(\sigma+it), ... , \zeta^{(n)}(\sigma+it))\ : \ t\in [0,\infty) \} 
$$
lies dense in $\C^{n+1}$. For $n=0$, this was already known to Bohr et. al \cite{bohrcourant:1914,bohrjessen:1930,bohrjessen:1932}. It follows basically from the Dirichlet representation and the functional equation that, for $\sigma< 0$ or $\sigma> 1$,
$$
\overline{V_0(\sigma)} \neq \C.
$$
On the assumption of the Riemann hypothesis, Garunk\v{s}tis \& Steuding \cite{garunkstissteuding:2010} proved that, for $\sigma<\frac{1}{2}$,
$$
\overline{V_0(\sigma)} \neq \C.
$$
However, even by assuming the truth of the Riemann hypothesis, it is not known whether the values of the zeta-function on the critical line lie dense in $\C$ or not. According to \textbf{\textit Ramachandra's denseness conjecture}, we expect that
$$
\overline{V_0(\tfrac{1}{2})} = \C.
$$
By assuming several moment conjectures arising from random matrix theory models for the Riemann zeta-function, Kowalski \& Nikeghbali \cite{kowalskinikeghbali:2012}  obtained that $\overline{V_0(\tfrac{1}{2})} = \C$. Garunk\v{s}tis \& Steuding \cite{garunkstissteuding:2010} showed that a multidimensional denseness statement for the zeta-function on the critical line does not hold. In particular, they proved that
$$
\overline{V_1(\tfrac{1}{2})} \neq \C^2.
$$

{\bf Mean-square value on vertical lines.} An essential ingredient in the proof of Bohr's denseness result and Voronin's universality theorem is the fact that 
$$
\lim_{T\rightarrow\infty}\frac{1}{T} \int_{-T}^T \left|\zeta(\sigma+it)\right|^2 \d t = \sum_{n=1}^{\infty} n^{-2\sigma}<\infty \qquad \mbox{ for }\sigma>\tfrac{1}{2}.
$$
On the critical line, the methods of Bohr and Voronin collaps, since
$$
\frac{1}{T} \int_{-T}^T \left|\zeta(\tfrac{1}{2}+it)\right|^2 \d t \sim \log T, \qquad \mbox{as }T\rightarrow\infty,
$$
according to Hardy \& Littlewood \cite{hardylittlewood:1936}.\par

{\bf Selberg's central limit law.} Due to Selberg (unpublished), the values of the Riemann zeta-function are Gaussian normally distributed, after some suitable normalization: for any measurable set $B\subset \C$ with positive Jordan content, as $T\rightarrow\infty$, 
$$
\frac{1}{T}\meas\left\{t\in (0,T]: \frac{\log\zeta\left(\frac{1}{2}+it\right)}{ \sqrt{\frac{1}{2}\log\log T}}\in {B}\right\}
\sim 
\frac{1}{2\pi}\iint_{B}\exp\left(-{\textstyle{\frac{1}{2}}}(x^2+y^2)\right)\d x\d y.
$$
For a first published proof, we refer to Joyner \cite{joyner:1986}. Note that $f(x,y):=\exp\left(-{\textstyle{\frac{1}{2}}}(x^2+y^2)\right)$ defines the density function of the bivariate Gaussian normal distribution.

{\bf Growth behaviour of the Riemann zeta-function} The Riemann zeta-function is a function of finite order. For $\sigma\in\R$ and any $\eps>0$, 
$$
\zeta(\sigma+it)\ll t^{\theta_{\zeta}(\sigma)+\eps}, \qquad\mbox{as }|t|\rightarrow\infty, 
$$
where $\theta_{\zeta}(\sigma)$ is a continuous, convex function in $\sigma$ with
$$
\theta_{\zeta}(\sigma)=\begin{cases}0 & \mbox{if }\sigma\geq 1,\\ \tfrac{1}{2}-\sigma & \mbox{if }\sigma\leq 0. \end{cases}
$$
According to the \textbf{\textit{Lindel\"of hypothesis (LH)}}, we expect that $\theta_{\zeta}(\frac{1}{2})=0$. This would imply that
$$
\theta_{\zeta}(\sigma)=\begin{cases}0 & \mbox{if }\sigma\geq \frac{1}{2},\\ \tfrac{1}{2}-\sigma & \mbox{if }\sigma<\frac{1}{2}. \end{cases}
$$
However, the Lindel\"of hypothesis is neither proven nor disproven. The best known upper bound for $\theta_{\zeta}(\frac{1}{2})$ is due to Huxley \cite{huxley:2002, huxley:2005}. He proved that
$$
\theta_{\zeta}(\tfrac{1}{2}) \leq \frac{32}{205} = 0.1560...\ .
$$
The truth of the Riemann hypothesis implies the truth of the Lindel\"of hypothesis. The Lindel\"of hypothesis can be reformulated in terms of power moments to the right of the critical line. Due to classical works of Hardy \& Littlewood \cite{hardylittlewood:1923}, the Lindel\"of hypothesis is true if and only if, for every $k\in\N$ and every $\sigma>\frac{1}{2}$,
\begin{equation}\label{Lind}
\lim_{T\rightarrow \infty}\frac{1}{T}\int_1^T \left|\zeta(\sigma+it) \right|^{2k} \d t =\sum_{n=1}^{\infty}\frac{d_k(n)^2}{n^{2\sigma}},
\end{equation}
where $d_k$ denotes the generalized divisor function appearing in the Dirichlet series expansion of $\zeta^{k}$. The latter formula is proved only in the cases $k=1,2$ by works of Hardy \& Littlewood \cite{hardylittlewood:1922} and Ingham \cite{ingham:1926}.

\section{The Selberg class and the extended Selberg class}\label{sec:selbergclass}
Selberg \cite{selberg:1992} made a promising attempt to describe axiomatically the class of all Dirichlet series for which an analogue of the Riemann hypothesis is expected to be true.\par
{\bf Definition of the Selberg class.} A function $\L$ belongs to the Selberg class $\Sc$ if it satisfies the following properties:
\begin{itemize}
 \item[(S.1)] {\it Dirichlet series representation.} In the half-plane $\sigma>1$, $\L$ is given by an absolutely convergent Dirichlet series
$$
\L(s) = \sum_{n=1}^{\infty} \frac{a(n)}{n^s}
$$
with coefficients $a(n)\in\C$.
\item[(S.2)] {\it Ramanujan hypothesis.} The Dirichlet series coefficients of $\L$ satisfy the growth condition $a(n)\ll n^{\eps}$ for any $\eps>0$, as $n\rightarrow\infty$; here, the implicit constant in the Vinogradov symbol may depend on $\eps$.
\item[(S.3)] {\it Euler product representation.} In the half-plane $\sigma>1$, $\L$ has a product representation
$$
\L(s) = \prod_{p\in\mathbb{P}} \L_p(s),
$$
where the product is taken over all prime numbers and 
$$
\L_p(s):=\exp \left(\sum_{k=1}^{\infty} \frac{b(p^k)}{p^{ks}} \right)
$$
with suitable coefficients $b(p^k)\in\C$ satisfying $b(p^k)\ll p^{k\theta}$ with some $\theta<\frac{1}{2}$.
\item[(S.4)] {\it Analytic continuation.} There exists a non-negative integer $k$ such that $(s-1)^k \L(s)$ defines an entire function of finite oder.
\item[(S.5)] {\it Riemann-type functional equation.} $\L$ satisfies a functional equation 
\begin{equation*}
\L(s) = \Delta_{\L} (s) \overline{\L(1-\overline{s})},
\end{equation*}
where 
\begin{equation*}
\Delta_{\L}(s):= \omega Q^{1-2s}\prod_{j=1}^f \frac{\Gamma\left( \lambda_j (1-s) + \overline{\mu_j}\right)}{\Gamma\left( \lambda_j s + \mu_j \right)},
\end{equation*}
with positive real numbers $Q, \lambda_1,...,\lambda_f$ and complex numbers $\mu_1,...,\mu_f, \omega$ with $\Re\ \mu_j \geq 0$ and $|\omega|=1$.
\end{itemize}
For a concise survey on the Selberg class and a motivation for the choice of the axioms, the reader is referred to Perelli \cite{perelli:2005}.\par 
An important parameter of a function $\L\in\Sc$ is its so called degree which is defined by $$d_{\L}:=2 \sum_{j=1}^f  \lambda_j$$ via the quantities $\lambda_j$ from the Riemann-type functional equation. The degree of $\L\in\Sc$ is uniquely determined. The Riemann zeta-function is an element of the Selberg class of degree one. Its $k$-th power ($k\in\N$) lies also in $\mathcal{S}$ and has degree $k$.\par 

Besides the Riemann zeta-function, the Selberg class contains many other arithmetical relevant $\L$-functions. Prominent examples of functions in $\mathcal{S}$ are Dirichlet $L$-functions attached to primitive characters, Dedekind zeta-functions, Hecke $L$-functions associated to algebraic number fields and, under appropriate normalizations, Hecke $L$-functions associated to certain modular forms.\par
The Euler product representation of these examples has a very special form:
\begin{itemize}
 \item[(S.3$^*$)] {\it Polynomial Euler product representation.} There exist an integer $m\in\N$ and $\alpha_1(p),...,\alpha_m(p)\in \C$ such that
\begin{equation*}
\L(s)=\prod_{p\in\mathbb{P}} \prod_{j=1}^m \left(1-\frac{\alpha_j(p)}{p^s}\right)^{-1}
\end{equation*} 
in the half-plane $\sigma>1$.
\end{itemize}

In the value-distribution of functions in the Selberg class, there appear similar phenomena as in the case of the Riemann zeta-function. It follows from the Euler product representation that $\L\in\Sc$ has no zeros in the half-plane $\sigma>1$. The function $\L$ has zeros which are generated by the poles of the $\Gamma$-factors appearing the functional equation. These zeros are called trivial zeros of $\L$ and are located at the points
$$
s= -\frac{\mu_j + k}{\lambda_j}, \qquad k\in\N_0, \; j=1,...,f
$$ 
All other zeros of $\L$ are said to be non-trivial zeros. According to the \textbf{\textit{Grand Riemann hypothesis}}, one expects that every function $\L\in\Sc$ satisfies an analogue of the Riemann hypothesis, i.e. the non-trivial zeros of every function $\L\in\Sc$ are located on the critical line $\sigma=\frac{1}{2}$. For general functions in $\Sc$ much less is known than for the special case of the Riemann zeta-function. For example, by now, it is not verified whether every function $\L\in\Sc$ satisfies the following zero-density estimate: 
\begin{itemize}
 \item[(DH)] {\it Selberg's zero-density estimate.} Let $\L\in\Sc$ and $N_0(\sigma,T)$ denote the number of non-trivial zeros $\rho=\beta+i\gamma$ of $\L$ with real part $\beta>\sigma$ and imaginary part $\gamma\in(0,T]$. Then, there exists a positive number $\alpha$ such that 
$$
N_{0}(\sigma, T) \ll T^{1-\alpha(\sigma-\frac{1}{2})} \log T
$$
uniformly in $\sigma>\frac{1}{2}$, as $T\rightarrow\infty$.
\end{itemize}
The \textbf{\textit{Grand density hypothesis}} asserts that (DH) is true for every $\L\in\Sc$. Besides the Riemann zeta-function, (DH) is verified for example for Dirichlet $L$-functions attached to primitive characters; see Selberg \cite{selberg:1946-2}. Certainly, the truth of the Grand Riemann hypothesis implies the truth of the Grand density hypothesis. Moreover, according to the \textbf{\textit{Grand Lindel\"of hypothesis}} we expect that every function $\L\in\Sc$ satisfies an analogue of the Lindel\"of hypothesis, i.e., for any $\eps>0$, 
$$
\L\left(\tfrac{1}{2}+it\right) \ll t^{\eps}, \qquad \mbox{as }t\rightarrow\infty.
$$

Besides many unsolved analytic questions concerning functions in $\Sc$, there are also several structural problems related to $\Sc$ as a class of Dirichlet series. For example, one expects that the Dirichlet series coefficients $a(n)$ of $\L\in\Sc$ satisfy the following prime mean-square condition; see Steuding \cite[Chapt. 6.6]{steuding:2007}:
\begin{itemize}
\item[(S.6)] {\it Prime mean-square condition.} For $\L\in\mathcal{S}$, there exist a positive constant $\kappa_{\L}$ such that
\begin{equation*} 
\lim_{x\rightarrow\infty}\frac{1}{\pi(x)} \sum_{p\leq x} |a(p)|^2 = \kappa_{\L},
\end{equation*}
here, the summation is taken over all primes $p\leq x$. 
\end{itemize}
Selberg \cite{selberg:1992} conjectured that the Dirichlet series coefficients $a(n)$ of any function $\L\in\Sc$ satisfy the following property: 
\begin{itemize}
\item[(S.6$^*$)] {\it Selberg's prime coefficient condition.} For $\L\in\Sc$, there exists a positive integer $n_{\L}$ such that
\begin{equation*}
\sum_{\begin{subarray}{c} p\in\mathbb{P} \\ p\leq x \end{subarray}} \frac{|a(p)|^2}{p} = n_{\L} \log\log x + O(1),
\end{equation*}
as $x\rightarrow\infty$.
\end{itemize}
We know that the Riemann zeta-function and Dirichlet $L$-functions attached to primitive characters satisfy (S.6) and (S.6$^*$); see Mertens \cite{mertens:1874} and Dirichlet \cite{dirichlet:1837}. The conditions (S.6) and (S.6$^*$) are closely related to one another; see Steuding \cite[Chapt. 6.6]{steuding:2007} for details. Selberg conjectured that (S.6$^*$) results from a deeper structure in $\mathcal{S}$: obviously, the Selberg class is multiplicatively closed. We call a function $\L\in\Sc$ primitive if
$$
\L = \L_1 \cdot \L_2 \qquad \mbox{ with }\L_1,\L_2\in\Sc
$$
implies that $\L_1=\L$ or $\L_2 = \L$. Roughly speaking, primitive functions $\L\in\Sc$ cannot be written as a non-trivial product of other functions in $\Sc$. According to  \textbf{\textit{Selberg's orthogonality conjecture}}, we expect that the following is true.
\begin{itemize}
\item[(S.6$^{**}$)] {\it Selberg's orthogonality conjecture.} For any primitive functions $\L_1,\L_2\in\Sc$ with Dirichlet series coefficients $a_{\L_1}(n)$, resp. $a_{\L_2}(n)$,
$$
\sum_{\begin{subarray}{c} p\in\mathbb{P} \\ p\leq x \end{subarray}} \frac{a_{\L_1}(p)\overline{a_{\L_2}(p)}}{p} = 
\begin{cases}
\log\log x + O(1) & \mbox{if }\L_1=\L_2,\\
O(1) & \mbox{otherwise.} 
\end{cases}
$$
\end{itemize}

Besides the Selberg class, we shall also work with certain subclasses or extensions of the Selberg class:

{\bf The extended Selberg class $\Sc^{\#}$.} A function $\L\not\equiv 0$ belongs to the extended Selberg class $\Sc^{\#}$ if it satisfies axioms (S.1), (S.4) and (S.5). The functions in $\mathcal{S}^{\#}$ do not ne\-cessari\-ly satisfy the Riemann hypothesis. The Davenport-Heilbronn zeta-function is an element of $\mathcal{S}^{\#}$, but not of $\mathcal{S}$ and has non-trivial zeros off the critical line. However, one expects that the Lindel\"of hypothesis remains still true for every function $\L\in\Sc^{\#}$. \par

{\bf The class $\mathcal{S}^{\#}_R$.} A function $\L\in \Sc^{\#}$ with $d_{\L}>0$ belongs to the class $\mathcal{S}^{\#}_R$ if it satisfies additionally the Ramanujan hypothesis (S.2). \par

{\bf The class  $\Sc^{*}$.} A function $\L\in\Sc$ belongs to the class $\mathcal{S}^*$ if $\L$ satisfies the zero-density estimate (DH) and Selberg's prime coefficient condition (S.6$^*$). One expects that both conditions hold for every function $\L\in\Sc$ and, thus, that $\Sc^*=\Sc$. We know that the Riemann zeta-function and Dirichlet $L$-functions attached to primitive characters are elements of $\Sc^{*}$.

\section{Statement of the main results and outline of the thesis}\label{sec:outline}

This thesis is divided into two parts. In part I we study the value-distribution of the Riemann zeta-function on and near the critical line. In particular, we focus on the collapse of the Voronin-type universality property at the critical line and the clustering of $a$-points around the critical line. We discuss the interplay of these two features and their connection to Ramachandra's denseness conjecture. In our argumentation, we shall use several results from the theory of normal families of meromorphic functions. For the convenience of the reader we summarize the results which we shall need in the appendix.\par 

The critical line is a natural boundary for the Voronin-type universality property of the Riemann zeta-function; see Section \ref{sec:failure}. In Chapter \ref{ch:conceptsuniv} we modify Voronin's universality concept. Roughly speaking, we add a scaling factor to the vertical shifts that appear in Voronin's universality theorem and regard
$$
\zeta_{\tau}(s):=\zeta\left(\tfrac{1}{2}+\lambda(\tau)s +i\tau \right),  \qquad s\in\D,
$$ 
with $\tau\in[2,\infty)$ and a positive function $\lambda$ satisfying $\lim_{\tau\rightarrow\infty}\lambda(\tau)=0$. By sending $\tau$ to infinity, this leads to a limiting process for the Riemann zeta-function in a funnel-shaped neighbourhood of the critical line, more precisely in the region
\begin{equation*}\label{region}
S_{\lambda}:=\left\{ \sigma+it\in\C \, : \,  \tfrac{1}{2}-\lambda(t)<\sigma<\tfrac{1}{2}+\lambda(t), \; \; t\geq 2\right\}.
\end{equation*}
We shall see in Proposition \ref{mainprop} that possible limit functions of this process depend on the choice of $\lambda$ and are strongly affected by the functional equation of the Riemann zeta-function. Our results do not only apply for the Riemann zeta-function but hold for meromorphic functions that satisfy a Riemann-type functional equation in general. For this purpose, we define in Chapter \ref{chapt:classG} the class $\mathcal{G}$, which generalizes the extended Selberg class $\mathcal{S}^{\#}$.\par
\begin{figure}\centering
\includegraphics[width=0.7\textwidth]{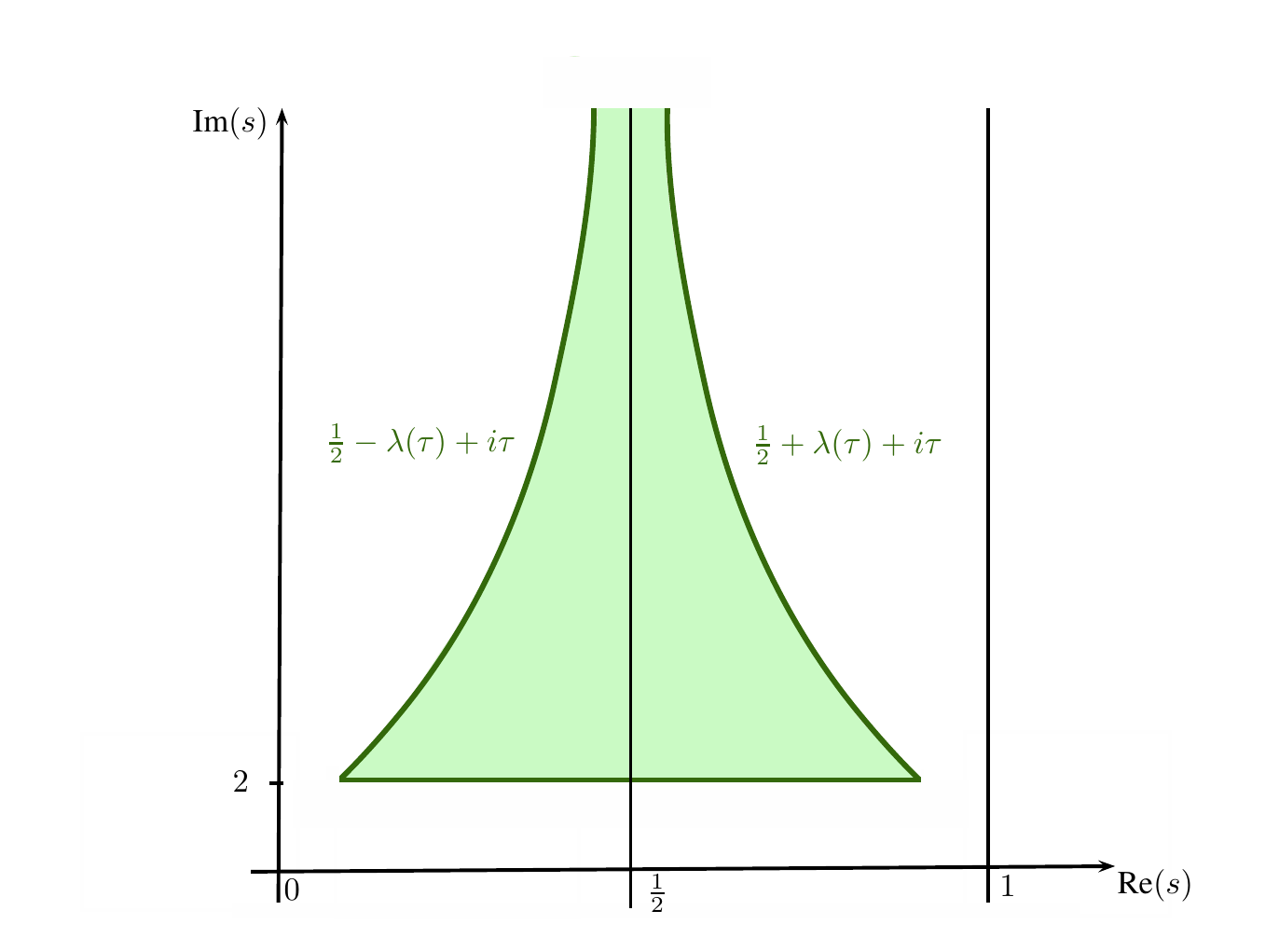}
\caption{The funnel-shaped region $S_{\lambda}$ attached to a certain positive function $\lambda$ with $\lim_{t\rightarrow\infty}\lambda(t)=0$.}
\label{fig:asymstrip} 
\end{figure}

In Chapter \ref{ch:smalllarge} we shall see that Selberg's central limit law implies that, for suitably chosen $\lambda$, the limiting process of Chapter \ref{ch:conceptsuniv} has a strong tendency to converge to $g\equiv 0$ or to $g\equiv \infty$; see Theorem \ref{th:largesmall}. This provides information on the frequency of small and large values of the Riemann zeta-function in certain regions $S_{\lambda}$ and complements recults of Laurin\v{c}ikas \cite[Chapt. 3, Theorem 3.5.1]{laurincikas:1991-2}, Bourgade \cite{bourgade:2010} and others who established certain extensions of Selberg's central limit law; see Section \ref{sec:selbergcentrallimit}.
For example, we deduce that the Riemann zeta-function assumes both arbitrarily small and arbitrarily large values on every path to infinity which lies inside $S_{\lambda}$, where the function $\lambda$ satisfies
$$
\lambda(t)=\frac{c}{\log t}, \qquad t\geq 2,
$$
with an arbitrary constant $c>0$; see Corollary \ref{cor:selbergsmalllarge} and Corollary \ref{cor:curvessmalllarge}.
Selberg's central limit law does not only apply to the Riemann zeta-function. Due to Selberg \cite{selberg:1992} it holds (with suitable adaptions) for every function in the class $\mathcal{S}^*$. Thus, most of our results in Chapter \ref{ch:smalllarge} hold for arbitrary functions $\L\in\Sc^*$. \par

Levinson \cite{levinson:1975} showed that the $a$-points of the Riemann zeta-function cluster around the critical line. In Chapter \ref{chapt:apoints} we investigate how to choose $\lambda$ such that almost all or infinitely many $a$-points of the Riemann zeta-function lie in the region $S_{\lambda}$. Levinson \cite{levinson:1975} relied essentially on a lemma of Littlewood which can be considered as an integrated version of the principle of argument. Endowed with a result of Selberg \cite{selberg:1992}, resp. Tsang \cite[\S 8]{tsang:1984}, Levinson's method yields that, for any $a\in\C$, almost all $a$-points of the Riemann zeta-function (in a certain density sense) lie inside the region $S_{\lambda}$, if $\lambda$ is chosen such that
$$
\lambda(t)=\frac{\mu(t)\sqrt{\log\log t}}{\log t}, \qquad t\geq 2,
$$
with an arbitrary positive function $\mu$ satisfying $\lim_{t\rightarrow\infty}\mu(t)=\infty$; see Theorem \ref{th:levinsonselberg}. 
Besides Levinson's method, we use certain arguments from the theory of normal families and rely on the notation of filling discs to study the $a$-point distribution of the Riemann zeta-function near the critical line. With these concepts we obtain new insights into the $a$-point distribution, complementing the observations of Levinson. In particular, we show that, for every $a\in\C$, with at most one exception, there are infinitely many $a$-points of the Riemann zeta-function inside the region $S_{\lambda}$, if $\lambda$ is chosen such that
$$
\lambda(t)=\frac{\mu(t)}{\log t}, \qquad t\geq 2,
$$
with any positive function $\mu$ satisfying $\lim_{t\rightarrow\infty}\mu(t)=\infty$; see Theorem \ref{th:levinsonselberg}. We shall see that, under quite general assumptions, the same is true for functions in $\mathcal{G}$. Beyond this, relying on a result of Ng \cite{ng:2008}, we prove that, under the assumption of the generalized Riemann hypothesis for Dirichlet $L$-functions, for every $a\in\C$, with at most one exception, there are infinitely many $a$-points of the Riemann zeta-function inside the region $S_{\lambda}$, if $\lambda$ satisfies
$$
\lambda(t) = \mu(t)\exp \left(-c_0 \frac{\log t}{\log\log t} \right), \qquad t\geq 2,
$$
where $\mu$ is any positive function satisfying $\lim_{t\rightarrow\infty}\mu(t)=\infty$ and $c_0$ any positive constant less than $\frac{1}{\sqrt{2}}$.\par

The results of Chapter \ref{ch:smalllarge} and \ref{chapt:apoints} help us to approach Ramachandra's denseness conjecture in Chapter \ref{ch:curve}. Obviously, we have
$$
0\in V_0(\tfrac{1}{2})=\left\{\zeta(\tfrac{1}{2}+it) \, : \, t\in[2,\infty)\right\}.$$ 
Ramachandra's conjecture suggests that zero is in particular an interior point of $\overline{V_0(\tfrac{1}{2})}$. This is, however, neither proven nor disproven. Relying on the results of Chapter \ref{ch:smalllarge}, we show that there is a subinterval $A\subset [0,2\pi)$ of length at least $\frac{\pi}{4}$ such that, for every $\theta\in A$, there is a sequence $(t_n)_n$ of numbers $t_n\in[2,\infty)$ with
$$
\zeta(\tfrac{1}{2}+it_n)\neq 0, \qquad \lim_{n\rightarrow\infty} \zeta(\tfrac{1}{2}+it_n) = 0 \qquad \mbox{ and } \qquad
\arg \zeta(\tfrac{1}{2}+it_n) \equiv \theta \mod 2\pi;
$$  
see Theorem \ref{th:zeroasintpoint}. This may be interpreted as a weak counterpart of a result of Kalpokas, Korolev \& Steuding \cite{kalpokaskorolevsteuding:2013} who showed that, for every $\theta\in[0,2\pi)$, there is a sequence $(t_n)_n$ of numbers $t_n\in[2,\infty)$ with $\lim_{n\rightarrow\infty} t_n = \infty$ such that, for $n\in\N$,
$$
 |\zeta(\tfrac{1}{2}+it_n)|\geq C (\log t_n)^{5/4}\qquad
\mbox{ and } \qquad
\arg \zeta(\tfrac{1}{2}+it_n) \equiv \theta \mod 2\pi
$$
with some positive constant $C$. \par
Moreover, we investigate in Chapter \ref{ch:curve} whether there are any curves 
$$
\gamma:[1,\infty)\rightarrow \C, \qquad t \mapsto \tfrac{1}{2}+\epsilon(t)+it
$$ 
with $\lim_{t\rightarrow\infty}\epsilon(t)=0$ such that the values of the Riemann zeta-function on these curves lie dense in $\C$. If we could establish a denseness result for the Riemann zeta-function on curves $\gamma$ with $\epsilon(t)$ tending to zero fast enough, then the truth of Ramachandra's conjecture would follow; see Theorem \ref{th:curvesmotivation}. In Theorem \ref{th:densenessavalues} and Theorem \ref{th:enumerationbohr} we prove that there exist certain curves $\gamma$ on which the values of the Riemann zeta-function lie dense in $\C$. We rely here both on the $a$-point results of Chapter \ref{chapt:apoints} and on Bohr's method. However, we shall not be able to derive a denseness statement for the zeta-values on the critical line.\par

In part II we study the value distribution of the Riemann zeta-function and related functions to the right of the critical line. We aim at a weak version of the Lindel\"of hypothesis. According to Hardy \& Littlewood \cite{hardylittlewood:1923}, the Lindel\"of hypothesis can be reformulated in terms of power moments to the right of the critical line. In particular, the Lindel\"of hypothesis is equivalent to statement \eqref{Lind}. Tanaka \cite{tanaka:2008} showed recently that \eqref{Lind} is true in the following measure-theoretical sense. Let $\pmb{1}_{X}$ denote the indicator function of a set $X\subset\R$ and $X^c := \R\setminus X$ its complement. Tanaka proved that there exists a subset $A\subset[1,\infty)$ of density
\begin{equation}\label{densA}
\lim_{T\rightarrow \infty} \frac{1}{T}\int_1^{T} \pmb{1}_{A}(t)\d t = 0,
\end{equation}
such that, for every $k\in\N$ and every $\sigma>\frac{1}{2}$,
\begin{equation}\label{t1}
\lim_{T\rightarrow\infty}\frac{1}{T} \int_1^T \left| \zeta(\sigma+it)\right|^{2k} \pmb{1}_{A^c}(t) \d t = \sum_{n=1}^{\infty} \frac{d_k(n)^2}{n^{2\sigma}},
\end{equation}
where $d_k$ denotes the generalized divisor function. Thus, Tanaka showed that \eqref{Lind} holds if one neglects a certain set $A\subset [1,\infty)$ of density zero from the path of integration. Tanaka used some ergodic reasoning and methods from abstract harmonic analysis to establish his results.\par

In the main theorem of Part II, Theorem \ref{th:probmom}, we extend Tanaka's result. We rely here essentially on his methods and ideas.\par

We provide an integrated and discrete version of \eqref{t1}. The discrete version, for example, implies the following:\par
\textit{Let $\alpha\in (\frac{1}{2},1]$ and $l>0$ such that $$l\notin \{2\pi k(\log\tfrac{n}{m})^{-1} \, : \, k,n,m\in\N, n\neq m\}.$$ Then, there is a subset $J\subset \N$ with
$$
\lim_{N\rightarrow\infty} \frac{1}{N}\sum_{\begin{subarray}{c}n\in J \\ n\leq N \end{subarray}} 1=0
$$ 
such that, for every $k\in\N$, uniformly for $\sigma\in[\alpha,2]$ and $\lambda\in [0,l]$,
\begin{equation*}\label{rzeta}
\lim_{N\rightarrow\infty}\frac{1}{N} \sum_{\begin{subarray}{c}n\in J \\ n\leq N \end{subarray}} \bigl|\zeta(\sigma + i\lambda + inl) \bigr|^{2k}  = \sum_{n=1}^{\infty} \frac{d_k(n)^2}{n^{2\sigma}}.
\end{equation*}}Moreover, we show that Tanaka's result holds for a large class of functions with Dirichlet series expansion in $\sigma>1$. Our result implies, for instance, the following:\par

\textit{Let $\L(s)$ be a Dirichlet series that satisfies the Ramanujan hypothesis. Suppose that $\L(s)$ extends to a meromorphic function of finite order in some half-plane $\sigma>u\geq \frac{1}{2}$ with at most finitely many poles. Suppose that
$$
\limsup_{T\rightarrow\infty} \frac{1}{T} \int_{1}^{T} \left|\L(\sigma+it)\right|^2 \d t < \infty  \qquad \mbox{for }\sigma>u.
$$
Then, for every $\alpha\in(u,1]$, there is a subset $A\subset[1,\infty)$ satisfying \eqref{densA} such that, for every $k\in\N$, uniformly for $\sigma\in[\alpha,2]$,
\begin{equation*}\label{r}
\lim_{T\rightarrow\infty} \frac{1}{T}\int_1^T \left| \L(\sigma+it)\right|^{2k} \pmb{1}_{A^c}(t) \d t = \sum_{n=1}^{\infty} \frac{|a_k(n)|^2}{n^{2\sigma}}
\end{equation*}
and
\begin{equation*}\label{rrr}
\lim_{T\rightarrow\infty} \frac{1}{T}\int_1^T  \L^k(\sigma+it) \pmb{1}_{A^c}(t) \d t = a_k(1),
\end{equation*}
where the $a_k(n)$ denote the coefficients appearing in the Dirichlet series expansion of $\L^k$. If $\L$ can be written additionally as a polynomial Euler product in $\sigma>1$, then we find a subset $A\subset[1,\infty)$ satisfying \eqref{densA} such that, for every $k\in\N$, uniformly for $\sigma\in[\alpha,2]$,
\begin{equation*}\label{rr}
\lim_{T\rightarrow\infty}\frac{1}{T} \int_1^T \left| \L(\sigma+it)\right|^{-2k} \pmb{1}_{A^c}(t) \d t = \sum_{n=1}^{\infty} \frac{|a_{-k}(n)|^2}{n^{2\sigma}},
\end{equation*}
\begin{equation*}\label{rrrr}
\lim_{T\rightarrow\infty} \frac{1}{T}\int_1^T \left| \log \L(\sigma+it)\right|^{2} \pmb{1}_{A^c}(t) \d t = \sum_{n=1}^{\infty} \frac{|a_{\log\L}(n)|^2}{n^{2\sigma}}
\end{equation*}
and
\begin{equation*}\label{rrrrr}
\lim_{T\rightarrow\infty} \frac{1}{T}\int_1^T \left| \frac{\L'(\sigma+it)}{\L(\sigma+it)}\right|^2 \pmb{1}_{A^c}(t) \d t = \sum_{n=1}^{\infty} \frac{|\Lambda_{\L}(n)|^2}{n^{2\sigma}},
\end{equation*}
where the $a_{-k}(n)$, $a_{\log\L}(n)$ and $\Lambda_{\L}(n)$ denote the coefficients of the Dirichlet series expansion of $\L^{-k}$, $\log\L$ and $\L'/\L$, respectively.}

By working with a certain normality feature we shall relax the conditions posed on $\L$ above; see Section \ref{sec:classN}. Moreover, we shall see that our results are connected to the Lindel\"of hypothesis in the extended Selberg class and complement existing mean-value results due to Carlson \cite{carlson:1922}, Potter \cite{potter:1940}, Steuding \cite{steuding:2007}, Reich \cite{reich:1980-2}, Good \cite{good:1978}, Selberg \cite{selberg:1992} and others; see Section \ref{sec:probmom}.

\newpage

\part[Value-distribution near the critical line]{Value-distribution near the critical line}

\renewcommand{\thechapter}{\arabic{chapter}}  
\renewcommand{\thesection}{\arabic{chapter}.\arabic{section}}  
\renewcommand{\theequation}{\arabic{chapter}.\arabic{equation}}
\renewcommand{\thetheorem}{\arabic{chapter}.\arabic{theorem}}
\setcounter{chapter}{0}

\chapter{A Riemann-type functional equation}\label{chapt:classG}
In this chapter we define the class $\mathcal{G}$ which gathers all meromorphic functions that satisfy a Riemann-type functional equation. The class $\mathcal{G}$ generalizes the extended Selberg class $\Sc^{\#}$. By investigating the behaviour of functions in $\mathcal{G}$ we are able to detect analytic properties of functions in $\Sc^{\#}$ which are solely induced by a Riemann-type functional equation and do not depend on the Dirichlet series representation.\par
In Section \ref{sec:Delta} we state some basic facts about the function $\Delta_p$ that characterizes a Riemann-type functional equation. In Section \ref{sec:classG} we define the class $\mathcal{G}$ and give a brief overview on its elements. 

\section{The factor \texorpdfstring{$\Delta_p$}{} of a Riemann-type functional equation}\label{sec:Delta}
{\bf Definition and basic properties of $\Delta_p$.} For a given parameter tuple
$$
p:=(\omega, Q, \lambda_1,...,\lambda_f, \mu_1,...,\mu_f), \qquad f\in\N_0,
$$ 
consisting of  positive real numbers $Q, \lambda_1,..., \lambda_f$ and complex numbers $\omega, \mu_1,...\mu_f$ with $|\omega| = 1$, we set
\begin{equation}\label{def:Delta_p}
\Delta_{p}(s):= \omega Q^{1-2s}\prod_{j=1}^f \frac{\Gamma\left( \lambda_j (1-s) + \overline{\mu_j}\right)}{\Gamma\left( \lambda_j s + \mu_j \right)},
\end{equation}
where $\Gamma$ denotes the Gamma-function. Here, in contrast to the functions $\Delta_p$ used in the definition of the (extended) Selberg class, we do not pose any restriction on the real parts of the $\mu_j$'s.\par 
If $f=0$, we read \eqref{def:Delta_p} as $\Delta_p(s):=\omega Q^{1-2s}$ and say that $\Delta_p(s)$ has degree $d_p=0$. In this degenerate case, $\Delta_p(s)$ defines an analytic, non-vanishing function in $\C$. Moreover, for every function $\Delta_p$ with $d_p=0$, the corresponding parameter tuple $p=(\omega,Q)$ is uniquely determined.\par 
If $f\geq 1$, we define the degree of $\Delta_p(s)$ by 
$$d_p := 2\sum_{j=1}^f \lambda_j.$$ 
Certainly, in this case, $d_p$ is always a positive real number. As the Gamma-function is non-vanishing and analytic in $\C$, except for simple poles at the non-positive integers, $\Delta_p(s)$ with $d_p>0$ defines a meromorphic function in $\C$ with possible poles located at 
$$
s= 1 + \frac{n+\overline{\mu_j}}{\lambda_j}, \qquad n\in\N_0, \quad j=1,...,f,
$$ 
and possible zeros located at
$$
s=\frac{-n-\mu_j}{\lambda_j}, \qquad n\in\N_0, \quad j=1,...,f.
$$
It might happen that zeros and poles arising from different Gamma-quotients cancel each other or lead to multiply zeros or poles. We observe that all poles and zeros of $\Delta_p$ lie in the horizontal strip defined by
\begin{equation}\label{stripzerospoles}
\min \left\{-\tfrac{|\Im\ \mu_j|}{\lambda_j}\, : \, j=1,...,f  \right\}\leq t \leq 
\max \left\{\tfrac{|\Im\ \mu_j|}{\lambda_j}\, : \, j=1,...,f  \right\}.
\end{equation}
For a given function $\Delta_p$ with $d_p>0$, its representation in the form \eqref{def:Delta_p} and, thus, its assigned parameter tuple $p$ is not unique: by means of the Gauss multiplication formula 
$$
(2\pi)^{(n-1)/2} n^{1/2 - n s} \Gamma\left(ns \right)
= \prod_{k=0}^{n-1} \Gamma\left( s + \frac{k}{n} \right), \qquad n\in\mathbb{N},
$$
and the factorial formula 
$$
\Gamma(s+1) = s\Gamma(s)
$$
for the Gamma-function, we can easily vary the shape of \eqref{def:Delta_p} and, thus, the data of $p$. However, as we shall see below, the degree $d_p$ of $\Delta_p$ remains invariant under these transformations.\par 
{\bf An asymptotic expansion for $\Delta_p$.} Using Stirling's formula
\begin{equation}\label{eq:stirling}
\log \Gamma(s) = \left(s-\frac{1}{2}\right) \log s  - s + \frac{1}{2}\log 2\pi + O\left(\frac{1}{|s|} \right)
\end{equation}
which is valid for $s\in\C$ satisfying $|s|\geq 1$ and $|\arg s | \leq \pi - \delta$ with any fixed $\delta>0$, uniformly in $\sigma$, as $|s|\rightarrow\infty$, we find that
\begin{align*}
\Gamma(\sigma+it) & = \sqrt{2\pi} |t|^{\sigma-1/2} \exp\left(-\frac{\pi}{2}|t| + 
i\left( t\log|t| - t + \frac{\pi t}{2|t|}\left(\sigma-\frac{1}{2} \right) \right) \right)\\
& \qquad \times \left(1 + O\left(\frac{1}{|t|}\right) \right)
\end{align*}
holds uniformly for $\sigma$ from an arbitrary bounded interval, as $|t|\rightarrow\infty$. From this, we derive by a straightforward computation the following asymptotic expansion for $\Delta_p$ and its logarithmic derivative.
\begin{lemma}\label{lem:asym_Delta_p}
Let $\Delta_p(s)$ be defined by \eqref{def:Delta_p}. Then, uniformly for $\sigma$ from an arbitrary bounded interval, as $|t|\rightarrow\infty$,
\begin{equation}\label{asymext_Delta_p}
\Delta_p (\sigma+it)  =
\omega_p \left(\lambda_p Q^2 |t|^{d_p}\right)^{1/2 - \sigma - it}
\exp\left( id_p t + i\Im \mu_p \log |t|  \right)
\left( 1+ O\left( \frac{1}{|t|} \right) \right)
\end{equation}
and
\begin{equation}\label{asymext_logdiff_Delta_p}
 \frac{\Delta_p'(\sigma+it)}{\Delta_p(\sigma+it)}  = - d_p \log |t| - \log Q^2 \lambda_p + O\left(\frac{1}{|t|} \right).
\end{equation}
Here, $d_p$ denotes the degree of $\Delta_p(s)$. The quantities $\mu_p$, $\lambda_p$ and $\omega_p$ are defined by
$$
\mu_p := \sum_{j=1}^f (1-2 \mu_j), \qquad \qquad \lambda_p := \prod_{j=1}^f \lambda_j^{2\lambda_j}
$$
and
$$
\omega_p := \omega \exp\left( i\frac{\pi}{4}(2\Re \mu_p - d_p) - i \Im \mu_p \right) \prod_{j=1}^f \lambda_j ^{-i2\Im \mu_j},
$$
if $d_p>0$; and by $\mu_p:=0$, $\lambda_p:=1$ and $\omega_p:=\omega$, if $d_p=0$.
\end{lemma}
We observe that the quantity $\omega_p$ in Lemma \ref{lem:asym_Delta_p} has modulus one.\par 
{\bf Invariant parameters of $\Delta_p$.} From the asymptotic expansion \eqref{asymext_Delta_p}, we deduce that, for every function $\Delta_p$, the quantities $d_p$, $Q^2\lambda_p$, $\Im \mu_p$ and $\omega_p$ are uniquely determined, although the parameter tuple $p$ is in general not unique. For a deeper understanding of the structure of parameter tuples $p$ leading to the same function $\Delta_p$, we refer to Kaczorowski \& Perelli \cite{kaczorowskiperelli:2000}.\par

{\bf The function $\Delta_p$ on the line $\sigma=\frac{1}{2}$.} By the definition of $\Delta_p(s)$, we have
$$
\Delta_p(s)\cdot \overline{\Delta_p(1-\overline{s})} = 1
$$ 
for $s\in\C$. This implies in particular that for real $t$
$$
\left|\Delta_p(\tfrac{1}{2}+it) \right| = 1.
$$ 
{\bf A logarithm and a square root function for $\Delta_p$.} For a given function $\Delta_p$, we define the slitted plane
\begin{equation}\label{CDelta}
\C_{\Delta_p} := \C \setminus \left( \bigcup_{z_0\in\mathcal{N}_{\Delta_p}} L_{z_0}\right), 
\end{equation}
where $\mathcal{N}_{\Delta_p}$ denotes the union of all zeros and poles of $\Delta_p$ and $L_{z_0}$ the vertical half-line defined by
$$
L_{z_0}:=\left\{\Re\ z_0 + it \, : \, -\infty < t\leq \Im\ z_0\right\}.
$$
As all zeros and poles of $\Delta_p$ can be located inside the strip \eqref{stripzerospoles}, we observe that $\C_{\Delta_p}$ contains the half-plane 
$$
t> \max \left\{\tfrac{|\Im\ \mu_j|}{\lambda_j}\, : \, j=1,...,f  \right\}.
$$ 
Certainly, $\C_{\Delta_p}$ is a simply connected domain on which $\Delta_p$ is analytic and free of zeros. Thus, there exists a continuous argument function of $\Delta_p$ which we denote by $\arg \Delta_p$ and normalize such that
$$
\arg \Delta_p(\tfrac{1}{2})\in[-\pi,\pi),
$$ 
provided that $\frac{1}{2}\in\C_{\Delta_p}$. If this is not the case, we normalize $\arg\Delta_p$ such that
$$
\lim_{\sigma \rightarrow \frac{1}{2}+} \arg\Delta_p(\sigma) \in [-\pi,\pi).
$$
With these conventions,
$$
\log \Delta_p(s):= \log |\Delta_p(s)| + i\arg \Delta_p(s)
$$ 
defines an analytic logarithm and 
\begin{equation}\label{Delta12}
\Delta_p(s)^{1/2}:=  |\Delta_p(s)|^{1/2}\exp\left(i\tfrac{1}{2}\arg \Delta_p(s)\right).
\end{equation}
an analytic square root function of $\Delta_p$ on $\C_{\Delta_p}$.

\section{The class \texorpdfstring{$\mathcal{G}$}{G} }\label{sec:classG}
In the following, let $\mathcal{D}$ denote the half-strip defined by
$$
0<\sigma<1, \qquad t>0.
$$

{\bf Definition of the class $\mathcal{G}$.} A function $G\not\equiv 0$ belongs to the class $\mathcal{G}$ if it is meromorphic in the half-strip $\mathcal{D}$ and if it satisfies a Riemann-type functional equation. By this, we mean that there is a parameter tuple 
$$
p:=(\omega, Q, \lambda_1,...,\lambda_f, \mu_1,...,\mu_f), \qquad f\in\N_0,
$$ 
which consists of  positive real numbers $Q, \lambda_1,..., \lambda_f$ and complex numbers $\omega, \mu_1,...\mu_f$ with $|\omega| = 1$ such that
\begin{equation}\label{fcteqG}
G(s) = \Delta_{p} (s) \overline{G(1-\overline{s})}
\end{equation}
for $s\in\mathcal{D}$, where $\Delta_p(s)$ is defined by \eqref{def:Delta_p}. \par

{\bf Uniqueness of the functional equation and invariants for $G\in\mathcal{G}$.} In the following, let $G\in\mathcal{G}$ and $p_G$ be an admissible parameter tuple for which $G$ solves the functional equation \eqref{fcteqG}. Suppose that there is a further admissible parameter tuple $p'_G\neq p_G$ for which $G$ satisfies \eqref{fcteqG}. Since $G(s)/\overline{G(1-\overline{s})}$ defines a meromorphic function in the half-strip $\mathcal{D}$, the identity principle yields that $\Delta_{p_G}(s)=\Delta_{p'_G}(s)$ for $s\in\C$. Thus, to every function $G\in\mathcal{G}$, there corresponds a unique functional equation of the form \eqref{fcteqG} with a uniquely determined function $\Delta_{p_G}$, which we denote from now on by $\Delta_G:=\Delta_{p_G}$. For $G\in\mathcal{G}$, the admissible parameter tuple $p_{G}$ leading to the function $\Delta_{G}$ is in general not uniquely determined. However, from the preceeding section we know that the quantities $d_{p_G}$, $Q^2\lambda_{p_G}$, $\Im \ \mu_{p_G}$ and $\omega_{p_G}$, defined in Lemma \ref{lem:asym_Delta_p} via the data of $p_G$, do not depend on the choice of $p_G$. As for every $G\in\mathcal{G}$ the function $\Delta_G$ of the functional equation is uniquely determined, we can understand these characteristic quantities of $\Delta_G$ also as characteristic quantities of $G$. In particular, we refer to $d_{p_G}$ not only as degree of $\Delta_G$ but also as degree of $G$ and write
$$
d_G := d_{p_G}= 2\sum_{j=1}^{f} \lambda_j .
$$

{\bf The critical line.} Due to the functional equation \eqref{fcteqG}, the elements of $\mathcal{G}$ obey a certain symmetry with respect to the line $\sigma=\frac{1}{2}$. For this reason, we refer to the line $\sigma=\frac{1}{2}$ as {\it critical line} of a function $G\in\mathcal{G}$.\par

{\bf Elements in $\mathcal{G}$.} Certainly, the class $\mathcal{G}$ contains all elements of the extended Selberg class $\mathcal{S}^{\#}$ and, thus, all elements of the Selberg class $\Sc$: 
$$
\Sc \subset \Sc^{\#} \subset \mathcal{G}.
$$ 
The set of parameter tuples $p$ for which we can actually find solutions of the functional equation \eqref{fcteqG} inside $\Sc$ or $\Sc^{\#}$ is limited. For $d\geq 0$, we set
$$
\Sc_d^{\#}:=\left\{\L\in\Sc^{\#} \, : \, d_{\L}=d \right\}\qquad \mbox{ and } \qquad
\Sc_d:=\left\{\L\in\Sc \, : \, d_{\L}=d \right\}.
$$
Obviously, $1\in\Sc_0$ and $\zeta^n \in \Sc_n$ for every $n\in\N$. The degree conjecture for the (extended) Selberg class asserts that
$$
\bigcup_{d\in\R^+_0 \setminus \N_0 } \Sc_d^{\#} = \emptyset \qquad \mbox{ and } \qquad \bigcup_{d\in\R^+_0 \setminus \N } \Sc_d = \{1\}.
$$
There are some results in support of this conjecture: Conrey \& Gosh \cite{conreygosh:1993} obtained that all functions $\L\in\Sc\setminus\{1\}$ have degree $d_{\L}\geq 1$. Essentially, it was already known to Richert \cite{richert:1957} and Bochner \cite{bochner:1958} that there are no functions $\L\in\Sc$ with degree $0<d_{\L}<1$. Using the machinery of linear and non-linear twists, Kaczorowski  \& Perelli \cite{kaczorowskiperelli:2002, kaczorowskiperelli:2005, kaczorowskiperelli:2011} succeeded to prove that there are neither functions $\L\in\Sc^{\#}$ of degree $0<d_{\L}<1$ nor of degree $1<d_{\L}<2$. Beyond the degree conjecture, one expects that $\Sc_n^{\#}$, $n\in\N_0$, does not contain `too many' elements. Hamburger's theorem for the Riemann zeta-function (see Titchmarsh \cite[\S 2.13]{titchmarsh:1986}) gives a first impression on how a Riemann-type functional equation invokes strong restrictions on the Dirichlet series coefficients of $\L\in\Sc^{\#}$. It is a challenging problem to classify all elements in $\Sc^{\#}$ of given degree $d\in\N_0$. Kaczorowski  \& Perelli \cite{kaczorowskiperelli:1999} proved that the Riemann zeta-function and shifts $L(s+i\theta,\chi)$ of Dirichlet $L$-functions attached to a primitive character with $\theta\in\R$ are the only functions in $\Sc_1$.\par
The situation changes if one is looking for solutions of the functional equation $\eqref{fcteqG}$ among generalized Dirichlet series. Let $\mathcal{C}$ denote the set of generalized Dirichlet series 
$$
A(s)=\sum_{n=1}^{\infty} a(n) e^{-\lambda_n s},
$$
which are absolutely convergent in some half-plane $\sigma\geq \sigma_0$, admit an analytic continuation to $\C$ as an entire function of finite order and satisfy 
$$
0=\lambda_0< \lambda_1 < \lambda_2 < ...\, , \qquad \lim_{n\rightarrow\infty}\lambda_n = \infty \qquad \mbox{ and } \qquad  a(n)\in\C \quad \mbox{ for }n\in\N .
$$
Kaczorowski \& Perelli \cite{kaczorowskiperelli:2004} showed that for every admissible parameter tuple $p$ with $\Re \ \mu_j \geq 0$ for $j=1,...,f$, the real vector space of all solutions $A\in\mathcal{C}$ of the functional equation
$$
A(s)=\Delta_p(s) \overline{A(1-\overline{s})}
$$
has an uncountable basis. \par

We have seen that $\mathcal{G}$ contains both, functions represented by an ordinary Dirichlet series and functions represented by a generalized Dirichlet series. Beyond this, there are also functions in $\mathcal{G}$ which cannot be written as a Dirichlet series. Let $\Delta_p(s)$ be as in \eqref{def:Delta_p} with an admissible parameter tuple $p$ and $f$ a meromorphic function in  $\mathcal{D}$, then the  function $G_{f,p}$ given by
$$
G_{f,p}(s):= f(s)+ \Delta_p(s)\overline{f(1-\overline{s})} 
$$
is meromorphic in $\mathcal{D}$ and satisfies the functional equation $$G_{f,p}(s)=\Delta_p(s)\overline{G_{f,p}(1-\overline{s})}.$$ Hence, $G_{f,p}$ is an element of $\mathcal{G}$. In general, $G_{f,p}$ does not have a Dirichlet series representation. Gonek modeled the Riemann zeta-function on the critical line by truncated Euler products. More precisely, he worked with
$$
\zeta_{X}(s):=P_X(s) + \Delta_{\zeta}(s)\overline{P_X(1-\overline{s})},
$$
where 
$$
P_X(s):= \exp\left(\sum_{n\leq X} \frac{\Lambda_X(n)}{n^s\log n} \right), \qquad X\geq 2,
$$ 
and $\Lambda_X$ is a suitably weighted version of the Riemann-von Mangoldt function. We deduce immediately that $\zeta_X$ is an element of $\mathcal{G}$. We refer to Gonek \cite{gonek:2012} and Christ, Kalpokas \& Steuding \cite{christkalpokassteuding:2010} for results on the analytic behaviour of $\zeta_X$ on the critical line and its intimate connection to the Riemann zeta-function. \par 

Other examples of functions in $\mathcal{G}$ can be constructed as follows: let $G\in\mathcal{G}$ and $f$ any meromorphic function $f$ in the strip $0<\sigma<1$. Then, it is easy to see that the function $H_{f,G}$ defined by
$$
H_{f,G}(s):= G(s)\left( f(s) + \overline{f(\overline{s})} + f(1-s) + \overline{f(1-\overline{s})} \right),
$$
is an element of $\mathcal{G}$. Functions of the latter type were used by Gauthier \& Zeron \cite{gauthierzeron:2004} to construct functions that share several properties with the Riemann zeta-function (same functional equation, simple pole at $s=1$, reflection principle) but have prescribed zeros off the critical line. \par

{\bf An analogue of Hardy's $Z$-function.} For $X\subset \C$ and $a,b\in\C$, we set
$$
aX+b := \left\{ax+b \, : \, x\in X \right\}.
$$
For $G\in\mathcal{G}$, let $\C_{\Delta_G}$ and $\Delta_G(\tfrac{1}{2}+it)^{1/2}$ be defined as in \eqref{CDelta} and \eqref{Delta12}. We set 
\begin{equation}\label{Dstar}
\mathcal{D}^* :=\mathcal{D}^*_{G}:= -i\cdot \left(\C_{\Delta_G}\cap \mathcal{D} \right) +i\tfrac{1}{2}.
\end{equation}
For $G\in\mathcal{G}$ and $t\in \mathcal{D}^*$, we define
$$
Z_G(t) := G(\tfrac{1}{2}+it) \Delta_G(\tfrac{1}{2}+it)^{-1/2}.
$$
The function $Z_G(t)$ forms the analogue for $G\in\mathcal{G}$ of Hardy's classical $Z$-function and allows us to model $G$ on the critical line as a real-valued function. 

\begin{lemma} \label{lem:hardyZ}
Let $G\in\mathcal{G}$ and $\mathcal{D}^*$ be defined by \eqref{Dstar}. Then, the function 
$$
Z_G(t) := G(\tfrac{1}{2}+it) \Delta_G(\tfrac{1}{2}+it)^{-1/2}
$$
is meromorphic on $ \mathcal{D}^*$. Moreover, for real $t\in\mathcal{D}^*$, 
$$
Z_G(t)\in\R \qquad \mbox{ and } \qquad \left| Z_G(t)\right| = \left|G(\tfrac{1}{2}+it) \right|.
$$
\end{lemma}
\begin{proof}
It is immediately clear that $Z_G(t) = G(\tfrac{1}{2}+it) \Delta_G(\tfrac{1}{2}+it)^{-1/2}$ defines a meromorphic function on the domain $\mathcal{D}^*$. In the sequel, we assume that $t\in\mathcal{D}^*$ is real. Then, it follows from the functional equation that 
$$
Z(t)=G(\tfrac{1}{2}+it) \Delta_G(\tfrac{1}{2}+it)^{-1/2} = \overline{G(\tfrac{1}{2}+it) }\Delta_G(\tfrac{1}{2}+it)^{1/2}.
$$
Using the relation $\Delta(s)\overline{\Delta(1-\overline{s})}=1$, we deduce that
$$
\Delta_G(\tfrac{1}{2}+it)^{1/2} = \overline{\Delta_G(\tfrac{1}{2}+it)^{-1/2} }.
$$
It follows that $\Im\ Z_G(t)=0$ and, consequently, that $Z_G(t)\in\R$. Moreover, since $$|\Delta_G(\tfrac{1}{2}+it)| =1,$$ we obtain that $\left| Z(t)\right| = \left|G(\tfrac{1}{2}+it) \right|$. 
\end{proof}

\begin{figure}
\centering
\includegraphics[width=0.9\textwidth, height=7cm]{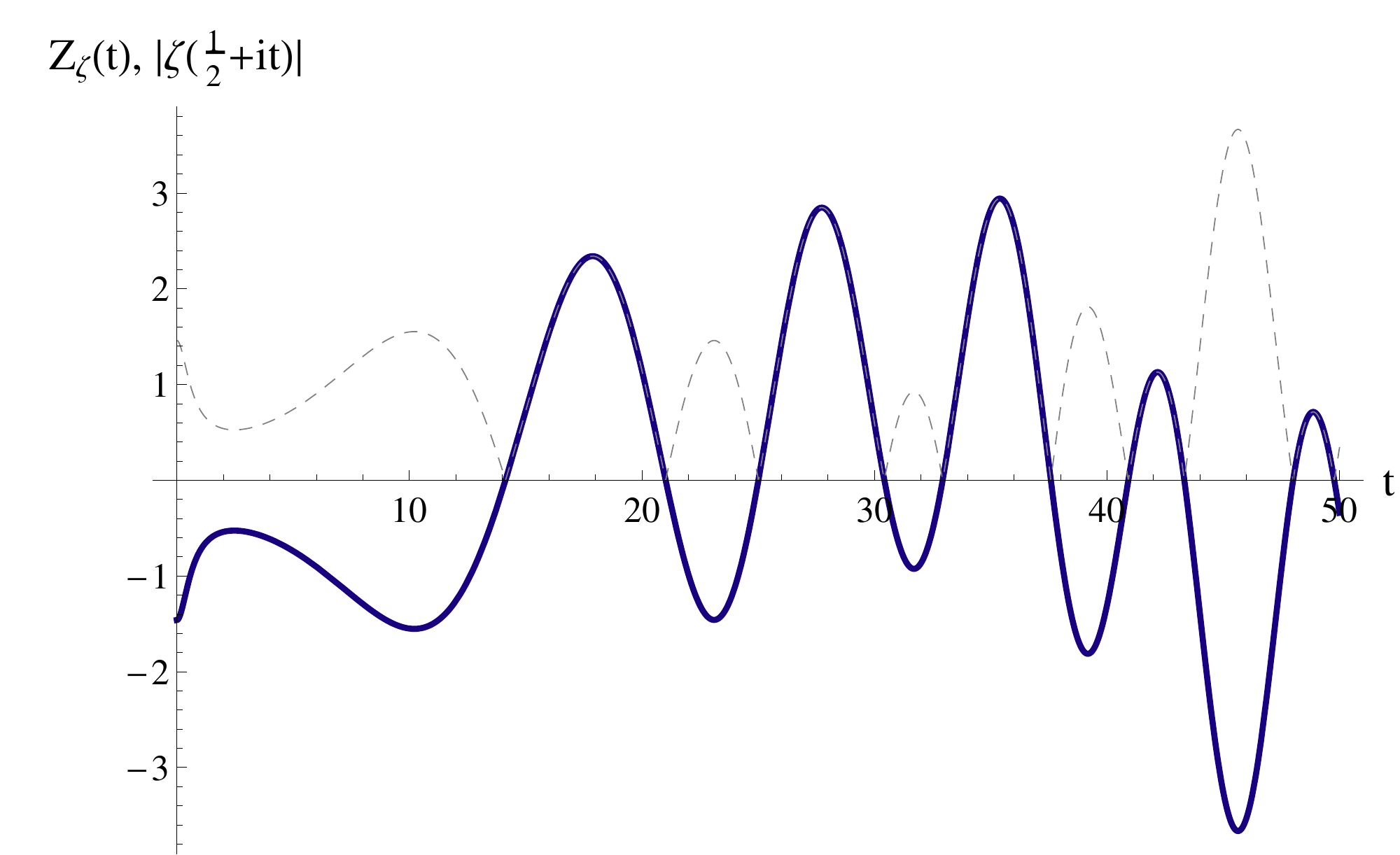}
\caption{Hardy's classical $Z$-function for the Riemann zeta-function $t\mapsto Z_{\zeta}(t)$ (blue) and $t\mapsto |\zeta(\frac{1}{2}+it)|$ (dashed), both plotted in the range $0\leq t\leq 50$. }
\end{figure}

{\bf Some special representations for $G\in\mathcal{G}$.} By Lemma \ref{lem:hardyZ}, we can write any given function $G\in\mathcal{G}$ in the form
\begin{equation}\label{repG1}
G(\tfrac{1}{2}+it) = Z_G (t) \Delta(\tfrac{1}{2}+it)^{1/2}, \qquad \mbox{}
t\in\mathcal{D}^*.
\end{equation}
There is a further possibility to represent a given function $G\in\mathcal{G}$. For $G\in\mathcal{G}$, we define
$$
\mathcal{D}_{1/2}:=\mathcal{D}_{1/2,G}:= \left(\C_{\Delta_G}\cap \mathcal{D} \right) - \tfrac{1}{2}. 
$$
Moreover, we set $f_G(z):=Z_G (-iz)$. Then,
\begin{equation}\label{repG}
G(\tfrac{1}{2}+z) = f_G (z) \Delta(\tfrac{1}{2}+z)^{1/2}\qquad \mbox{ for }z\in\mathcal{D}_{1/2}.
\end{equation}
Here, the function $f_G$ satisfies a certain reflection principle. Since $Z_G(t)$ is real for real $t\in\mathcal{D}^*$, the relation 
$$
f_G(z) = Z_G(-iz) = \overline{Z_G(\overline{-iz})}  = \overline{Z_G(i\overline{z})} = \overline{f_G(-\overline{z})}
$$
holds for all purely imaginary $z\in\mathcal{D}_{1/2}$ and, thus, by the identity principle, for all $z\in\mathcal{D}_{1/2}$. This implies that $f_G$ is real on the intersection of $\mathcal{D}_{1/2}$ with the imaginary axis. 

\chapter{A modified concept of universality near the critical line}\label{ch:conceptsuniv}
In this chapter, we study the collapse of the Voronin-type universality property of the Riemann zeta-function at the critical line and discuss a modified concept of universality.\par

In section \ref{sec:failure}, we briefly discuss for which $\L$-functions a Voronin-type universality statement is known to be true. We provide a heuristic explanation that this universality property cannot persist beyond the critical line.\par
In section \ref{sec:shiftingshrinking}, we try to maintain universality on the critical line by slightly changing the concept. Roughly speaking, we add a rescaling (or zooming) factor to the shifts that occur in Voronin's universality theorem and establish a limiting process in funnel-shaped neighbourhoods of the critical line. It will turn out that it is essentially the symmetry given by the functional equation that restricts the functions to be obtained by this process. For this reason, we investigate this process not only for the Riemann zeta-function but for all functions of the class $\mathcal{G}$, i.e. all functions that are meromorphic around the line $\sigma=\frac{1}{2}$ and satisfy a Riemann-type functional equation.\par
In section \ref{sec:convnonconv}, we discuss convergence and non-convergence issues of this limiting process.

\section{Failure of Voronin's universality theorem around the critical line}\label{sec:failure}
Building on works of Bohr \cite{bohrcourant:1914,bohrjessen:1930,bohrjessen:1932} and his collaborators, Voronin \cite{voronin:1975} established a remarkable universality theorem for the Riemann zeta-function; see Theorem \ref{th:universality}. In the meantime, similar universality properties were discovered for many other $\L$-functions. Examples include Dirichlet $L$-functions (see Voronin \cite{voronin:1977}, Gonek \cite{gonek:1979} and Bagchi \cite{bagchi:1982}), Dedekind zeta-functions (see Voronin \cite{voronin:1977}, Gonek \cite{gonek:1979} and Reich \cite{reich:1977,reich:1980}) and Hecke $L$-functions to gr\"ossencharacters (see Mishou \cite{mishou:2003}).\par
There are even $\L$-functions with a stronger universality property. For them the restriction on the target function $g$ in Voronin's universality theorem to be non-vanishing on $\mathcal{K}$ can be omitted. Here, prominent examples are Hurwitz zeta-functions whose allied parameter is either transcendental or rational but not equal to $\frac{1}{2}$ or $1$; see Bagchi \cite{bagchi:1981} and Gonek \cite{gonek:1979}. For a comprehensive account on different universal $\L$-functions, we refer the reader to Steuding \cite[Sect. 1.4-1.6]{steuding:2007}.\par
Steuding \cite[Sect. 5.6]{steuding:2007} established a universality theorem for a large class of Dirichlet series. In particular, his results imply that every function $\L\in\Sc$ which has a polynomial Euler product (S.3*) and satisfies the prime mean-square condition (S.6) is universal in the sense of Voronin inside the strip
$
\sigma_m<\sigma<1,
$
where $\sigma_m$ denotes the abscissa of bounded mean-square of $\L$ which we define rigorously in Section \ref{subsec:meansquare}. For $\L\in\Sc$, we know that
$ \sigma_m \leq \max\{\tfrac{1}{2},1-\tfrac{1}{d_{\L}}\}<1$, and under the assumption of the Lindel\"of hypothesis, that $\sigma_m\leq \frac{1}{2}$. 
\footnote{For details we refer to Section \ref{sec:unboundedness}}

The critical line is a natural boundary for universality in the Selberg class. For a heuristic explanation, we restrict to the Riemann zeta-function and assume the truth of the Riemann hypothesis. Firstly, we observe that Voronin's universality theorem for the Riemann zeta-function implies that, for any compact set $\mathcal{K}$ inside the strip $\frac{1}{2}<\sigma<1$ with connected complement and any continuous, non-vanishing function $g$ (resp. $g\equiv\infty$) on $\mathcal{K}$ which is analytic in the interior of $\mathcal{K}$, there is a sequence $(\tau_k)_k$ of positive real numbers tending to infinity such that 
\begin{equation}\label{eq:limitshifts}
\zeta(s+i\tau_k) \rightarrow g(s)
\end{equation}
uniformly on $\mathcal{K}$, as $k\rightarrow\infty$. This phenomenon collapses around the critical line $\sigma=\frac{1}{2}$.  Let $D_r(\frac{1}{2})$ be the open disc with center $\frac{1}{2}$ and radius $0<r<\frac{1}{4}$. Assume that there is a sequence  $(\tau_k)_k$ of positive real numbers tending to infinity such that $\zeta(s+i\tau_k)$ converges locally uniformly on $D_r(\frac{1}{2})$ to an analytic function $g\not\equiv 0$ (resp. to $g\equiv \infty$). As $g$ is analytic on $D_r(\frac{1}{2})$, it follows that $g$ has at most finitely many zeros in $D_r(\frac{1}{2})$. According to a result of Littlewood \cite{littlewood:1924}, there is a positive constant $A$ such that, for every sufficiently large $t$, the interval 
$$\left(t- \frac{A}{\log\log\log t}, t + \frac{A}{\log\log\log t}\right)$$  contains at least one ordinate of a non-trivial zero of the Riemann zeta-function. As we assumed the truth of the Riemann hypothesis, all non-trivial zeros can be located on the critical line. Thus, the number of zeros of $\zeta(s+i\tau_k)$ in $D_{r/2}(\frac{1}{2})\subset D_r(\frac{1}{2})$ tends to infinity. By the theorem of Hurwitz (Theorem \ref{th:hurwitz} in the appendix), we can conclude that $g\equiv 0$ is the only possible limit function that can be obtained in this case.\par
 
Proposition \ref{mainprop} (a) of the next section provides an unconditional proof for the failure of Voronin's universality theorem around the critical line. We shall see that it is essentially the functional equation that is responsible for the collapse of universality.

\section{A limiting process in neighbourhoods of the critical line}\label{sec:shiftingshrinking}

In this section, we try to `rescue' universality on $\sigma=\frac{1}{2}$ by slightly changing the concept.\par
Concepts of universality appear in various areas of analysis. There are real universal functions due to Fekete (see Steuding \cite[Appendix]{steuding:2007}), entire universal functions of Birkhoff-type due to Birkhoff \cite{birkhoff:1929} and entire universal functions of MacLane-type due to MacLane \cite{maclane:1952}. There are universal Taylor series due to Luh \cite{luh:1986} and a concept of universality for Fourier series of continuous functions on $\partial\D$ due to M\"uller \cite{mueller:2010}. Rubel \cite{rubel:1981} discovered a universal differential equation; see Elsner \& Stein \cite{elsnerstein:2011} for an overview on recent developments. There is a concept of differential universality for the Riemann zeta-function for which we refer to Christ, Steuding \& Vlachou \cite{christsteudingvlachou:2013}. The theory of hypercyclic operators investigates universality phenomena in a rather abstract topological and functional analytic setting; see the textbook of Bayart \& Matheron \cite{bayartmatheron:2009}. For a comprehensive survey on different concepts of universality the reader is referred to Grosse-Erdmann \cite{grosseerdmann:1999} and Steuding \cite[Appendix]{steuding:2007}.\par

Andersson \cite{andersson:arxiv-1} discussed a universality property for the Riemann zeta-function on the critical line. He restricted the target functions $g$ to compact line segments $L_{\alpha}:=[\frac{1}{2}-i\alpha, \frac{1}{2}+i\alpha]$ with $\alpha>0$, i.e. connected compact sets without interior points, and asked for approximating $g$ by vertical shifts of the zeta-function on $\sigma=\frac{1}{2}$. He found out that among all continuous functions on $L_{\alpha}$ only the function $g\equiv 0$ (resp. $g\equiv\infty$) might possibly be approximated in this way. \par

We shall persue the following approach: we add a scaling factor to the vertical shifts in \eqref{eq:limitshifts} and try to figure out which target functions $g$ are to be obtained by this modified limiting process. We investigate the latter not only for the Riemann zeta-function but for general functions of the class $\mathcal{G}$. \par

In the sequel, let $\mathcal{M}(\Omega)$ denote the set of meromorphic functions on a domain $\Omega\subset \C$. Moreover, let $\mathcal{H}(\Omega)\subset\mathcal{M}(\Omega)$ denote the set of analytic functions on $\Omega$. For families $\mathcal{F}\subset \mathcal{M}(\Omega)$ there is a notion of normality. We refer to the appendix for basic definitions and fundamental results of this concept.\par

Let $\mu:[2,\infty) \rightarrow\R^+ $ be a positive function satisfying $\mu(\tau)\leq \frac{1}{2}\log\tau$ for all $\tau\in[2,\infty)$. Every function $\mu$ induces a corresponding family $\{\varphi_{\mu(\tau)}\}_{\tau\in[2,\infty)}$ of linear conformal mappings 
\begin{equation}\label{eq:varphidef}
\varphi_{\mu(\tau)}:\mathbb{D}\rightarrow \C \qquad \mbox{ by } \qquad
s:=\varphi_{\mu(\tau)} (z):=\frac{1}{2} + \frac{\mu(\tau)}{\log\tau} z + i\tau .
\end{equation}
We observe that $\varphi_{\mu(\tau)}$ maps the center of the unit disc $\D$ to $\frac{1}{2}+i\tau$ and shrinks its radius linearly by a factor 
$$\lambda(\tau):=\frac{\mu(\tau)}{\log \tau}.$$
We call $\lambda(\tau)$ the scaling (or zooming) factor of $\varphi_{\mu(\tau)}$. The action of the map $\varphi_{\mu(\tau)}$ is illustrated in Figure \ref{fig:phi}. \par
\begin{figure}\centering
\includegraphics[width=0.9\textwidth]{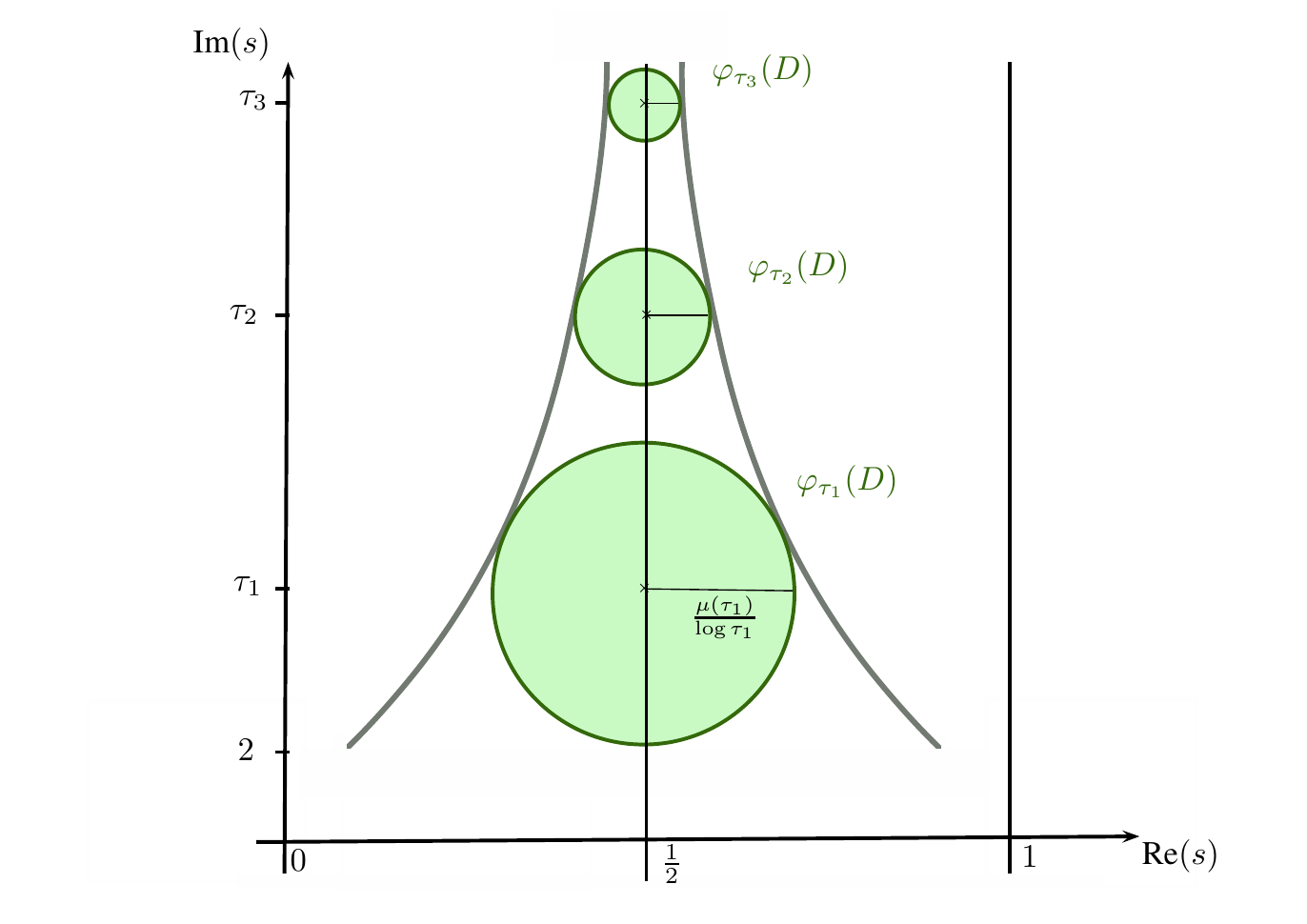}
\caption{The action of the conformal maps $\varphi_{\tau}:=\varphi_{\mu(\tau)}$ with $\tau\in[2,\infty)$ on $\D$. The function $\varphi_{\tau}$ maps the center of the unit disc $\D$ to $\frac{1}{2}+i\tau$ and shrinks its radius linearly by the factor $\lambda(\tau)=\frac{\mu(\tau)}{\log \tau}$.}
\label{fig:phi} 
\end{figure}

The condition $\mu(\tau)\leq \frac{1}{2}\log\tau$ assures that, for any $\tau\in[2,\infty)$, the image domain $\varphi_{\mu(\tau)}(\D)$ lies  completely inside the half-strip $\mathcal{D}$. Thus, for any function $G\in\mathcal{G}$ and any $\tau\in[2,\infty)$,
\begin{equation}\label{Gtau}
G_{\mu(\tau)}(z):=(G\circ \varphi_{\mu(\tau)}) (z) = G( \varphi_{\mu(\tau)} (z) ) = G\left(\frac{1}{2}+\frac{\mu(\tau)}{\log\tau} z + i\tau \right)
\end{equation}
defines a meromorphic function on $\D$. For sake of simplicity, we usually write $\varphi_{\tau}$ instead of $\varphi_{\mu(\tau)}$ and $G_{\tau}$ instead of $G_{\mu(\tau)}$.\par 

In the following we regard, for a given function $G\in\mathcal{G}$ and a given conformal mapping $\varphi_{\tau}$, the family $\mathcal{F}:=\{G_{\tau}\}_{\tau\in[2,\infty)}\subset\mathcal{M}(\D)$. We try to figure out which functions $g\in\mathcal{M}(\D)$ appear as limit functions of convergent sequences in $\mathcal{F}$. The following proposition shows that the set of possible limit functions depend essentially on the speed with which the scaling factor $\lambda(\tau)$ tends to zero as $\tau\rightarrow\infty$ and that the shape of the limiting functions is strongly affected by the functional equation of $G$.

\begin{proposition}\label{mainprop}
Let $G\in\mathcal{G}$ with degree $d_G>0$ and $\mu:[2,\infty) \rightarrow\R^+ $ be a positive function satisfying $\mu(\tau)\leq \frac{1}{2}\log\tau$ for $\tau\in[2,\infty)$. Let $\{G_{\tau}\}_{\tau\in[2,\infty)}$ be the family of functions on $\D$, generated by $G$ and $\mu$ via \eqref{Gtau}.\\ Assume that there is a sequence $(\tau_k)_k$ of real numbers $\tau_k\in[2,\infty)$ with $\lim_{k\rightarrow \infty} \tau_k = \infty$ such that $(G_{\tau_k})_k $ converges locally uniformly on $\mathbb{D}$ to a limit function $g$. 
\begin{itemize}
 \item[(a)] If $\lim_{k\rightarrow\infty}\mu(\tau_k) = \infty$, then 
$
g\equiv 0$ or $ g\equiv \infty.
$
\item[(b)] If $\lim_{k\rightarrow\infty}\mu(\tau_k) = c$ with some $c\in(0,\infty)$, then $g\equiv\infty$ or $g$ is of the form
\begin{equation}\label{shape1}
g(z)=f_g(z)\exp\left(-\tfrac{cd_G}{2}z + i \ell\right) 
\end{equation}
with some $\ell \in [0,\pi)$ and some meromorphic function $f_g$ on $\D$ satisfying $$f_g(z)=\overline{f_g(-\overline{z})} \qquad \mbox{ for } z\in\D.$$
\item[(c)] If $\lim_{k\rightarrow\infty}\mu(\tau_k) = 0$, then $g\equiv \infty$ or $g$ is of the form
\begin{equation}\label{shape2}
g(z)=f_g(z)\exp\left( i \ell\right) 
\end{equation}
with some $\ell \in [0,\pi)$ and some meromorphic function $f_g$ on $\D$ satisfying $$f_g(z)=\overline{f_g(-\overline{z})}\qquad \mbox{ for } z\in\D.$$
\end{itemize}
\end{proposition}
The condition $f_g(z)=\overline{f_g(-\overline{z})}$ implies that $f_g$ is real on the intersection of the unit disc with the imaginary axis. As we shall see from the proof of Proposition \ref{mainprop}, the shapes \eqref{shape1} and \eqref{shape2} of the limit functions $g$ actually result from the representation of $G\in\mathcal{G}$ as 
$$
G(\tfrac{1}{2}+z) = f_G (z) \Delta(\tfrac{1}{2}+z)^{1/2}\qquad \mbox{ for every }z\in\mathcal{D}_{1/2}
$$
with a certain function $f_G$ satisfying $f_G(z)=\overline{f_G(-\overline{z})}$; see \eqref{repG}.\par
Proposition \ref{mainprop} is just hypothetical: we assume the convergence of the limiting process. In general, however, it seems very difficult to verify that a given sequence  $(G_{\tau_k})_k$ converges locally uniformly on $\mathbb{D}$ or not. We will postpone convergence, resp. non-convergence issues to Section \ref{sec:convnonconv}. \par
If we restrict in Proposition \ref{mainprop} to the Riemann zeta-function and assume the truth of the Riemann hypothesis, we get additional constraints on the shape of the possible limit functions:
\begin{itemize}
\item[(i)] According to Littlewood \cite{littlewood:1924}, there is a constant $A>0$ such that, for every sufficiently large $t$, the interval $$\left(t-\frac{A}{\log\log\log t}, t + \frac{A}{\log\log\log t}\right)$$ contains at least one imaginary part of a non-trivial zero of the Riemann zeta-function. Under the assumption of the Riemann hypothesis, all non-trivial zeros lie on the critical line. Thus, if 
$$\lambda(\tau)=\frac{\mu(\tau)}{\log\tau}>\frac{2A}{\log\log \log \tau}$$ 
for sufficiently large $\tau$, we can exclude the case $g\equiv\infty$ in Proposition \ref{mainprop} (a). 
\item[(ii)] The truth of the Riemann hypothesis implies that the logarithmic derivative of Hardy's $Z$-function $Z_{\zeta}(t)$ is monotonically decreasing between two sufficiently large consecutive zeros of $Z_{\zeta}(t)$; see Ivi\'{c} \cite[Chapt. 2.3]{ivic:2013}. By \eqref{repG1} and \eqref{repG}, this has effects on the shape of the limit functions $g$ in Proposition \ref{mainprop}.\par
\end{itemize}

With suitable adaptions it is also possible to state Proposition \eqref{mainprop} for functions $G\in\mathcal{G}$ with degree $d_G =0$.
\begin{proposition}\label{mainpropzero}
Let $G\in\mathcal{G}$ with degree $d_G=0$ and $\mu:[2,\infty) \rightarrow\R^+ $ be a positive function satisfying $\mu(\tau)\leq \frac{1}{2}\log\tau$ for $\tau\in[2,\infty)$. Let $\{G_{\tau}\}_{\tau\in[2,\infty)}$ be the family of functions on $\D$, generated by $G$ and $\mu$ via \eqref{Gtau}.\\ Assume that there is a sequence $(\tau_k)_k$ of real numbers $\tau_k\in[2,\infty)$ with $\lim_{k\rightarrow \infty} \tau_k = \infty$ such that $(G_{\tau_k})_k $ converges locally uniformly on $\mathbb{D}$ to a limit function $g$. 
\begin{itemize}
 \item[(a)] If $\lim_{k\rightarrow\infty}\frac{\mu(\tau_k)}{\log\tau_k} =c$ with some $c\in(0,\frac{1}{2}]$, then $g\equiv\infty$ or $g$ is of the form
\begin{equation*}
g(z)=f_g(z)\exp\left(-\tfrac{c \log Q^2}{2}z + i \ell\right) 
\end{equation*}
with some $\ell \in [0,\pi)$ and some meromorphic function $f_g$ on $\D$ satisfying $$f_g(z)=\overline{f_g(-\overline{z})} \qquad \mbox{ for } z\in\D.$$
\item[(c)] If $\lim_{k\rightarrow\infty}\frac{\mu(\tau_k)}{\log\tau_k} = 0$, then $g\equiv \infty$ or $g$ is of the form
\begin{equation*}
g(z)=f_g(z)\exp\left( i \ell\right) 
\end{equation*}
with some $\ell \in [0,\pi)$ and some meromorphic function $f$ on $\D$ satisfying $$f_g(z)=\overline{f_g(-\overline{z})}\qquad \mbox{ for } z\in\D.$$
\end{itemize}
\end{proposition}
As Proposition \ref{mainpropzero} can be proved by essentially the same method as Proposition \ref{mainpropzero} and as we are not too much interested in functions of degree zero in the further course of our investigations, we omit a proof.\par

According to Kaczorowski \& Perelli \cite{kaczorowskiperelli:1999},  every function  $\L\in\Sc^{\#}$ of degree $d_{\L}=0$ is given by a certain Dirichlet polynomial. This implies that $\L\in\Sc^{\#}$ with $d_{\L}=0$ is bounded in the half-strip $\mathcal{D}$. Thus, if we restrict ourselves in Proposition \ref{mainpropzero} to functions $\L\in\Sc^{\#}$ with $d_{\L}=0$, Montel's theorem assures that  the family $\{\L_{\tau}\}_{\tau\in[2,\infty)}$ is normal for any admissible function $\mu$.

{\bf Proof of Proposition \ref{mainprop}:} Before proving Proposition \ref{mainprop}, we start with some lemmas for the function $\Delta_p$. Recall that $\Delta_p$ depends on the parameter tuple $p=(\omega, Q,\lambda_1,...,\lambda_f,\mu_1,...,\mu_f)$ for which we defined the quantities $d_p$, $\omega_p$, $\mu_p$ and $\lambda_p$; see Lemma \ref{lem:asym_Delta_p}.\par 
The following lemma provides an asymptotic expansion for 
$$
\Delta_{p,\tau}(z):=\Delta_p(\varphi_{\tau}(z)) = \Delta_p\left(\frac{1}{2}+\frac{\mu(\tau)}{\log\tau}z + i\tau\right)
$$ 
on $\D$, as $\tau\rightarrow\infty$.
\begin{lemma}\label{lem:Delta_p_phi} 
Let $\Delta_p$ be defined by \eqref{def:Delta_p} and suppose that $d_p>0$. Let $\mu:[2,\infty)\rightarrow \R^+$ be a positive function with $\mu(\tau)\leq \frac{1}{2}\log\tau$ for $\tau\in[2,\infty)$. Let $\{\varphi_{\tau}\}_{\tau\in[2,\infty)}$ be the family of conformal mappings generated by $\mu$ via \eqref{eq:varphidef}. Then, uniformly for $z\in\mathbb{D}$, as $\tau\rightarrow\infty$,
\begin{align*}
\Delta_{p,\tau}(z):=\Delta_p(\varphi_{\tau}(z)) &= \omega_p \exp \left(-d_p\mu(\tau) z -i\nu_p(\tau) \right) \left(1+O\left(\frac{\mu(\tau)}{\log\tau} \right) +O\left(\frac{1}{\tau} \right) \right)
\end{align*}
with 
$$
\nu_p(\tau):= d_p \tau\log\tau + \tau\log(\lambda_p Q^2) -d_p \tau - \Im\ \mu_p \log\tau.
$$
\end{lemma}
\begin{proof}
Respecting the conditions posed on the function $\mu$, the assertion follows directly from the asymptotic expansion for $\Delta_p(s)$ in Lemma \ref{lem:asym_Delta_p} and a short computation.
\end{proof}

By means of the asymptotic expansion for $\Delta_{p,\tau}$ on $\D$, as $\tau\rightarrow\infty$, we are now able to describe the limit behaviour of sequences in $\{ \Delta_{p,\tau}\}_{\tau\in[2,\infty)}$.

\begin{lemma}\label{mainlemma}
Let $\Delta_p$ be defined by \eqref{def:Delta_p} and suppose that $d_p>0$. Let $\mu:[2,\infty)\rightarrow \R^+$ be a positive, (not necessarily strictly) monotonically decreasing or increasing function with $\mu(\tau)\leq \frac{1}{2}\log\tau$ for $\tau\in[2,\infty)$. Let $\{\varphi_{\tau}\}_{\tau\in[2,\infty)}$ be the family of conformal mappings generated by $\mu$  via \eqref{eq:varphidef} and set $\Delta_{p,\tau}(z):=\Delta_p(\varphi_{\tau}(z))$ for $z\in\D$ and $\tau\geq 2$.

\begin{itemize}
\item[(a)] If $\lim_{\tau\rightarrow\infty}\mu(\tau) = \infty$, then there is no sequence $(\tau_k)_k$ of real numbers $\tau_k\in[2,\infty)$ with $\lim_{k\rightarrow\infty}\tau_k = \infty$ such that $(\Delta_{p,\tau_k})_k$ converges locally uniformly in some neighbourhood of zero.
\item[(b)]  If $\lim_{\tau\rightarrow\infty}\mu(\tau) = c$ with some $c\in(0,\infty)$,
then, for every unbounded subset $A\subset[2,\infty)$, there exists a sequence $(\tau_k)_k$ of real numbers $\tau_k\in A$ with $\lim_{k\rightarrow\infty}\tau_k = \infty$ such that $(\Delta_{p,\tau_{k}})_k$ converges uniformly on $\mathbb{D}$ to a limit function $g$ given by
$$
g(z)=a\exp(-cd_p z) \qquad \mbox{ for }z\in\D
$$
with some $a\in\mathbb{C}$ satisfying $|a|=1$.\\
Conversely, if $\lim_{\tau\rightarrow\infty}\mu(\tau) = c$ with some $c\in(0,\infty)$, then, for every function $g$ of the form above, there exists a sequence $(\tau_k)_k$ of real numbers $\tau_k\in[2,\infty)$ with $\lim_{k\rightarrow\infty}\tau_k = \infty$ such that $(\Delta_{p,\tau_k})_k$ converges uniformly on $\mathbb{D}$ to $g$.  
\item[(c)]  If $\lim_{\tau\rightarrow\infty}\mu(\tau) = 0$, then, for every unbounded subset $A\subset[2,\infty)$, there exists a sequence $(\tau_k)_k$ of real numbers $\tau_k\in A$ with $\lim_{k\rightarrow\infty}\tau_k = \infty$ such that $(\Delta_{p,\tau_{k}})_k$ converges uniformly on $\mathbb{D}$ to a constant limit function 
$$
g\equiv a
$$
with some $a\in\C$ satisfying $|a|=1$.\\
Conversely, if $\lim_{\tau\rightarrow\infty}\mu(\tau) = 0$, then, for every constant function $g\equiv a$ with $|a|=1$,  there exists a sequence $(\tau_k)_k$ of real numbers $\tau_k\in[2,\infty)$ with $\lim_{k\rightarrow\infty}\tau_k = \infty$ such that 
$(\Delta_{p,\tau_k})_k$ converges uniformly on $\mathbb{D}$ to $g$.  
\end{itemize}
\end{lemma}


Lemma \ref{mainlemma} (a) follows immediately from the asymptotic expansion for $\Delta_{p,\tau}$ on $\D$ and implies that, for any unbounded subset $A\subset[2,\infty)$, the family $\{\Delta_{p,\tau}\}_{\tau\in A}$ is not normal in any neighbourhood $U\subset \D$ of zero. Thus, by the rescaling lemma of Zalcman (Theorem \ref{th:zalcman}), we find a sequence $(\tau_k)_k$ of real numbers $\tau_k\in A$ with $\lim_{k\rightarrow\infty}\tau_k = \infty$, a sequence $(z_k)_k$ of complex numbers $z_k\in\D$ with $\lim_{k\rightarrow\infty}z_k=0$ and a sequence $(\rho_k)_k$ of positive real numbers $\rho_k$ with $\lim_{k\rightarrow\infty} \rho_k = 0 $ such that 
$$
h_k(z):= \Delta_{p,\tau_{k}}(z_k+\rho_k z) = \Delta\left(\tfrac{1}{2} + \frac{\mu(\tau_{k})}{\log\tau_k}(z_k+\rho_k z) + i\tau_{k}\right)
$$
converges locally uniformly on $\C$ to a non-constant entire function. Having this in mind, the statement of Lemma (b) \ref{mainlemma} might not be too surprising. As the functions $\Delta_p$ are rather smooth, it makes sense that the limit functions of $(\Delta_{p,\tau_k})_k$ are constant, if the underlying scaling factor $\lambda(\tau_k)$ tends to zero fast enough.

\begin{proof}[Proof of Lemma \ref{mainlemma}]
As all poles and zeros of $\Delta_p$ are located in some horizontal strip, it follows that $\Delta_{p,\tau}$ is analytic and non-vanishing on $\D$ for  sufficiently large $\tau$.\par
{\bf Case (a): } Let $U\subset \D$ be an arbitrary neighbourhood of zero. Taking into account that $\lim_{\tau\rightarrow\infty} \mu(\tau) = \infty$, the asymptotic expansion of Lemma \ref{lem:Delta_p_phi} yields that, for $z=x+iy\in U$ with real part $x\neq 0$,
\begin{align}\label{p1}
 \lim_{\tau\rightarrow\infty}  \Delta_{p,\tau} (z)  = \begin{cases} 0  & \mbox{if } x>0,\\  \infty  & \mbox{if } x<0.\end{cases}
\end{align}
If there is a sequence $(\tau_k)_k$ of real numbers $\tau_k\in[2,\infty)$ with $\lim_{k\rightarrow\infty}\tau_k = \infty$ such that $(\Delta_{p,\tau_{k}})_k$ converges locally uniformly in $U$, then, according to the theorem of Weierstrass (Theorem \ref{th:weierstrass}), either
$$
g(z):=\lim_{k\rightarrow\infty}  \Delta_{p,\tau_k} (z) , \qquad z\in U,
$$
defines an analytic function in $U$ or $g\equiv \infty$ in $U$. Both cases, however, are in contradiction to \eqref{p1}.\par

{\bf Case (b): } According to our assumption on $\mu$ and the asymptotic expansion for $\Delta_{p,\tau}(z)$ of Lemma \ref{lem:Delta_p_phi}, we have, uniformly for $z=x+iy\in \mathbb{D}$,
\begin{align}\label{p2}
\lim_{\tau\rightarrow\infty} \left|\Delta_{p,\tau} (z) \right| =|\omega_p|\exp(-cd_px) \leq \exp(cd_p). 
\end{align}
This implies that the family $\mathcal{F}:=\{\Delta_{p,\tau}\}_{\tau\in[2,\infty)}$ 
is uniformly bounded on $\mathbb{D}$. According to the theorem of Montel, every sequence in  $\mathcal{F}$ has a subsequence that converges locally uniformly on $\mathbb{D}$ to an analytic function. Thus, we can extract from any given unbounded set $A\subset[2,\infty)$ a sequence $(\tau_k)_k$ with $\lim_{k\rightarrow\infty}\tau_k = \infty$ such that $(\Delta_{p,\tau_{k}})_k$ converges locally uniformly on $\D$ to an analytic function $g$. Next we shall figure out the shape of $g$: By means of \eqref{p2}, we have, uniformly for $z=x+iy\in\mathbb{D}$,
$$
|g(z)| = \lim_{k\rightarrow\infty} \left|\Delta_{p,\tau_{k}} (z)  \right|= |\omega_p| \exp(-cd_px) = \exp(-cd_px).
$$
Obviously, $g$ is non-vanishing on $\mathbb{D}$. This allows us to write 
$$
g(z)=\eta_p \exp(-cd_px + i f(x,y))  \qquad \mbox{ for }z=x+iy\in\D
$$
with some continuously differentiable function $f:\Omega \rightarrow\R$, where 
$$\Omega=\{(x,y)\in\R^2 \ : \ x+iy\in\D\}.$$ 
The Cauchy-Riemann differential equations,
\begin{align*}
 \tfrac{\partial}{\partial x}\ \Re\ g &= \tfrac{\partial}{\partial y}\ \Im\ g,\\
 \tfrac{\partial}{\partial y}\ \Re\ g &= -  \tfrac{\partial}{\partial x}\ \Im\ g,
\end{align*}
 yield that, for $(x,y)\in\Omega$,  
\begin{align*}
\tfrac{\partial}{\partial x}f+\tfrac{\partial}{\partial y}f &=- cd_p,\\
\tfrac{\partial}{\partial x}f-\tfrac{\partial}{\partial y}f &=  cd_p.
\end{align*}
Hence, $f(x,y)= - cy + \ell $ with some constant $\ell\in\R$. By setting $a:=\omega_p e^{i\ell}$, we get
$$
g(z)=a\exp(-cd_p z) \qquad \mbox{ for }z\in\mathbb{D}.
$$ 
By rescaling $\varphi_{\tau}$ in a suitable manner, we deduce from Lemma \ref{lem:Delta_p_phi}, that $(\Delta_{p,\tau_k})_k$ converges not only locally uniformly, but even uniformly on $\D$.\par

Conversely, let $g:\mathbb{D}\rightarrow\C$ be a function of the form $g(z)=a\exp(-cd_p z)$ with arbitrary $a\in\C$ satisfying $|a|=|1|$. We choose $\ell\in[0,2\pi)$ such that $\omega_p e^{i\ell} = a$ and regard the function
$$
\nu_p(\tau):= d_p \tau\log\tau + \tau\left(\log(\lambda_p Q^2)-d_p \right) - \Im\ \mu_p \log \tau
$$
which is monotonically increasing for sufficiently large $\tau$. We define $(\tau_k)_k$ to be the sequence of all solutions $\tau\geq 2$ of 
$$
-\nu_p(\tau) \equiv \ell\mod 2\pi
$$
in ascending order. As the sequence $(\tau_k)_k$ tends to infinity, the asymptotic expansion for $\Delta_{p,\tau}$ in Lemma \ref{lem:Delta_p_phi} yields that, uniformly for $z\in\mathbb{D}$,
\begin{align*}
\lim_{k\rightarrow\infty}\Delta_{p,\tau_k}(z) 
& = \omega_p\exp\left(-cd_p z + i\ell \right)= g(z).
\end{align*}
The assertion is proved.\par
{\bf Case (c):} According to our assumption on $\mu$ and the asymptotic expansion for $\Delta_{p,\tau}$ of Lemma \ref{lem:Delta_p_phi}, we have, uniformly for $z=x+iy\in \mathbb{D}$,
\begin{align}\label{p3}
\lim_{\tau\rightarrow\infty} \left|\Delta_{p,\tau} (z)  \right| =|\omega_p| =1. 
\end{align}
This implies that the family $\{\Delta_{p,\tau}\}_{\tau\in[2,\infty)}$ is uniformly bounded on $\mathbb{D}$. Thus, Montel's theorem assures that, from every unbounded subset $A\subset[2,\infty)$,  we can extract a sequence $(\tau_k)_k$ of real numbers $\tau_k\in A$ with $\lim_{k\rightarrow\infty}\tau_k = \infty$ such that $(\Delta_{p,\tau_k})_k$ converges locally uniformly  on $\mathbb{D}$ to an analytic function $g$. By means of \eqref{p3}, we have, uniformly for $z\in\mathbb{D}$,
$$
|g(z)| = \lim_{k\rightarrow\infty} \left|\Delta_{p,\tau_{k}} (z) ) \right|=  1 
$$
Thus, it follows from the open mapping theorem that $g\equiv a$ with some $a\in\C$ satisfying $|a|=1$.\par
Conversely, for a given function $g\equiv a$ with $a\in\C$ satisfying $|a|=1$, we can construct a sequence $(\tau_k)_k$ such that $(\Delta_{p,\tau_k})_k$ converges uniformly on $\D$ to $g$  in the same manner as in case (b).
\end{proof}

\begin{proof}[Proof of Proposition \ref{mainprop}] For $G\in\mathcal{G}$, let $\Delta_G$ denote the factor of the functional equation for $G$. In the following we set $\Delta_{G,\tau}(z):= \Delta_G (\varphi_{\tau}(z))$.\par 
Note that the locally uniform convergence of a sequence $(G_{\tau_k})_{k}$ on $\mathbb{D}$ implies that its limit function $g$ is either meromorphic on $\D$ or $g\equiv \infty$ on $\D$.

{\bf Case (a):} Suppose that $(G_{\tau_k})_{k}$ converges locally uniformly on $\D$ to a meromorphic limit function $g\not\equiv 0$. Then, we find a compact set $\mathcal{K}\subset \D$ with non-empty interior which lies completely to the right of the imaginary axis, i.e. $\Re\ z >0$ for $z\in\mathcal{K}$, and a constant $m>0$ such that $|g(z)|\geq m$ for $z\in\mathcal{K}$. In particular, we have
$$
\left|\lim_{k\rightarrow\infty} G(\varphi_{\tau_k}(z))\right| = \left|g(z)\right| \geq m \qquad \mbox{for }z\in\mathcal{K}.
$$
If $z\in\D$ with $\Re\ z >0$, then the point $-\overline{z}$ lies also in $\D$ and satisfies $\Re (-\overline{z}) < 0$. Respecting the growth condition posed on $\mu$, the asymptotic expansion of Lemma \ref{lem:asym_Delta_p} yields that
$$
\lim_{k\rightarrow\infty}\Delta_{G,\tau_k} (-\overline{z}) = \infty\qquad \mbox{for }z\in\mathcal{K}.
$$
According to the functional equation of $G$, we get
\begin{align*}
g(-\overline{z}) & = \lim_{k\rightarrow\infty} G\left(\varphi_{\tau_k}(-\overline{z})\right) = \lim_{k\rightarrow\infty}\Delta_G (\varphi_{\tau_k}(-\overline{z})) \overline{G\left( 1- \overline{\varphi_{\tau_k}(-\overline{z})}\right)} \\
& = \lim_{k\rightarrow\infty}\Delta_p (\varphi_{\tau_k}(-\overline{z})) \overline{G\left( \varphi_{\tau_k}(z)\right)}\\
& = \infty
\end{align*}
for $z\in\mathcal{K}$. Thus, it follows from the identity principle, applied to the meromorphic function defined by $\overline{g(-\overline{z})}$ on $\D$, that $g\equiv \infty$ in $\D$. This contradicts our assumption. 

{\bf Case (b):} Suppose that $(G_{\tau_k})_k$ converges locally uniformly on $\D$ to a meromorphic limit function $g$. Certainly, we have for an arbitrary $z \in\mathbb{D}$,
\begin{align}\label{eq:lim1}
\lim_{k\rightarrow\infty} G(\varphi_{\tau_k}(z)) = g(z).
\end{align}
Moreover, according to the functional equation, we get that
\begin{align*} 
G(\varphi_{\tau_k }(z)) & = \Delta_p(\varphi_{\tau_k}(z)) \overline{G(1-\overline{\varphi_{\tau_k}(z)})} = \Delta_p(\varphi_{\tau_k}(z)) \overline{G(\varphi_{\tau_k}(-\overline{z}))}.
\end{align*}
If $z\in\mathbb{D}$, then the point $-\overline{z}$ is also in $\mathbb{D}$. Thus, we obtain that
$$
\lim_{k\rightarrow\infty} \overline{ G (\varphi_{\tau_k}(-\overline{z})) } = \overline{ g(-\overline{z})}.
$$
Since $(G(\varphi_{\tau_k}(z)))_k$ converges locally uniformly on $\mathbb{D}$, this is also true for $(\overline{G(\varphi_{\tau_k}(-\overline{z}))})_k$. Consequently, $\overline{g(-\overline{z})}$ defines an analytic function on $\mathbb{D}$. According to case (b) of Lemma \ref{mainlemma} we find a subsequence $(\tau_{k_j})_j$ of $(\tau_k)_k$ such that, uniformly for $z\in\mathbb{D}$,
$$
\lim_{j\rightarrow\infty} \Delta(\varphi_{\tau_{k_j}}(z)) = \exp(-cd_p z + i\ell)
$$
with some $\ell\in[0,2\pi)$. Thus,
\begin{equation}\label{eq:lim2}
\lim_{j\rightarrow\infty} G(\varphi_{\tau_{k_j}}(z)) = \overline{g(-\overline{z})} \exp(-cd_p z + i\ell)
\end{equation}
holds locally uniformly for $z\in\mathbb{D}$. By the uniqueness of the limit function, we deduce from \eqref{eq:lim1} and \eqref{eq:lim2} that $g$ satisfies the  functional equation
$$
g(z)=\overline{g(-\overline{z})} \exp(-cd_p z + i\ell), \qquad \mbox{ }z\in\mathbb{D}.
$$
By setting
$$
f_g(z):= g(z)\exp\left(\frac{cd_p}{2} z - i\frac{\ell}{2}\right)
$$
the functional equation translates to
$$
f_g(z)=\overline{f_g(-\overline{z})}, \qquad \mbox{}z\in\D.
$$
It follows that $g(z)=f_g(z)\exp\left(-\frac{cd_p}{2} z + i\frac{\ell}{2}\right) $ for $z\in\D$. The assertion follows by substituting $\ell/2 \mapsto \ell'$.\par
{\bf Case (c):} To prove case (c), we follow the lines of the proof for case (b). This time, however, according to Lemma \ref{mainlemma} (c), we find a subsequence $(\tau_{k_j})_j$ of $(\tau_k)_k$ such that, uniformly for $z\in\mathbb{D}$, 
$$
\lim_{j\rightarrow\infty} \Delta_{\tau_{k_j}}(z) = \exp( i\ell)
$$
with some $\ell \in[0,2\pi)$. This leads to the functional equation 
$$
g(z)=\overline{g(-\overline{z})} \exp( i\ell), \qquad \mbox{}z\in\mathbb{D}.
$$
Again, by setting
$$
f_g(z):= g(z)\exp\left(- i\frac{\ell}{2}\right),
$$
this translates to
$$
f_g(z)=\overline{f_g(-\overline{z})} \qquad \mbox{ }z\in\D,
$$
and yields the representation $g(z)=f_g(z)\exp\left( i\frac{\ell}{2}\right)$. The assertion follows.
\end{proof}

\section{Convergence and non-convergence of the limiting process}\label{sec:convnonconv}
In this section we describe natural mechanisms that enforce the limiting process introduced in the preceeding section to converge or not to converge.  

\subsection{Non-convergence of the limiting process}\label{subsec:noncon}
Suppose that, for a given function $G\in\mathcal{G}$ and a given family  $\{\varphi_{\tau}\}_{\tau\in[2,\infty)}$ of conformal mappings on $\D$, there is an unbounded subset $A \subset[2,\infty)$ such that the corresponding family $\{G_{\tau}\}_{\tau\in A}$ is not normal in $\D$. By Montel's fundamental normality test (Theorem \ref{th:FNT1}), this bears information on the $a$-point-distribution of $G$. Thus, certain non-convergence statements for the limiting process of Section \ref{sec:shiftingshrinking} are of equal interest as convergence statements. We will study non-convergence statements and their connection to the $a$-point-distribution of $G\in\mathcal{G}$ in Chapter \ref{chapt:apoints} in details. Here, we only state the following observation which follows immediately from Proposition \ref{mainprop} (a).

\begin{lemma}\label{lem:nonconvergence}
Let $G\in\mathcal{G}$ and $\mu:[2,\infty) \rightarrow\R^+$ be a positive function with $\lim_{\tau\rightarrow\infty}\mu(\tau) = \infty$. For $\tau\in[2,\infty)$ and $z\in\D$, let $\varphi_{\tau}(z):=\tfrac{1}{2}+\frac{\mu(\tau)}{\log\tau}\cdot z + i\tau$ and $G_{\tau}(z):=G(\varphi_{\tau}(z))$.
Suppose that there is an $\alpha\in(0,\infty)$ and a sequence $(\tau_k)_k$ of real numbers $\tau_k\in[2,\infty)$ with $\lim_{k\rightarrow\infty}\tau_k = \infty$ such that
\begin{equation}\label{ass:nonconvergence}
\lim_{k\rightarrow\infty} \left|G(\tfrac{1}{2}+i\tau_k)\right| = \alpha.
\end{equation}
Then, the sequence $(G_{\tau_{k}})_k$ has no subsequence which converges locally uniformly on $\D$.
\end{lemma}
\begin{proof}
Assume that $(G_{\tau_k})_k$ has a subsequence $(G_{\tau_{k_j}})_j$ that converges locally uniformly on $\D$ to a function $g$. Then, it follows from Proposition \ref{mainprop} (a) that $g\equiv 0$ or $g\equiv \infty$. Our assumption \eqref{ass:nonconvergence}, however, yields that
$$
g(0)= \lim_{j\rightarrow\infty} G_{\tau_{k_j}}(0) = \lim_{j\rightarrow\infty} G(\tfrac{1}{2}+\tau_{k_j}) = a 
$$
with some $a\in\C$ satisfying $|a|=\alpha\in(0,\infty)$. This gives a contradiction and the lemma is proved.
\end{proof}

Let $G\in\mathcal{G}$. Due to the condition \eqref{ass:nonconvergence} in Lemma \ref{lem:nonconvergence}, it seems reasonable to determine the quantities
\begin{equation}\label{def:alpha_inf}
\alpha_{G,\scalebox{0.8}{\mbox{inf}}}:=\liminf_{\tau\rightarrow\infty} \left|G(\tfrac{1}{2}+i\tau)\right|
\quad\mbox{and}\quad
\alpha_{G,\scalebox{0.8}{\mbox{sup}}}:=\limsup_{\tau\rightarrow\infty} \left|G(\tfrac{1}{2}+i\tau)\right|.
\end{equation}
The intermediate value theorem for continuous functions implies that, for every $\alpha\in[\alpha_{G,\scalebox{0.8}{\mbox{inf}}},\alpha_{G,\scalebox{0.8}{\mbox{sup}}}]$, we find a sequence $(\tau_k)_k$ of real numbers $\tau_k\in[2,\infty)$ with $\lim_{k\rightarrow\infty}\tau_k=\infty$ such that
$$
\lim_{k\rightarrow\infty} \left| G(\tfrac{1}{2}+it)\right| = \alpha.
$$
For $\L\in\Sc^{\#}$ with $d_{\L}>0$, we expect that $\alpha_{\L,\scalebox{0.8}{\mbox{inf}}}=0$ and $\alpha_{\L,\scalebox{0.8}{\mbox{sup}}}=\infty$. However, it appears to be quite challenging to prove this for functions $\L\in\Sc^{\#}$ in general. We tackle this problem later on in Chapter~\ref{ch:smalllarge}.\par

In the class $\mathcal{G}$, things are qualitatively different. Here, we find, for every $\alpha_1,\alpha_2\in[0,\infty)$ with $\alpha_1\leq \alpha_2$, a function $G\in\mathcal{G}$ such that $\alpha_{G,\scalebox{0.8}{\mbox{inf}}}=\alpha_1$ and $\alpha_{G,\scalebox{0.8}{\mbox{sup}}}=\alpha_2$: let $\Delta_{\zeta}(s)$ be the factor of the functional equation for the Riemann zeta-function. We recall that 
\begin{equation}\label{tools}
\Delta_{\zeta}(s)\Delta_{\zeta}(1-s) = 1 \quad\mbox{and}\quad \overline{\Delta_{\zeta}(\overline{s})} = \Delta_{\zeta}(s)\quad \mbox{for }s\in\C 
\end{equation}
and observe that $\Delta_{\zeta}(s)$ defines an analytic non-vanishing function in the half-strip $\mathcal{D}$. Let the square-root function $\Delta_{\zeta}(s)^{1/2}$ be defined according to Section \ref{sec:Delta} and let $\alpha_1,\alpha_2\in[0,\infty)$ with $\alpha_1\leq \alpha_2$. Suppose that $\pmb{\alpha}:=(\alpha_1,\alpha_2)\neq (0,0)$. Then, by means of \eqref{tools}, we verify easily that
$$
G_{\pmb{\alpha}}(s):=\left(\frac{\alpha_1+\alpha_2}{2} + \frac{\alpha_2-\alpha_1}{2}\cdot\sin\left(-i(s-\tfrac{1}{2}) \right) \right) \Delta_{\zeta}(s)^{1/2}
$$ 
defines a function in $\mathcal{G}$ which fulfills the functional equation $G_{\pmb{\alpha}}(s)=\Delta_{\zeta}(s)\overline{G_{\pmb{\alpha}}(1-\overline{s})}$ and satisfies $\alpha_{G_{\pmb{\alpha}},\scalebox{0.8}{\mbox{inf}}}=\alpha_1$ and $\alpha_{G_{\pmb{\alpha}},\scalebox{0.8}{\mbox{sup}}}=\alpha_2$. If $\pmb{\alpha}:=(\alpha_1,\alpha_2)=(0,0)$, then the function 
$$
G_{\pmb{0}}(s):= \exp(-s(1-s)) \Delta_{\zeta}(s)^{1/2}
$$
lies in $\mathcal{G}$ and satisfies $\alpha_{G_{\pmb{0}},\scalebox{0.8}{\mbox{inf}}}=\alpha_{G_{\pmb{0}},\scalebox{0.8}{\mbox{sup}}}=0$. We observe further that, uniformly for $0<\sigma<1$, as $t\rightarrow\infty$,
\begin{equation}\label{eq:Galphaabsch}
G_{\pmb{0}}(\sigma+it) \ll \exp\left(-Ct^2 \right)
\end{equation}
with some constant $C>0$. Let $\mu:[2,\infty)\rightarrow\R^+$ be a positive function with $\mu(\tau)< \frac{1}{2}\log\tau$ for $\tau\geq 2$ and set $G_{\pmb{0},\tau}(z):=G(\frac{1}{2}+\frac{\mu(\tau)}{\log\tau}z+i\tau)$ for $\tau\geq 2$ and $z\in\D$. Then, \eqref{eq:Galphaabsch} implies that, for every sequence $(\tau_k)_k$ of real numbers $\tau_k\in[2,\infty)$ with $\lim_{k\rightarrow\infty}\tau_k = \infty$, the corresponding sequence $(G_{\pmb{0},\tau_k})_k$ converges locally uniformly on $\D$ to the function $g\equiv 0$. This yields a very first example that convergence of the limiting process in Proposition \ref{mainprop} is actually possible.\par

With a little more effort we can adjust the function $G_{\pmb{\alpha}}$ above such that analogous results hold if we admit $\alpha_1,\alpha_2\in[0,\infty)\cup\{\infty\}$.

\subsection{Convergence via the growth behaviour in \texorpdfstring{$\Sc^{\#}$}{}  }
\label{sec:orderofgrowth}

For $\L\in\Sc^{\#}$, we can enforce the limiting process introduced in the Section \ref{sec:shiftingshrinking} to converge by adjusting the underlying conformal mapping $\varphi_{\tau}$ according to the growth behaviour of $\L$.\par

{\bf Dirichlet series of finite order.} We say that a Dirichlet series $A(s)$, resp. its meromorphic continuation which we suppose to have only finitely many poles, is of finite order in the strip $-\infty \leq \sigma_1 \leq \sigma \leq \sigma_2 \leq \infty$, if there exists a non-negative real number $c$ such that for all $\sigma_1\leq \sigma \leq \sigma_2$
\begin{equation}\label{finiteorder}
A(\sigma+it) \ll |t|^c \qquad \mbox{ as } |t|\rightarrow\infty .
\end{equation}
For a given $\sigma_1 \leq \sigma \leq \sigma_2$, we define $\theta_A(\sigma)$ to be the infimum of all $c\geq 0$ such that \eqref{finiteorder} holds; see Steuding \cite[Chapt. 2.1]{steuding:2007} and Titchmarsh \cite[\S 9.4]{titchmarsh:1939}.  It follows from a Phragm\'{e}n-Lindel\"of type argument that the function $\theta_A(\sigma)$ is continuous, non-decreasing and convex-downwards; see Titchmarsh \cite[\S 9.41]{titchmarsh:1939}.\par

{\bf The growth behaviour in the extended Selberg class.} Let $\L\in\Sc^{\#}$. Then, $\L$ is of finite order in every strip $-\infty < \sigma_1 \leq \sigma \leq \sigma_2 < \infty$. By the absolute convergence of $\L\in\Sc^{\#}$ in $\sigma>1$, we have 
$$
\theta_{\L}(\sigma)=0 \qquad  \mbox{for} \qquad \sigma>1.
$$ 
It follows then basically from the functional equation together with the asymptotic estimate $\Delta_{\L}(s)\asymp \left(|t|/2\pi  \right)^{d_{\L}(1/2 - \sigma)}$, as $|t|\rightarrow\infty$, that 
$$
\theta_{\L}(\sigma)=d_{\L}\left(\tfrac{1}{2}-\sigma \right) \qquad  \mbox{for} \qquad \sigma<0.
$$ 
The convexity of the function $\theta_{\L}(\sigma)$ implies that
$$
\max\left\{0,\, d_{\L}\left(\tfrac{1}{2}-\sigma \right) \right\}\leq \theta_{\L}(\sigma)\leq  \frac{d_{\L}}{2}\left(1- \sigma \right) \qquad \mbox{for} \qquad 0\leq \sigma \leq 1.
$$ 
For the peculiar case $\sigma=\frac{1}{2}$, this yields the bounds $0\leq \theta_{\L}(\frac{1}{2})\leq \frac{d_{\L}}{4}$. According to the Lindel\"of hypothesis, we expect that $\theta_{\L}(\frac{1}{2})=0$. \par

{\bf The growth-behaviour of the derivatives.} For a non-negative integer $\ell$, let $\L^{(\ell)}$ denote the $\ell$-th derivative of $\L\in\Sc^{\#}$. Then, Cauchy's integral formula, applied to discs with center $\sigma+it$ and radius $1/\log |t|$, assures that, as $|t|\rightarrow\infty$,
\begin{equation}\label{eq:finiteorderderivative}
\L^{(\ell)}(\sigma+it) \ll_{\ell, \eps} |t|^{\theta_{\L}(\sigma) + \eps}
\end{equation}
with any $\eps>0$. Thus, $\L$ and $\L^{(\ell)}$ have the same order of growth.\par 

{\bf Convergence of the limiting process.} First, we establish the following lemma which results from a `smoothness' argument.
\begin{lemma}\label{lem:growth}
Let $\L\in\Sc^{\#}$, $\ell\in\N_0$ and $\delta>0$. For $\tau\geq 2$, let $D_{\tau}$ be the disc defined by
\begin{equation}\label{def:disclemgrowth}
|s-\tfrac{1}{2}-i\tau| < \tau^{-\theta_{\L}(\frac{1}{2})-\delta}.
\end{equation}
Then, as $\tau\rightarrow\infty$,
$$
\max_{s_1,s_2 \in D_{\tau}} \left|\L^{(\ell)}(s_1) - \L^{(\ell)}(s_2) \right| \ll_{\delta, \ell} \tau^{-\delta/2}.
$$
\end{lemma}
\begin{proof}
Let $\L\in\Sc^{\#}$, $\ell\in\N_0$ and $\delta>0$. By \eqref{eq:finiteorderderivative} and the continuity of the function $\theta_{\L}(\sigma)$, the estimate
\begin{equation}\label{eq:sebicu}
 \left| \L^{(\ell + 1)} (s)\right| \ll_{\delta,\ell} \tau^{\theta_{\L}(\sigma) + \delta/2}
\end{equation}
holds uniformly for $s\in D_{\tau}$, as $\tau\rightarrow\infty$. For $s_1,s_2\in D_{\tau}$, we denote the line segment connecting $s_1$ and $s_2$ by $[s_1,s_2]$. By \eqref{eq:sebicu} and a trivial estimation, we obtain that, uniformly for $s_1,s_2\in D_{\tau}$, as $\tau\rightarrow\infty$,
$$
\left|\L^{(\ell)}(s_1) - \L^{(\ell)}(s_2) \right| = \left|\int_{[s_1,s_2]} \L^{(\ell + 1)}(s)\d s \right| \ll_{\delta,\ell} \ \tau^{-\theta_{\L}(\frac{1}{2})-\delta} \ \cdot \ \tau^{\theta_{\L}(\sigma) + \delta/2} = \tau^{-\delta/2}.
$$
This proves the lemma.
\end{proof}

Now, we are ready to give a very first convergent statement for the limiting process introduced in Section \ref{sec:shiftingshrinking}. 

\begin{theorem}\label{th:growth-conceptofuniv}
Let $\L\in\Sc^{\#}$ and $\delta>0$. For $\tau\in[2,\infty)$ and $z\in\D$, let $$\varphi_{\tau}(z):=\tfrac{1}{2}+\tau^{-\theta_{\L}(\frac{1}{2})-\delta}\cdot z + i\tau \qquad \mbox{and}\qquad \L_{\tau}(z):=\L(\varphi_{\tau}(z)).$$
\begin{itemize}
 \item[(a)] Suppose that there is an $\alpha\in[0,\infty)$ and a sequence $(\tau_k)_k$ of real numbers $\tau_k\in[2,\infty)$ with $\lim_{k\rightarrow\infty}\tau_k = \infty$ such that
\begin{equation}\label{9o}
\lim_{k\rightarrow\infty} \left|\L(\tfrac{1}{2}+i\tau_k)\right| = \alpha.
\end{equation}
Then, there is an $a\in\C$ with $|a|=\alpha$ and a subsequence of $(\L_{\tau_{k}})_k$ which converges locally uniformly on $\D$ to $g\equiv a$.
 \item[(b)] Suppose that there is a sequence $(\tau_k)_k$ of real numbers $\tau_k\in[2,\infty)$ with $\lim_{k\rightarrow\infty}\tau_k = \infty$ such that
\begin{equation}\label{9oo}
\lim_{k\rightarrow\infty} \left|\L(\tfrac{1}{2}+i\tau_k)\right| = \infty.
\end{equation}
Then, there is a subsequence of $(\L_{\tau_{k}})_k$ which converges locally uniformly on $\D$ to $g\equiv \infty$.
\end{itemize}
\end{theorem}

\begin{proof}
We observe that $\varphi_{\tau_k}(\D) = D_{\tau_k}$, where $D_{\tau_k}$ is defined by \eqref{def:disclemgrowth} in Lemma \ref{lem:growth}. Suppose that there exists a sequence $(\tau_k)_k$ of real numbers $\tau_k\in[2,\infty)$ with $\lim_{k\rightarrow\infty}\tau_k = \infty$ such that
$$
\lim_{k\rightarrow\infty} \left|\L(\tfrac{1}{2}+i\tau_k)\right| = \infty,\qquad
\mbox{resp. }\lim_{k\rightarrow\infty} \left|\L(\tfrac{1}{2}+i\tau_k)\right| = 0.
$$
Then, Lemma \ref{lem:growth} yields immediately that $(\L_{\tau_k})_k$ converges locally uniformly on $\D$ to $g\equiv \infty$, resp. $g\equiv 0$. Let $(\tau_k)_k$ be a sequence of real numbers $\tau_k\in[2,\infty)$ with $\lim_{k\rightarrow\infty}\tau_k = \infty$ such that
\begin{equation}\label{limk}
\lim_{k\rightarrow\infty} \left|\L(\tfrac{1}{2}+i\tau_k)\right| = \alpha
\end{equation}
with some $\alpha\in(0,\infty)$. Lemma \ref{lem:growth} assures that, for sufficiently large $k$, 
$$
\left|\L_{\tau_k}(z)\right|< \alpha + 1, \qquad \mbox{for }z\in\D.
$$
Thus, the family $\{\L_{\tau_k}\}_k$ is bounded on $\D$ and, consequently, by Montel's theorem, normal in $\D$. This means that there is a subsequence $(\tau_{k_j})_j$ of $(\tau_k)_k$ such that $(\L_{\tau_{k_j}})_j$ converges locally uniformly on $\D$ to an analytic function $g$. Due to \eqref{limk} we have 
$$
g(0)= \lim_{j\rightarrow\infty} \L_{\tau_{k_j}}(0) =\lim_{j\rightarrow\infty} \L(\tfrac{1}{2}+i\tau_{k_j}) = a
$$
with some $a\in\C$ satisfying $|a|=\alpha$. From Lemma \ref{lem:growth} we deduce that 
$$
g(z)= \lim_{j\rightarrow\infty} \L_{\tau_{k_j}}(z) = a
$$
for $z\in\D$. Hence, $g\equiv a$ on $\D$. 
\end{proof}

With confinements on the scaling factor of the mapping $\varphi_{\tau}$, an analogous statement of Theorem \ref{th:growth-conceptofuniv} can be made for every function $G\in\mathcal{G}$ by means of certain continuity arguments for which we refer the reader to Christ \cite[Lemma 2]{christ:2012}.\par

In view of the conditions \eqref{9o} and \eqref{9oo} in Theorem \ref{th:growth-conceptofuniv}, it seems again worth to determine, for given function $G\in\mathcal{G}$, the quantities $\alpha_{G,\scalebox{0.8}{\mbox{inf}}}$ and $\alpha_{G,\scalebox{0.8}{\mbox{sup}}} $ defined by \eqref{def:alpha_inf}.

\subsection{Convergence via the a-point-distribution in \texorpdfstring{$\Sc^{\#}_R$}{} } \label{subsec:convapoints}
\begin{figure}\centering
\includegraphics[width=0.9\textwidth]{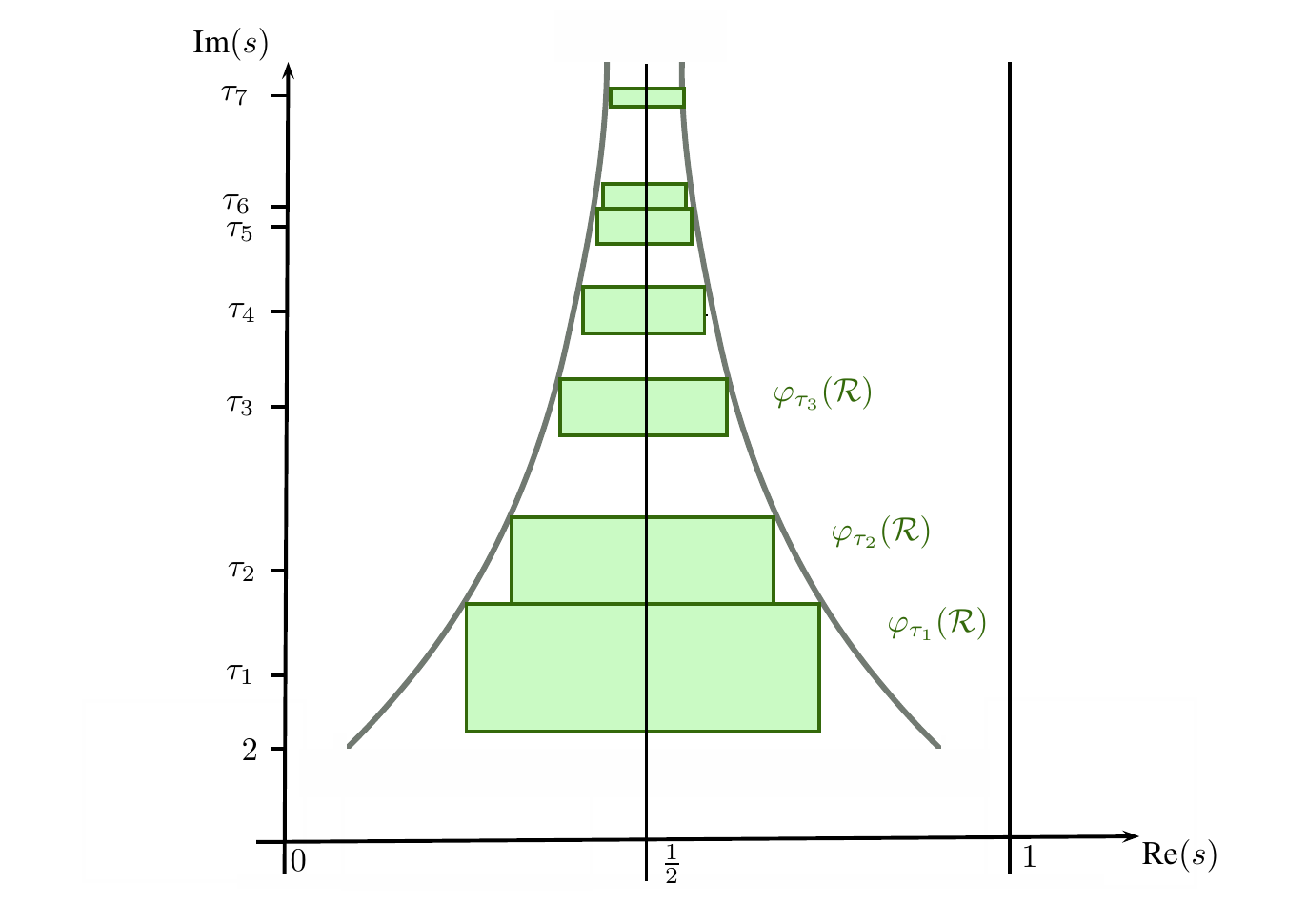}
\caption{The images of the rectangular domain $\mathcal{R}:=\mathcal{R}(3,1)$ under a certain conformal mapping $\varphi_{\tau}$ for some $\tau\in[2,\infty)$.}
\label{fig:phirect} 
\end{figure}

There is a further possibility to enforce convergence of the limiting process introduced in Section \ref{sec:shiftingshrinking}, namely by adjusting the underlying conformal mapping $\varphi_{\tau}$ according to the $a$-point-distribution of $\L\in\mathcal{S}^{\#}$.\par

We shall study the $a$-point-distribution of functions in the extended Selberg class in details later on in Chapter \ref{chapt:apoints}. Here, we anticipate only some very basic observations. Let $\L\in\Sc^{\#}$ and $a\in\C$. A complex number $\rho_a$ is said to be an $a$-point of $\L$ if $\L(\rho_a)=a$. We distinguish between trivial and non-trivial $a$-points. For a given function $\L\in\Sc^{\#}$, we can fix real numbers $R_a> 1$ and $L<0$ such that there are no $a$-points of $\L$ in the half-plane $\sigma>R_a$, only trivial ones in the half-plane $\sigma< L$ and only non-trivial ones in the strip $L \leq \sigma\leq R_a$. \par

Steuding \cite[Theorem 7.7]{steuding:2007} established a Riemann von-Mangoldt formula for a rather large subclass of the extended Selberg class. Let $\Sc^{\#}_R$ be the set of all functions $\L\in\Sc^{\#}$ with $d_{\L}>0$ which satisfy the Ramanujan hypothesis; see Section \ref{sec:selbergclass}. Obviously, we have
$$
\Sc\setminus\{1\} \subset  \Sc^{\#}_R\subset \Sc^{\#}.
$$ 
For given $\L\in\Sc^{\#}_R$, let $N_a(T)$ denote the number of non-trivial $a$-points with imaginary part $0<\gamma_a\leq T$. Then, according to Steuding \cite[Theorem 7.7]{steuding:2007},   
\begin{equation}\label{eq:RvMextS}
N_{a}(T) \sim  \frac{d_{\L}}{2\pi} T \log T + O(T), \qquad\mbox{ as $T\rightarrow\infty$.}
\end{equation}

For a subset $X\subset \R^+$, we define a density function by
$$
\nu_{T}(X) := \frac{1}{T} \meas \left( X \cap (T, 2T] \right),\qquad T>0.
$$
Further, let $\mathcal{R}(x,y)$ with $x,y> 0$ denote the rectangular domain defined by the vertices $\pm x\pm iy$. Now, we are ready to prove the following convergence statement.
\begin{theorem}\label{th:convergenceviaapoints}
Let $\L\in\Sc^{\#}_R$ and $a,b\in\C$ with $a\neq b$. Let $c>0$ and $\mathcal{R}:=\mathcal{R}(c,1)$. Further, let $\mu:[2,\infty) \rightarrow\R^+$ be monotonically decreasing such that
$$
\lim_{\tau\rightarrow\infty}\mu(\tau) < \kappa_{\L} \qquad \mbox{with}\qquad \kappa_{\L}:= \frac{\pi }{2 d_{\L}}.
$$
For $\tau\geq 2$ and $z\in\mathcal{R}$, let $\varphi_{\tau}(z):= \frac{1}{2}+\frac{\mu(\tau)}{\log \tau} z + i\tau$ and $\L_{\tau}(z):= \L\left(\varphi_{\tau}(z)\right)$.
Then, for every $0<q<1$ with $\lim_{\tau\rightarrow\infty}\mu(\tau) < \kappa_{\L} q$, there exists a subset $\mathcal{W}\subset[2,\infty)$ with 
$$
\liminf_{T\rightarrow\infty}\ \nu_T(\mathcal{W})\geq 1-q
$$
such that the family $\mathcal{F}:=\{\L_{\tau}\}_{\tau\in\mathcal{W}}\subset\mathcal{H}(\mathcal{R})$ omits the values $a$ and $b$ on $\mathcal{R}$. In particular, $\mathcal{F}$ is normal in $\mathcal{R}$.
\end{theorem}
\begin{proof}
Let $\L\in\Sc^{\#}_R$ and let $\mu:[2,\infty)\rightarrow\R^+$ satisfy the conditions of Theorem \ref{th:convergenceviaapoints}. For any $a\in\C$, let $\Gamma_{a}$ denote the set of all imaginary parts $\gamma_a$ of $a$-points $\rho_a = \beta_a + i\gamma_a$ of $\L$ which lie in the region defined by
\begin{equation}\label{apointstrip}
\tfrac{1}{2}-\frac{\mu(t-1)}{\log (t-1)} \leq \sigma \leq \tfrac{1}{2}+\frac{\mu(t-1)}{\log (t-1)} , \qquad t> 2.
\end{equation}
Furthermore, for any two distinct $a,b\in\C$, we set $\Gamma_{a,b} := \Gamma_a \cup \Gamma_b$ and denote the number of elements $\gamma\in\Gamma_{a,b}$ with $0<\gamma \leq T$ by $N_{\Gamma_{a,b}}(T)$.\par
Now, we fix two distinct $a,b\in\C$ as required in the assumptions of Theorem \ref{th:convergenceviaapoints}. Certainly, the inequality 
$$
N_{\Gamma_{a,b}}(T) \leq N_a(T) + N_b(T).
$$
holds for $T>0$; here, $N_a(T)$ denotes the number of non-trivial $a$-points of $\L$ with imaginary part $0<\gamma_a \leq T$. The Riemann-von Mangoldt formula \eqref{eq:RvMextS} for functions in $\Sc_R^{\#}$ yields that, as $T\rightarrow\infty$,
\begin{equation}\label{NGamma-1}
N_{\Gamma_{a,b}}(T)\leq  \frac{d_{\L}}{\pi} T \log T + o(T\log T). 
\end{equation}
This asymptotic bound for $N_{\Gamma_{a,b}}$ seems to be very rough. However, as we shall briefly discuss at the end of this section, in most cases this bound might not be too far from the truth. \par
For $\gamma\in\R$, we set $\gamma':= \gamma-1$ and define
$$
\Lambda := \bigcup_{\gamma\in\Gamma_{a,b}} \left[\gamma - \frac{\mu(\gamma')}{\log \gamma'},\ \gamma + \frac{\mu(\gamma')}{\log \gamma'} \right];
$$
By means of \eqref{NGamma-1}, we obtain that
$$
\nu_T(\Lambda) \leq \frac{1}{T}\cdot \frac{2\mu(T-1)}{\log (T-1)} \cdot N_{\Gamma_{a,b}}(T) \leq \frac{2d_{\L}}{\pi}\cdot \mu(T-1) + o(1),
$$
as $T\rightarrow\infty$. Let $0<q<1$ be such that $\lim_{\tau\rightarrow\infty} \mu(\tau) < \kappa_{\L}q$. Then, we have
\begin{equation}\label{eq:limsupconvergenceapoints}
\limsup_{T\rightarrow\infty} \nu_T(\Lambda) \leq q .
\end{equation}
We choose $\tau_0\geq 2$ large enough such that $0<\frac{\mu(\tau)}{\log \tau}<\frac{1}{2}$ holds for $\tau\geq \tau_0$ and set $\mathcal{W}:= [\tau_0,\infty) \setminus \Lambda$. It follows immediately from \eqref{eq:limsupconvergenceapoints} that
\begin{equation}\label{sel:measA}
\liminf_{T\rightarrow\infty}\nu_T( \mathcal{W}) \geq 1-q.
\end{equation}
By the construction of the set $\mathcal{W}$ and the monotonicity of $\mu$, the family 
$
\{\L_{\tau}\}_{\tau\in\mathcal{W}}
$
omits the values $a$ and $b$ on $\mathcal{R}$. Our choice of $\tau_0$ assures that a possible pole of $\L$ at $s=1$ does not generate a pole for any function $\L_{\tau}$ with $\tau\in[\tau_0,\infty)$ on $\mathcal{R}$. Thus, the functions $\L_{\tau}$ with $\tau\in[\tau_0,\infty)$ are analytic on $\mathcal{R}$. By Montel's fundamental normality test (Theorem \ref{th:FNT1}) we conclude that the family $\{\L_{\tau}\}_{\tau\in\mathcal{W}}$ is normal in $\mathcal{R}$. 
\end{proof}
It would be interesting to know whether one can replace the constant $\kappa_{\L}$ in Theorem \ref{th:convergenceviaapoints} by a larger one. And, in fact, there are two estimates in our proof which appear to be quite rough.

\begin{itemize}
\item[(i)] We used the very rough bound 
\begin{equation*}\label{sebastian}
N_{\Gamma_{a,b}}(T) \leq N_a(T) + N_b(T).
\end{equation*}
We could improve this bound if we knew more about the $a$-point-distribution of $\L$ in the funnel-shaped strip \eqref{apointstrip}. As we shall see in Chapter \ref{chapt:apoints}, we expect that, for $\L\in\Sc^{\#}_R$, almost all of its $a$-points lie arbitrarily close to the critical line. However, it seems to be difficult to get detailed information on how the $a$-points cluster around the critical line. Due to works of Levinson \cite{levinson:1975}, Selberg \cite{selberg:1992} and Tsang \cite{tsang:1984}, there is slightly more information at our disposal if a function $\L\in\Sc$ possesses a rich arithmetical structure. But, in most cases, our knowledge is not sufficient to estimate the number of $a$- and $b$-points in the domain \eqref{apointstrip} better than above. If we restrict to the Riemann zeta-function and assume the Riemann hypothesis, then we deduce from a conditional $a$-point result of Selberg (see Section \ref{apointslittlewood}) that, for $a\neq 0$, about half of the $a$-points lie outside the domain \eqref{apointstrip}, provided that $\lim_{t\rightarrow\infty}\mu(t)>0$. This allows us to replace the constant $\kappa_{\L}$ in Theorem \ref{th:convergenceviaapoints} by $2\kappa_{\L}$, provided that $a,b\neq 0$.

\item[(ii)] In bounding $\nu_{T}(\Lambda)$ we did not respect that there might be quite many intervals
$$
\left[\gamma - \frac{\mu(\gamma')}{\log \gamma'},\ \gamma + \frac{\mu(\gamma')}{\log \gamma'} \right]\qquad \mbox{with }\gamma':=\gamma-1\;\mbox{and}\;\gamma\in\Gamma_{a,b}
$$
which overlap. However, to deal with this overlapping seems to be out of reach. We refer here to the many obstacles that prevent us from getting control over the gap conjecture and Montgomery's pair correlation conjecture in the case of the Riemann zeta-function.
\end{itemize} 

\newpage

\chapter{Small and Large values near the critical line} \label{ch:smalllarge}
For functions $G\in\mathcal{G}$, we introduced in Chapter \ref{ch:conceptsuniv} a limiting process in neighbourhoods of the critical line and found out that the functional equation has strong effects on the shape of possible limit functions. In Theorem \ref{th:growth-conceptofuniv} and \ref{th:convergenceviaapoints} we investigated two natural mechanisms by which we can enforce the limiting process to converge. For a given $\L\in\Sc^{\#}$, the limit functions to be obtained by such a convergent process are connected with the quantities
$$
\alpha_{\L,\scalebox{0.8}{\mbox{inf}}}:=\liminf_{\tau\rightarrow\infty} \left|\L(\tfrac{1}{2}+i\tau) \right|\qquad \mbox{and}\qquad
\alpha_{\L,\scalebox{0.8}{\mbox{sup}}}:=\limsup_{\tau\rightarrow\infty} \left|\L(\tfrac{1}{2}+i\tau) \right|.
$$
It appears to be quite challenging to determine $\alpha_{\L,\scalebox{0.8}{\mbox{inf}}}$ and $\alpha_{\L,\scalebox{0.8}{\mbox{sup}}}$ for general functions $\L\in\Sc^{\#}$. We postpone this problem to the end of this chapter and tackle it in Section \ref{sec:unboundedness}.\par 
If a function $\L\in\Sc$ has a sufficiently rich arithmetic structure in its Dirichlet series coefficients, Selberg's central limit implies that $\alpha_{\L,\scalebox{0.8}{\mbox{inf}}}=0$ and $\alpha_{\L,\scalebox{0.8}{\mbox{sup}}}=\infty$ and, thus, provides a more satisfactory answer than we can state for general functions $\L\in\Sc^{\#}$. We present Selberg's central limit law and several of its extensions in Section \ref{sec:selbergcentrallimit}. 
For suitable functions $\L\in\Sc$, we deduce from Selberg's central limit law in Section \ref{sec:smalllargeONcritline} information on the frequency of small and large values on the critical line.\par
In Section \ref{sec:largesmall}, we discover that Selberg's central limit law implies that, for suitable functions $\L\in\Sc$, the limiting process of Theorem \ref{th:growth-conceptofuniv} and \ref{th:convergenceviaapoints} has a strong tendency to converge either to $g\equiv 0$ or to $g\equiv \infty$.

\section{Selberg's central limit law} \label{sec:selbergcentrallimit}
In the class $\Sc^*$ we gather all functions from the Selberg class for which both Selberg's prime coefficient condition (S.6$^*$) and Selberg's zero-density estimate (DH) are true; see Section \ref{sec:selbergclass}. \par

{\bf Selberg's central limit law.} Selberg \cite{selberg:1992} derived that, for $\L\in\Sc^*$, the values of $\log \L(\frac{1}{2} + it)$ are Gaussian normally distributed after some suitable normalization. We set
$$
\kappa_{\L,T}(\sigma,t):= \frac{\log \L(\sigma+it)}{\sqrt{\frac{1}{2}n_{\L} \log \log T}},
$$ 
where $n_{\L}$ is defined by Selberg's prime coefficient condition (S.6$^*$).
Then, for any measurable set $B\subset \C$ with positive Jordan content, we have  
\begin{equation*}
\frac{1}{T}\ \meas \left\{ t\in (T,2T]\ : \ \kappa_{\L,T}(\tfrac{1}{2},t) \in B \right\} \sim \frac{1}{2\pi} \iint_{B} e^{-\frac{1}{2}(x^2+y^2)}\d x \d y,
\end{equation*}
as $T\rightarrow\infty$. Note that $\varphi(x,y):= \frac{1}{2\pi}e^{-\frac{1}{2}(x^2+y^2)}$ defines the density function of the bivariate Gaussian normal distribution. Moreover, for any real numbers $\alpha$ and $\beta$ with $\alpha<\beta$, as $T\rightarrow\infty$,
\begin{equation}\label{centrallimitlaw}
\frac{1}{T}\ \meas \left\{ t\in (T,2T]\ : \  \alpha \leq  \Re \left(\kappa_{\L,T}(\tfrac{1}{2},t)\right) \leq  \beta  \right\} \qquad\qquad\qquad\mbox{ }
\end{equation}
$$\qquad\qquad\qquad\qquad\qquad\qquad
= \frac{1}{\sqrt{2\pi}}  \int_{\alpha}^{\beta}e^{-\frac{1}{2}x^2 } \d x + O\left( \frac{(\log\log\log T)^2}{\sqrt{\log\log T}} \right)
$$
and, similarly, with a slightly better error term,
$$
\frac{1}{T}\ \meas \left\{ t\in (T,2T]\ : \  \alpha \leq  \Im\left(\kappa_{\L,T}(\tfrac{1}{2},t)\right)\leq \beta \right\}  \qquad\qquad\qquad\mbox{ }
$$
$$\qquad\qquad\qquad\qquad\qquad\qquad
= \frac{1}{\sqrt{2\pi}}  \int_{\alpha}^{\beta}e^{-\frac{1}{2}x^2} \d x + O\left( \frac{\log\log\log T}{\sqrt{\log\log T}} \right).
$$

In the case of the Riemann zeta-function, the asymptotic of Selberg's limit law for $\Re \left( \kappa_{\zeta,T}(\frac{1}{2},t)\right)$ was also discovered by Laurin\v{c}ikas \cite{laurincikas:1987}, independently of Selberg's work; however, without explicit error term.\footnote{We refer to Ivi\'{c} \cite{ivic:2002} for a short historical overview on preliminary works leading to Selberg's limit law.} Selberg himself never published a rigorous proof of his limit theorems. For a precise description of Selberg's method, we refer to Tsang \cite{tsang:1984}, who carried out all details in the case of the Riemann zeta-function. Joyner \cite{joyner:1986} proved Selberg's central limit law for a large class of Dirichlet series. Laurin\v{c}ikas \cite{laurincikas:1991-2} provided proofs for various limit laws connected with the Riemann zeta-function. Hejhal \cite{hejhal:2000} sketches the proof of Selberg's central limit law for linear combinations of functions in the class $\Sc^*$ with polynomial Euler product (S.3$^*$).\par

{\bf Extensions of Selberg's central limit law.} There are several directions to extend Selberg's central limit law. In the following, we mainly restrict to the Riemann zeta-function and to limit laws with respect to $\Re\left(\kappa_{\zeta,T}(\frac{1}{2},t) \right)$. \par 

Laurin\v{c}ikas proved that, for the Riemann zeta-function, Selberg's central limit theorem is also valid on certain line segments to the right of the critical line. Let
$$
0\leq \epsilon(T)\leq \frac{\mu(T)\sqrt{\log\log T}}{\log T}, \qquad T\geq 2,
$$
with an arbitrary positive function $\mu$ satisfying $ \mu(T) \rightarrow \infty$ and $\mu(T)=o(\log\log T)$, as $T\rightarrow\infty$. Then, 
$$
\frac{1}{T}\ \meas \left\{ t\in (T,2T]\ : \ \alpha \leq  \Re \left(\kappa_{\zeta,T}(\tfrac{1}{2}+\epsilon(T),t)\right) \leq \beta  \right\} \sim
\frac{1}{\sqrt{2\pi}}  \int_{\alpha}^{\beta}e^{-\frac{1}{2}x^2} \d x ,
$$
as $T\rightarrow\infty$; see Laurin\v{c}ikas \cite[Chapt. 3, Theorem 3.5.1]{laurincikas:1991-2}. It is also possible to obtain normal distribution results if 
$$
\epsilon(T)>\frac{\mu(T)\sqrt{\log\log T}}{\log T}, \qquad T\geq 2.
$$
In this case, however, a change of normalization is necessary: one has to work with
$$
\kappa'_{\zeta,T} (\tfrac{1}{2} +\epsilon(T) , t) : = \frac{\log| \zeta(\tfrac{1}{2}+\epsilon(T)+i t)|}{\sqrt{-\log\epsilon(T)}}.
$$
For details we refer to Laurin\v{c}ikas \cite[Chapt. 3.4, Theorem 3.4.1 and Corollary 3.4.2]{laurincikas:1991-2}.\par

Under the assumption of the Riemann hypothesis, Hejhal \cite{hejhal:1989} established a central limit law for the modulus of the first derivative of the Riemann zeta-function on the critical line. Let $A(t):=(t/2\pi)\log (t/2\pi e)$. If the Riemann hypothesis is true, then
$$
\frac{1}{T} \meas \left\{ t\in (T,2T]\ : \ \alpha \leq  
\frac{\log\left|\zeta'(\frac{1}{2}+it)/A'(\frac{1}{2}+it)\right|}{\sqrt{\frac{1}{2} \log\log T}}
\leq \beta  \right\} \sim
\frac{1}{\sqrt{2\pi}}  \int_{\alpha}^{\beta}e^{-\frac{1}{2}x^2} \d x,
$$
as $T\rightarrow\infty$.\par

Hughes, Nikeghbali \& Yor \cite{hughesnikeghbaleyour:2008} and Bourgade \cite{bourgade:2010} gave multidimensional extensions of Selberg' central limit law. In the following, let $\omega$ be a random variable which is uniformly distributed on the interval $(0,1)$. In the terminology of probability theory, Selberg's central limit law states that
$$
\frac{\log| \zeta(\tfrac{1}{2}+i\omega t)|}{\sqrt{\log\log t}}
$$
converges in distribution, as $t\rightarrow\infty$, to a Gaussian normally-distributed random variable. Hughes, Nikeghbali \& Yor \cite{hughesnikeghbaleyour:2008} gave a multidimensional extension of Selberg's central limit law by showing that, for any $0<\lambda_1<...<\lambda_n$, the vector
$$
\frac{1}{\sqrt{\log\log t}} \left(\log| \zeta(\tfrac{1}{2}+i\omega e^{(\log t)^{\lambda_1}}),...,\log| \zeta(\tfrac{1}{2}+i\omega e^{(\log t)^{\lambda_n}})| \right)
$$
converges in distribution, as $t\rightarrow\infty$, to $(\lambda_1 \mathcal{N}_1, ..., \lambda_n \mathcal{N}_n)$, where $\mathcal{N}_1,...,\mathcal{N}_n$ are independent Gaussian normally-distributed random variables. Bourgade \cite{bourgade:2010} investigated vectors with respect to smaller shifts and revealed some interesting correlation structure. His results imply, for instance, that, for any $0\leq \delta\leq 1$, 
$$
\frac{1}{\sqrt{\log\log t}} \left(\log| \zeta(\tfrac{1}{2}+i\omega t),\log\left| \zeta\left(\tfrac{1}{2}+i\omega t +i\tfrac{1}{(\log t)^{\delta}}\right)\right| \right)
$$
converges in distribution, as $t\rightarrow\infty$, to $(\mathcal{N}_1 , \delta \mathcal{N}_1  + \sqrt{1-\delta^2}\mathcal{N}_2)$, where $\mathcal{N}_1$ and $\mathcal{N}_2$ are independent Gaussian normally-distributed random variables.\par

\section{Small and large values on the critical line}\label{sec:smalllargeONcritline}
Selberg's central limit law provides information on the frequency of small and large values of the Riemann zeta-function on the critical line. \par
For $\L\in\Sc$ and any two positive real functions $l,u:[2,\infty)\rightarrow\R^+$ satisfying $l(t)\leq u(t)$ for $t\in[2,\infty)$, we define
\begin{equation*}\label{W_lu}
W_{l(t),u(t)} := \left\{ t\in [2,\infty)\ : \  l(t) \leq |\L(\tfrac{1}{2}+it)| \leq  u(t) \right\}.
\end{equation*}
In consistency with this notation, we set
\begin{align*}
W_{-\infty,u(t)} &:= \left\{ t\in [2,\infty)\ : \   |\L(\tfrac{1}{2}+it)| \leq  u(t) \right\},\\
W_{l(t),+\infty} &:= \left\{ t\in [2,\infty)\ : \   |\L(\tfrac{1}{2}+it)| \geq  l(t) \right\}.
\end{align*}
Recall that, for a given subset $X\subset \R^+$, we defined a density function by
$$
\nu_{T}(X) := \frac{1}{T} \meas \left( X \cap (T, 2T] \right),\qquad T>0.
$$
The subsequent theorem is an immediate consequence of Selberg's central limit law.
\begin{theorem}\label{th:measselberglimitlaw}
Let $\L\in\Sc^{*}$. Furthermore, let
\begin{equation}\label{gx}
g_{u}(t):= \exp\left( u \sqrt{\tfrac{1}{2} n_{\L} \log \log t} \right),\qquad t\geq 2,
\end{equation}
where $u$ is a real parameter and $n_{\L}$ is defined by Selberg's prime coefficient condition (S.3$^*$), and
\begin{equation}\label{ET}
E(T):= \frac{(\log\log\log T)^2}{\sqrt{\log\log T}}, \qquad T>1.
\end{equation}
\begin{itemize}
 \item[(a)] Let $\alpha,\beta\in\R$ with $\alpha<\beta$. Then, as $T\rightarrow\infty$,
\begin{align*}\label{meas:normal1}
\nu_{T}\left(W_{-\infty, g_{\alpha}(t)}\right)  &= \frac{1}{\sqrt{2\pi}}  \int_{-\infty}^{\alpha}e^{-x^2/2 } \d x + O\left( E(T) \right), \\
\nu_{T}\left(W_{g_{\alpha}(t),g_{\beta}(t)}\right)  &= \frac{1}{\sqrt{2\pi}}  \int_{\alpha}^{\beta}e^{-x^2/2 } \d x + O\left( E(T) \right),\\
\nu_{T}\left(W_{g_{\beta}(t),+\infty}\right)  &= \frac{1}{\sqrt{2\pi}}  \int_{\beta}^{+\infty}e^{-x^2/2 } \d x + O\left( E(T)\right).
\end{align*}
\item[(b)] Let $m$ be a fixed positive real number. Then, as $T\rightarrow\infty$,
\begin{align*}
\nu_{T}\left(W_{-\infty,\frac{1}{m}}\right) & = \frac{1}{2}  + O\left( E(T) \right),\\
\nu_{T}\left(W_{\frac{1}{m}, m}\right) & = O\left( E(T) \right),\\
\nu_{T}\left(W_{m,\infty}\right) & = \frac{1}{2}  + O\left( E(T) \right).
\end{align*}
\item[(c)] Let $m:[2,\infty)\rightarrow\R^+$ be a positive function with $\lim_{t\rightarrow\infty}m(t) = \infty$ and such that, for any $\eps>0$, the inequality $m(t)\leq g_{\eps}(t)$ holds for sufficiently large $t$. Then, as $T\rightarrow\infty$,
\begin{align*}
\nu_{T}\left(W_{-\infty,\frac{1}{m(t)}}\right) &= \frac{1}{2}  + o(1),\\
\nu_{T}\left(W_{\frac{1}{m(t)},m(t)}\right) &= o(1),\\
\nu_{T}\left(W_{m(t),\infty}\right) &= \frac{1}{2}  + o(1).
\end{align*}
\end{itemize}
\end{theorem}
\begin{proof} Statement (a) follows immediately from Selberg's central limit law by noticing that $\Re (\log \L(\frac{1}{2}+it)) = \log |\L(\frac{1}{2}+it)|$. The statements (b) and (c) can be deduced from (a) by coupling the parameters $\alpha:=\alpha(T)$ and $\beta:=\beta(T)$ with $T$ in a suitable manner and  by observing that
$$
\frac{1}{\sqrt{2\pi}}  \int_{-\infty}^{0}e^{-x^2/2 } \d x = \frac{1}{\sqrt{2\pi}}  \int_{0}^{\infty}e^{-x^2/2 } \d x =\frac{1}{2}.
$$
\end{proof}
We close this section with a brief remark on large deviations in Selberg's central limit law. Let $\L\in\Sc^*$ and $\alpha:[2,\infty)\rightarrow\R^+$ be a positive function with $\lim_{T\rightarrow\infty}\alpha(T)=\infty$. We set 
$$
l(T):=\exp\left( \alpha(T) \sqrt{\tfrac{1}{2} n_{\L} \log \log T} \right),\qquad T\geq 2.
$$
Selberg's central limit law implies that the asymptotic
 $$
\nu_T (W_{l(T),\infty}) \sim
\frac{1}{\sqrt{2\pi}}  \int_{\alpha(T)}^{\infty}e^{-\frac{1}{2}x^2} \d x, \qquad\mbox{as }T\rightarrow\infty ,
 $$
holds whenever $\alpha(T)\leq (\log\log\log T)^{\frac{1}{2}-\eps}$ for sufficiently large $T$ with an arbitrary fixed $\eps>0$. For larger deviations, i.e. if 
$$
\alpha(T)=\Omega\left( (\log\log\log T)^{\frac{1}{2}} \right), \qquad\mbox{ as } T\rightarrow\infty,
$$ 
we obtain from Selberg's central limit law only the trivial bound
\begin{equation*}\label{123}
\nu_T (W_{l(T),\infty}) \ll E(T), \qquad\mbox{as }T\rightarrow\infty,
\end{equation*}
where $E(T)$ is defined by \eqref{ET}. In the case of the Riemann zeta-function, there are unconditional results due to Jutila \cite{jutila:1983}, Soundararajan \cite{soundararajan:2008} and  Radziwi\l\l \  \cite{radziwill:2011} and conditional results (on the assumption of the Riemann hypothesis) due to Soundararajan \cite{soundararajan:2009} at our disposal which allow to describe the frequency of large deviations in Selberg's central limit law in a more precise manner than we get from the trivial bound \eqref{123}.

\section{Small and large values near the critical line}\label{sec:largesmall}
Relying on Selberg's central limit law, in particular on Theorem \ref{th:measselberglimitlaw}, we shall deduce that, for $\L\in\Sc^{*}$, the limiting processes of Theorem \ref{th:growth-conceptofuniv} and \ref{th:convergenceviaapoints} have a strong tendency to converge either to $g\equiv 0$ or to $g\equiv \infty$. \par

\begin{theorem}\label{th:largesmall}
Let $\L\in\Sc^{*}$ and $a\in\C\setminus\{0\}$. Let $c>0$ and $\mathcal{R}:=\mathcal{R}(c,1)$ be the rectangular domain defined by the vertices $\pm c \pm i$. Let $\mu:[2,\infty) \rightarrow\R^+$ be monotonically decreasing such that
$$
\lim_{\tau\rightarrow\infty}\mu(\tau) < \tfrac{1}{2}\kappa_{\L} \qquad \mbox{with}\qquad \kappa_{\L}:= \frac{\pi }{2 d_{\L}}.
$$
For $\tau\geq 2$ and $z\in\mathcal{R}$, we set $\varphi_{\tau}(z):= \frac{1}{2}+\frac{\mu(\tau)}{\log \tau} z + i\tau$ and $ \L_{\tau}(z):= \L\left(\varphi_{\tau}(z)\right)$. Then, for every $0<q<\frac{1}{2}$ with $\lim_{\tau\rightarrow\infty}\mu(\tau) < \kappa_{\L} q$, there exist subsets $\mathcal{W}_0,\mathcal{W_{\infty}}\subset[2,\infty)$ with 
\begin{align*}
\liminf_{T\rightarrow\infty}\nu_T(\mathcal{W}_0) & \geq \tfrac{1}{2}-q, \\
\liminf_{T\rightarrow\infty}\nu_T(\mathcal{W}_{\infty}) & \geq \tfrac{1}{2}-q,\\
\liminf_{T\rightarrow\infty}\nu_T(\mathcal{W}_0\cup\mathcal{W}_{\infty}) & \geq 1-q,
\end{align*}
such that the following holds: 
\begin{itemize}
\item[(a)] The family $\mathcal{F}:=\{\L_{\tau}\}_{\tau\in\mathcal{W}_0\cup\mathcal{W}_{\infty}}\subset\mathcal{H}(R)$ omits the two values $0$ and $a$ on $\mathcal{R}$. In particular, $\mathcal{F}$ is normal in $\mathcal{R}$.
\item[(b)] For every sequence $(\tau_k)_k$ with $\tau_k \in \mathcal{W}_0$ and $\lim_{k\rightarrow\infty} \tau_k =\infty$, the sequence $(\L_{\tau_k})_k$ converges locally uniformly on $\mathcal{R}$ to $g\equiv 0$. 
\item[(c)] For every sequence $(\tau_k)_k$ with $\tau_k \in \mathcal{W}_0$ and $\lim_{k\rightarrow\infty} \tau_k =\infty$, the sequence $(\L_{\tau_k})_k$ converges locally uniformly on $\mathcal{R}$ to $g\equiv \infty$. 
\end{itemize}
\end{theorem}

\begin{proof}
Let $\L\in\Sc^{*}$ and $a\in\C\setminus\{0\}$. Furthermore, let the function $\mu:[2,\infty)\rightarrow\R^+$ satisfy the conditions of the theorem.\par 
Then, according to Theorem \ref{th:convergenceviaapoints}, we find, for every $0<q<\frac{1}{2}$ with $\lim_{\tau\rightarrow\infty}\mu(\tau) < \kappa_{\L} q$, a subset $\mathcal{W}\subset[2,\infty)$ with 
\begin{equation*}
\liminf_{T\rightarrow\infty} \nu_T(\mathcal{W})\geq 1-q
\end{equation*}
such that every function $\L_{\tau}$ with $\tau\in\mathcal{W}$ is analytic on $\mathcal{R}$ and omits there the values $0$ and $a$. It follows from Montel's fundamental normality test (Theorem \ref{th:FNT1}) that the family $\{\L_{\tau}\}_{\tau\in\mathcal{W}}$ is normal in $\mathcal{R}$. \par
Let $g_x(t)$ be defined by \eqref{gx} and let $m:[2,\infty)\rightarrow\R^+$ be a positive function with $\lim_{t\rightarrow\infty}m(t) = \infty$ such that, for any $\eps>0$, the inequality $m(t)\leq g_{\eps}(t)$ holds for sufficiently large $t$. Furthermore, let  the sets $ W_{-\infty,1/m(t)} , W_{m(t), +\infty}\subset[2,\infty)$ be defined by 
\begin{align*}
W_{-\infty,\frac{1}{m(t)}} &:= \left\{ t\in [2,\infty)\ : \   |\L(\tfrac{1}{2}+it)| \leq  \tfrac{1}{m(t)} \right\},
\shortintertext{and}
W_{m(t),+\infty} &:= \left\{ t\in [2,\infty)\ : \   |\L(\tfrac{1}{2}+it)| \geq  m(t) \right\}.
\end{align*}
According to Theorem \ref{th:measselberglimitlaw} (c),
\begin{align*}
\liminf_{T\rightarrow\infty} \nu_T\left( W_{-\infty, \frac{1}{m(t)}} \right) &\geq \frac{1}{2},\\
 \liminf_{T\rightarrow\infty} \nu_T\left( W_{m(t), +\infty} \right) &\geq \frac{1}{2}
\end{align*}
\vspace{-0.5cm} 
and
$$
\liminf_{T\rightarrow\infty} \nu_T\left( W_{-\infty, \frac{1}{m(t)}}\cup W_{m(t), +\infty} \right) =1.
$$

We set 
$$
\mathcal{W}_0:=\mathcal{W}\cap W_{-\infty, \frac{1}{m(t)}} \qquad \mbox{and}\qquad
\mathcal{W}_{\infty}:= \mathcal{W}\cap  W_{m(t), +\infty}.
$$
The lower density estimates for the sets $\mathcal{W}$, $ W_{-\infty,1/m(t)}$  and $W_{m(t), +\infty}$ above imply that
\begin{align*}
\liminf_{T\rightarrow\infty} \nu_T\left( \mathcal{W}_{0} \right) & \geq \frac{1}{2}-q,\\
 \liminf_{T\rightarrow\infty} \nu_T\left( W_{\infty} \right) & \geq \frac{1}{2}-q
\end{align*}
\vspace{-0.5cm} 
and
$$
\liminf_{T\rightarrow\infty} \nu_T\left( \mathcal{W}_{0}\cup \mathcal{W}_{\infty} \right) \geq 1-q.
$$
As $\mathcal{W}_0\cup \mathcal{W}_{\infty}\subset \mathcal{W}$, statement (a) of the theorem follows immediately from the construction of the set $\mathcal{W}$.\par

Now, let $(\tau_k)_k$ be a sequence with $\tau_k\in \mathcal{W}_0$ and $\lim_{k\rightarrow\infty} \tau_k =\infty$. The observations that
$$
\left| \L_{\tau}(0)\right| = \left|\L(\tfrac{1}{2}+i\tau) \right|\leq \frac{1}{m(\tau)} \quad\mbox{ for }\tau\in\mathcal{W}_0 
\qquad \mbox{ and } \lim_{\tau\rightarrow\infty} \frac{1}{m(\tau)}=0,
$$
yield that 
\begin{equation}\label{sel:cond1}
\lim_{k\rightarrow\infty}\L_{\tau_k}(0)=0.
\end{equation}
By the normality of $\{\L_{\tau}\}_{\tau\in\mathcal{W}_0}$ in $\mathcal{R}$, every subsequence of $(\L_{\tau_k})_{k}$ converges locally uniformly on $\mathcal{R}$. Assume that there is a subsequence $(\L_{\tau_{k_j}})_{j}$ of $ (\L_{\tau_k})_{k}$ that converges locally uniformly on $\mathcal{R}$ to a limit function $g\not\equiv 0$. Then, due to \eqref{sel:cond1}, the function $g$ has a zero at $z=0$. Thus, by the theorem of Hurwitz, every function $\L_{\tau_{k_j}}$ with sufficiently large $j$ has at least one zero in $\mathcal{R}$. This, however, contradicts the assumption that $\{\L_{\tau}\}_{\tau\in\mathcal{W}_0}$ omits the value $0$. Statement (b) of our theorem is proved.\par

Now, let $(\tau_k)_k$ be a sequence with $\tau_k\in \mathcal{W}_{\infty}$ and $\lim_{k\rightarrow\infty}\tau_k = \infty$. Then, it follows from
$$
\left| \L_{\tau}(0)\right| = \left|\L(\tfrac{1}{2}+i\tau) \right|\geq m(\tau) \quad\mbox{ for }\tau\in\mathcal{W}_{\infty} 
\qquad \mbox{ and } \lim_{\tau\rightarrow\infty} m(\tau)=\infty,
$$
that
$$
\lim_{k\rightarrow\infty} \L_{\tau_k} (0) = \infty.
$$
As $\{\L_{\tau}\}_{\tau\in\mathcal{W}_{\infty}}$ is normal in $\mathcal{R}$ and omits the value `$\infty$', we conclude that $(\L_{\tau_k})_k$ converges locally uniformly on $\mathcal{R}$ to $g\equiv \infty$. This proves statement (c) of the theorem.
\end{proof}

Theorem \ref{th:largesmall} provides information on the frequency of small and large values of $\L\in\Sc^{*}$ in funnel-shaped neighbourhoods of the critical line. This information complements Selberg's central limit law and its extensions stated in Section \ref{sec:selbergcentrallimit}: Selberg's central limit law measures the number of points $\frac{1}{2}+it$ on the critical line for which a given function $\L\in\Sc^*$ takes small or large values. Laurin\v{c}ikas' extension allows us to do the same for points on certain line segments, which lie close to the right of the critical line at distances of order not exceeding $1/\log t$. By Bourgade's multidimensional extension we can measure the number of points $\frac{1}{2}+it$  on the critical line such that both $\L(\tfrac{1}{2}+it)$ and $\L(\tfrac{1}{2}+it + i \frac{1}{(\log t)^{\delta}})$, $0\leq\delta\leq 1$ take either small or large values. By means of Theorem \ref{th:largesmall} we can measure the number of certain rectangular subsets of the region
$$
\frac{1}{2} - \frac{c}{\log t} \leq \sigma \leq \frac{1}{2}+ \frac{c}{\log t}, \qquad t\geq 2, 
$$
with $c>0$, on which a given function $\L\in\Sc^{*}$ assumes small or large values:

\begin{corollary}\label{cor:selbergsmalllarge}
Let $\L\in\Sc^{*}$. Let $c>0$ and $\mathcal{R}:=\mathcal{R}(c,1)$ be the rectangular domain defined by the vertices $\pm c \pm i$. Let $\mu:[2,\infty) \rightarrow\R^+$ be monotonically decreasing such that
$$
\lim_{\tau\rightarrow\infty}\mu(\tau) < \tfrac{1}{2}\kappa_{\L} \qquad \mbox{with}\qquad \kappa_{\L}:= \frac{\pi }{2 d_{\L}}.
$$
For $\tau\geq 2$, we set $\varphi_{\tau}(z):= \frac{1}{2}+\frac{\mu(\tau)}{\log \tau} z + i\tau$ and denote by $\mathcal{R}_{\tau}$ the image of the rectangle $\mathcal{R}$ under the mapping $\varphi_{\tau}$, i.e. 
$
\mathcal{R}_{\tau}:= \varphi_{\tau}\bigl(\mathcal{R}\bigr).
$ 
Let $\overline{\mathcal{R}}_{\tau}$ denote the closure of $\mathcal{R}_{\tau}$.
Then, for every real number $m>1$, every positive integer $\ell$ and every $0<q<\frac{1}{2}$ with $\lim_{\tau\rightarrow\infty}\mu(\tau) < \kappa_{\L} q$, there exist subsets $\mathcal{W}_{1/m},\mathcal{W}_{m}\subset[2,\infty)$ with 
\begin{align*}
\liminf_{T\rightarrow\infty}\nu_T(\mathcal{W}_{\frac{1}{m}}) &\geq \tfrac{1}{2}-q,\\ 
\liminf_{T\rightarrow\infty}\nu_T(\mathcal{W}_{m}) &\geq \tfrac{1}{2}-q,\\
\liminf_{T\rightarrow\infty}\nu_T(\mathcal{W}_{\frac{1}{m}}\cup\mathcal{W}_{m}) &  \geq 1-q,
\end{align*}
such that the following holds. 
\begin{itemize}
\item[(a)] For any $s\in\overline{\mathcal{R}}_{\tau}$ with $\tau\in\mathcal{W}_{1/m}$ and any integer $0\leq l\leq \ell$
$$
|\L^{(l)}(s)|\leq \frac{1}{m}\cdot  \left(\frac{\log\tau }{\mu(\tau)} \right)^{l}   \qquad \mbox{ and } \qquad \L(s)\neq 0 .
$$
\item[(b)] For any $s\in\overline{\mathcal{R}}_{\tau}$ with $\tau\in\mathcal{W}_m$ 
$$|\L(s)|\geq m \qquad \mbox{and}\qquad \left|  \L' (s) \right| \leq  \frac{1}{m} \cdot   \frac{\log\tau}{  \mu(\tau)}\cdot \left|  \L(s)\right|^2 . $$
\end{itemize}
\end{corollary}

\begin{proof}
Corollary \ref{cor:selbergsmalllarge} is a direct consequence of Theorem \ref{th:largesmall}. However, to deduce the statements rigorously from Theorem \ref{th:largesmall}, some rescaling is necessary. First, we choose a sufficiently small $\eta>1$ such that $\lim_{t\rightarrow\infty} \eta \mu(t) < \kappa_{\L}$. Then, we set $\varphi_{\tau}^*(z) = \frac{1}{2}+\frac{ \eta\mu(t)}{\log \tau} z +i\tau$ and define
$\L_{\tau}^*(z):= \L(\varphi_{\tau}^*(z))$. For the family $\{\L^*_{\tau}\}_{\tau\in[2,\infty)}$, the rectangle $\mathcal{R}:=\mathcal{R}(c,1)$ and an arbitrary $0<q<\frac{1}{2}$ with $\lim_{\tau\rightarrow\infty} \mu^*(\tau)< \kappa_{\L}q$, we choose the sets $\mathcal{W}_0, \mathcal{W}_{\infty}\subset[2,\infty)$ according to Theorem \ref{th:largesmall}. Let 
$$\mathcal{K}:=\overline{\mathcal{R}(\tfrac{c}{\eta}, \tfrac{1}{\eta})}$$ be the compact rectangular domain defined by the vertices $\pm \frac{c}{\eta}\pm i \frac{1}{\eta}$. As $\eta>1$, $\mathcal{K}$ is a compact subset of $\mathcal{R}$. Moreover, we have 
$$
\varphi_{\tau}^* (\mathcal{K}) =\overline{\varphi_{\tau}(\mathcal{R})} =\overline{\mathcal{R}}_{\tau}
$$ for every $\tau\in[2,\infty)$. Let $(f_k)_k$ be a sequence of functions $f_k\in \mathcal{H}(\mathcal{R})$ which converges locally uniformly on $\mathcal{R}$ to $f\equiv 0$. Then, according to the theorem of Weierstrass, the corresponding sequence $(f_k^{(l)})_k$ of $l$-th derivatives with $l\in\N_0$ converges also locally uniformly on $\mathcal{R}$ to $f\equiv 0$. Thus, it follows from Theorem \ref{th:largesmall} (b) and (c) that, for arbitrary $m>1$ and $\ell\in\N$, we find a number $\tau_0\geq 2$ such that both
$$
\left| \frac{\d^{l}}{\d z^{l}}\ \L^*_{\tau}(z)\right| \leq \frac{1}{m}
$$
for every $z\in\mathcal{K}$, every $\tau\in\mathcal{W}_{\frac{1}{m}}:=\mathcal{W}_0\cap[\tau_0,\infty)$ and every integer $0\leq l\leq \ell$, and 
$$
\left| \frac{\d^{l}}{\d z^{l}}\ \frac{1}{\L^*_{\tau}(z)}\right| \leq \frac{1}{m}
$$ 
for every $z\in\mathcal{K}$, every $\tau\in\mathcal{W}_{m}:=\mathcal{W}_{\infty}\cap[\tau_0,\infty)$ and every integer $0\leq l\leq \ell$. We observe that
$$
\frac{\d^{l}}{\d z^{l}}\ \L^*_{\tau}(z) = \left( \frac{\eta \mu(\tau)}{\log\tau}\right)^l \cdot \L_{\tau}^{*(l)} (z) 
$$
for $0\leq l \leq \ell$ and
$$
\frac{\d^{l}}{\d z^{l}}\ \frac{1}{\L^*_{\tau}(z)} =  \left( \frac{\eta \mu(\tau)}{\log\tau}\right)^l \cdot \frac{ \L_{\tau}^{*(l)}(z)}{(\L_{\tau}^{*}(z))^2}.
$$
for $l=0$ and $l=1$. For sake of simplicity, we omit here to evaluate analogous expressions for the $l$-th derivative if $l\geq 2$. Consequently, we obtain that
$$
  \left| \L_{\tau}^{*(l)} (z) \right|\leq\frac{1}{m} \cdot  \left( \frac{\log\tau}{ \eta \mu(\tau)}\right)^l \leq \frac{1}{m} \cdot  \left( \frac{ \mu(\tau)}{\log\tau}\right)^l 
$$
for every $z\in\mathcal{K}$, every $\tau\in\mathcal{W}_{\frac{1}{m}}$ and every integer $0\leq l\leq \ell$ and  
$$
\left|  \frac{\L_{\tau}^{*(l)} (z)}{ (\L_{\tau}^{*}(z))^2} \right| \leq \frac{1}{m} \cdot   \left( \frac{\log\tau}{ \eta \mu(\tau)}\right)^l\leq  \frac{1}{m} \cdot   \left( \frac{\log\tau}{  \mu(\tau)}\right)^l 
$$ 
for every $z\in\mathcal{K}$, every $\tau\in\mathcal{W}_{m}$ and $l=0,1$. Note that, for the choice $l=0$, the latter inequality implies that
$$
 \left|\L_{\tau}^{*}(z)\right| \geq m.
$$
By setting $s:=\varphi^*_{\tau}(z)$, we can identify $\L_{\tau}^{*(l)}(z)$ on $K$  with $\L^{(l)}(s)$ on $\varphi_{\tau}^*(\mathcal{K}) = \overline{\mathcal{R}}_{\tau}$. The estimates
\begin{align*}
\liminf_{T\rightarrow\infty}\nu_T(\mathcal{W}_{\frac{1}{m}}) &\geq \tfrac{1}{2}-q,\\ 
\liminf_{T\rightarrow\infty}\nu_T(\mathcal{W}_{m}) &\geq \tfrac{1}{2}-q,\\
\liminf_{T\rightarrow\infty}\nu_T(\mathcal{W}_{\frac{1}{m}}\cup\mathcal{W}_{m}) &  \geq 1-q,
\end{align*}
hold due to the definition of the sets $\mathcal{W}_{1/m}\subset \mathcal{W}_0$, $\mathcal{W}_m\subset \mathcal{W}_{\infty}$ and the choice of $\mathcal{W}_0$ and $\mathcal{W}_{\infty}$. The corollary is proved.
\end{proof}

Corollary \ref{cor:selbergsmalllarge} has some nice applications in the further course of our investigations. We shall deduce the following: 
\begin{itemize}
\item[(i)] The Riemann zeta-function assumes both arbitrarily small and arbitrarily large values on every path to infinity which lies inside the region defined by
$$
\tfrac{1}{2}-\frac{c}{ \log t}<\sigma < \tfrac{1}{2}+\frac{c}{ \log t}, \qquad t\geq 2
$$ 
with any fixed $c>0$; see Corollary \ref{cor:curvessmalllarge}.
\item[(ii)] Let $\alpha,c>0$. Then, there is an $a\in\C$ with $|a|=\alpha$ and a sequence $(t_n)_n$ of numbers $t_n\in[2,\infty)$ such that
$$
\lim_{n\rightarrow\infty} \zeta\left(\tfrac{1}{2} - \frac{c}{\log t_n} +it_n \right) = a;
$$ 
see Corollary \ref{cor:apointsleft}.
\item[(iii)] There is a subinterval $A\subset [0,2\pi)$ of length at least $\frac{\pi}{4}$ such that, for every $\theta\in A$, there is a sequence $(t_n)_n$ of numbers $t_n\in[2,\infty)$ with 
$$
\zeta(\tfrac{1}{2}+it_n)\neq 0, \qquad \lim_{n\rightarrow\infty} \zeta(\tfrac{1}{2}+it_n) = 0 \qquad \mbox{ and } \qquad
\arg \zeta(\tfrac{1}{2}+it_n) \equiv \theta \mod 2\pi;
$$  
see Theorem \ref{th:zeroasintpoint}.
\end{itemize}

\section{Unboundedness on the critical line in the extended Selberg class}\label{sec:unboundedness}
Selberg's central limit law implies that, for $\L\in\Sc^{*}$,
\begin{equation}\label{eq:alphainfsup}
\alpha_{\L,\scalebox{0.8}{\mbox{inf}}}:=\liminf_{\tau\rightarrow\infty} \left|\L(\tfrac{1}{2}+i\tau) \right|=0\quad \mbox{and}\quad
\alpha_{\L,\scalebox{0.8}{\mbox{sup}}}:=\limsup_{\tau\rightarrow\infty} \left|\L(\tfrac{1}{2}+i\tau) \right|=\infty.
\end{equation}
Roughly speaking, every $\L\in\Sc^{*}$ assumes both arbitrarily small and arbitrarily large values on the critical line. It seems reasonable to expect that \eqref{eq:alphainfsup} holds for every function $\L\in\Sc^{\#}$ with $d_{\L}>0$. However, it turns out to be quite challenging to prove \eqref{eq:alphainfsup} for a general function $\L\in\Sc^{\#}$. In this section, we derive some sufficient conditions for a function $\L\in\Sc^{\#}$ to be unbounded on the critical line. Besides some fundamental insights in the extended Selberg class due to Kaczorowski \& Perelli \cite{kaczorowskiperelli:2002, kaczorowskiperelli:2005, kaczorowskiperelli:2011}, we rely here basically on the general theory of ordinary Dirichlet series, for which the reader is referred to the textbook of Titchmarsh \cite[Chapter 9]{titchmarsh:1939}. At the end of this section, we give some specific examples of functions in $\Sc$ for which our considerations imply that \eqref{eq:alphainfsup} is true.

\subsection{Characteristic convergence abscissae in the extended Selberg class}\label{sec:charconvabs}
For an ordinary Dirichlet series
\begin{equation}\label{dirichletseries1}
A(s) = \sum_{n=1}^{\infty} \frac{a(n)}{n^s}
\end{equation}
with coefficients $a(n)\in\C$, we can define certain characteristic convergence abscissae. If a Dirichlet series converges in a point $s_0\in\C$, then it converges uniformly in any angular domain 
$$
A_{\delta}(s_0):=\left\{ s\in\C\, : \, \left| \arg (s-s_0) \right| \leq \frac{\pi}{2} - \delta \right\}
$$
with an arbitrary real number $0<\delta<\frac{\pi}{2}$. Consequently, the region of convergence of a Dirichlet series is always a half-plane and it is reasonable to define its abscissa of convergence as the real number $\sigma_c\in\R\cup\{\pm\infty\}$ such that the Dirichlet series converges in the half-plane $\sigma>\sigma_c$ and diverges in the half-plane $\sigma<\sigma_c$. It follows essentially from Abel's summation formula that the abscissa of convergence is given by
\begin{equation}\label{eq:sigmac}
\sigma_c = \limsup_{x\rightarrow\infty} \frac{\log \left|\sum_{n\leq x} a(n) \right|}{\log x} \qquad \mbox{or} \qquad \sigma_c = \limsup_{x\rightarrow\infty} \frac{\log \left|\sum_{n> x} a(n) \right|}{\log x},
\end{equation}
according to whether $\sum_{n=1}^{\infty}a(n)$ diverges or converges.\par
By a similar argument, the region of absolute convergence of a Dirichlet series is also a half-plane. For a Dirichlet series, we define the abscissa of absolute convergence as the real number $\sigma_a\in\R\cup\{\pm\infty\}$ such that the Dirichlet series converges absolutely in the half-plane $\sigma>\sigma_a$, but does not converge absolutely in the half-plane $\sigma<\sigma_a$. Abel's summation formula yields that
$$
\sigma_a = \limsup_{x\rightarrow\infty} \frac{\log \sum_{n\leq x} \left| a(n) \right|}{\log x} \qquad \mbox{or} \qquad \sigma_a = \limsup_{x\rightarrow\infty} \frac{\log \sum_{n> x} \left| a(n) \right|}{\log x},
$$
according to whether $\sum_{n=1}^{\infty}\left| a(n) \right|$ diverges or converges.\par
Besides $\sigma_c$ and $\sigma_a$, we define the abscissa of uniform convergence $\sigma_u$ as the infimum of all $\sigma^* \in\R \cup\{\pm\infty\}$ for which the Dirichlet series converges uniformly in the half-plane $\sigma\geq \sigma^*$.\par
The abscissae $\sigma_c$, $\sigma_u$ and $\sigma_a$ of a given Dirichlet series do not necessarily coincide. Trivially, one has
\begin{equation}\label{eq:abs1}
-\infty\leq\sigma_c \leq \sigma_u \leq \sigma_a \leq \infty.
\end{equation}
It can be shown that
\begin{equation}\label{eq:abs2}
\sigma_a - \sigma_c \leq 1
\end{equation}
if at least one of the two abscissae $\sigma_c$ and $\sigma_a$ is finite. In particular, the latter inequality is sharp; equality holds, for example, for Dirichlet $L$-functions with non-principle characters. Moreover, one has
\begin{equation}\label{eq:abs3}
\sigma_a - \sigma_u \leq \frac{1}{2},
\end{equation}
if at least one of the two abscissae $\sigma_u$ and $\sigma_a$ is finite.
According to a result of Bohnenblust \& Hille \cite{bohnenblusthille:1931}, this inequality is also sharp.\par

{\bf The analytic character of Dirichlet series.} We suppose in the following that $A(s)$ is an ordinary Dirichlet series with finite convergence abscissae $\sigma_a$, $\sigma_c$ and $\sigma_u$. As a consequence of the theorem of Weierstrass, the Dirichlet series $A(s)$ defines an analytic function in its half-plane of convergence $\sigma>\sigma_c$. Possibly, this function may be continued meromorphically to a larger half-plane $\sigma > \sigma_0$ with $\sigma_0\leq \sigma_c$. If existent, we denote this meromorphic extension also by $A(s)$.\par 

{\bf Boundedness in the half-plane of uniform convergence.} Bohr \cite{bohr:1913} proved that a Dirichlet series $A(s)$ is bounded in every half-plane $\sigma\geq \sigma^*$ with $\sigma^*>\sigma_u$ and that $A(s)$ is unbounded in every half-plane $\sigma\geq \sigma^*$ with $\sigma^* <\sigma_u$ to which $A(s)$ can be continued meromorphically. \par

In view of Bohr's result, for given $\L\in\Sc^{\#}$, it makes sense to localize the abscissa $\sigma_u$ in order to retrieve information on the boundedness and unboundedness of $\L$.\par

{\bf The abscissae of convergence and absolute convergence for functions in the extended Selberg class.}
The definition of the extended Selberg class implies that $\sigma_a\leq 1$ for every $\L\in\Sc^{\#}$. According to Kaczorowski \& Perelli \cite{kaczorowskiperelli:1999}, the elements in the extended Selberg class of degree $d_{\L}=0$ are given by certain Dirichlet polynomials. Thus, in this case, we have $\sigma_c=\sigma_u=\sigma_a = -\infty$. For all functions $\L\in\Sc^{\#}$ with non-zero degree, we expect that $\sigma_a=1$. However, it seems difficult to prove this in general.\par
Of course, if $\L\in\Sc^{\#}$ has a pole at $s=1$, then $\sigma_c = \sigma_u=\sigma_a = 1$. Perelli \cite{perelli:2007} states that $\sigma_a=1$ holds for all $\L\in\Sc$ if Selberg's orthonormality conjecture (S.6$^{**}$) is true.\par 
Relying on non-linear twists of functions in the extended Selberg class, we are able to deduce lower bounds for $\sigma_a$ and $\sigma_c$ which depend on the degree $d_{\L}$ of $\L\in\Sc^{\#}$. Kaczorowski \& Perelli \cite{kaczorowskiperelli:2002, kaczorowskiperelli:2005, kaczorowskiperelli:2011} introduced linear and non-linear twists of functions $\L\in\Sc^{\#}$ to study the structure of the Selberg class, resp. the extended Selberg class. Using this machinery, they were able to obtain partial results towards the degree conjecture; see Section \ref{sec:classG}. In the sequel, we do not want to go too deep into the theory of non-linear twists. Thus, we state the results of Kaczorowski \& Perelli \cite{kaczorowskiperelli:2005} only in a very weak form which is sufficient for our purpose.\par
Let $\L\in\Sc^{\#}$ with $d_{\L}>0$ and Dirichlet series representation 
\begin{equation}\label{eq:dirichletreprtwists}
\L(s)=\sum_{n=1}^{\infty}\frac{a(n)}{n^s}, \qquad\sigma>1.
\end{equation}
For a parameter $\alpha>0$, the standard non-linear twist of $\L$ is defined by
$$
\L(s,\alpha) = \sum_{n=1}^{\infty} \frac{a(n)}{n^s} \exp\left( -2\pi i \alpha n^{1/d_{\L}}\right), \qquad \sigma>1.
$$ 
It follows from Kaczorowski \& Perelli \cite[Theorem 1 and 2]{kaczorowskiperelli:2005} that, for every parameter $\alpha>0$, the function $\L(s,\alpha)$ can be continued meromorphically to the whole complex plane and that there exists an $\alpha^*>0$ such that $\L(s,\alpha^*)$ has a simple pole at 
$$s=s_0:= \frac{d_{\L}+1}{2d_{\L}} + i\frac{\Im \mu_{\L}}{d_{\L}},$$ 
where $\mu_{\L}:=\sum_{j=1}^f (1-2 \mu_j)$ is defined by the data of the functional equation of $\L$; see Chapter \ref{chapt:classG}. It is essentially the pole of $\L(s,\alpha^*)$ at $s=s_0$ which gives us a lower bound for the abscissa of absolute convergence of the Dirichlet series \eqref{eq:dirichletreprtwists}.
\begin{corollary}\label{cor:twists}
Let $\L\in\Sc^{\#}$ with $d_{\L}>0$. Then, the abscissa of absolute convergence $\sigma_a$ of the Dirichlet series defining $\L$ is bounded by
\begin{equation}\label{eq:abs4}
\frac{1}{2} + \frac{1}{2 d_{\L}} \leq \sigma_a \leq 1.
\end{equation}
\end{corollary}
\begin{proof}
Let $\L\in\Sc^{\#}$ with $d_{\L}>0$ and Dirichlet series representation \eqref{eq:dirichletreprtwists}. Let $\sigma_a$ denote the abscissa of absolute convergence of the Dirichlet series defining $\L$. The upper bound $\sigma_a \leq 1$ follows directly from axiom (S.1) in the definition of the extended Selberg class. Suppose that $\sigma_a < \frac{1}{2} + \frac{1}{2 d_{\L}} $. Since 
$$
 \left| \frac{a(n)}{n^{s}} \exp\left( -2\pi i \alpha n^{1/d_{\L}}\right) \right|
= \left| \frac{a(n)}{n^{s}} \right|
$$
for every $n\in\N$ and every $\alpha>0$, the Dirichlet series
$$
\sum_{n=1}^{\infty}   \frac{a(n)}{n^{s}} \exp\left( -2\pi i \alpha n^{1/d_{\L}}\right) 
$$
converges also absolutely in the half-plane $\sigma>\sigma_a$. By the identity principle, we conclude that, for every parameter $\alpha>0$,
$$
\L(s,\alpha) = \sum_{n=1}^{\infty}   \frac{a(n)}{n^{s}} \exp\left( -2\pi i \alpha n^{1/d_{\L}}\right) ,\qquad \sigma>\sigma_a.
$$ 
In particular, $\L(s,\alpha)$ with $\alpha>0$ is analytic in the half-pane $\sigma>\sigma_a$. However, due to \cite[Theorem 1 and 2]{kaczorowskiperelli:2005}, there exists an $\alpha^*>0$ such that $\L(s,\alpha^*)$ has a simple pole at the point $s=s_0: =  \frac{1}{2} + \frac{1}{2 d_{\L}}  + i\frac{\Im \mu_{\L}}{d_{\L}}$ which lies in the half-plane $\sigma>\sigma_a$ according to our assumption on $\sigma_a < \frac{1}{2} + \frac{1}{2 d_{\L}}$. This yields a contradiction and the corollary is proved.
\end{proof}
As a byproduct of Corollary \ref{cor:twists}, we get that there are no functions $\L\in\Sc^{\#}$ of degree $0<d_{\L}<1$.\par

From the simple pole of $\L(s,\alpha^*)$ at the point $s=s_0$, Kaczorowski \& Perelli \cite{kaczorowskiperelli:2005} deduced an $\Omega$-result for truncated sums of the Dirichlet coefficients of $\L$. They showed that, for $\L\in\Sc^{\#}$ with $d_{\L}\geq 1$ and Dirichlet series representation \eqref{eq:dirichletreprtwists},
\begin{equation}\label{eq:omegacoeffsum}
\sum_{n\leq x} a(n) = x \cdot \mbox{res}_{s=1}\ \L(s) + \Omega\left(x^{\frac{d_{\L}-1}{2d_{\L}}} \right).
\end{equation}
By means of \eqref{eq:sigmac}, this yields a lower bound for the abscissa of convergence of $\sum_{n=1}^{\infty} \frac{a(n)}{n^s}$.
\begin{corollary}\label{cor:sigmac}
Let $\L\in\Sc^{\#}$ with $d_{\L}>0$. Then, the abscissa of convergence $\sigma_c$ of the Dirichlet series defining $\L$ is bounded by
$$
\frac{1}{2} - \frac{1}{2 d_{\L}} \leq \sigma_c \leq 1.
$$
\end{corollary}
\begin{proof} Let $\L\in\Sc^{\#}$ with $d_{\L}>0$ and Dirichlet series representation $\L(s)=\sum_{n=1}^{\infty}\frac{a(n)}{n^s}$ in $\sigma>1$. The assumption $d_{\L}>0$ implies that $d_{\L}\geq 1$; see Section \ref{sec:classG}. Let $\sigma_c$ denote the abscissa of convergence of the Dirichlet series defining $\L$. The upper bound $\sigma_c\leq 1$ follows directly from axiom (S.1) in the definition of the extended Selberg class. If $\L$ has a pole at $s=1$, then we conclude immediately that $\sigma_c =1$. Thus, we may suppose that $\L$ has no pole at $s=1$. In this case, $ \mbox{res}_{s=1}\ \L(s) = 0$ and we get by combining \eqref{eq:sigmac} with \eqref{eq:omegacoeffsum} that 
$$
\sigma_c = \limsup_{x\rightarrow\infty} \frac{\log \left|\sum_{n\leq x} a(n) \right|}{\log x}  \geq \frac{1}{2} - \frac{1}{2 d_{\L}}.
$$
The assertion is proved.
\end{proof}
By means of \eqref{eq:abs1} and \eqref{eq:abs3}, we deduce from Corollary \ref{cor:sigmac} the following bounds for the abscissa of uniform convergence $\sigma_u$.
\begin{corollary}
Let $\L\in\Sc^{\#}$ with $d_{\L}>0 $. Then, the abscissa of uniform convergence $\sigma_u$ of the Dirichlet series defining $\L$ is bounded by
$$
\max\left\{\sigma_a - \frac{1}{2}, \frac{1}{2}-\frac{1}{2d_{\L}} \right\}\leq \sigma_u \leq 1 .
$$
\end{corollary}
We expect that $\sigma_u = 1$ for all $\L\in\Sc^{\#}$. If $\L\in\Sc$ has a polynomial Euler product representation (S.3$^*$) and satisfies the prime mean-square condition (S.6), Steuding \cite[Chapt. 5]{steuding:2007} showed that $\L$ is universal in the sense of Voronin at least inside the strip 
$$
\max\left\{\frac{1}{2},\, 1-\frac{1}{d_{\L}}\right\}<\sigma<1.
$$ 
This implies, in particular, that $\L$ is unbounded on every vertical line inside this strip and, consequently, by Bohr's fundamental observation stated above, that $\sigma_u = 1$.

{\bf Problem.} Is it possible to prove that, for every $\L\in\Sc^{\#}$ with $d_{\L}>0$,
$$
\sigma_u = \sigma_a = 1?
$$ 

\subsection{Almost periodicity and a Phragm\'{e}n-Lindel\"of argument}
According to Bohr \cite{bohr:1913}, we know that $\L\in\Sc^{\#}$ is unbounded in every open half-plane containing the line $\sigma=\sigma_u$, where $\sigma_u$ is the abscissa of uniform convergence of the Dirichlet series defining $\L$. Almost periodicity and a Phragm\'{e}n-Lindel\"of argument allow us to make statements about unboundedness on vertical half-lines to the left of $\sigma_u$.

{\bf Almost periodicity in $\sigma>\sigma_u$.} Bohr \cite{bohr:1922} revealed that every Dirichlet series is almost periodic in its half-plane of uniform convergence. 

\begin{theorem}[Bohr, 1922]\label{th:almostperiod}
Let $A(s)$ be an ordinary Dirichlet series and $\sigma_u$ its abscissa of uniform convergence. Then, for every $\sigma>\sigma_u$ and every $\varepsilon>0$, there exists a positive real number $\ell:=\ell(\sigma,\varepsilon)$ such that every interval $[t_0,t_0+\ell]\subset\R$ of length $\ell$ contains at least one number $\tau$ with the property that
$$
\left|A(\sigma+it) - A(\sigma + i(t+\tau)) \right|< \varepsilon \qquad \mbox{ for all }t\in\mathbb{R}.
$$  
\end{theorem}
For the general theory of almost periodic functions, the reader is referred to Bohr \cite{bohr:1924, bohr:1925, bohr:1926} and Besicovitch \cite{besicovitch:1932}.\par

{\bf Unboundedness on vertical half-lines in $\sigma<\sigma_u$.} We shall prove the following lemma.
\begin{lemma}
\label{lem:unbounded}
Let $\L\in\Sc^{\#}$. \begin{itemize}
 \item[(a)] Let $\sigma_u$ denote the abscissa of uniform convergence of the Dirichlet series defining $\L$. Suppose, that $\L$ is unbounded in the half-plane $\sigma> \sigma_u$. Then, for every $t_0>0$, the function $\L$ is unbounded both in the region 
defined by
$$\sigma> \sigma_u, \qquad t\geq t_0,$$
and in the region defined by
$$\sigma> \sigma_u, \qquad t\leq -t_0.$$
\item[(b)] Let $\sigma_0\leq 1$ and $t_0>0$. Suppose that $\L$ is unbounded in the region defined by 
$$
\sigma> \sigma_0,\qquad t\geq t_0.
$$ 
Then, for every $\sigma^* \leq \sigma_0$, the function $\L$ is unbounded on the vertical half-line $L_{\sigma^*}:=\{\sigma^* + it\, :\, t\geq t_0\}$.
\end{itemize}
\end{lemma}
\begin{proof}
Statement (a) follows directly from the almost periodicity of $\L$ in the half-plane $\sigma>\sigma_u$.\par
To prove statement (b) we shall apply a Phragm\'{e}n-Lindel\"of theorem for half-strips. Let $\L\in\Sc^{\#}$ and let $\sigma_0\leq 1$ and $t_0>0$ such that $\L$ is unbounded in the region defined by $\sigma> \sigma_0$, $t\geq t_0$. Note that this necessarily implies that $\L$ is unbounded in the half-strip
$$
S_1:= \{\sigma+it \, :\, \sigma_0\leq \sigma\leq 2, \, t\geq t_0\}.
$$ 
Suppose that there is a $\sigma^*\leq \sigma_0$ such that $\L$ is bounded on the half-line $L_{\sigma^*}$. By the absolute convergence, we know that $\L\in\Sc^{\#}$ is bounded on the half-line $L_2:=\{2+it\,:\,t\geq t_0\}$. Certainly, $\L$ is also bounded on the horizontal line segment $\{\sigma+it_0\, :\, \sigma^*\leq\sigma\leq 2 \}$. Altogether, we obtain that $\L$ is bounded on the boundary $\partial S_2$ of the half-strip
$$
S_2:= \{\sigma+it \, :\, \sigma^*\leq \sigma\leq 2, \, t\geq t_0\}.
$$
Thus, we can find a constant $M>0$ such that $|\L(s)|\leq M$ for all $s\in\partial S_2$. As $\L$ is analytic and of finite order in $S_2$, it follows from a Phragm\'{e}n-Lindel\"of theorem (see for example Levin \cite[Chapt. I, \S 14]{levin:1964}) that $|\L(s)|\leq M$ for all $s\in S_2$. This is a contradiction to our assumption that the function $\L$ is unbounded in $S_1\subset S_2$.
\end{proof}
In Theorem \ref{lem:unbounded} (a) we demand that $\L\in\Sc^{\#}$ is unbounded in its half-plane of uniform convergence $\sigma>\sigma_u$. We know that there are functions  $\L\in\Sc^{\#}$ which are unbounded in the half-plane $\sigma>\sigma_u$. For example, if $\L$ has a pole at $s=1$, then $\L$ is necessarily unbounded in the half-plane $\sigma>\sigma_u=1$. In general, however, we cannot exclude that there are functions $\L\in\Sc^{\#}$ which are bounded in $\sigma>\sigma_u$. \par

\subsection{Mean-square values in the extended Selberg class}\label{subsec:meansquare}
In the theory of Dirichlet series, mean values on vertical lines play an important role. The following fundamental result goes back to Carlson \cite{carlson:1922}. 
\begin{theorem}[Carlson's theorem, 1922]\label{th:carlson}
Let the function $A(s)$ be defined by a Dirichlet series of the form \eqref{dirichletseries1}. Suppose that, for $\sigma\geq\sigma_0$, the function $A(s)$ is analytic except for finitely many poles, of finite order and satisfies
$$
\limsup_{T\rightarrow\infty} \frac{1}{T} \int_{-T}^T \left|A (\sigma_0+it) \right|^2 \d t <\infty.
$$
Then, for all $\sigma>\sigma_0$,
$$
\lim_{T\rightarrow\infty}\frac{1}{T} \int_{-T}^T \left|A (\sigma+it) \right|^2 \d t = \sum_{n=1}^{\infty} \frac{|a_n|^2}{n^{2\sigma}}.
$$
\end{theorem}
Carlson's theorem may be interpreted as a special case of Parseval's theorem in the theory of Hilbert spaces. For a proof, we refer to the original paper of Carlson \cite{carlson:1922} or to the textbook of Titchmarsh \cite[\S 9.51]{titchmarsh:1939}.\par
Let $A(s)$ be a Dirichlet series which can be continued meromorphically to the half-plane $\sigma>\sigma_0$ with some $\sigma_0\in\R$ such that $A(s)$ is of finite order in $\sigma>\sigma_0$. In view of Carlson's theorem, it makes sense to define for $A(s)$ the abscissa of bounded mean-square $\sigma_m$ by taking $\sigma_m$ as the infimum of all $\sigma^*\in(\sigma_0,\infty)$ for which
\begin{equation}\label{limsup}
\limsup_{T\rightarrow\infty} \frac{1}{T} \int_{-T}^T \left|A (\sigma^*+it) \right|^2 \d t < \infty.
\end{equation}
In the following, we call the half-plane $\sigma>\sigma_m$ the mean-square half-plane of $A(s)$. According to Titchmarsh \cite[\S 9.52]{titchmarsh:1939}, the abscissae $\sigma_m$ and $\sigma_a$ are related as follows: 
\begin{equation}\label{eq:sigmama}
\sigma_m \geq \max \{\sigma_a - \tfrac{1}{2}, \sigma_0\}.
\end{equation}
Landau \cite[\S 226, Theorem 41]{landau:1953} showed that
$$
\sigma_m \leq \tfrac{1}{2}\left(\sigma_a + \sigma_c\right).
$$
According to Bohr, $A(s)$ is bounded on every vertical line in the half-plane $\sigma>\sigma_u$. This implies that the inequality
$$
\sigma_m \leq \sigma_u
$$ 
holds.\par
Let $\theta_A(\sigma)$ denote the growth order of $A(s)$ as defined in Section \ref{sec:orderofgrowth}. Then, we have $\theta_A(\sigma)\leq \frac{1}{2}$ for all $\sigma>\sigma_m$; see Titchmarsh \cite[\S 9.55]{titchmarsh:1939}.\par

{\bf Mean-square value in the extended Selberg class.} Relying on a result of Potter \cite{potter:1940}, who studied the mean-square value for Dirichlet series satisfying a quite general functional equation, Steuding \cite[Chapt. 6, Corollary 6.11\,]{steuding:2007} deduced that, for every function $\L\in\Sc$ with $d_{\L}>0$,
\begin{equation}\label{eq:absmean}
\sigma_m \leq \max \left\{\frac{1}{2},\, 1-\frac{1}{d_{\L}}\right\}.
\end{equation}
We easily deduce from Potter's result that \eqref{eq:absmean} holds not only for every function in $\Sc$ but also for every function in $\Sc^{\#}_R$. \par

If $\L\in\Sc^{\#}_R$ has degree $d_{\L}=1$, then the inequalities \eqref{eq:abs4}, \eqref{eq:sigmama} and \eqref{eq:absmean} assure that $\sigma_m=\frac{1}{2}$. We expect that $\sigma_m = \frac{1}{2}$ for every function $\L\in\Sc^{\#}_R$. However, to prove this in general seems to be very difficult. If we assume that $\L\in\Sc^{\#}_R$ satisfies the Lindel\"of hypothesis, then a result of Steuding \cite[Chapt. 2.4, set $\sigma_{\L}=\mu_{\L}=0$ in Lemma 2.4]{steuding:2007} implies that $\sigma_m\leq\frac{1}{2}$. If we assume that $\L\in \Sc^{\#}_R$ has abscissa of absolute convergence $\sigma_a = 1$, then we deduce from \eqref{eq:sigmama} that $\sigma_m\geq \frac{1}{2}$. \par

{\bf A sufficient condition for unboundedness on the critical line.} Let $\L\in\mathcal{\Sc}^{\#}$. If
$$
\limsup_{T\rightarrow\infty}\frac{1}{2T} \int_{-T}^T \left|\L(\sigma+it) \right|^2 \d t =\infty,
$$
then it follows immediately that 
$$
\alpha_{\L,\scalebox{0.8}{\mbox{sup}}}(\sigma):=\limsup_{t\rightarrow\infty} \left|\L(\tfrac{1}{2}+it)\right| = \infty \qquad\mbox{or}\qquad
\alpha^{-}_{\L,\scalebox{0.8}{\mbox{sup}}}(\sigma):=\limsup_{t\rightarrow-\infty} \left|\L(\tfrac{1}{2}+it)\right|.
$$
This observation allows us to formulate sufficient conditions for $\L\in\Sc^{\#}$ to be unbounded on certain lines by relying on mean-square value results.
\begin{lemma}\label{lem:unboundmeanvalue}
Let $\L\in\Sc^{\#}$ with Dirichlet series representation $\L(s)=\sum_{n=1}^{\infty} \frac{a(n)}{n^s}$ in $\sigma>1$. Suppose that $\sum_{n=1}^{\infty}\frac{|a(n)|^2}{n}$ is divergent, then
$$
\limsup_{T\rightarrow\infty} \frac{1}{2T} \int_{-T}^{T} \left|\L(\tfrac{1}{2}+it) \right|^2 \d t = \infty;
$$
and, in particular, 
$$
\alpha_{\L,\scalebox{0.8}{\mbox{sup}}}:=
\alpha_{\L,\scalebox{0.8}{\mbox{sup}}}(\tfrac{1}{2})=\infty 
\qquad \mbox{or} \qquad \alpha^{-}_{\L,\scalebox{0.8}{\mbox{sup}}}:=
\alpha^{-}_{\L,\scalebox{0.8}{\mbox{sup}}}(\tfrac{1}{2})=\infty.
$$
\end{lemma}
The statement follows from Carlson's theorem (Theorem \ref{th:carlson}) and a convexity theorem for the mean-square value which goes back to Hardy, Ingham \& Polya \cite{hardyinghampolya:1927}. 
\begin{theorem}[Hardy, Ingham \& Polya, 1927] \label{th:convexity_meanvalue}
Let the function $f$ be analytic in the strip $\sigma_1<\sigma<\sigma_2$ and such that $|f|$ is continuous on the closure of the strip. Suppose that $f$ satisfies
$$
f(\sigma+it) \ll e^{e^{k|t|}} \qquad \mbox{with } \qquad 0<k<\tfrac{\pi}{\sigma_2 - \sigma_1}
$$
uniformly in $\sigma_1<\sigma<\sigma_2$. If, for arbitrary $p>0$, there are constants $A,B>0$ such that for every $T\geq 0$
$$
\frac{1}{2T} \int_{-T}^{T} \left| f(\sigma_1+it) \right|^{p} \d t \leq A \qquad \mbox{and} \qquad  \frac{1}{2T} \int_{-T}^{T} \left| f(\sigma_2+it) \right|^{p} \d t \leq B ,
$$ 
then
$$
\frac{1}{2T} \int_{-T}^{T} \left| f(\sigma+it) \right|^{p} dt \leq A^{\frac{\sigma_2 - \sigma}{\sigma_2 - \sigma_1}}B^{\frac{\sigma - \sigma_1}{\sigma_2 - \sigma_1}}
$$
for every $\sigma_1\leq \sigma \leq \sigma_2$ and $T\geq 0$.
\end{theorem}
We shall now proceed to prove Lemma \ref{lem:unboundmeanvalue}.

\begin{proof}[Proof of Lemma \ref{lem:unboundmeanvalue}]
For $\sigma\in\R$, we define
$$
J(\sigma):=\limsup_{T\rightarrow\infty} \frac{1}{2T} \int_{-T}^{T} \left|\L(\sigma+it) \right|^2 \d t .
$$
Suppose that there exists a constant $A>0$ such that $J(\frac{1}{2})\leq A$. Then, according to Carlson's theorem, 
$$
J(\sigma)=\sum_{n=1}^{\infty}\frac{|a(n)|^2}{n^{2\sigma}}\qquad\mbox{ for } \sigma>\tfrac{1}{2}.
$$ 
Due to the divergence of $\sum_{n=1}^{\infty} |a(n)|^2 n^{-1}$, we have
\begin{equation}\label{J12}
\lim_{\sigma\rightarrow \frac{1}{2} +}J(\sigma) = \infty.
\end{equation}
Theorem \ref{th:convexity_meanvalue}, however, implies that 
$$
J(\sigma) \leq A \cdot J(\tfrac{3}{4}) \qquad \mbox{ for } \tfrac{1}{2}\leq \sigma \leq \tfrac{3}{4}. 
$$
This yields a contradiction to \eqref{J12}. Hence, $J(\frac{1}{2})=\infty$ and, consequently, $\alpha_{\L,\scalebox{0.8}{\mbox{sup}}}=\infty$ or $\alpha^{-}_{\L,\scalebox{0.8}{\mbox{sup}}}=\infty$.
\end{proof}

\subsection{Summary: The quantities \texorpdfstring{$\alpha_{\L ,inf}$}{} and \texorpdfstring{$\alpha_{\L ,sup}$}{} for \texorpdfstring{$\L\in\Sc^{\#}$}{}} 
\label{subsec:summaryunboundedness}
In the following corollary we gather sufficient conditions which assure that, for a given function $\L\in\Sc^{\#}$, 
$$
\alpha_{\L,\scalebox{0.8}{\mbox{sup}}}:=\limsup_{t\rightarrow\infty} \left|\L(\tfrac{1}{2}+it)\right| = \infty.
$$

\begin{corollary}\label{cor:unbound1}
Let $\L\in\Sc^{\#}$ with $d_{\L}>0$. Suppose that $\L$ satisfies at least one of the following conditions.
\begin{itemize}
\item[(a)] $\L$ is unbounded in some region defined by $\sigma> \sigma_0$, $t\geq t_0$ with some $\sigma_0\geq \frac{1}{2}$ and $t_0>0$.
\item[(b)] $\L$ has a pole at $s=1$.
\item[(c)] $\L$ is universal in the sense of Voronin in some strip $\frac{1}{2}\leq\sigma_1<\sigma<\sigma_2\leq 1$.
\item[(d)] $\L\in\Sc^*$.
\end{itemize}
Then, $\alpha_{\L,\scalebox{0.8}{\mbox{sup}}} = \infty.$
\end{corollary}
\begin{proof}
If $\L\in\mathcal{S}^{\#}$ satisfies property (a), then the statement of the corollary follows from Lemma \ref{lem:unbounded} (b). If $\L\in\mathcal{S}^{\#}$ has a pole at $s=1$, then $\L$ is unbounded in the half-plane $\sigma>\sigma_u=1$. By Lemma \ref{lem:unbounded} (a), $\L$ is unbounded in the region defined by $\sigma>1$, $t\geq 1$. Thus, property (b) is a special case of property (a). Similarly, property (c) is also a special case of (a): a Voronin-type universality property for a given function $\L\in\Sc^{\#}$ implies that $\L$ is unbounded in the region defined by $\sigma>\frac{1}{2}$, $t\geq 1$. If $\L$ satisfies property (d), the statement follows from Selberg's central limit law; see for example Theorem \ref{th:largesmall} (b).
\end{proof}

In the following corollary we gather sufficient conditions which assure that, for a given function $\L\in\Sc^{\#}$, 
$$
\alpha_{\L,\scalebox{0.8}{\mbox{sup}}}:=\limsup_{t\rightarrow\infty} \left|\L(\tfrac{1}{2}+it)\right| = \infty \qquad\mbox{or}\qquad
\alpha^{-}_{\L,\scalebox{0.8}{\mbox{sup}}}:=\limsup_{t\rightarrow-\infty} \left|\L(\tfrac{1}{2}+it)\right|=\infty.
$$
\begin{corollary}
Let $\L\in\Sc^{\#}$ with $d_{\L}>0$. Suppose that $\L$ satisfies at least one of the following conditions.
\begin{itemize}
\item[(a)] The series $\sum_{n=1}^{\infty}\frac{|a(n)|^2}{n}$ is divergent. 
\item[(b)] $\L$ satisfies Selberg's prime coefficient condition (S.6$^*$).
\end{itemize}
Then, $\alpha_{\L,\scalebox{0.8}{\mbox{sup}}}=\infty$ or $\alpha^{-}_{\L,\scalebox{0.8}{\mbox{sup}}}=\infty$.
\end{corollary}
\begin{proof}
If $\L\in\Sc^{\#}$ satisfies property (a), then the statement of the theorem follows directly from Lemma \ref{lem:unboundmeanvalue}. If $\L\in\Sc^{\#}$ satisfies property (b), then it is immediately clear that $\L$ also satisfies (a).
\end{proof}

To prove that, for given $\L\in\Sc^{\#}$,
$$
\alpha_{\L,\scalebox{0.8}{\mbox{inf}}} :=\limsup_{t\rightarrow\infty} \left|\L(\tfrac{1}{2}+it)\right| = 0,
$$
seems to be even harder than to show that $\alpha_{\L,\scalebox{0.8}{\mbox{sup}}} = \infty$. In the subsequent corollary we gather some more or less trivial conditions which assure that $\alpha_{\L,\scalebox{0.8}{\mbox{inf}}} = 0$.
\begin{corollary}
Let $\L\in\Sc^{\#}$ with $d_{\L}>0$. Suppose that $\L$ satisfies at least one of the 
following conditions.
\begin{itemize}
\item[(a)] There are infinitely many zeros of $\L$ with positive imaginary parts located on the critical line. 
\item[(b)] $\L\in\Sc^*$.
\end{itemize}
Then, $\alpha_{\L,\scalebox{0.8}{\mbox{inf}}} =0$.
\end{corollary}
\begin{proof}
If $\L\in\Sc^{\#}$ satisfies property (a), then it is trivially clear that $\alpha_{\L,\scalebox{0.8}{\mbox{inf}}} =0$. If $\L\in\Sc^{\#}$ satisfies property (b), then the statement follows from Selberg's central limit law; see Theorem \ref{th:largesmall}.
\end{proof}

To close this section, we give some specific examples of functions in $\mathcal{S}$ for which we know that \eqref{eq:alphainfsup} holds. According to Hardy \cite{hardy:1914}, the Riemann zeta-function has infinitely many zeros on the critical line with positive imaginary part. Moreover, the zeta-function is unbounded on the critical line as follows for example from the mean-value result of Hardy \& Littlewood \cite{hardylittlewood:1936}:
$$
\frac{1}{T}\int_{1}^{T} \left|\zeta(\tfrac{1}{2}+it) \right|^2 dt \sim \log T , \qquad \mbox{as }T\rightarrow\infty. 
$$
Thus, the Riemann zeta-function satisfies \eqref{eq:alphainfsup}.\par

We can partially transfer this reasoning to the Selberg class. Although the Grand Riemann hypothesis asserts that every functions $\L\in\Sc$ has all its non-trivial zeros on the critical line, very few can be verified about zeros located on $\sigma=\frac{1}{2}$. There are partial results only for some $\L\in\Sc$ of small degree, say $d_{\L}\leq 2$. Besides the Riemann zeta-function, it is known for Dirichlet $L$-functions with primitive character that a positive proportion of their non-trivial zeros lie on the critical line; see Zuravlev \cite{zuravlev:1978}. Moreover, Chandrasekharan \& Narasimhan \cite{chandrasekharannarasimhan:1968}, resp. Berndt \cite{berndt:1969}, proved that infinitely many non-trivial zeros of a Dedekind zeta-function $\zeta_K (s)$ associated to a quadratic field $K$ lie on the critical line. Recently, Mukhopadhyay, Srinivas \& Rajkumar \cite{mukhopadhyay:2008} showed that all functions in the Selberg class of degree $d_{\L} \leq 2$ satisfying some rather general conditions\footnote{For $\L\in\Sc$ with $d_{\L}=2$, the functional equation has to be such that the quantity $(2\pi)^{d_{\L}/2} Q \lambda^{1/2}$ is irrational and such that the quantity $\mu_{\L}:=\mu_{p_{\L}}$ is real. Moreover, the Dirichlet coefficients of $\L$ have to satisfy $\sum_{n\leq x} |a(n)|^2 = O(x)$. } have infinitely many zeros on the critical line. Mukhopadhyay et al. rely on a method due to Landau,\footnote{For a description of the method, see Titchmarsh \cite[\S 10.5]{titchmarsh:1986}.} which is based on the different asymptotic behaviour of the integrals
$$
\int_T^{2T} Z_{\L}(t)\ \d t \qquad \mbox{ and } \int_T^{2T} |Z_{\L}(t)|\ \d t,\qquad \mbox{as } T\rightarrow\infty,
$$
where $Z_{\L}(t)$ is the analogue of Hardy's $Z$-function for $\L\in\Sc$. For all these $\L$-functions mentioned above, the respective results on their zeros assure that $\alpha_{\L,\scalebox{0.8}{\mbox{inf}}}=0$. \par 
Both Dirichlet $L$-functions with primitive character and Dedekind zeta-functions of quadratic fields have a sufficiently `nice' behaving Euler product such that a Voronin-type universality theorem can be verified for them in the strip $\frac{1}{2}<\sigma<1$ (see Bagchi \cite{bagchi:1982} and Reich \cite{reich:1980, reich:1982}, respectively). By Corollary \ref{cor:unbound1}, this implies that they satisfy $\alpha_{\L,\scalebox{0.8}{\mbox{sup}}}=\infty$.

\chapter{a-point-distribution near the critical line} \label{chapt:apoints}
Let $\L\in\Sc^{\#}$. For given $a\in\C$, we refer to the roots of the equation $\L(s)=a$ as $a$-points of $\L$ and denote them by $\rho_a = \beta_a + i\gamma_a$. In view of the Riemann hypothesis, the case $a=0$ is of special interest. Nevertheless, it is reasonable to study the distribution of the $a$-points for general $a\in\C$. For the Riemann zeta-function, the distribution of $a$-points was studied, amongst others, by Bohr, Landau \& Littlewood \cite{BohrLandauLittlewood:1913}, Bohr \& Jessen \cite{bohrjessen:1932}, Levinson \cite{levinson:1974, levinson:1975}, Levinson \& Montgomery \cite{levinsonmontgomery:1974} and Tsang \cite{tsang:1984}.  Selberg \cite{selberg:1992} discussed the distribution of $a$-points in the Selberg class. Steuding \cite[Chapt. 7]{steuding:2003, steuding:2007} investigated to which extent these methods, in particular the ones initiated by Levinson \cite{levinson:1975}, can be transferred to functions of the extended Selberg class. Their methods rely essentially on a lemma of Littlewood which may be interpreted as an integrated version of the principle of argument or as an analogue of Jensen's formula for rectangular domains.
\begin{lemma}[Lemma of Littlewood, 1924]
Let $b<c$ and $T>0$. Let $f$ be an analytic function on the rectangular region
$$
\mathcal{R}:= \left\{s=\sigma+it\in\C \, : \, b\leq\sigma \leq c, \, T\leq t \leq 2T\right\}.
$$
and denote the zeros of $f$ in $\mathcal{R}$ by $\rho=\beta+i\gamma$. Suppose that $f$ does not vanish on the right edge $\sigma = c$ of $\mathcal{R}$. Let $\mathcal{R}'$ be $\mathcal{R}$ minus the union of the horizontal cuts from the zeros of $f$ in $\mathcal{R}$ to the left edge of $\mathcal{R}$, and choose an analytic branch of $\log f(s)$ in the interior of $\mathcal{R}'$. Then,
$$
- \frac{1}{2\pi i} \int_{\partial \mathcal{R}} \log f(s) \d s = \sum_{\begin{subarray}{c} b<\beta<c \\ T<\gamma\leq 2T \end{subarray}} (\beta-b),
$$
where the integral on the lefthand-side is taken over the counterclockwise orientated rectangular contour $\partial\mathcal{R}$.
\end{lemma}
With slight deviations, we took the formulation of Littlewood's lemma from Steuding \cite[Lemma 7.2]{steuding:2007}. For a proof, the reader is referred to the original paper of Littlewood \cite{littlewood:1924-2} or to Titchmarsh \cite[\S 9.9]{titchmarsh:1986}.\par

In Section \ref{sec:apointsgeneral} we summarize some general results on the $a$-point distribution of functions in the extended Selberg class.\par

In section \ref{apointslittlewood} we study in detail the $a$-point distribution of $\L\in\Sc^{\#}$ near the critical line. Levinson \cite{levinson:1975} (and conditionally under the Riemann hypothesis also Landau \cite{BohrLandauLittlewood:1913}) revealed an interesting feature of the Riemann zeta-function in its $a$-point distribution: almost all $a$-points of the Riemann zeta-function are located arbitrarily close to the critical line. Levinson's method builds essentially on the lemma of Littlewood and can be used to detected similar properties in the $a$-point distribution of many functions from the extended Selberg class; see Steuding \cite[Chapt. 7]{steuding:2003, steuding:2007}. By using a result of Selberg, we shall refine the statement of Levinson's theorem for functions in $\Sc^*$

In Section \ref{sec:apointsnormality} we use the notation of filling discs and certain arguments of the theory of normal families to describe the clustering of $a$-points near the critical line. As far as the author knows, the concept of filling discs was not yet used to study the value-distribution of $\L$-functions and yields some new insights in their analytic behaviour near the critical line. In fact, we shall see that the existence of filling discs for $\L\in\Sc^{\#}$ near the critical line is strongly connected to the non-convergence of the limiting process introduced in Section \ref{sec:shiftingshrinking}.\par

\section{General results on the a-point-distribution in the extended Selberg class}
\label{sec:apointsgeneral}

{\bf Trivial $a$-points and half-planes free of non-trivial $a$-points.} Let $\L\in\Sc^{\#}$. Suppose that $\L$ has positive degree and Dirichlet series expansion
\begin{equation}\label{di}
\L(s)=\sum_{n=1}^{\infty} \frac{a(n)}{n^s}, \qquad \sigma>1,
\end{equation}
with leading coefficient $a(1)=1$. By the definition of the extended Selberg class, the Dirichlet series \eqref{di} converges absolutely in $\sigma>1$. If $\L\in\Sc$, the normalization $a(1)=1$ holds trivially due to the Euler product representation. Let $q>1$ denote the least integer such that the coefficient $a(q)$ of the Dirichlet expansion of $\L$ is not equal to zero. Then, due to the absolute convergence of \eqref{di} in $\sigma>1$, we obtain that
\begin{equation}\label{eq:expansionL}
\L(\sigma+it) = 1 +\frac{a(q)}{q^{\sigma+it}} + O\left(\frac{1}{(q+1)^{\sigma}}\right), \qquad \mbox{as }\sigma\rightarrow\infty.
\end{equation}
From this, we derive that, for every $a\in\C$, there exists a real number $R_a\geq 1$ such that $\L$ is free of $a$-points in the half-plane $\sigma>R_a$.\par 

Besides the right half-plane $\sigma> R_a$, which is free of $a$-points of $\L$, there is also a left half-plane which contains not too many $a$-points of $\L$. In the particular case of the Riemann zeta-function this observation is due to Landau \cite{BohrLandauLittlewood:1913} and in the general setting of the extended Selberg class due to Steuding \cite[Chapt. VII]{steuding:2007}. Their results rely basically on the functional equation and the principle of argument: let $M$ be the set of all $\sigma^*>1$ for which we find a constant $m(\sigma^*)>0$ such that 
\begin{equation}\label{bo}
\left|\L(\sigma+it)\right| \geq m(\sigma^*) \qquad \mbox{ for }\sigma\geq \sigma^*.
\end{equation}
It follows from \eqref{eq:expansionL} that $M\neq\emptyset$. We define $L:=1-\inf M<0$. In the half-plane $\sigma<L$, there are $a$-points connected to the trivial zeros of $\L$. The number of these $a$-points with real part $-R < \beta_a < L$ coincides asymptotically, as $R\rightarrow\infty$, with the number of non-trivial zeros with real part $-R<\beta_0<L$ and, thus, grows linear in $R$; see Steuding \cite[Chapt. VII]{steuding:2007}. It follows from \eqref{bo} and the functional equation that, apart from the $a$-points generated by the non-trivial zeros of $\L$, there are at most finitely many other $a$-points in the half-plane $\sigma<L$. We call the $a$-points in the half-plane $\sigma<L$ {\it trivial $a$-points} and refer to all other $a$-points as {\it non-trivial $a$-points} of $\L$. We observe that all non-trivial $a$-points of $\L$ are located in the strip $L\leq\sigma\leq R_a$. \par

{\bf Counting non-trivial $a$-points.} From now on, we assume additionally that $\L$ satisfies the Ramanujan hypothesis. Recall that we defined the class $\Sc^{\#}_R$ to contain all elements of the extended Selberg class which have positive degree and satisfy the Ramanujan hypothesis. Steuding \cite[Chapt. 7]{steuding:2007}  generalized a result of Levinson \cite{levinson:1975} to the class $\Sc^{\#}_R$, which the latter established for the Riemann zeta-function:  
\begin{lemma}[Steuding, 2003]\label{lem:steuding1}
Let $\L\in\Sc^{\#}_R$ with $a(1)=1$. Then, for $a\in\C\setminus\{1\}$ and sufficiently large negative $b$, as $T\rightarrow\infty$,
$$
\sum_{T<\gamma_a \leq 2T} \left(\beta_a - b \right) = \left(\frac{1}{2}-b\right)\left(\frac{d_{\L}}{2\pi}T\log \frac{4T}{e} + T\log(\lambda Q^2) \right) - T \log |1-a| + O(\log T).
$$
\end{lemma}
The proof of Lemma \ref{lem:steuding1} relies essentially on Littlewood's lemma. 
Steuding \cite[Chapt. 7]{steuding:2007} deduced from Lemma \ref{lem:steuding1} a precise Riemann-von Mangoldt type formula for the number of $a$-points of $\L\in\Sc^{\#}_R$. Let $N_a (T)$ denote the number of non-trivial $a$-points of $\L$ with imaginary part $0<\gamma_a \leq T$. 
\begin{theorem}[Steuding, 2003]\label{th:riemannmangoldt}
Let $\L\in\Sc^{\#}_R$ with $a(1)=1$. Then, for any $a\in\C\setminus\{1\}$, as $T\rightarrow\infty$,
\begin{equation}\label{eq:riemannmangoldt}
N_a(T) = \frac{d_{\L}}{2\pi} T \log \frac{T}{e} + \frac{T}{\pi} \log(\lambda Q^2) + O(\log T).
\end{equation}
\end{theorem}
For special functions in $\Sc^{\#}_R$, asymptotic extensions for $N_a(T)$, in particular in the case $a=0$, were obtained already before. Exemplarily, we discuss the case of the Riemann zeta-function. Here, it was Riemann \cite{riemann:1859} who stated the asymptotic formula \eqref{th:riemannmangoldt} for $a=0$. A rigorous proof was given by von Mangoldt \cite{mangoldt:1895}. The case $a\neq 0$ was first established by Landau \cite[Chapt. II, \S 4]{BohrLandauLittlewood:1913}. Their original methods were based on contour integration with respect to the logarithmic derivative of $\zeta(s)$.\par

{\bf $a$-points in the mean-square half-plane.} Let $\L\in\Sc^{\#}_R$ and $\sigma_m$ denote its abscissa of bounded mean-square. Let $N_{a}(\sigma,T)$ denote the number of $a$-points of $\L$ with real part $\beta_a >\sigma$ and imaginary part $0<\gamma_a\leq T$. As $\L$ is of finite order in any strip $-\infty < \sigma_1 \leq \sigma \leq \sigma_2 < \infty$, it follows from the general theory of Dirichlet series that, for every $a\in\C$ and every $\sigma>\sigma_m$,
\begin{equation}\label{eq:NaT_meansquare}
N_a(\sigma, T) \ll T;
\end{equation}
see for example Titchmarsh \cite[\S 9.622]{titchmarsh:1939}.
According to the mean-square results due to Steuding \cite[Chapt. 4 \& 6]{steuding:2007}, we know that unconditionally $\sigma_m \leq \max\{\tfrac{1}{2},1-\tfrac{1}{d_{\L}}\}$ and that $\sigma_m\leq \frac{1}{2}$, if $\L$ satisfies the Lindel\"of hypothesis; see Section \ref{subsec:meansquare} for details.\par

{\bf $a$-points in the strip of universality.} Suppose that $\L$ is an element of the Selberg class, is represented by a polynomial Euler product in $\sigma>1$ and satisfies the prime mean-square condition (S.6). Under these assumptions, Steuding \cite[Chapt. 5, Theorem 5.14]{steuding:2007} verified a Voronin-type universality property for $\L$ in the intersection of its mean-square half-plane $\sigma>\sigma_m$ with the strip $\frac{1}{2}<\sigma<1$. Let $N_{a}(\sigma_1,\sigma_2,T)$ denote the number of $a$-points of $\L$ with real part $\sigma_1<\beta_a <\sigma_2$ and imaginary part $0<\gamma_a\leq T$. As an immediate consequence of the universality property, we obtain that
$$
\liminf_{T\rightarrow\infty} \frac{1}{T} N_a(\sigma_1,\sigma_2,T) >0
$$
for every $a\in\C\setminus\{0\}$ and every $\sigma_m <\sigma_1<\sigma_2<1$. Taking into account the upper bound \eqref{eq:NaT_meansquare}, this implies that,
\begin{equation}\label{ud}
N_a(\sigma_1,\sigma_2,T) \asymp T,
\end{equation}
for every $a\in\C\setminus\{0\}$ and every  $\sigma_m<\sigma_1<\sigma_2<1$; see Steuding \cite{steuding:2007}. Kaczorowski \& Perelli \cite{kaczorowskiperelli:2003} established a zero-density estimate for functions in the Selberg class. They showed that, for any $\eps>0$, uniformly for $\frac{1}{2}\leq \sigma\leq 1$, 
$$N_0(\sigma,T)\ll T^{4(d_{\L}+3)(1-\sigma)+\eps}.$$ Thus, in particular, $N_0(\sigma,T)=o(T)$, if $1-\frac{1}{4(d_{\L}+3)}<\sigma\leq 1.$ Together with \eqref{ud}, this reveals a quantitative difference in the $a$-point distribution of $\L$ between $a=0$ and $a\neq 0$.\par 

For the Riemann zeta-function, Bohr \& Jessen \cite{bohrjessen:1932} obtained that 
$$
N_a(\sigma_1,\sigma_2,T)\sim c  T
$$
for every $a\in\C\setminus\{0\}$ and every $\frac{1}{2}<\sigma<1$ with a positive constant $c:=c(\sigma_1,\sigma_2,a)$ depending on $\sigma_1$, $\sigma_2$ and $a$. Moreover, in the case of the Riemann zeta function there are much more precise zero-density estimates at our disposal than the ones provided by Kaczorowski \& Perelli \cite{kaczorowskiperelli:2003}. We mention here a result of Selberg \cite{selberg:1946} who obtained that, uniformly for $\frac{1}{2}\leq \sigma \leq 1$,
\begin{equation}\label{eq:Selbergzerodensityriemann}
N_0(\sigma,T)\ll T^{1-\frac{1}{4}(\sigma-\frac{1}{2})} \log T.
\end{equation}
For more advanced results on zero-density estimates for the Riemann zeta-function the reader is referred to Titchmarsh \cite[\S 9]{titchmarsh:1986} and Ivi\'{c} \cite[Chapt. 11]{ivic:1985}.
\par

{\bf The special case $a=1$.} The case $a=1$ is special, as our assumption $a(1)=1$ yields that $\lim_{\sigma\rightarrow\infty} \L(s)=1$. This leads to some technical problems in the proofs of Lemma \ref{lem:steuding1} and Theorem \ref{th:riemannmangoldt}. However, one can easily overcome these obstacles by working with
$$
\frac{q^s}{a(q)}(\L(s)-1),
$$
where $q> 1$ is the least integer such that $a(q)\neq 0$; see Steuding \cite[Chapt. 7]{steuding:2007} for details. In this way, we get analogous results in Lemma \ref{lem:steuding1} and Theorem \ref{th:riemannmangoldt} for the case $a=1$ with a minor change in the asymptotic extensions of magnitude $O(T)$, respectively.\par

{\bf $a$-points in the lower half-plane.} All the results stated in this section with respect to $a$-points in the upper half-plane hold in an analogous manner for $a$-points in the lower half-plane.

\section{a-points near the critical line - approach via Littlewood's lemma}
\label{apointslittlewood}

In this section we study in detail the $a$-point distribution of functions in the extended Selberg class near the critical line. Under quite general assumptions on $\L\in\Sc^{\#}$, it is known that the $a$-points of $\L$ cluster around the critical line. We give an overview on existing results and provide a refinement of a theorem of Levinson \cite{levinson:1975}.\par 
Let $\L\in\Sc_R^{\#}$. By assuming a certain growth condition for the mean-value of $\L$ on the critical line, Steuding \cite[Chapt. 7.2]{steuding:2007} showed that almost all $a$-points lie arbitrarily close to the critical line. His methods build on works of Levinson \cite{levinson:1975}, who investigated the particular case of the Riemann zeta-function.
\begin{theorem}[Steuding, 2003]\label{th:steuding3}
Let $\L\in\Sc^{\#}_R$ with $a(1)=1$ and let $a\in\C$. Suppose that, for any $\eps>0$, as $T\rightarrow\infty$, 
\begin{equation}\label{eq:LHgrowth}
\frac{1}{T}\int_{T}^{2T} \left| \L\left(\tfrac{1}{2}+it\right) \right|^2 dt \ll T^{\eps}.
\end{equation}
Then, for any $\delta>0$, all but $O(\delta T\log T)$ of the $a$-points of $\L$ with imaginary part $T<\gamma_a\leq 2T$ lie inside the strip
$$
\tfrac{1}{2}-\delta < \sigma < \tfrac{1}{2}+\delta.
$$ 
\end{theorem}
In his original formulation of Theorem \ref{th:steuding3}, Steuding \cite{steuding:2007} demands that $\L\in\Sc^{\#}_R$ satisfies the Lindel\"of hypothesis. However, a close look at his proof reveals that only the somehow weaker condition \eqref{eq:LHgrowth} is needed. We expect that all functions in $\Sc^{\#}_R$ satisfy the Lindel\"of hypothesis. Thus, all functions $\L\in\Sc^{\#}_R$ should in particular satisfy the growth condition \eqref{eq:LHgrowth}.\par

In the case of the Riemann zeta-function, Landau \cite[Chapt. II, \S 5]{BohrLandauLittlewood:1913} was the first who noticed that, for general $a\in\C$, almost all $a$-points lie arbitrarily close to the critical line. However, for his reasoning, he had to assume the Riemann hypothesis. Levinson \cite{levinson:1975} provided unconditional results for the Riemann zeta-function, exceeding both Landau's conditional observations and the information that may be retrieved from the general situation of Theorem \ref{th:steuding3}.
\begin{theorem}[Levinson, 1975]\label{th:levinson}
Let $a\in\C$. Then, as $T\rightarrow\infty$, all but\\ $O(T\log T / \log\log T)$ of the $a$-points of the Riemann zeta-function with imaginary part $T<\gamma_a\leq 2T$ lie inside the strip
$$
\tfrac{1}{2}-\frac{(\log\log T)^2}{\log T} < \sigma < \tfrac{1}{2}+\frac{(\log\log T)^2}{\log T}.
$$ 
\end{theorem}
In the special situation of $a=0$, there are certain zero-density estimates for the Riemann zeta-function at our disposal which allow a more precise statement than the one provided by Theorem \ref{th:levinson} for general $a\in\C$. The following theorem is an immediate consequence of Selberg's zero-density estimate \eqref{eq:Selbergzerodensityriemann}.
\begin{theorem}[Selberg, 1946] \label{th:selbergzero}
Let $a\in\C$ and $\mu:[2,\infty)\rightarrow\R^+$ be a positive function with $\lim_{t\rightarrow\infty}\mu(t)=\infty$. Then, as $T\rightarrow\infty$, all but $O(T\log T \exp(-\mu(T)/4))$ of the zeros of the Riemann zeta-function with imaginary part $T<\gamma\leq 2T$ lie inside the strip
$$
\tfrac{1}{2}-\frac{\mu(T)}{\log T} < \sigma < \tfrac{1}{2}+\frac{\mu(T)}{\log T}.
$$ 
\end{theorem}
Leaning on a result of Selberg \cite{selberg:1992}, we can slightly refine Levinson's result of Theorem \ref{th:levinson} and extend it to the class $\Sc^*$. We are pretty sure that, at least in the case of the Riemann zeta-function, both Levinson,\footnote{This is suggested by Levinson's remark at the end of his paper {\it Almost all roots of $\zeta(s)=a$ are arbitrarily close to $\sigma=\frac{1}{2}$}, \cite{levinson:1975}.} Selberg and Tsang\footnote{Tsang \cite{tsang:1984} states in his corollary after Theorem 8.2 that almost all $a$-points lie to the left of the line $\sigma = \mu(t)\sqrt{\log\log t}/\log t$ with any function $\lim_{t\rightarrow\infty}\mu(t)=\infty$. This is a one-sided version of Theorem \ref{th:levinsonselberg} in the special case of the Riemann zeta-function.} were aware of this refinement. However, apart from a brief hint by Heath-Brown in Titchmarsh \cite[\S 11.12]{titchmarsh:1986}, we could not find the following theorem stated in the literature explicitly. We recall that both the Riemann zeta-function and Dirichlet $L$-functions attached to primitive characters lie in $\Sc^*$ and that we expect that $\Sc^*=\Sc$; see Section \ref{sec:selbergclass}.
\begin{theorem}\label{th:levinsonselberg}
Let $\L\in\Sc^*$ and $a\in\C$. Let $\mu:[2,\infty)\rightarrow\R^+$ be a positive function with $\lim_{t\rightarrow\infty}\mu(t)=\infty$. Then, as $T\rightarrow\infty$, all but $O(T\log T / \mu(T))$ of the $a$-points of $\L$ with imaginary part $T<\gamma_a\leq 2T$ lie inside the strip defined by
\begin{equation}\label{eq:striplevinsonselberg}
\frac{1}{2}-\frac{\mu(T)\sqrt{\log\log T}}{\log T} <\sigma < 
\frac{1}{2}+\frac{\mu(T)\sqrt{\log\log T}}{\log T} .
\end{equation}
\end{theorem}
In the proof of Theorem \ref{th:levinsonselberg}, we follow strongly the ideas of Levinson \cite{levinson:1975}.
\begin{proof}
Let $\L\in\Sc$ and let $a\in\C\setminus\{1\}$. It follows from Littlewood's lemma that
\begin{equation}\label{eq:momentselberg}
\sum_{\begin{subarray}{c} T< \gamma_a \leq 2T \\ \beta_a>\frac{1}{2} \end{subarray} } \left(\beta_a - \frac{1}{2} \right) = \frac{1}{2\pi}\int_0^T \log \left| \L(\tfrac{1}{2}+it) - a\right| \d t - \frac{T}{2\pi}\log \left|1-a \right| + O(\log T).
\end{equation}
For details we refer to Levinson \cite[Lemma 2]{levinson:1975}, Steuding \cite[Chapt. VII, proof of Theorem 7.1]{steuding:2007} or Selberg \cite[eq. (3.5)]{selberg:1992}. In the case of the Riemann zeta-function, Levinson \cite{levinson:1975} obtained the assertion of Theorem \ref{th:levinson} by bounding the left-hand side of \eqref{eq:momentselberg} by $O(T\log\log T)$; here, he basically used Jensen's inequality in combination with the asymptotic formula $\int_0^T |\zeta(\frac{1}{2}+it)|^2\d t\sim T\log T$, as $T\rightarrow\infty$. Similarly, Steuding \cite{steuding:2007} obtained the assertion of Theorem \ref{th:steuding3} by bounding the left-hand side of \eqref{eq:momentselberg} by  $O(\eps T\log T)$ and then following basically Levinson's ideas. In our case, we brush up Levinson's approach, by using a precise asymptotic expansion for the integral on the right-hand side of \eqref{eq:momentselberg}, which is due to Selberg \cite{selberg:1992}. The latter obtained that, for any $\L\in\Sc^{*}$,
\begin{equation}\label{eq:momentselberg2}
\int_{T}^{2T} \log \left| \L(\tfrac{1}{2}+it) - a\right| \d t = \frac{\sqrt{n_{\L}}}{2\sqrt{\pi}} T \sqrt{\log\log T} + O_{|a|}(T);
\end{equation}
where the quantity $n_{\L}$ is defined by Selberg's prime coefficient condition (S.6$^*$). In some places, Selberg's proof seems a bit sketchy. We refer to Tsang \cite[\S 8]{tsang:1984}, who carried out all details in the case of the Riemann zeta-function, and to Hejhal \cite[\S 4]{hejhal:2000}, who provided a thorough description of Selberg's method.\footnote{Hejhal \cite{hejhal:2000} assumes additionally that $\L\in\Sc^*$ has a polynomial Euler product (S.3$^*$). However, this assumption is only needed for later purposes and not for the results in \S 4.}\par 
Suppose that $\L\in\Sc^{*}$. Then, by combining \eqref{eq:momentselberg} with \eqref{eq:momentselberg2}, we get that
\begin{equation}\label{eq:apointslittlewood}
\sum_{\begin{subarray}{c} \beta_a>\frac{1}{2} \\ T<\gamma_a\leq 2T \end{subarray}} \left(\beta_a-\frac{1}{2}\right)
= \frac{\sqrt{n_{\L}}}{4\pi^{3/2}} T \sqrt{\log\log T} + O_{|a|}(T).
\end{equation}
Let $\mu:[2,\infty)\rightarrow\R^+$ be a positive function with $\lim_{t\rightarrow\infty}\mu(t) = \infty$. For positive $T$, we denote by $N_{a}^{(1)}(T)$ the number of $a$-points $\rho_a = \beta_a + i\gamma_a$ of $\L$ with
$$
\beta_a > \frac{1}{2}+ \frac{\mu(T)\sqrt{\log\log T}}{\log T}, \qquad T<\gamma_a\leq 2T; 
$$ 
by $N_{a}^{(2)}(T)$ the number of $a$-points with 
$$
 \frac{1}{2}- \frac{\mu(T)\sqrt{\log\log T}}{\log T} \leq \beta_a \leq \frac{1}{2}+ \frac{\mu(T)\sqrt{\log\log T}}{\log T}, \qquad T<\gamma_a\leq 2T; 
$$ 
and by $N_{a}^{(3)}(T)$ the number of $a$-points with
$$
 \beta_a < \frac{1}{2} -\frac{\mu(T)\sqrt{\log\log T}}{\log T}, \qquad T<\gamma_a\leq 2T.
$$ 
The trivial estimate
$$
\sum_{\begin{subarray}{c} \beta_a>\frac{1}{2} \\ T<\gamma_a\leq 2T \end{subarray}} \left(\beta_a-\frac{1}{2}\right) \geq \frac{\mu(T)\sqrt{\log\log T}}{\log T} N_{a}^{(1)}(T)
$$
yields in combination with \eqref{eq:apointslittlewood} that, for sufficiently large $T$,
$$
N_{a}^{(1)}(T) \ll \frac{T\log T}{\mu(T)}.
$$
Moreover, for any real $b$,
\begin{align*}
\sum_{T<\gamma_a \leq 2T} \left(\beta_a + b \right) & = 
\sum_{\begin{subarray}{c} \beta_a>\frac{1}{2} \\ T<\gamma_a\leq 2T \end{subarray}} \left(\beta_a - \frac{1}{2} \right) + \sum_{\begin{subarray}{c} \beta_a\leq \frac{1}{2} \\ T<\gamma_a\leq 2T \end{subarray}}\left(\beta_a - \frac{1}{2} \right)  + \sum_{T<\gamma_a \leq 2T} \left(\frac{1}{2} + b \right) \\
&\leq \sum_{\begin{subarray}{c} \beta_a>\frac{1}{2} \\ T<\gamma_a\leq 2T \end{subarray}} \left(\beta_a - \frac{1}{2} \right) - \frac{\mu(T)\sqrt{\log\log T}}{\log T} N_a^{(3)}(T) + \\
&\qquad \qquad\qquad +\left(\frac{1}{2} + b \right)\left( N_{a}^{(1)}(T)+N_{a}^{(2)}(T)+N_{a}^{(3)}(T)\right).
\end{align*}
If we take $b$ sufficiently large, we get by means of the asymptotic extensions \eqref{lem:steuding1} and \eqref{eq:apointslittlewood} and the Riemann-von Mangoldt  formula \eqref{eq:riemannmangoldt} that, for sufficiently large $T$,
$$
N_{a}^{(3)}(T) \ll \frac{T\log T}{\mu(T)}.
$$
This proves the assertion for $a\in\C\setminus\{1\}$. The case $a=1$ can be treated in a similar manner by relying on suitably adjusted formulas; see the remark at the end of the preceeding section.
\end{proof}
Theorems \ref{th:steuding3}, \ref{th:levinson} and \ref{th:levinsonselberg} hold in an analogous manner for $a$-points in the lower half-plane.\par

{\bf $a$-points close to the critical line with real part $\beta_a<\frac{1}{2}$.} Suppose that $\L\in\Sc^*$. Unconditionally, almost nothing is known about how the $a$-points of Theorem \ref{th:levinsonselberg}, which lie in the strip \eqref{eq:striplevinsonselberg}, are distributed to the left and to the right of the critical line. Under the assumption of the Riemann hypothesis, Selberg \cite{selberg:1992} obtained the following: for $T\geq 2$ and $\mu>0$, let
$$
\sigma(T,\mu):= \frac{1}{2}-\mu\frac{\sqrt{\log\log T}}{\log T}.
$$
Then, for every $a\in\C\setminus\{0\}$, as $T\rightarrow\infty$,
\begin{equation}\label{eq:selberg_aleft}
N_a(\sigma(T,\mu),T) \sim N_a (T) \cdot \frac{1}{\sqrt{2\pi}}\int_{-\mu'}^{\infty}e^{-x^2/2} \d x
\end{equation}
with
$$
\mu':= \frac{\sqrt{2}d_{\L}}{n_{\L}}\mu.
$$
Recall that $\varphi(x)= \frac{1}{\sqrt{2\pi}}e^{-x^2/2}$ defines the density function of the Gaussian normal distribution. Thus, roughly speaking, Selberg's result states that about half of the $a$-points of $\L$ lie to the left of the critical line and are statistically well distributed at distances of order $\sqrt{\log\log T}/\log T$. In particular, this means that the left bound of the strip \eqref{eq:striplevinsonselberg} seems to be best possible. Moreover, by assuming the Riemann hypothesis, Selberg deduced, that, as $T\rightarrow\infty$, most of the other $a$-points with $T<\gamma_a\leq 2T$ in the strip \eqref{eq:striplevinsonselberg} lie quite close to the critical line at distances of order not exceeding
$$
O\left(\frac{(\log\log\log T)^3}{\log T \sqrt{\log\log T}} \right).
$$
The proof presented by Selberg \cite{selberg:1992} is quite sketchy. For a rigorous proof in the special case of the Riemann zeta-function, we refer to Tsang \cite[Theorem 8.3]{tsang:1984}. Selberg \cite{selberg:1992} also provides some heuristic reason in support of the following conjecture.\par
{\bf Selberg's $a$-point conjecture.} {\it
Let $\L\in\Sc$. For any $a\in\C\setminus\{0\}$, as $T\rightarrow\infty$, about $3/4$-th of all non-trivial $a$-points with $T<\gamma_a\leq 2T$ lie to the left of the critical line, and about $1/4$-th of all non-trivial $a$-points with $T<\gamma_a\leq 2T$ to its right.}

\newpage

\section{a-points near the critical line - approach via normality theory}\label{sec:apointsnormality}
In this section we use the concept of filling discs to investigate the $a$-point distribution of functions $\L\in\Sc^{\#}$ near the critical line. Many of our results hold even for functions in the more general class $\mathcal{G}$, introduced in Chapter \ref{chapt:classG}. 
\subsection[Filling discs, Julia directions and Julia lines - Basic properties]{Filling discs, Julia directions and Julia lines - Definitions and basic properties}\label{sec:fillingdiscs_basics}
We introduce the notion of filling discs, Julia directions and Julia lines. Roughly speaking, these concepts allow a more precise formulation of Picard's great theorem.\par 
Recall that $\mathcal{M}(\Omega)$, resp. $\mathcal{H}(\Omega)$, denotes the set of all functions which are meromorphic, resp. analytic, on a domain $\Omega\subset\C$. We say that $f\in\mathcal{M}(\Omega)$ satisfies the {\it Picard-property} on $E\subset \Omega$ if $f$ assumes on $E$ every value $a\in\widehat{\C}$, with at most two exceptions.\par
{\bf Filling discs.} We call a sequence of discs $D_{\lambda_n}(z_n) \subset \Omega$, $n\in\N$, a sequence of {\it filling discs for $f\in\mathcal{M}(\Omega)$} if, for arbitrary $0<\eps\leq 1$, the function $f$ satisfies the Picard-property on every infinite union of the discs $D_{\eps\lambda_n}(z_n)$, i.e. on every set of the form
$$
E:= \bigcup_{k\in\N} D_{\eps \lambda_{n_k}} (z_{n_k}) \qquad \mbox{with }0<\eps\leq1 .
\footnote{In the literature it is convenient to demand additionally the growth condition $\lambda_n\leq |z_n|$ for the radii of the filling discs. For our purpose, however, this restriction is not relevant.}
$$
We have necessarily that $\lim_{n\rightarrow\infty} z_n \in \partial \Omega$. Due to rich contributions by French mathematical schools,\footnote{Amongst others, the names of Julia, Milloux and Valiron are here to mention.} filling discs are sometimes referred to as {\it cercles de remplissage}.\par 
By Montel's fundamental normality test (Theorem \ref{th:FNT1}) and basic convergence properties, the existence of a sequence of filling discs is strongly connected to the non-normality of a certain family. In fact it is Montel's fundamental normality test that motivates the definition of filling discs and provides a first characterization of the latter.
\begin{proposition}\label{prop:charfilldiscs}
Let $f\in\mathcal{M}(\Omega)$. Suppose that $(z_n)_n$ is a sequence of points $z_n\in \Omega$ and $(\lambda_n)_n$ a sequence of positive real numbers such that $D_{\lambda_n} (z_n) \subset \Omega$ for $n\in\N$. Let $f_n:\D\rightarrow\C$ be defined by $f_n(z):=f(z_n +\lambda_n z)$.\\
Then, the discs $D_{\lambda_n}(z_n)$, $n\in\N$, form a sequence of filling discs for $f$ if and only if every infinite subset of the family $\mathcal{F}:=\{f_n\}_n$ is not normal at zero.
\end{proposition}
The assertion of Proposition \ref{prop:charfilldiscs} was observed by many people and, essentially, goes back to Montel \cite{montel:1912} and Julia \cite{julia:1919}.
\begin{proof}
Certainly, the functions $f_n$ are well-defined on $\D$.
Suppose that every infinite subset of the family $\mathcal{F}:=\{f_n\}_n$ is not normal at zero. Then, Montel's fundamental normality test (Theorem \ref{th:FNT1}) implies that every infinite subset of $\mathcal{F}$ can omit at most two values $a,b\in\widehat{\C}$ on any  neighbourhood $D_{\eps}(0)$ of zero with $0<\eps\leq 1$. By observing that $f_n(D_{\eps}(0)) = D_{\eps\lambda_n}(z_n)$, this is equivalent to the statement that the discs $D_{\lambda_n}(z_n)$, $n\in\N$, form a sequence of filling discs for $f$.\par
Now, suppose that the discs $D_{\lambda_n}(z_n)$, $n\in\N$, form a sequence of filling discs for $f$. Assume that there exists a subsequence $(f_{n_k})_k$ of $(f_n)_n$ which converges locally uniformly on some disc $D_{\eps}(0)$ with $0<\eps\leq 1$. Then, by the theorem of Weierstrass (Theorem \ref{th:weierstrass}), its limit function $\phi$ is meromorphic on $D_{\eps}(0)$. Consequently, taking $0<\eps'<\eps$ small enough, $\phi(\overline{D_{\eps'}(0)})$ omits a non-empty open subset of $\C$. By uniform convergence, this implies that $\bigcup_{k\geq K} f_{n_k}(\overline{D_{\eps'}(0)})$ omits more than three values for any sufficiently large $K\in\N$, contradicting the definition of filling discs. Consequently, every infinite subset of $\mathcal{F}$ is not normal at zero.
\end{proof}
By combining Proposition \ref{prop:charfilldiscs} with Marty's theorem (Theorem \ref{th:marty}), Lehto  \cite{lehto:1958}
derived a powerful characterization of filling discs; we refer here also to Clunie \& Hayman \cite{cluniehayman:1966} and Sauer \cite{sauer:2002} who pointed out a slight inexactness in Lehto's original formulation. For $f\in\mathcal{M}(\Omega)$, the spherical derivative $f^{\#}$ is defined by
$$
f^{\#} (z) :=  \frac{|f'(z)|}{1+|f(z)|^2}, \qquad z\in\Omega;
$$
for details we refer to the appendix.
\begin{theorem}[Lehto's criterion, 1958]\label{th:lehto}
Let $f\in\mathcal{M}(\Omega)$. Suppose that $(z_n)_n$ is a sequence of points $z_n\in \Omega$ and $(\lambda_n)_n$ a sequence of positive real numbers such that $D_{\lambda_n} (z_n) \subset \Omega$ for $n\in\N$. Then, the discs $D_{\lambda_n}(z_n)$, $n\in\N$, form a sequence of filling discs for $f$ if and only if there is a sequence $(w_n)_n$ of points $w_n\in \Omega$ such that
\begin{equation}\label{crit_fillingdiscs}
\lim_{n\rightarrow\infty}\lambda_n f^{\#}(w_n) = \infty
\qquad \mbox{and} \qquad
|z_n-w_n|=o(\lambda_n).
\end{equation}
\end{theorem}
Lehto's criterion will play a central role in our further investigations.\par

In a half-strip setting, we equip sequences of filling discs with certain counting functions. Let $S$ be a vertical half-strip in the upper half-plane defined by
\begin{equation}\label{halfstrip}
-\infty< x_1 < \Re \ z < x_2 <+\infty, \qquad \Im\ z >0
\end{equation}
For a sequence $(z_n)_n$ of points $z_n\in S$ with $\lim_{n\rightarrow\infty} z_n = \infty$, we define $N_{\{z_n\}_n} (T) $ to be the number of elements in $\{z_n\}_n$ with imaginary part less than $T$. Similarly, for $f\in\mathcal{M}(S)$, any subset $E\subset S$ and any $a\in\widehat{\C}$, let $N_{a}(E,T)$ denote the number of $a$-points of $f$ on $E$ with imaginary part less than $T$. 

\begin{lemma}\label{lem:countingfillingdiscs}
Let $S$ be a half-strip defined by \eqref{halfstrip} and $f\in\mathcal{H}(S)$.
Let $D_{\lambda_n}(z_n)$, $n\in\N$, be a sequence of filling discs for $f$ in $S$ such that $\lim_{n\rightarrow\infty} z_n = \infty$ and $D_{\lambda_n}(z_n)\cap D_{\lambda_m}(z_m)=\emptyset$ for $n\neq m$. Let $E:=\bigcup_{n\in\N} D_{\lambda_n}(z_n)$. Then, for all but at most one $a\in\C$,
$$
\limsup_{T\rightarrow\infty} \frac{N_{a}(E,T)}{N_{\{z_n\}_n}(T)} > 0.
$$
\end{lemma}

\begin{proof}\label{lem:fill_counting}
Assume that there is an $a\in\C$ such that 
$$
\limsup_{T\rightarrow\infty} \frac{N_{a}(E,T)}{N_{\{z_n\}_n}(T)} = 0.
$$
Then, since the discs $D_{\lambda_n}(z_n)$, $n\in\N$, are pairwise disjoint, there exists a subsequence $(z_{n_k})_k$ of $(z_n)_n$ with 
\begin{equation}\label{eq:countingsequence}
\lim_{T\rightarrow\infty}\frac{N_{\{z_{n_k}\}_k}(T)}{N_{\{z_n\}_n}(T)} =1
\end{equation}
such that $f$ omits the values $a$ and $\infty$ on $\bigcup_{k\in\N} D_{\lambda_{n_k}}(z_{n_k})$. Let $b\in\C\setminus\{a\}$. Then, we deduce from Proposition \ref{prop:charfilldiscs} and an extended version of Montel's fundamental normality test (Theorem \ref{th:FNTextension} (b)), for any $m\in\N$ there is an integer $K$ such that $f$ assumes the value $b$ more than $m$-times on every disc  $D_{\lambda_{n_k}}(z_{n_k})$ with $k\geq K$. Thus, since the discs $D_{\lambda_{n_k}}(z_{n_k})$, $k\in\N$, are pairwise disjoint, we obtain that for every $b\in\C\setminus\{a\}$ and every $m\in\N$
$$
\liminf_{T\rightarrow\infty} \frac{N_{b}(E,T)}{N_{\{z_{n_k}\}_k}(T)} > m.
$$
Together with \eqref{eq:countingsequence}, this implies that
$$
\liminf_{T\rightarrow\infty} \frac{N_{b}(E,T)}{N_{\{z_{n}\}_n}(T)} = \infty.
$$
As this holds for any $b\in\C\setminus\{a\}$, the assertion is proved.
\end{proof}

{\bf Julia directions.} A complex number $e^{i\theta_0}$ with $\theta_0\in\R$ is called {\it Julia direction for $f\in\mathcal{M}(\Omega)$ at $z_0\in\partial \Omega$} if $f$ satisfies the Picard-property in every sectorial domain
$$
E_{R,\eps}:=\left\{z_0 + re^{i\theta}\in\C \; : \; 0< r \leq R, \; |\theta-\theta_0|<\eps   \right\} \subset \Omega
$$
with any $\eps>0$ and any $R>0$. To define a Julia direction for $f$ at $z_0 = \infty$, one has to replace $E_{R,\eps}$ by
$$
E'_{R,\eps}:=\left\{re^{i\theta}\in\C \; : \;  r \geq R, \; |\theta-\theta_0|<\eps   \right\} \subset \Omega .
$$ 
If $e^{i\theta_0}$ is a Julia direction for $f$ at $z_0$, then $f$ assumes every  value $a\in\widehat{\C}$, with at most two exceptions, {\it infinitely often} on every set $E_{R,\eps}$, resp. $E'_{R,\eps}$. This follows immediately from the definition of a Julia direction.\par
Let $f$ be a function that is analytic in a punctured neighbourhood of an essential singularity $z_0 \in \C$. Lehto \cite{lehto:1958} proved that 
\begin{equation}\label{limsup_esssing}
\limsup_{z\rightarrow z_0} |z-z_0| f^{\#}(z) = \infty;
\end{equation}
here, one has to replace \eqref{limsup_esssing} by 
$
\limsup_{z\rightarrow z_0} |z| f^{\#}(z) = \infty
$
if the essential singularity $z_0$ lies at infinity. By means of Lehto's criterion, we deduce that there is a sequence of points $z_n\in\Omega$ with $\lim_{n\rightarrow\infty}z_n = z_0$ such that the discs $D_{\lambda_n}(z_n)$, $n\in\N$, with $\lambda_n= |z_n-z_0|$ form a sequence of filling discs for $f$. This implies, in particular, that $f$ satisfies the Picard-property in every punctured neighbourhood of $z_0$. Thus, we have reproduced the analytic version of Picard's classical theorem. Moreover, by choosing $e^{i\theta_0}$ as an accumulating point of the set $\{e^{i\arg z_n}\, :\, n\in\N\} \subset \partial\D$, we deduce that their exists a Julia direction for $f$ at $z_0$. Thus, we have reproduced the statement of Julia's classical result on Julia directions (see Burckel \cite[Theorem 12.27]{burckel:1979} and Julia \cite{julia:1919}).

{\bf Julia lines.} We call
$$
L := \{z_0 + re^{i\theta_0} \; : \; r\in\R \} \qquad \mbox{ with fixed } z_0\in\C, \; \theta_0\in[0,2\pi)
$$
a {\it Julia line} for $f\in\mathcal{M}(\C)$ if $f$ satisfies the Picard-property in every open strip containing the line $L$.\footnote{In some literature, the ray $re^{i\theta_0}$, $r\in\R^+$, connected to a Julia direction $e^{i\theta_0}$ is also called Julia line.} Certainly, one can use Lehto's criterion to detect Julia lines. The following corollary is a direct consequence thereof.
\begin{corollary}\label{cor:julia}
Let $f\in\mathcal{M}(\C)$. Suppose that, for a fixed $x\in\R$, there exists a sequence $(y_n)_n$ of positive real numbers $y_n$ with $\lim_{n\rightarrow\infty}y_n=\infty$ and a positive real number $\alpha$ such that
$$
\lim_{n\rightarrow\infty} |f(x+iy_n)| = \alpha  \qquad
\mbox{ and } \qquad \lim_{n\rightarrow\infty} |f'(x+iy_n)| = \infty. 
$$
Then, $\Re\ z = x$ defines a Julia line for $f$.
\end{corollary}
\begin{proof}
For arbitrary $\lambda>0$,
$$
\lim_{n\rightarrow\infty} \lambda f^{\#}(x_0+iy_n) =  \lim_{n\rightarrow\infty}  \lambda \ \frac{|f'(x_0+iy_n)|}{1+|f(x_0+iy_n)|^2} = \infty.
$$
By Lehto's criterion (Theorem \ref{th:lehto}), the discs $D_{\lambda}(x_0 + iy_n)$, $n\in\N$, form a sequence of filling discs for $f$ and the assertion follows.
\end{proof}

Mandelbrojt \& Gergen \cite{mandelbrojtgergen:1931} investigated Julia lines of entire functions defined by a generalized Dirichlet series
$$
A(s)=\sum_{n=0}^{\infty} a(n) e^{-\lambda_n s},
$$
which is absolutely convergent in $\C$ and satisfies 
$$
0=\lambda_0< \lambda_1 < \lambda_2 < ...\, , \qquad \lim_{n\rightarrow\infty}\lambda_n = \infty, \qquad a(n)\neq 0 \quad \mbox{ for all }n>0. 
$$
They derived conditions on the exponents $\lambda_n$ and on the growth behaviour of $A(s)$ which lead to the existence of horizontal Julia lines. Their results mainly rely on the fact that the quantities
$$
\limsup_{N\rightarrow\infty}\frac{\#\{\lambda_n \, : \, n\leq N\}}{N}, \qquad \mbox{ resp. } \quad \liminf_{n\rightarrow\infty}  (\lambda_{n+1} - \lambda_n), 
$$
are, at least in some mean sense, bounded from above, resp. from below. Thus, their results do not apply to ordinary Dirchlet series  with non-vanishing coefficients $a(n)$. For extended results in this direction, we refer to the monographies of Mandelbrojt \cite{mandelbrojt:1952, mandelbrojt:1969}; in \cite{mandelbrojt:1952} particularly to Chapter II for results on analytic functions in strips and Chapter VII for Picard-type results on generalized Dirichlet series in strips, in \cite{mandelbrojt:1969} to Chapter VI again for Picard-type results. One might interpret the theory developed by Mandelbrojt as a counterpart of lacunary power series for Dirichlet series. It seems that his theory is not appropriate to describe the value-distribution of functions in the extended Selberg class.

\subsection{Julia directions and Julia lines for the Riemann zeta-function}
We investigate Julia directions and Julia lines of the Riemann-zeta function. With suitable adaptions, similar results can be stated for many functions in the extended Selberg class. Using another terminology, the following theorem was basically stated by Garunkt\v{s}tis \& Steuding \cite{garunkstissteuding:2010}.
\begin{theorem}
The Julia directions of the Riemann zeta-function are given by $e^{i\pi/2}$, $e^{i\pi}$ and $e^{i 3\pi/2}$. \\
All Julia lines of the Riemann zeta-function are vertical lines. Every line defined by $\Re\ z =\sigma \in [\frac{1}{2},1]$ is a Julia line. There are no Julia lines among the lines $\Re \ z = \sigma \in (-\infty,0) \cup (1,\infty)$.  Assuming the Riemann hypothesis, the lines $\Re\ z =\sigma \in [0,\frac{1}{2})$ are no Julia lines, either.
\end{theorem}
\begin{proof}
For a given $a\in\C$, the discs $D_{\lambda}(-2n)$, $n\in\N$, with any $\lambda>0$, cover all but finitely many trivial $a$-points of the Riemann zeta-function. This follows by a slight refinement of Landau's proof of Lemma 1 and Lemma 2 in Bohr, Landau \& Littlewood \cite[Chapt. II]{BohrLandauLittlewood:1913}.  Thus, they form a sequence of filling discs. Consequently, $e^{i\pi}$ is a Julia direction and $\Im \ z = 0$ defines a Julia line for the zeta-function. As there are left and right half-planes free of non-trivial $a$-points, the only further Julia directions are given by $e^{i\pi/2}$ and $e^{i 3\pi/2}$ and all further Julia lines are vertical lines. Bohr's and Voronin's denseness results in combination with Lemma \ref{cor:julia} yield that every line $\Re \ z = \sigma \in(\frac{1}{2},1]$ is a Julia line. The clustering of $a$-points around the critical line (see Theorem \ref{th:levinsonselberg}) implies that the same is true for $\Re \ z = \frac{1}{2}$. By the inequality
$$
\frac{\zeta(2\sigma)}{\zeta(\sigma)} \leq |\zeta(s)|\leq \zeta(\sigma),
$$
which is valid in the half-plane $\sigma>1$, and the functional equation, there is no Julia direction among $\Re\ z \in (-\infty, 0) \cup (1,\infty)$. Assuming the Riemann hypothesis, we know that, for $\sigma>\frac{1}{2}$, as $t\rightarrow\infty$, 
$$\zeta(\sigma+it)\gg t^{-\eps}$$ with any $\eps>0$; see Titchmarsh \cite[\S 14.2]{titchmarsh:1986}. This in combination with the functional equation yields that, under assumption of the the Riemann hypothesis, the lines $\Re\ z = \sigma\in[0,\frac{1}{2})$ are no Julia lines either (see Garunk\v{s}tis \& Steuding \cite[Lemma 4 and Proposition 5]{garunkstissteuding:2010}). 
\end{proof}
In the next section, we shall see that, in main parts, the functional equation is responsible for $\sigma=\frac{1}{2}$ being a Julia line. 

\newpage

\subsection{Filling discs induced by a Riemann-type functional equation}
{\bf A Riemann-type functional equation and its symmetry line $\sigma=\frac{1}{2}$.} A Riemann-type functional equation for a function $G\in\mathcal{G}$ invokes a strong connection between $G$ and $G'$ on its symmetry line $\sigma=\frac{1}{2}$. 
\begin{lemma}\label{lem:connectionGG'}
Let $G\in\mathcal{G}$ with $d_G>0$. Then, for $t\in\R$ with $G(\frac{1}{2}+it)\neq 0$ and $G(\frac{1}{2}+it)\neq \infty$, as $|t|\rightarrow\infty$,
$$
 \left| \frac{G'(\frac{1}{2}+it)}{G(\frac{1}{2}+it)} \right| \geq  \frac{d_G}{2}\log |t|-\frac{1}{2}\log (Q^2 \lambda) + O\left( \frac{1}{|t|}\right).
$$
\end{lemma}
\begin{proof}
Logarithmic differentiation of the functional equation yields that
$$
\frac{G'(s)}{G(s)} = \frac{\Delta'(s)}{\Delta(s)} - \overline{\left(\frac{G'(1-\overline{s})}{G(1-\overline{s})}\right)}.
$$
For $s=\frac{1}{2}+it$, this and the asymptotic extension \eqref{asymext_logdiff_Delta_p} of $\Delta'/\Delta$  imply that 
\begin{equation}\label{Re_logdiff_Delta}
2 \ \Re \left( \frac{G'(\frac{1}{2}+it)}{G(\frac{1}{2}+it)} \right) = \frac{\Delta'(\frac{1}{2}+it)}{\Delta(\frac{1}{2}+it)} = -d_G\log |t|-\log (Q^2 \lambda) + O\left( \frac{1}{|t|}\right),
\end{equation}
provided that $|t|$ is sufficiently large and that $s=\frac{1}{2}+it$ is not a pole of $G'(s)/G(s)$. The assertion follows from \eqref{Re_logdiff_Delta} by using the relation $|z|\geq |\Re \ z|$.
\end{proof}

{\bf Filling discs induced by a Riemann-type functional equation on $\sigma=\frac{1}{2}$.} By Lehto's criterion, Lemma \ref{lem:connectionGG'} implies that, in most cases, a Riemann-type functional equation turns $\sigma=\frac{1}{2}$  into a Julia line. The following theorem is a special case of Corollary \ref{cor:julia} for functions satisfying a Riemann-type functional equation. 
\begin{theorem}\label{th:avalues-case1}
Let $G\in\mathcal{G}$ with $d_G>0$. Suppose that there exists a sequence $(\tau_k)_k$ with $\tau_k\in[2,\infty)$ and $\lim_{k\rightarrow\infty} \tau_k = \infty$ and an $\alpha\in(0,\infty)$ such that 
\begin{equation}\label{cond1}
\lim_{k\rightarrow \infty} |G(\tfrac{1}{2}+i\tau_k)| = \alpha .
\end{equation}
Then, for any positive function $\mu:[2,\infty)\rightarrow\R^+$ with $\lim_{\tau\rightarrow\infty}\mu(\tau) = \infty$, the discs defined by
$$
|s-\tfrac{1}{2}-i\tau_k|< \frac{\mu(\tau_k)}{\log\tau_k}, \qquad k\in\N,
$$
form a sequence of filling discs for $G$. In particular, $\sigma=\frac{1}{2}$ is a Julia line of $G$; and $G$ assumes every value $a\in\widehat{\C}$, with at most two exceptions, infinitely often inside the region defined by
\begin{equation}\label{eq:strip-case1}
\frac{1}{2}-\frac{\mu(t)}{\log t} < \sigma < \frac{1}{2}+\frac{\mu(t)}{\log t}, \qquad t\geq 2.
\end{equation}

\end{theorem}

\begin{proof}
 Due to our conditions, we can assume without loss of generality that $G(\frac{1}{2}+i\tau_k)\neq 0,\infty$ for $k\in\N$. We deduce from Lemma \ref{lem:connectionGG'} that
$$
\left| G'(\tfrac{1}{2}+i\tau_k) \right| \geq  \frac{d_G}{4}  \log \tau_k
\left|G(\tfrac{1}{2}+i\tau_k )\right|,
$$ 
provided that $k$ is large enough.
Hence, for sufficiently large $k$,
$$
G^{\#}(\tfrac{1}{2}+i\tau_k) = \frac{|G'(\frac{1}{2}+i\tau_k)|}{1 + |G(\frac{1}{2}+i\tau_k)|^2} \geq \frac{d_G}{4}\ \log\tau_k \ \frac{|G(\frac{1}{2}+i\tau_k)|}{1 + |G(\frac{1}{2}+i\tau_k)|^2}
$$
Together with
$$
\lim_{k\rightarrow\infty} \frac{|G(\frac{1}{2}+i\tau_k)|}{1 + |G(\frac{1}{2}+i\tau_k)|^2} = \frac{\alpha}{1+\alpha^2}\in (0,\tfrac{1}{2}],
$$
which is true according to our assumption, this yields that
$$
\lim_{k\rightarrow\infty }\frac{\mu(\tau_k)}{\log \tau_k} G^{\#}(\tfrac{1}{2}+i\tau_k) = \infty
$$
for any positive function $\mu$ with $\lim_{\tau\rightarrow\infty} \mu(\tau) = \infty$. The assertion of the theorem follows immediately from Theorem \ref{th:lehto}.
\end{proof}

For general functions $G\in\mathcal{G}$, the conditions posed on $G$ in Theorem \ref{th:avalues-case1} are best possible:\par
Theorem \ref{th:avalues-case1} does not necessarily apply to functions in $\mathcal{G}$ of degree zero: according to Kaczorowski \& Perelli \cite{kaczorowskiperelli:1999} any function $\L\in\Sc^{\#}\subset\mathcal{G}$ with $d_{\L}=0$ is given by a Dirichlet polynomial. Hence, $\L$ is bounded in any vertical strip around the critical line. \par

Condition \eqref{cond1} cannot be removed in general as the following two examples show: Let
the function 
$$
G_{\pmb{0}}(s):= \exp(-s(1-s)) \Delta_{\zeta}(s)^{1/2}, \qquad s\in\C_{\Delta_{G_{\pmb{0}}}}
$$
be defined as in Section \ref{subsec:noncon}. Then, $G_{\pmb{0}}\in\mathcal{G}$ with $d_{G_{\pmb{0}}}=1$ and $\lim_{t\rightarrow\infty}G_{\pmb{0}}(\frac{1}{2}+it)=0$.  Since
\begin{equation*}
G_{\pmb{0}}(\sigma+it) \ll \exp\left(-Ct^2 \right)
\end{equation*}
holds uniformly for $0<\sigma<1$, as $t\rightarrow\infty$, with a certain constant $C>0$, the function $G_{\pmb{0}}\in\mathcal{G}$ assumes any given $a\in\C\setminus\{0\}$ at most finitely often inside the half-strip $0<\sigma<1$, $t\geq 2$. Thus, the assertion of Theorem \ref{th:avalues-case1} does not hold for $G_{\pmb{0}}$. Similarly, the function $G_{\pmb \infty}$ defined by
$$
G_{\pmb \infty}(s):=\Delta(s)^{\frac{1}{2}}\exp(s(1-s)),\qquad s\in\C_{\Delta_{G_{\pmb{0}}}}
$$
yields an example for a function in $\mathcal{G}$ with $d_{G_{\pmb{\infty}}}=1$ and $\lim_{\tau\rightarrow\infty} |G_2(\frac{1}{2}+i\tau)|= \infty$, but not satisfying the assertion of Theorem \ref{th:avalues-case1}.\par
In the general setting of the class $\mathcal{G}$, the radii $\frac{\mu(\tau_k)}{\log\tau_k}$ of the filling discs are best possible. This can be seen by considering the function
$$
G_{1,\zeta}(s):=1+\Delta_{\zeta}(s).
$$
Due to the asymptotic estimate $\Delta_{\zeta} (s) \asymp \left(|t|/2\pi \right)^{1/2 - \sigma}$, which holds uniformly for $0\leq \sigma\leq 1$, as $t\rightarrow\infty$, the function $G_{1,\zeta}(s)$ is bounded in any region defined by
$$
\frac{1}{2} - \frac{c}{\log t} \leq \sigma \leq \frac{1}{2} + \frac{c}{\log t}, \qquad t>2, 
$$
with an arbitrary constant $c>0$. \par
It seems reasonable to determine the quantities
$\alpha_{G,\scalebox{0.8}{\mbox{inf}}}$ and $\alpha_{G,\scalebox{0.8}{\mbox{sup}}} $ defined by \eqref{def:alpha_inf} to check whether a given function $G\in\mathcal{G}$ satisfies condition \eqref{cond1}. By the intermediate value theorem, we find a sequence $(\tau_k)_k$ with $\tau_k\in[2,\infty)$ and $\lim_{k\rightarrow\infty} \tau_k = \infty$ and an $\alpha\in(0,\infty)$ such that \eqref{cond1} holds whenever
$$
\alpha_{G,\scalebox{0.8}{\mbox{inf}}}\neq \alpha_{G,\scalebox{0.8}{\mbox{sup}}} 
\qquad \qquad \mbox{or}\qquad \qquad
0<\alpha_{G,\scalebox{0.8}{\mbox{inf}}}=\alpha_{G,\scalebox{0.8}{\mbox{sup}}}<\infty
$$
For $\L\in\Sc^{*}$ we know that $\alpha_{\L,\scalebox{0.8}{\mbox{inf}}}=0$ and $\alpha_{\L,\scalebox{0.8}{\mbox{sup}}}=\infty$. We expect that the same is true for every function $\L\in\Sc^{\#}$. We refer to Section \ref{subsec:summaryunboundedness} for partial results.

{\bf Normality approach vs. Levinson's method.} We shall compare the $a$-point result of Theorem \ref{th:avalues-case1}, given in terms of filling discs, with the $a$-point results, obtained by Levinson's method (Theorems \ref{th:steuding3}, \ref{th:levinson} and \ref{th:levinsonselberg}). Both concepts yield qualitatively and quantitatively different information and complement each other: Levinson's method is appropriate to count the number of $a$-values in neighbourhoods of line segments $\frac{1}{2}+it$, $t\in[T,2T]$ in a very good manner. Filling discs provide locally more precise information on the location of some of the $a$-points (combine for example Theorem \ref{th:FNTextension} with Proposition \ref{prop:charfilldiscs}). Moreover, the conditions in Theorem \ref{th:avalues-case1} for the existence of filling discs are rather weak compared to the conditions needed to evaluate the integrals involved in Levinson's method in a precise manner.\par 
In the case of the Riemann zeta-function, Theorem \ref{th:levinsonselberg} provides that almost all non-trivial $a$-values with ordinates $0<\gamma_a\leq T$ have a distance less than $$\frac{\mu(\gamma_a) \sqrt{\log\log \gamma_a}}{\log \gamma_a}$$ to the critical line; where $\mu$ is any positive function tending to infinity. Theorem \ref{th:avalues-case1} yields that, among them, there are infinitely many at distance less than $$\frac{\mu(\gamma_a)}{\log \gamma_a}$$ to the critical line.\par
By assuming certain conjectures, we find filling discs for the Riemann zeta-function whose radii are significantly smaller than the ones provided by Theorem \ref{th:avalues-case1}. We refer to Section \ref{sec:filling_zeta}. However, for $\L\in\Sc^{\#}$ the radii of the filling discs cannot shrink arbitrarily fast, as $\tau\rightarrow\infty$.\par

{\bf Lower bound for the radii of filling discs in $\Sc^{\#}$.} 
The radii $\lambda_k$ of a sequence of filling discs $D_{\lambda_k}(\frac{1}{2}+i\tau_k)$, $k\in\N$, for $\L\in\Sc^{\#}$ are trivially bounded from below by $|\tau_k|^{-\theta_{\L}(\frac{1}{2})-\delta}$, where $\delta>0$ is an arbitrary positive real number and $\theta_{\L}(\frac{1}{2})$ defined as in Section \ref{sec:orderofgrowth}. This observation follows from the subsequent lemma.
\begin{lemma}\label{lem:lowerboundrad}
Let $\L\in\Sc^{\#}$. Suppose that there are sequences  $(\lambda_k)_k$, $(\tau_k)_k$ with $\lambda_k,\tau_k \in \R^+$ and $\lim_{k\rightarrow\infty}\tau_k = \infty$ such that $D_{\lambda_k}(\frac{1}{2}+i\tau_k)$, $k\in\N$, forms a sequence of filling discs for $\L$. Then, for any $\delta>0$,
$$
\lim_{k\rightarrow\infty}
\left( \lambda_k\tau_k^{\theta_{\L}(\frac{1}{2})+\delta}\right) =\infty.
$$
\end{lemma}
\begin{proof}
Assume that there exists a $\delta>0$ and a subsequence of $(\lambda_k)_k$, which we assume to be $(\lambda_k)_k$, such that
$$
\lambda_k \ll \tau_k^{-\theta_{\L}(\frac{1}{2})-\delta}
$$
as $k\rightarrow\infty$. We set $f_k(z):=\L(\frac{1}{2} + i\tau_k + \lambda_k z)$ and regard the family $\mathcal{F}:=\{f_k\}_k$ of functions on $\D$. Then, \eqref{eq:finiteorderderivative} and the continuity of the function $\theta_{\L}(\sigma)$ assure that 
\begin{equation}
 \left| \L' (\tfrac{1}{2}+i\tau_k + \lambda_k z)\right| \ll_{\delta,\ell} \tau_k^{\theta_{\L}(\sigma) + \delta/2},
\end{equation}
uniformly for $z\in \D$, as $k\rightarrow\infty$. Consequently, we obtain for the spherical derivatives
$$
f^{\#}_k(z) \leq \lambda_k   \left| \L' (\tfrac{1}{2}+i\tau_k + \lambda_k z)\right| \ll \tau_k^{-\theta_{\L}(\frac{1}{2})-\delta} \ \cdot \ \tau_k^{\theta_{\L}(\sigma) + \delta/2} = \tau_k^{-\delta/2}
$$
uniformly for $z\in \D$, as $k\rightarrow\infty$. Hence, the family $\mathcal{F}^{\#}:= \{f_k^{\#}\}_k$ is bounded on $\D$. According to Marty's theorem, $\mathcal{F}$ is normal in $\D$. This is in contradiction to Proposition \ref{prop:charfilldiscs} which states that $\mathcal{F}$ is not normal in $\D$.
\end{proof}

{\bf Generalizations of Theorem \ref{th:avalues-case1}.} In Theorem \ref{th:avalues-case1} we studied filling discs induced by a Riemann-type functional equation on $\sigma=\frac{1}{2}$. It would be interesting to investigate whether similar statements can be retrieved for functions satisfying a different type of functional equation, for example the Selberg zeta-function.

\subsection{Filling discs induced by Selberg's central limit law}\label{sec:Selberglimitlaw}

In this section, we consider filling discs for functions $\L\in\Sc^{*}\subset\Sc^{\#}$ which satisfy Selberg's prime coefficient condition (S.6$^*$) and the zero-density estimate (DH). Let $n_{\L}$ be defined by (S.6$^*$) and set
$$
g_{u}(t):= \exp\left( u \sqrt{\tfrac{1}{2} n_{\L} \log \log t} \right),\qquad t\geq 2,\qquad u\in\R.
$$ 
For $\L\in\Sc^*$ and $\alpha,\beta\in\R$ with $\alpha<\beta$, we set
$$
W_{g_{\alpha}(t),g_{\beta}(t)}:= \left\{ t\in [2,\infty)\ : \ g_{\alpha}(t) \leq  \left|\L(\tfrac{1}{2}+it) \right| \leq g_{\alpha}(t)\right\},
$$
as in Section \ref{sec:smalllargeONcritline}. 

\begin{theorem}\label{th:selberg1}
Let $\L\in\Sc^*$ and $\alpha,\beta\in\R$ with $\alpha<\beta$. Then, for any sequence  $(\tau_k)_k$ with $\tau_k\in W_{g_{\alpha}(t),g_{\beta}(t)}\subset[2,\infty)$ and $\lim_{k\rightarrow\infty}\tau_k = \infty$, the discs defined by
$$
|s-\tfrac{1}{2}-i\tau_k|< \frac{\mu(\tau_k)\lambda_{\alpha,\beta}(\tau_k)}{\log \tau_k},\qquad k\in\N,
$$
where $\mu:[2,\infty)\rightarrow\R^+$ is any positive function with $\lim_{t\rightarrow\infty}\mu(t)=\infty$ and
\begin{equation}\label{def:lambda}
\lambda_{\alpha,\beta}(t):= \exp( (-\alpha+\max\{0,2\beta\}) \sqrt{\tfrac{1}{2} n_{\L} \log \log t} ),
\end{equation}
form a sequence of filling discs for $\L$.
\end{theorem}
\begin{proof}
The existence of a sequence  $(\tau_k)_k$ with $\tau_k\in W_{g_{\alpha}(t),g_{\beta}(t)}\subset[2,\infty)$ and $\lim_{k\rightarrow\infty}\tau_k = \infty$ is assured by  Theorem \ref{th:measselberglimitlaw} (a). By Lemma \ref{lem:connectionGG'}, we know that, for sufficiently large $k\in\N$,
$$
\left| \L^{\#}(\tfrac{1}{2}+i\tau_k)\right| >\tfrac{1}{3} d_{\L} \log \tau_k \ \frac{|\L(\frac{1}{2}+i\tau_k)|}{1+|\L(\frac{1}{2}+i\tau_k)|^2}.
$$ 
Due to the definition of $W_{g_{\alpha}(t),g_{\beta}(t)}\subset[2,\infty)$, this implies that, for sufficiently large $k\in\N$,
$$
\left|\L^{\#}(\tfrac{1}{2}+i\tau_k) \right|>\tfrac{1}{4} d_{\L} \log \tau_k \cdot \frac{ \exp( \alpha \sqrt{\tfrac{1}{2} n_{\L} \log \log \tau_k} ) }{1+  \exp( 2\beta \sqrt{\tfrac{1}{2} n_{\L} \log \log \tau_k} )} .
$$
Thus, we obtain that
$$
\lim_{k\rightarrow\infty}  \frac{\mu(\tau_k)\lambda_{\alpha,\beta}(\tau_k)}{\log \tau_k } \ \L^{\#}(\tfrac{1}{2}+i\tau_k) = \infty 
$$
for any positive function $\mu$  satisfying $\lim_{t\rightarrow\infty}\mu(t)=\infty$ and $\lambda_{\alpha,\beta}(t)$ defined by \eqref{def:lambda}. The assertion follows by Lehto's criterion (Theorem \ref{th:lehto}).
\end{proof}
Theorem \ref{th:measselberglimitlaw} (a) provides a quantitative description of the set $W_{g_{\alpha}(t),g_{\beta}(t)}\subset[2,\infty)$, which allows us to estimate the number of pairwise disjoint filling discs for $\L\in\Sc^{*}$ in Theorem  \ref{th:selberg1}.\par
Recall the definition of the counting functions $N_{\{z_n\}_n}(T)$ and $N_a(E,T)$ introduced in Section \ref{sec:fillingdiscs_basics}. Let $\L\in\Sc^*$ and $\alpha,\beta\in\R$ with $\alpha<\beta$. Further, let $\lambda_{\alpha,\beta}$ be defined by \eqref{def:lambda} and $\mu:[2,\infty)$ be a positive function with  $\lim_{t\rightarrow\infty}\mu(t)=\infty$. According to Theorem \ref{th:measselberglimitlaw} and Theorem \ref{th:selberg1}, we find a sequence $(\tau_k)_k$ of positive real numbers tending to infinity with 
$$
N_{\{\frac{1}{2}+i\tau_k\}_k}(T) \geq   N_{\{\frac{1}{2}+i\tau_k\}_k}(T) - N_{\{\frac{1}{2}+i\tau_k\}_k}\left(\tfrac{T}{2}\right) \gg \frac{T\log T}{\mu\left(\frac{T}{2}\right) \lambda_{\alpha,\beta}(T)}
$$
such that the discs defined by
$$
\left|s-\tfrac{1}{2}-i\tau_k\right| <\frac{\mu(\tau_k)\lambda_{\alpha,\beta}(\tau_k)}{\log \tau_k}, \qquad k\in\N,
$$
form a sequence of {\it pairwise disjoint} filling discs for $\L$. By Lemma \ref{lem:countingfillingdiscs}, this implies that, for every $a\in\C$ with at most one exception, the number $N_a(E,T)$ of $a$-points with $0<\gamma_a\leq T$ inside the region $E$ defined by
$$
\tfrac{1}{2} -   \frac{\mu(t)\lambda_{\alpha,\beta}(t)}{\log t} < \sigma < \frac{1}{2} +   \frac{\mu(t)\lambda_{\alpha,\beta}(t)}{\log t}, \qquad t\geq 2.
$$
satisfies
$$
N_a(E,T) \gg  \frac{T}{\mu\left(\frac{T}{2}\right) \lambda_{\alpha,\beta}(T)}.
$$
However, by noticing that, for any real $\alpha<\beta$ and sufficiently large $t$
$$
\sqrt{\log\log t} < \lambda_{\alpha,\beta}(t),
$$
we deduce from Theorem \ref{th:levinsonselberg} that $N_a(E,T)\sim \frac{d_{\L}}{2\pi} T \log T$, as $T\rightarrow\infty$. Thus, the filling disc setting of Theorem \ref{th:selberg1} does not beat Theorem \ref{th:levinsonselberg} in counting the $a$-points inside the region $E$.\par

Nevertheless, as the information retrieved from filling discs is qualitatively different than the one provided by Levinson's method in Theorem \ref{th:levinsonselberg}, it is worth stating Theorem \ref{th:selberg1}.

\subsection{Filling discs for the Riemann zeta-function via \texorpdfstring{ $\Omega$}{Omega}-results \texorpdfstring{for $\zeta'(\rho_a)$}{}}\label{sec:filling_zeta}
In the case of the Riemann zeta-function there are (conditional) $\Omega$-results for $\zeta'(\rho_a)$ at our disposal which enable us to detect sequences of filling discs whose radii are significantly smaller than the ones in the general setting of Theorem \ref{th:avalues-case1}.

{\bf Discrete moments of $\zeta'(s)$ with respect to non-trivial $a$-points.} Information about large values of $|\zeta'(\rho_a)|$ can be retrieved from asymptotic estimates for discrete moments 
\begin{equation*}
I_{a}(T):=\frac{1}{N_a(T)} \sum_{0<\gamma_a \leq T} \zeta'(\rho_a), \quad \mbox{resp.} \quad J_{a,k}(T):=\frac{1}{N_a(T)} \sum_{0<\gamma_a \leq T} \left| \zeta'(\rho_a) \right|^{2k};
\end{equation*}
here $N_a(T)$ denotes as usual the number of non-trivial $a$-points $\rho_a$ of the zeta-function with imaginary part $0<\gamma_a \leq T$. Asymptotic estimates for $I_{a}(T)$ and $J_{a,k}(T)$ can be established by residue methods. Naturally, these asymptotics are powerful tools to estimate the number of simple $a$-points, in particular the number of simple zeros of the Riemann zeta-function. For works in this direction, we refer to Garunk\v{s}tis \& Steuding \cite{garunkstissteuding:2010} concerning simple $a$-points and exemplarily to Conrey, Gosh \& Gonek \cite{conreygoshgonek:1988, conreygoshgonek:1998}, Bui, Conrey \& Young \cite{buiconreyyoung:2011} and Bui \& Heath-Brown \cite{buiheathbrown:2013} concerning simple zeros.\par
Garunk\v{s}tis \&  Steuding \cite{garunkstissteuding:2010} proved that for any fixed $a\in\C$, as $T\rightarrow\infty$,
\begin{equation}\label{eq:asympIak}
I_{a}(T) = \frac{1}{N_a(T)}\sum_{0<\gamma_a \leq T} \zeta'(\rho_a) \sim 
(\tfrac{1}{2}-a) \log\frac{T}{2\pi}  
+ c(a)
\end{equation}
with some computable complex constant $c(a)$ depending on $a$.\footnote{Actually, Garunk\v{s}tis \& Steuding \cite{garunkstissteuding:2010} obtained even a more precise asymptotic formula. They also give a heuristic explanation why the leading term $\log T$ vanishes in the case $a=\frac{1}{2}$.} For $a=0$, this asymptotic was already established by Conrey, Gosh \& Gonek \cite{conreygoshgonek:1988} and later refined by Fujii \cite{fujii:1994}. From the asymptotic extension \eqref{eq:asympIak} of $I_{a}(T)$, we can immediately deduce the following corollary.

\begin{corollary}
Let $a\in\C\setminus\{\frac{1}{2}\}$. Then, for every constant $0<c< \left|\frac{1}{2}-a\right|$, there are infinitely many non-trivial $a$-points $\rho_a = \beta_a + i\gamma_a$ of the Riemann zeta-function such that
\begin{equation}\label{zetaprime-large1}
\left|\zeta'(\rho_a) \right| \geq c \log |\gamma_a|.
\end{equation}
\end{corollary}
\begin{proof}
Let $a\in\C\setminus\{\frac{1}{2}\}$ and $0<c< \left|\frac{1}{2}-a\right|$. Assume that for all but finitely many non-trivial $a$-points $\rho_a$, we have
$$
\left|\zeta'(\rho_a) \right| < c \log |\gamma_a|.
$$
Then, as $T\rightarrow\infty$,
\begin{equation}
|I_{a}(T)|\leq \frac{1}{N_a(T)} \sum_{0<\gamma_a \leq T} \left| \zeta'(\rho_a) \right| < c \log T + o(1).
\end{equation}
This is, however, in contradiction to \eqref{eq:asympIak}.
\end{proof}

By Lehto's criterion (Theorem \ref{th:lehto}), we are now able to show the following.
\begin{theorem}\label{th:avalues-case1-zeta}
For $a\in\C\setminus\{\frac{1}{2}\}$, there exists a sequence $(\rho_{a,k})_k$ of non-trivial $a$-points of the Riemann zeta-function with ordinates $\gamma_{a,k}>0$ such that the discs defined by
$$
|s-\rho_{a,k}|< \frac{\mu(\gamma_{a,k})}{ \log \gamma_{a,k} },\qquad k\in\N,
$$
with any positive function $\mu:[2,\infty)\rightarrow\R^+$ satisfying $\lim_{t\rightarrow\infty}\mu(t)=\infty$, form a sequence of filling discs for $\zeta$.
\end{theorem}

\begin{proof}[Proof of Theorem \ref{th:avalues-case1-zeta}] Let $(\rho_{a,k})_k$ be a sequence of non-trivial $a$-points with positive imaginary parts such that \eqref{zetaprime-large1} holds. Then, for any positive function with $\lim_{t\rightarrow\infty}\mu(t)$, we have
$$
\lim_{k\rightarrow\infty} \frac{\mu(\gamma_{a,k})}{ \log \gamma_{a,k} } \zeta^{\#}(\rho_{a,k}) = \lim_{k\rightarrow\infty}  \frac{\mu(\gamma_{a,k})}{ \log \gamma_{a,k} } \frac{|\zeta'(\rho_{a,k})|}{1+|a|^2} = \infty.
$$
The theorem follows by Lehto's criterion (Theorem \ref{th:lehto}).
\end{proof}
As far as the author knows, non-trivial conditional or unconditional asymptotic estimates for $J_{a,k}(T)$ are only known in the case  $a= 0$.\par

{\bf Discrete moments of $\zeta'(s)$ with respect to non-trivial zeros.}
Gonek \cite{gonek:1989} and Heijhal \cite{hejhal:1989} conjectured independently that $J_{0,k}(T)\asymp (\log T) ^{k^2 +2k}$ for $k\in\R$, as $T\rightarrow\infty$. Hughes, Keating \& O'Connell \cite{hugheskeatingoconnell:2000} conjectured that $J_{0,k}(T)\sim C_k (\log T) ^{k^2 +2k}$ for $k>-3/2$, as $T\rightarrow\infty$, with an explicit computable constant $C_k$ derived from models for the Riemann zeta-function by random matrix theory. Assuming the Riemann hypothesis, Gonek \cite{gonek:1984} obtained that $J_{0,1}\sim \frac{1}{12}(\log T)^3$ and Ng \cite{ng:2004} that $J_{0,2}\asymp (\log T)^8$. 
Using a method developed by Rudnick \& Soundararajan \cite{rudnicksoundararajan:2005} and assuming the generalized Riemann hypothesis (GRH) for Dirichlet $L$-functions, Milinovich \& Ng \cite{milinovichng:2013} could derive lower bounds for the discrete moments $J_{0,k}(T)$. For $k\in\mathbb{N}$, as $T\rightarrow\infty$,
\begin{equation}\label{eq:asym_JkT}
J_{0,k}(T)=\frac{1}{N(T)} \sum_{0<\gamma\leq T} \left|\zeta'(\rho) \right|^{2k} \gg (\log T) ^{k^2 +2k}. \;\;\; 
\end{equation}
where $N(T):=N_0(T)$ denotes, as usual, the number of non-trivial zeros of $\zeta$ with imaginary part $0<\gamma\leq T$. The assumption of the GRH is required to derive a suitable asymptotic formula for sums of the shape
$$
\sum_{0<\gamma\leq T} \zeta'(\rho)A_X(\rho)^{k-1}A_X(1-\rho)^k 
$$
with $A_X(s)=\sum_{n\leq X}n^{-s}$. Originally, Milinovich \cite{milinovich:phd} deduced this asymptotic formula from the main theorem of Ng \cite{ng:2007} where it was enough to assume the Riemann hypothesis. Later, they discovered a serious mistake in the error term of the asymptotic expansion in the main theorem of Ng \cite{ng:2007}. However, as they indicate in \cite{milinovichng:2013}, it appears that this mistake may be fixed with additional effort. Under the assumption of the Riemann hypothesis, Milinovich \cite{milinovich:2010} could bound $J_{0,k}(T)$ from above by $\ll (\log T) ^{k^2 +2k+\eps}$ with an arbitrary $\eps>0$. This upper bound together with the lower bound \eqref{eq:asym_JkT} supports the conjecture by Gonek \cite{gonek:1989}, Heijhal \cite{hejhal:1989} and Hughes, Keating \& O'Connell \cite{hugheskeatingoconnell:2000}. The following corollary is an immediate consequence of Milinovich's \& Ng's lower bound \eqref{eq:asym_JkT}.

\begin{corollary}\label{cor:lowerboundzetaprimerho}
Assume the GRH for Dirichlet $L$-functions. For $\delta>0$, let $A$ be the set of all non-trivial zeros $\rho$ of the Riemann zeta function satisfying
$$
\left|\zeta'(\rho)\right| \geq (\log |\gamma|)^{\delta}.
$$
Then, for any $\eta>0$, as $T\rightarrow\infty$,
$N_{A}(T) \gg T^{1-\eta}.$ \,  \footnote{The function $N_A(T)$ counts all elements in $A$ with positive imaginary part less or equal to $T$. For a precise definition, we refer to Section \ref{sec:fillingdiscs_basics}.}
\end{corollary}

\begin{proof} For a given $\delta>0$, we choose $K\in\N$ such that $K^2 + 2K > 2K\delta$. We denote by $A$ the set of all non-trivial zeros $\rho$ of the Riemann zeta function satisfying $\left|\zeta'(\rho)\right| \geq (\log |\gamma|)^{\delta}
$. Analogously, we denote by $B$ the set of all non-trivial zeros $\rho$ for which $\left|\zeta'(\rho)\right| < (\log |\gamma|)^{\delta}$ is true. Assume that there exists a real number $\eta>0$ such that $N_A(T)\ll T^{1-\eta}$, as $T\rightarrow\infty$. Then, it follows from the Riemann-von Mangoldt formula that $N_B(T)\sim N(T)$, as $T\rightarrow\infty$. Under the assumption of the Riemann hypothesis, we know that all non-trivial zeros are on the line $\sigma=\frac{1}{2}$ and that $\zeta(\frac{1}{2}+it)\ll t^{\eps}$ for any $\eps>0$, as $t\rightarrow\infty$. Taking $\eps=\eta/2K$ and putting everything together, we obtain that, as $T\rightarrow\infty$,
\begin{align*}
J_{0,K}(T)=&\frac{1}{N(T)} \sum_{0<\gamma\leq T} \left|\zeta'(\rho) \right|^{2K} \\
& \ll \frac{1}{N(T)} \left( N_A (T)\cdot T^{2K \cdot \frac{\eta}{2K}} + N_{B}(T)\cdot (\log T)^{2K\delta}  \right)\\
& \ll (\log T)^{2K\delta}.
\end{align*}
By our choice of $K$, this is in contradiction to \eqref{eq:asym_JkT}. The assertion is proved.
\end{proof}

By Lehto's criterion (Theorem \ref{th:lehto}), we are now able to show the following.
\begin{theorem}\label{th:avalues-case3}
Assume the GRH for Dirichlet $L$-functions. Then, for every $\delta> 0$, there is a set $A$ of non-trivial zeros $\rho = \frac{1}{2}+i\gamma$ of the Riemann zeta-function with $N_A(T)\gg T^{1-\eta}$ for any $\eta>0$, as $T\rightarrow\infty$, such that the discs defined by 
$$
|s-\rho|<\frac{1}{(\log|\gamma|)^{\delta}}, \qquad \rho\in A ,
$$ 
form a sequence of filling discs for $\zeta(s)$. 
\end{theorem}
\begin{proof}
For a given $\delta>0$, we define $A$ to be the set of all non-trivial zeros $\rho=\frac{1}{2}+i\gamma$ of the Riemann zeta-function satisfying
$$
 |\zeta'(\tfrac{1}{2}+i\gamma)| \geq (\log|\gamma|)^{\delta + 1}.
$$
Then, Corollary \ref{cor:lowerboundzetaprimerho} assures that $N_{A}(T)\gg T^{1-\eta}$ for any $\eta>0$, as $T\rightarrow\infty$. Moreover,
$$
\lim_{\begin{subarray}{c} \, |\gamma|\rightarrow\infty \\ \frac{1}{2}+i\gamma\in A \end{subarray}} \frac{1}{(\log |\gamma|)^{\delta}} \, \zeta^{\#}(\tfrac{1}{2}+i\gamma) 
= \lim_{\begin{subarray}{c} \, |\gamma|\rightarrow\infty \\ \frac{1}{2}+i\gamma\in A \end{subarray}} \frac{1}{(\log |\gamma|)^{\delta}} \, \left|\zeta'(\tfrac{1}{2}+i\gamma) \right| = \infty.
$$
and the theorem follows by Lehto's criterion (Theorem \ref{th:lehto}).
\end{proof}

{\bf Soundararajan's resonance method.} A resonance method developed by Sounda\-rara\-jan \cite{soundararajan:2008} yields another approach to retrieve information about large values of $\zeta'(\rho)$. Ng \cite{ng:2008} showed that, under assumption of the GRH for Dirichlet $L$-functions, there are infinitely many non-trivial zeros such that
$$
\zeta'(\rho) \gg \exp \left(c_0 \frac{\log |\gamma|}{\log\log |\gamma|} \right)
$$ 
where $c_0 = \frac{1}{\sqrt{2}} - \eps$ with any $\eps>0$. By Lehto's criterion, we deduce the following.
\begin{theorem}\label{th:resonance}
Assume the GRH for Dirichlet $L$-functions. Then, there exists a sequence $(\rho_k)_k$ of non-trivial zeros $\rho_k = \frac{1}{2}+i\gamma_k$ of the Riemann zeta-function with $\gamma_k>0$ such that the discs defined by 
$$
|s-\rho_k|<\mu(\gamma_k)\exp \left(-c_0 \frac{\log \gamma_k}{\log\log \gamma_k} \right), \qquad k\in\N,
$$ 
form a sequence of filling discs for $\zeta(s)$, where $\mu:[2,\infty)\rightarrow\R^+$ is any positive function satisfying $\lim_{t\rightarrow\infty}\mu(t)=\infty$.
\end{theorem}
\begin{proof}
According to Ng \cite{ng:2008}, for any fixed $0<c_0<\sqrt{2}$, there exist a constant $C>0$ and a sequence $(\rho_k)_k$ of non-trivial zeros $\rho_k = \frac{1}{2}+i\gamma_k$ of the Riemann zeta-function with $\gamma_k>0$ such that
$$
\left| \zeta'(\rho_k) \right| \geq C \exp \left(c_0 \frac{\log \gamma_k}{\log\log \gamma_k} \right)
$$
for every $k\in\N$. Hence, for any positive function $\mu$ satisfying $\lim_{t\rightarrow\infty}\mu(t)=\infty$,
\begin{align*}
&\lim_{k\rightarrow\infty}\, \mu(\gamma_k)\exp \left(-c_0 \frac{\log |\gamma_k|}{\log\log |\gamma_k|} \right)\, \zeta^{\#}(\rho_k)\\
&= \lim_{k\rightarrow\infty}\, \mu(\gamma_k)\exp \left(-c_0 \frac{\log |\gamma_k|}{\log\log |\gamma_k|} \right)\, |\zeta'(\rho_k)|\\
&=\infty.
\end{align*}
The assertion follows by Lehto's criterion (Theorem \ref{th:lehto}).
\end{proof}

\section{Non-trivial a-points to the left of the critical line}
According to Selberg's $a$-point conjecture, we expect that, for any $a\in\C\setminus\{0\}$, about $3/4$-th of the non-trivial $a$-points of the Riemann zeta-function lie to the left of the critical line. However, unconditionally, it is not even known whether there are {\it infinitely many} non-trivial $a$-points to the left of the critical line. This motivated Steuding to raise the following question at the problem session of the 2011 Palanga conference in honour of Jonas Kubilius. Is it possible to find for every $a\in\C$ a sequence of points $s_k=\sigma_k + it_k$, $k\in\N$, with $\sigma_k<\frac{1}{2}$ and $\lim_{k\rightarrow\infty} t_k = \infty$ such that
$$
\lim_{k\rightarrow\infty} \zeta(s_k) = a.
$$
By means of our results in Section \ref{sec:largesmall} concerning small and large values in a neighbourhood of the critical line, we can give a very first answer to this question.
\begin{corollary}\label{cor:apointsleft}
Let $c>0$ and let $\epsilon:[2,\infty)\rightarrow\R$ be a function satisfying
$$
0<\epsilon(t)\leq \frac{c}{ \log t}\qquad \mbox{for }t\geq 2.
$$ 
Then, for every $\alpha>0$, there exists an $a\in\C$ with $|a|=\alpha$ and a sequence $(t_k)$ with $t_k\in[1,\infty)$ such that
$$
\lim_{k\rightarrow\infty} \zeta(\tfrac{1}{2}-\epsilon(t_k)+it_k) = a.
$$
\end{corollary}
\begin{proof} 
The case $\alpha=0$ follows directly from Corollary \ref{cor:selbergsmalllarge}.
Thus, suppose that $\alpha>0$. If we take $m=\alpha+1$ in Corollary \ref{cor:selbergsmalllarge}, the intermediate value theorem assures the existence of a sequence $(t_k)_k$ with $t_k\in[1,\infty)$ and $\lim_{k\rightarrow\infty} t_k = \infty$ such that
$$
\zeta(\tfrac{1}{2}-\epsilon(t_k)+it_k) \in \partial D_{\alpha}(0)
$$
for all $k\in\N$. As the circle $\partial D_{\alpha}(0)$ is compact, the set $\{\zeta(\tfrac{1}{2}-\epsilon(t_k)+it_k)\}_{k\in\N}\subset \partial D_{\alpha}(0)$ has at least one accumulation point $a\in \partial D_{\alpha}(0)$. Thus, there is a subsequence $(t_{k_j})_j$ of $(t_k)_k$ such that 
$$
\lim_{j\rightarrow\infty} \zeta(\tfrac{1}{2}-\epsilon(t_{k_j})+it_{k_j}) = a.
$$
\end{proof}

\section{Summary: a-points of the Riemann zeta-function near the critical line}
Let $\lambda$ be a positive function on $\R$ with $\lim_{t\rightarrow\infty} \lambda(t)=0$ and $S_{\lambda}$ be the region defined by
$$
 \tfrac{1}{2} - \lambda(t) <\sigma < \tfrac{1}{2} + \lambda(t), \qquad t\geq 2.
$$
In the following corollary, we summarize our knowledge on how fast $\lambda(t)$ can tend to zero, as $t\rightarrow\infty$, such that, for a given $a\in\C$, there are almost all, a positive proportion, resp. infinitely many of all non-trivial $a$-points of the Riemann zeta-function inside $S_{\lambda}$. For more detailed information the reader is referred to the corresponding theorems of the preceeding sections.
\begin{corollary}[Levinson, Selberg, Tsang, Christ]\label{cor:summaryapoints}
$\mbox{ }$
\begin{itemize}
 \item[(a)] Unconditionally, for every $a\in\C$, there are almost all $a$-points of the Riemann zeta-function (in the sense of density) inside the region $S_{\lambda}$ if one chooses
$$
\lambda(t)=\frac{\mu(t)\sqrt{\log\log t}}{\log t}, \qquad t\geq 2,
$$
with any positive function $\mu$ satisfying $\lim_{t\rightarrow\infty}\mu(t)=\infty$.\footnote{This follows from Theorem \ref{th:levinsonselberg} which was proved in the framework of the class $\Sc^*$ by refining a method developed by Levinson \cite{levinson:1975} with results of Selberg \cite{selberg:1992}, Tsang \cite{tsang:1984} and Steuding \cite{steuding:2007}.}
\item[(b)] Under the assumption of the Riemann hypothesis, for every $a\in\C$, there is a positive proportion of all $a$-points of the Riemann zeta-function (in the sense of density) inside the region $S_{\lambda}$ if one chooses
$$
\lambda(t)=\frac{c\sqrt{\log\log t}}{\log t}, \qquad t\geq 2,
$$
with any constant $c>0$.\footnote{This is due to Selberg \cite{selberg:1992}; see Section \ref{sec:apointsgeneral}.}
 \item[(c)] Unconditionally, for every $a\in\C$, with at most one exception, there are infinitely many $a$-points of the Riemann zeta-function inside the region $S_{\lambda}$ if one chooses
$$
\lambda(t)=\frac{\mu(t)}{\log t}, \qquad t\geq 2,
$$
with any positive function $\mu$ satisfying $\lim_{t\rightarrow\infty}\mu(t)=\infty$.\footnote{This follows from Theorem \ref{th:avalues-case1} which was proved in the framework of the class $\mathcal{G}$ by relying on normality arguments.}\\
Moreover, one knows unconditionally that almost all zeros of the Riemann zeta-function (in the sense of density) lie in this region.\footnote{This is due to Selberg \cite{selberg:1946}; see Theorem \ref{th:selbergzero}.}
\item[(d)] Under the assumption of the Riemann hypothesis, for every $a\in\C$, there are infinitely many $a$-points of the Riemann zeta-function inside the region $S_{\lambda}$ if one chooses
$$
\lambda(t)=\frac{\mu(t)(\log\log\log t)^3}{\log t \sqrt{\log\log t}}, \qquad t\geq 2,
$$
with any positive function $\mu$ satisfying $\lim_{t\rightarrow\infty}\mu(t)=\infty$.\footnote{This is due to Selberg \cite{selberg:1992}; see Section \ref{sec:apointsgeneral}.}
\item[(e)] Under the assumption of the GRH for Dirichlet $L$-functions,\footnote{Very likely this can be also obtained by only assuming RH; see the remarks in Section \ref{sec:filling_zeta}.} for every $a\in\C$, with at most one exception, there are infinitely many $a$-points of the Riemann zeta-function inside the region $S_{\lambda}$ if one chooses
$$
\lambda(t)= \frac{1}{(\log t)^{\delta}}, \qquad t\geq 2,
$$
with any $\delta>0$.\footnote{This follows from Theorem \ref{th:avalues-case3}which was deduced by certain normality arguments from a result of Milinovich \& Ng \cite{milinovichng:2013}.}
 \item[(f)] Under the assumption of the GRH for Dirichlet $L$-functions, for every $a\in\C$, with at most one exception, there are infinitely many $a$-points of the Riemann zeta-function inside the region $S_{\lambda}$ if one chooses
$$
\lambda(t) = \mu(t)\exp \left(-c_0 \frac{\log t}{\log\log t} \right), \qquad t\geq 2
$$
with any positive function $\mu$ satisfying $\lim_{t\rightarrow\infty}\mu(t)=\infty$ and any constant $0<c_0 < \frac{1}{\sqrt{2}}$.\footnote{This follows from Theorem \ref{th:resonance} which was deduced by certain normality arguments from a result of Ng \cite{ng:2008}.}
\end{itemize}
\end{corollary}

\chapter{Denseness results for the Riemann zeta-function in the critical strip}\label{ch:curve}
The critical line is a natural boundary for the universality property of the Riemann zeta-function. Even if we slightly change the concept, the functional equation strongly restricts the set of approximable functions; see Chapter \ref{ch:conceptsuniv}. 

In this chapter we investigate a concept of universality for the Riemann zeta-function on the critical line that is significantly weaker than the one regarded in Chapter \ref{ch:conceptsuniv}. Roughly speaking, we set the scaling factor of the limiting process introduced in Section \ref{sec:shiftingshrinking} to be constantly equal to zero and drop our request to approximate analytic functions, but restrict to the approximation of complex points $a\in\C$ by shifts of the Riemann zeta-function on the critical line.\par 

We know that, for every $\alpha\in[0,\infty)$, there exist a sequence $(\tau_k)_k$ of real numbers $\tau_k\in[2,\infty)$ with $\lim_{k\rightarrow\infty}\tau_k=\infty$ such that
$$
\lim_{k\rightarrow\infty} \left| \zeta(\tfrac{1}{2}+i\tau_k) \right| = \alpha;
$$
see Section \ref{subsec:summaryunboundedness}. However, except for $a=0$, we do not know explicitly any other example of an accumulation point $a\in\C$ of the set $$V(\tfrac{1}{2}):=\{\zeta(\tfrac{1}{2}+it) \, : \, t\in[1,\infty)\}.$$ According to a conjecture of Ra\-ma\-chandra, we expect that 
$$
\overline{V(\tfrac{1}{2})} =\C,
$$ 
i.e., we conjecture that the values of the Riemann zeta-function on the critical lie dense in $\C$. However, to prove or disprove this conjecture with present day's methods seems to be out of reach.\par

Bohr \cite{bohrcourant:1914, bohr:1915, bohrjessen:1930, bohrjessen:1932} and his collaborators established denseness results for the zeta-values on vertical lines inside the strip $\frac{1}{2}<\sigma\leq 1$. It were their pioneering works that directed Voronin towards his universality theorem. In Section \ref{sec:bohr}, we state the main results of Bohr and his collaborators on the value-distribution of the Riemann zeta-function to the right of the critical line. We briefly report on their methods and discuss why these methods fail to obtain a denseness statement for the zeta-values on the critical line.\par

In Section \ref{sec:qualdiff}, we point out that a denseness result for the critical line is qualitatively different from the one on vertical lines in $\frac{1}{2}<\sigma\leq 1$. \par

We know that $0\in V(\tfrac{1}{2})$. According to Ramachandra's conjecture, we expect that $0$ is in particular an interior point of $\overline{V(\tfrac{1}{2})}$. In Section \ref{sec:appr}, we show that there is a subinterval $A\subset [0,2\pi)$ of length at least $\frac{\pi}{4}$ such that, for every $\theta\in A$, there is a sequence $(t_n)_n$ of numbers $t_n\in[2,\infty)$ with 
$$
\zeta(\tfrac{1}{2}+it_n)\neq 0, \qquad \lim_{n\rightarrow\infty} \zeta(\tfrac{1}{2}+it_n) = 0 \qquad \mbox{ and } \qquad
\arg \zeta(\tfrac{1}{2}+it_n) \equiv \theta \mod 2\pi.
$$  

In Section \ref{sec:curves}, we approach Ramachandra's conjecture by asking whether there are curves $[1,\infty)\ni t \mapsto \tfrac{1}{2}+\epsilon(t)+it$ with $\lim_{t\rightarrow\infty}\epsilon(t)=0$ such that the values of the Riemann zeta-function on these curves lie dense in $\C$. We obtain some positive answers both by relying on the $a$-point results of Section \ref{sec:apointsnormality} and by relying on Bohr's work.\par

In Section \ref{sec:densesigma1}, we see that the latter question is much easier to handle if we regard curves which do not approach the left but the right boundary line of the strip of universality, i.e. the line $\sigma=1$. \par

Finally, in Section \ref{sec:universalityoncurves}, we briefly indicate what happens to the limiting process of Section \ref{sec:shiftingshrinking} if we adjust the underlying conformal mappings such that they map $\D$ to discs which lie completely inside the strip $\frac{1}{2}<\sigma<1$, but arbitrarily close to the critical line. \par

Although we mainly restrict to the Riemann zeta-function in this chapter, most of the results are true for related functions from quite general classes.

\section{The works of Bohr and Voronin}\label{sec:bohr}
At the beginning of the 20th century, Bohr and his collaborators studied the value-distribution of the Riemann zeta-function to the right of the critical line.

{\bf Value-distribution on vertical lines in $\sigma>1$.} Due to the pole of the Riemann zeta-function at $s=1$, the characteristic convergence abscissae of the Dirichlet series expansion of the Riemann zeta-function are given by  $\sigma_c=\sigma_u=\sigma_a=1$. Thus, according to Bohr \cite{bohr:1922}, the behaviour of the zeta-function on vertical lines in the half-plane $\sigma>1$ is ruled by almost periodicity; see Theorem \ref{th:almostperiod}. For every fixed $\sigma>1$, we know that
$$
\frac{\zeta(2\sigma)}{\zeta(\sigma)} \leq |\zeta(\sigma+it)| \leq \zeta(\sigma) \qquad \mbox{ for }t\in\R.
$$
and that these inequalities are sharp; see Apostol \cite[Chapt. 7.6]{apostol:1990}. As
$$
\lim_{\sigma\rightarrow\infty} \frac{\zeta(2\sigma)}{\zeta(\sigma)} = \lim_{\sigma\rightarrow\infty} \zeta(\sigma) = 1, \qquad \lim_{\sigma\rightarrow 1+} \frac{\zeta(2\sigma)}{\zeta(\sigma)} = 0 \qquad \mbox{ and }\qquad \lim_{\sigma\rightarrow 1+} \zeta(\sigma) = \infty,
$$ 
the set
$$
V(\sigma):=\left\{ \zeta(\sigma+it) \, : \, t\in[1,\infty) \right\}
$$
contracts to $1$, as $\sigma\rightarrow\infty$, and contains both arbitrarily small and arbitrarily large values, as $\sigma\rightarrow 1$. Beyond this observation, we know that, for every $\sigma>1$, the set 
$$
\left\{\log\zeta(\sigma + it) \, : \, t\in\R \right\}
$$
lies dense in an area of $\C$ which is either simply connected and bounded by a convex curve or which is ring-shaped and bounded by two convex curves; see Bohr \& Jessen \cite{bohrjessen:1930} or Titchmarsh \cite[\S 11.6]{titchmarsh:1986}. Essential ingredients in the proof are the Euler product representation of the Riemann zeta-function and diophantine approximation.\par

{\bf Bohr's denseness results on vertical lines in $\frac{1}{2}<\sigma\leq 1$.} Bohr \& Courant \cite{bohrcourant:1914} proved that the values taken by the Riemann zeta-function on an arbitrary vertical line inside the strip $\frac{1}{2}<\sigma \leq 1$ form a dense set in $\mathbb{C}$, i.e.
$$
\overline{V(\sigma)} = \C \qquad \mbox{ for every }\sigma\in(\tfrac{1}{2},1].
$$
In fact, their proof yields the stronger statement that, for every $\sigma\in(\frac{1}{2},1]$, every $a\in\C$ and every $\eps>0$,
\begin{equation}\label{eq:denseBohr}
\liminf_{T\rightarrow\infty} \frac{1}{T} \meas\left\{ t\in[1,\infty) \, : \, \left|\zeta(\sigma+it)-a \right|<\eps \right\}>0.
\end{equation}
Even more precise results were obtained by Bohr \& Jessen \cite{bohrjessen:1932} and Laurin\v{c}ikas \cite[Chapt. 4]{laurincikas:1991-2}, who established probabilistic limit theorems for the values of the logarithm of the zeta-function on vertical lines. In particular, Laurin\v{c}ikas \cite[Chapt. 4, Theorem 4.1]{laurincikas:1991-2} showed that, for every $\sigma>\frac{1}{2}$, there exists a Borel probability measure $\mu_{\sigma}$ such that, for every continuous and bounded function $f:\C\rightarrow\C$,
$$
\lim_{T\rightarrow\infty} \ \frac{1}{2T} \int_{-T}^{T} f\bigl( \log \zeta(\sigma+it)\bigr) \d t = \int_{\C} f(z) \d\mu(z).
$$
If $\sigma\in(\frac{1}{2},1]$ the support of $\mu_{\sigma}$ is the whole complex plane.\par

{\bf Voronin's denseness results on vertical lines in $\frac{1}{2}<\sigma\leq 1$.} Voronin \cite{voronin:1972} established multidimensional extensions of Bohr's denseness result. For any $n\in\N_0$ and any function $f\in\mathcal{H}(\Omega)$, we define
$$
\Delta_n f(s):=\left( f (s), f'(s),..., f^{(n)}(s)\right)
$$ 
to be the $(n+1)$-dimensional vector consisting of the values of $f$ and its first $n$ derivatives evaluated at the point $s\in\Omega$. Among other things, Voronin \cite{voronin:1972} obtained that
\begin{equation}\label{eq:multdimVoronin}
\overline{\left\{ \Delta_n\zeta(\sigma+it) \, : \, t\in[1,\infty) \right\} }= \C^{n+1} \qquad \mbox{ for every }\sigma\in(\tfrac{1}{2},1] 
\end{equation}
and every $n\in\N_0$. By a slight refinement of Voronin's proof, one obtains that, for every $\sigma\in(\frac{1}{2},1]$, every $a\in\C^{n+1}$ and every $\eps>0$
\begin{equation}\label{eq:denseVoronin}
\liminf_{T\rightarrow \infty}\frac{1}{T}\meas\left\{t\in (0,T]: \left\| \Delta_n\zeta(\sigma+it) -a \right\|<\varepsilon \right\}  >0;
\end{equation}
here and in the following $\|\cdot\|$ denotes the maximum-norm in the complex vector space $\mathbb{C}^{n+1}$. It was this multidimensional extension of Bohr's result that inspired Voronin \cite{voronin:1975} to prove his universality theorem (Theorem \ref{th:universality}). In fact, for $\frac{1}{2}<\sigma<1$, the denseness results \eqref{eq:denseVoronin} and \eqref{eq:denseBohr} are direct consequences of Voronin's universality theorem.

{\bf Bohr's method.} To prove his denseness result, Bohr modeled the Riemann zeta-function by truncated Euler products
$$
\zeta_N(\sigma+it) = \prod_{p\leq N} \left(1-p^{-\sigma-it} \right)^{-1};
$$
here the product is taken over all prime numbers $p\in\mathbb{P}$ with $p\leq N$. We fix $\frac{1}{2}<\sigma\leq 1$. Bohr noticed that the quantities $p^{-it}=e^{-it \log p}$ with $p\in\mathbb{P}$, $p\leq N$, behave like independent random variables, although they all depend on the common variable $t$.\footnote{see Bohr \& Jessen \cite[p. 6]{bohrjessen:1930}.} Relying on the theory of convex curves and diophantine approximation, in particular a theorem of Kronecker and Weyl, he was able to prove that, for given $a\in\C$, there exists a large subset $A\subset[1,\infty)$ such that, for all $t\in A$, the truncated Euler product $\zeta_N(\sigma+it)$ is quite close to $a$. 
Although $\zeta_N(\sigma+it)$ does not converge in $\frac{1}{2}<\sigma\leq 1$, as $N\rightarrow\infty$, Bohr showed that the truncated Euler product approximates the zeta-function in mean-square. Here, the essential tool is the existence of the mean-square value
$$
\lim_{T\rightarrow\infty }\int_1^T \left|\zeta(\sigma+it)\right|^2 \d t < \infty
$$ 
for $\frac{1}{2}<\sigma\leq 1$ in combination with Carlson's theorem. From the approximation in mean-square, Bohr deduced that there exists a large subset $B\subset [1,\infty)$ such that, for all $t\in B$, the truncated Euler product $\zeta_N(\sigma+it)$ is close to $\zeta(\sigma+it)$. Finally, certain density estimates assure that $A\cap B \neq \emptyset$ and the result follows. Voronin obtained his denseness statement \eqref{eq:multdimVoronin} basically by refining the ideas of Bohr.

{\bf Non-denseness results on vertical lines in $\sigma<\frac{1}{2}$.} On vertical lines in the half-plane $\sigma<\frac{1}{2}$, we expect that the zeta-function grows too fast than that its values could lie dense in $\C$. It follows essentially from the functional equation that 
\begin{equation}\label{nondense}
\overline{V(\sigma)} \neq \C \qquad \mbox{ for }\sigma\leq 0;
\end{equation}
see Garunk\v{s}tis \& Steuding \cite{garunkstissteuding:2010}. By assuming the Riemann hypothesis, Garunk\v{s}tis \& Steuding \cite{garunkstissteuding:2010} proved that \eqref{nondense} persists for all $\sigma<\frac{1}{2}$. Their proof uses a conditional $\Omega$-result for the zeta-function on vertical lines in $\sigma>\frac{1}{2}$ together with the functional equation.\par 

We know that $\overline{V(\sigma)}=\C$ for $\frac{1}{2}<\sigma\leq 1$ and we expect that $\overline{V(\sigma)}\neq \C$ for $0<\sigma<\frac{1}{2}$. But how is the situation if $\sigma=\frac{1}{2}$?

{\bf Is $\overline{V(\frac{1}{2})}=\C$?} During the 1979 Durham conference, Ramachandra formulated the conjecture that the values of the Riemann zeta-function on the critical line lie dense in $\C$. Until now, this could not be proved or disproved. 
The difficulty in handling the values of the zeta-function on the critical line is that, due to Hardy \& Littlewood \cite{hardylittlewood:1936}, 
$$
\frac{1}{T} \int_{-T}^T \left|\zeta(\tfrac{1}{2}+it)\right|^2 dt \sim \log T,
$$ 
as $T\rightarrow\infty$. Hence, Bohr's method collapses on the critical line.\par

Jacod, Kowalski \& Nikeghbali \cite{jacodkowalksinikeghbali:2011} introduced a new type of convergence in probability theory, which they called `mod-Gaussian convergence'. For a given sequence $(Z_n)_n$ of real-valued random variables $Z_n$, this concept builds basically on working with the corresponding sequence of characteristic functions $(\mathbb{E}[e^{iuZ_n}])_n$. Leaning on this type of convergence, Kowalski \& Nikeghbali \cite{kowalskinikeghbali:2012} were able to show that $$\overline{V(\tfrac{1}{2})}=\C$$ would follow rather directly from a suitable version of the Keating-Snaith moment conjectures. By generalizing the notion of `mod-Gaussian convergence', Delbaen, Kowalski \& Nikeghbali \cite{delbaenkowalskinikeghbali:2011} obtained the following quantitative result.
\begin{theorem}[Delbaen, Kowalski \& Nikeghbali, 2011]
If for any $k>0$ there exist a real number $C_k\geq 0$ such that
\begin{equation}\label{eq:momentconjecturedenseness}
\left|\frac{1}{T}\int_0^T \exp\left(it\cdot \log\zeta(\tfrac{1}{2}+iu) \right)\d u \right|\leq
\frac{C_k}{1+|t|^4 (\log\log T)^2}
\end{equation}
holds for $T\geq 1$ and $t\in\R$ with $|t|\leq k$, then, for any bounded Jordan measurable subset $B\subset \C$,
$$
\lim_{T\rightarrow\infty} \frac{\frac{1}{2}\log\log T}{T} \,
\meas\left\{t\in (0,T] \, : \, \log\zeta(\tfrac{1}{2}+it)\in B \right\} = \tfrac{1}{2\pi} \meas B .
$$
\end{theorem}
Models for the zeta-function based on links to random matrix theory suggest that \eqref{eq:momentconjecturedenseness} should hold; see Keating \& Snaith \cite{keatingsnaith:2000-1, keatingsnaith:2000-2}.

\begin{figure}%
\centering
\subfigure[$V(\frac{3}{2})$]{
\centering
\includegraphics[width=0.3\columnwidth]{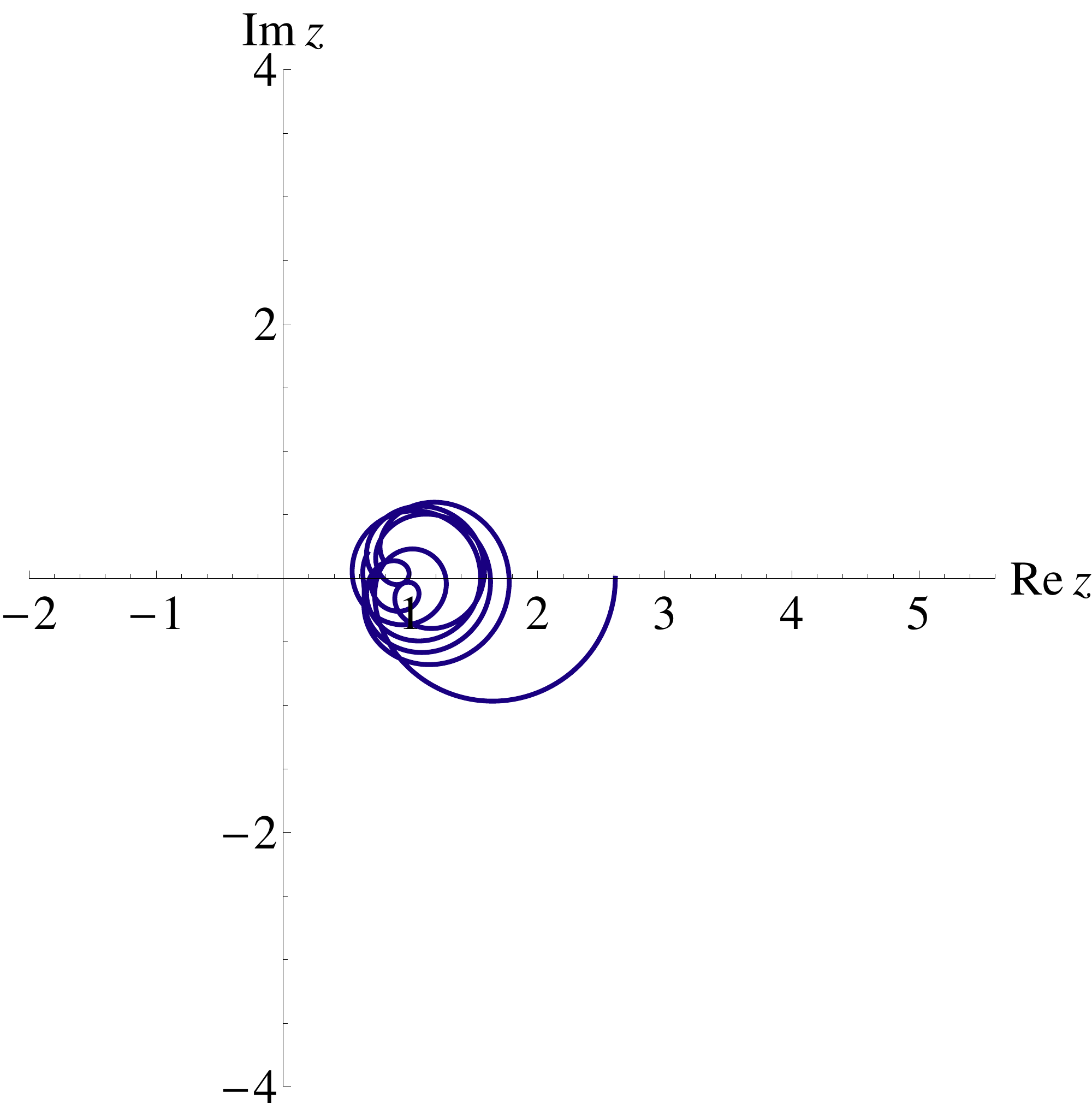}\label{fig:zeta32}%
}
\subfigure[$V(1)$]{
\centering
\includegraphics[width=0.3\columnwidth]{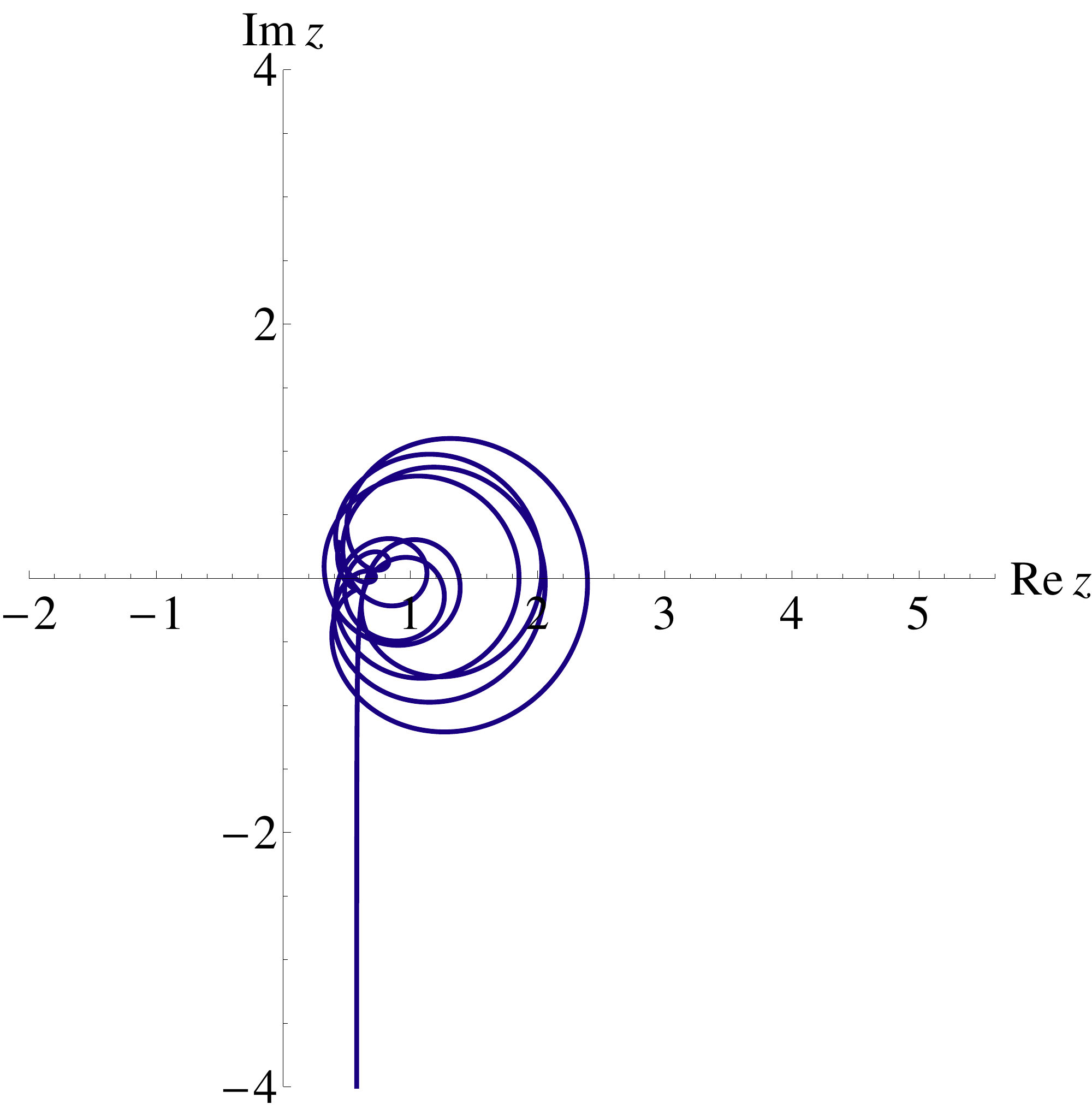}\label{fig:zeta88}%
}
\subfigure[$V(\frac{7}{8})$]{
\includegraphics[width=0.3\columnwidth]{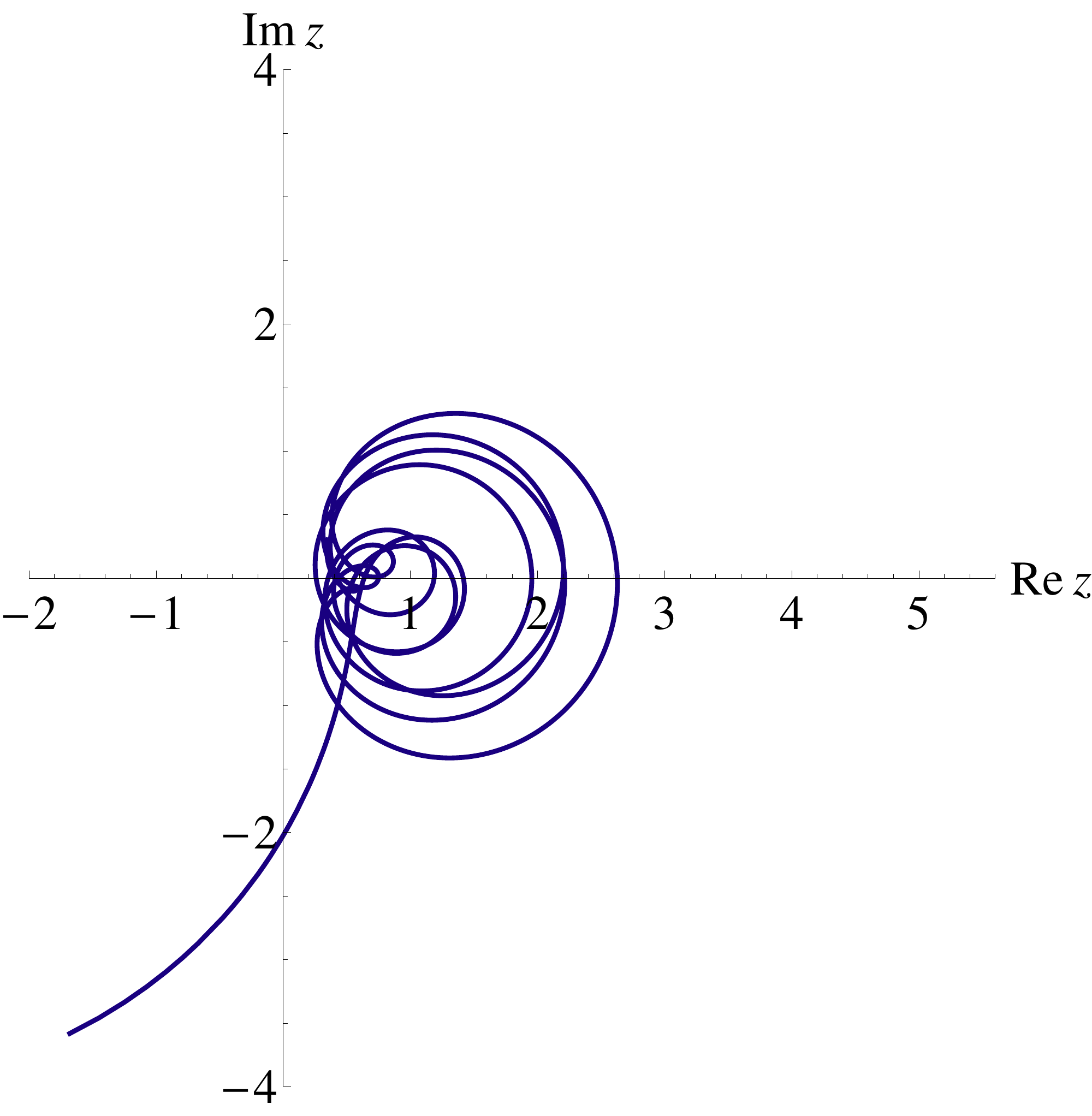}\label{fig:zeta78}%
}\\
\subfigure[$V(\frac{3}{4})$]{
\includegraphics[width=0.3\columnwidth]{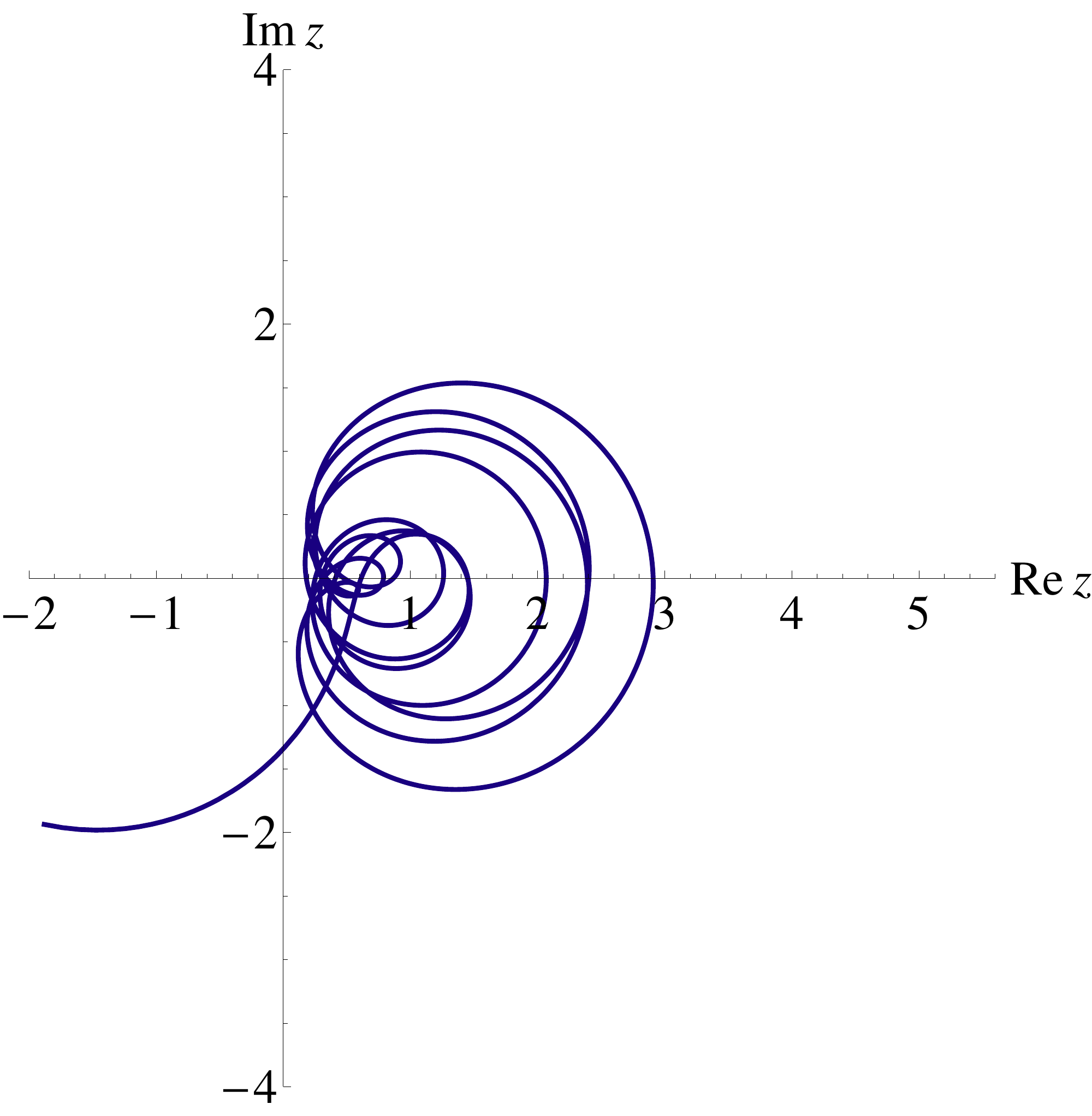}\label{fig:zeta68}%
}
\subfigure[$V(\frac{5}{8})$]{
\includegraphics[width=0.3\columnwidth]{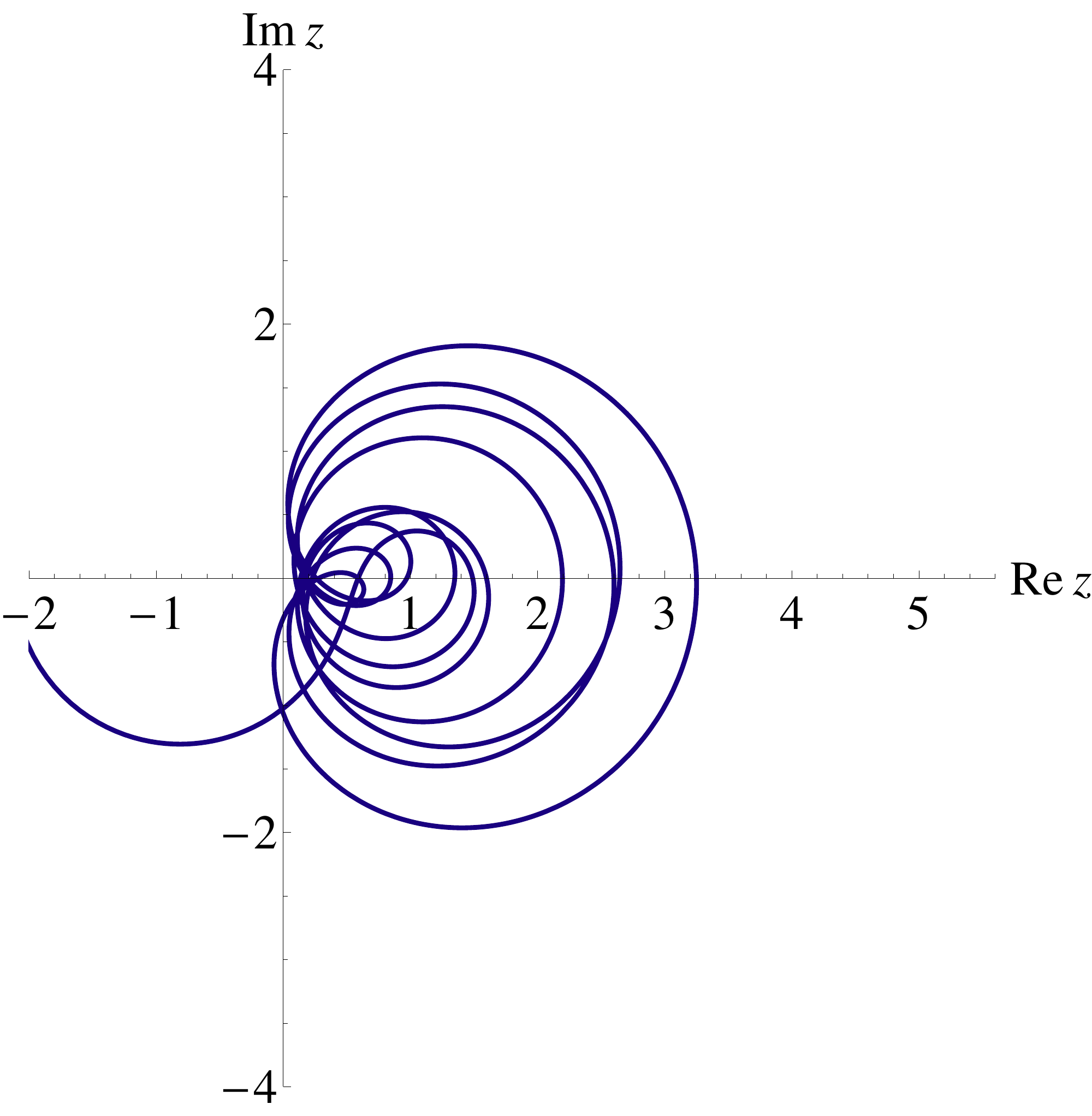}\label{fig:zeta58}%
}
\subfigure[$V(\frac{1}{2})$]{
\includegraphics[width=0.3\columnwidth]{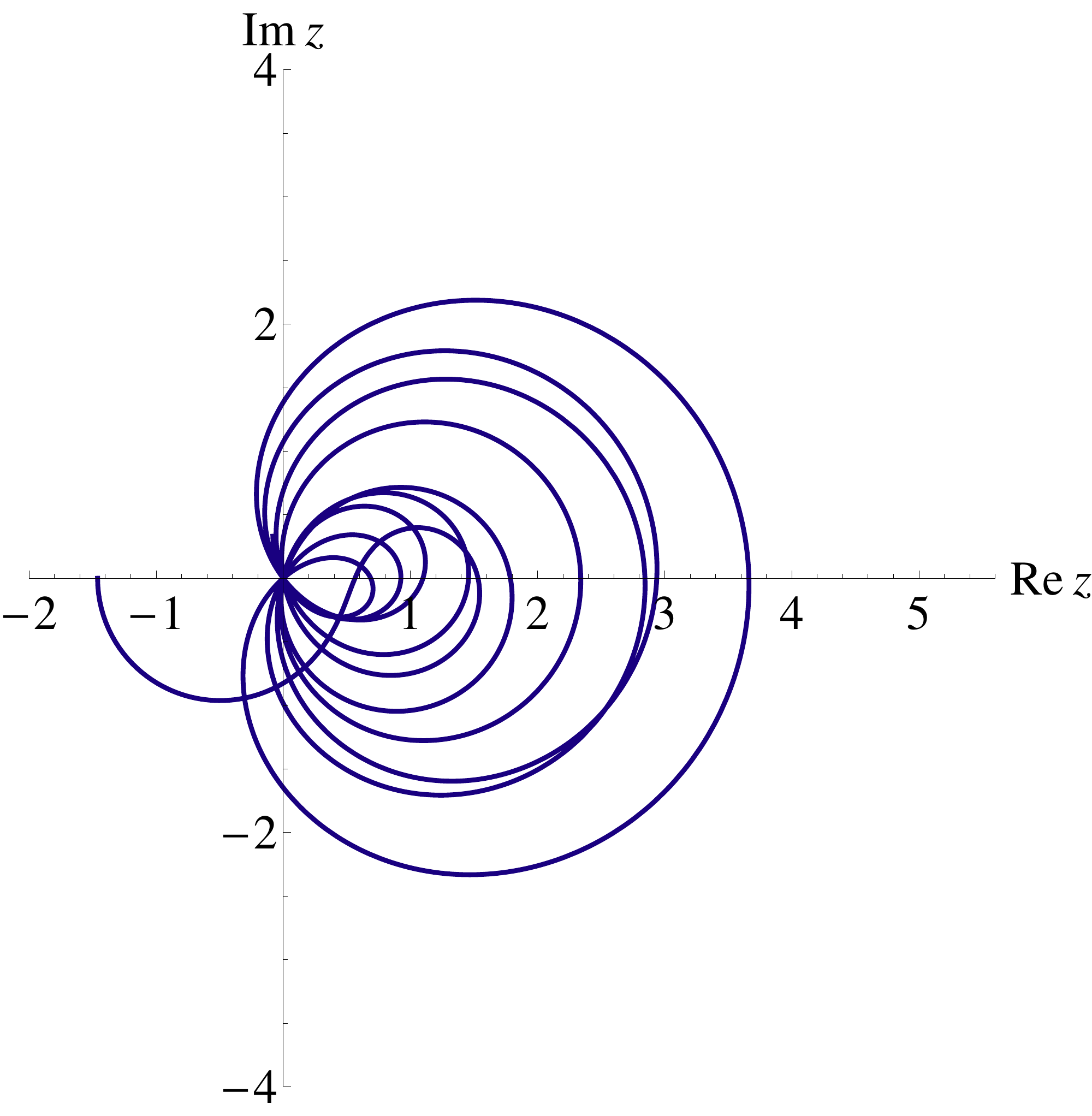}\label{fig:zeta48}%
}\\
\subfigure[$V(\frac{3}{8})$]{
\includegraphics[width=0.3\columnwidth]{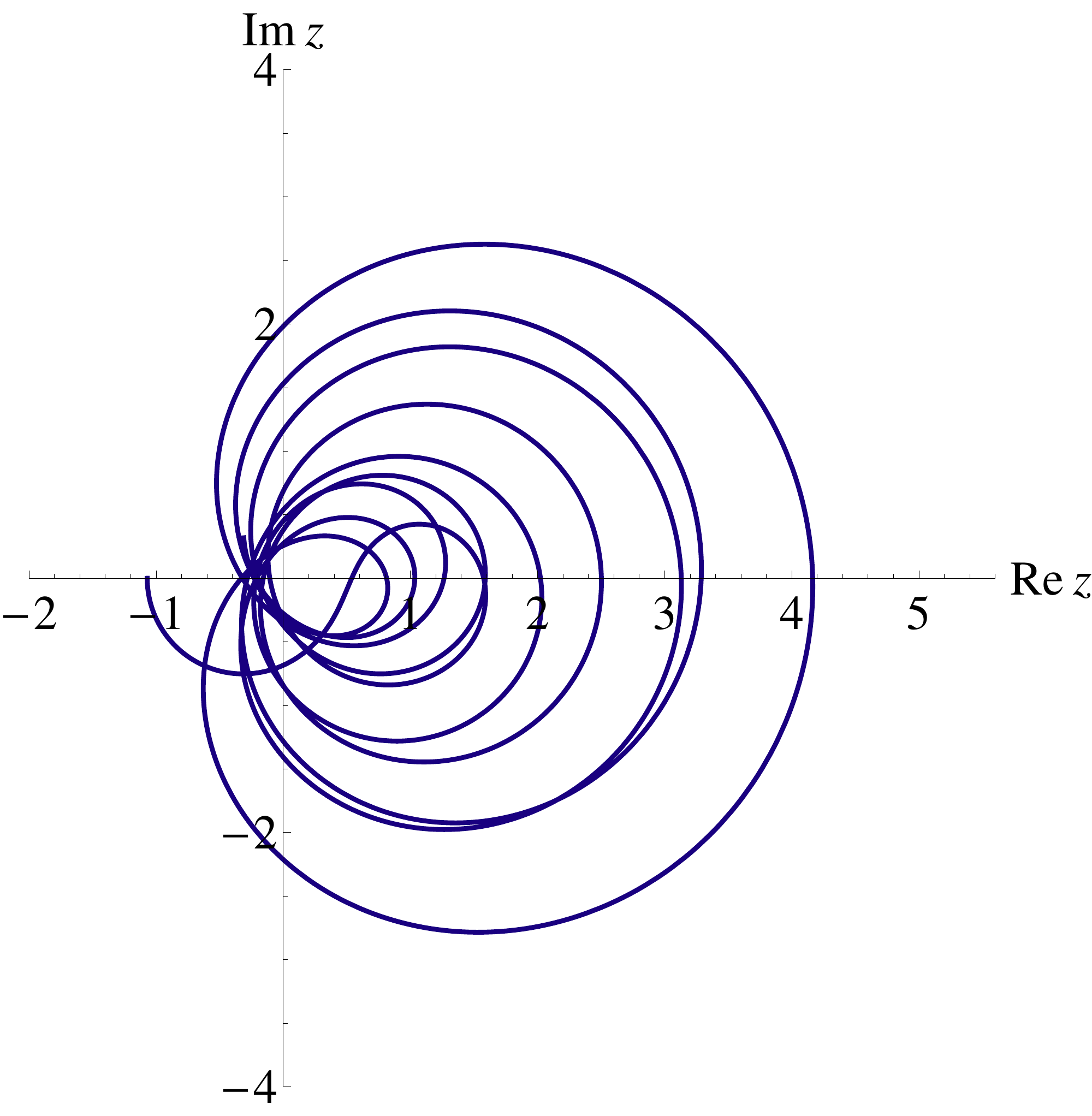}\label{fig:zeta38}%
}
\subfigure[$V(\frac{1}{4})$]{
\includegraphics[width=0.3\columnwidth]{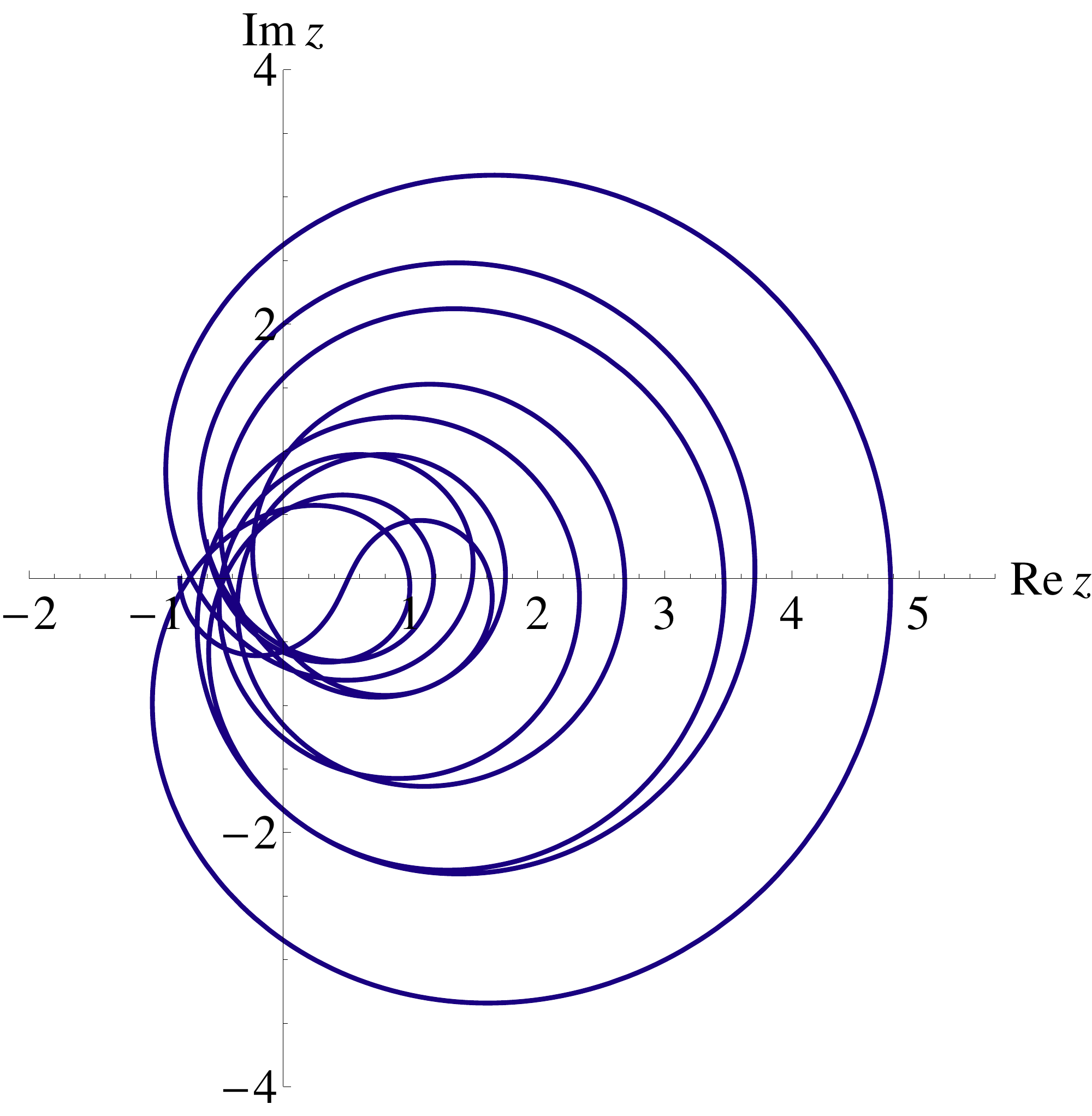}\label{fig:zeta28}%
}
\subfigure[$V(\frac{1}{8})$]{
\includegraphics[width=0.3\columnwidth]{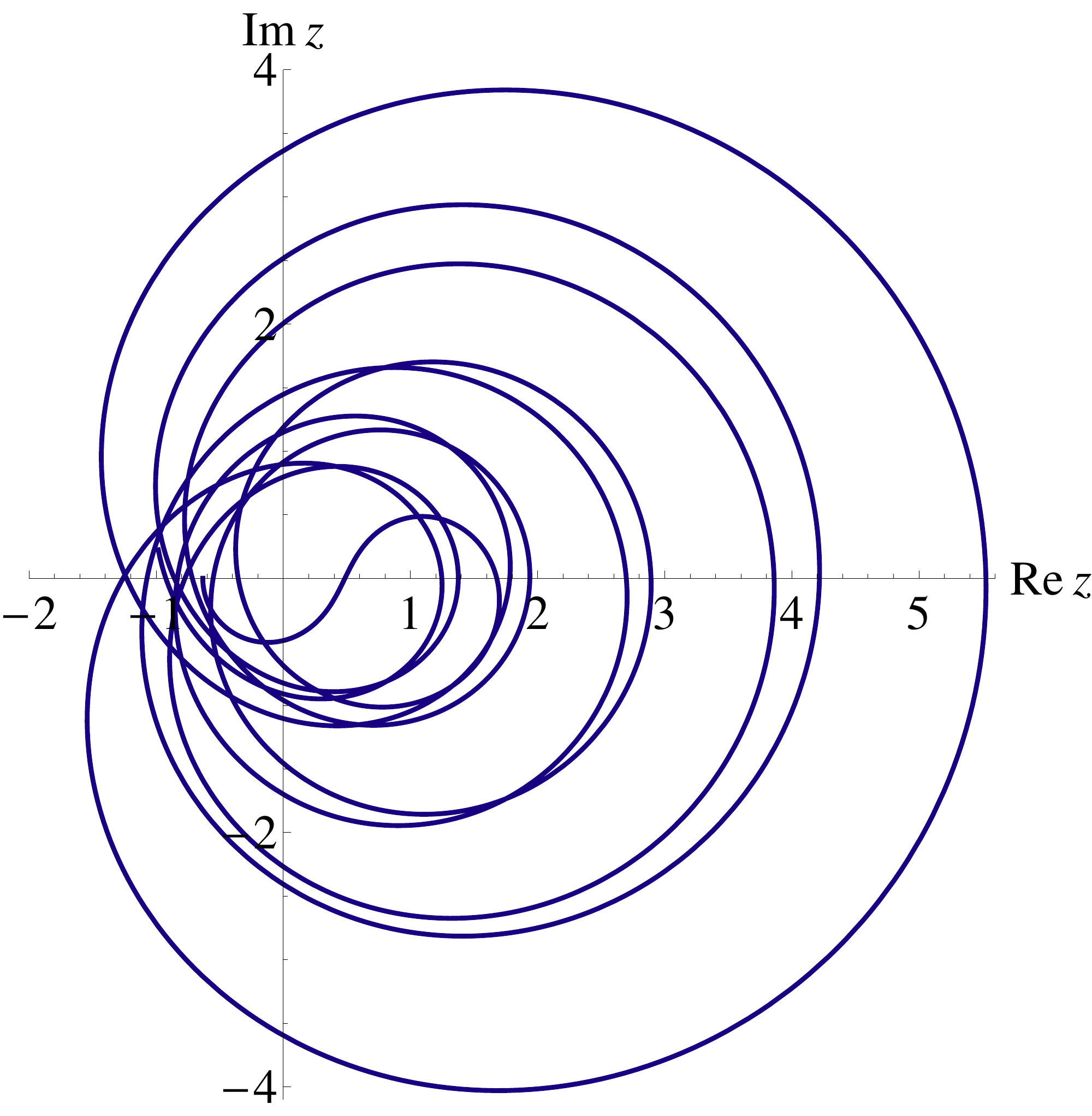}\label{fig:zeta18}%
}
\caption{For some $\sigma\in\R$, we plotted the set $V_{50}(\sigma):=\{\zeta(\sigma+it)\,:\, t\in[0,50]\}$ 
restricted to the range $-2 \leq \Re z \leq 6$, $-4 \leq \Im z \leq 4$.}
\end{figure}

\section{Qualitative difference of the value-distribution on the critical line}\label{sec:qualdiff}
A possible denseness statement for the values of the Riemann zeta-function on the critical line is qualitatively different from the one on vertical lines in $\frac{1}{2}<\sigma\leq 1$.\par
Garunk\v{s}tis \& Steuding \cite{garunkstissteuding:2010} showed that, due to the functional equation, a multidimensional denseness result in the sense of \eqref{eq:multdimVoronin} does not hold on the critical line. We can easily generalize their result to the class $\mathcal{G}$.
\begin{theorem}\label{th:nondensenesscritlineG}
Let $G\in\mathcal{G}$. Then,
$$
\overline{\left\{ (G(\tfrac{1}{2}+it), G'(\tfrac{1}{2}+it)) \; : \; t\in[1,\infty) \right\}} \neq \widehat{\C}^2.
$$
\end{theorem}
\begin{proof} We follow the ideas of Garunk\v{s}tis \& Steuding \cite{garunkstissteuding:2010}. The basic ingredient in the proof is the observation that for real $t$ of sufficiently large modulus with $G(\frac{1}{2}+it)\neq 0$ and $G(\frac{1}{2}+it)\neq \infty$,
\begin{equation}\label{eq:lindau}
 \left| \frac{G'(\frac{1}{2}+it)}{G(\frac{1}{2}+it)} \right| \geq  \frac{d_G}{2}\log |t|-\frac{1}{2}\log (Q^2 \lambda) + O\left( \frac{1}{|t|}\right);
\end{equation}
see Lemma \ref{lem:connectionGG'}. We fix $(a,b)\in\C^2$ with $a,b\neq 0$ and set, for a given $0<\eps<\min\{|a|,|b|\}$,
$$
D_{\eps}(a,b) := \left\{ (s_1,s_2)\in\C^2 \, : \, |s_1-a|<\eps \, \mbox{ and } |s_2-b|<\eps \right\}
$$
Assume that
$$
W:=\overline{\left\{ (G(\tfrac{1}{2}+it), G'(\tfrac{1}{2}+it)) \; : \; t\in[1,\infty) \right\}} = \widehat{\C}^2.
$$
Then, in particular, $D_{\eps}(a,b) \subset W$. Due to \eqref{eq:lindau}, we can fix $t_0>1$ such that
$$
\left| \frac{G'(\frac{1}{2}+it)}{G(\frac{1}{2}+it)} \right| > \frac{|b|+\eps}{|a|-\eps}
$$
for all $t\geq t_0$ with $G(\frac{1}{2}+it)\neq 0$ and $G(\frac{1}{2}+it)\neq \infty$.
By the choice of $t_0$, it follows that
$$
\overline{\left\{ (G(\tfrac{1}{2}+it), G'(\tfrac{1}{2}+it)) \; : \; t\in[t_0,\infty) \right\}} \subset \widehat{\C}\setminus D_{\eps}(a,b).
$$
Consequently, we have
$$
D_{\eps}(a,b) \subset \overline{\left\{ (G(\tfrac{1}{2}+it), G'(\tfrac{1}{2}+it)) \; : \; t\in[1,t_0] \right\}}=:W^*.
$$
This, however, yields a contradiction.
\end{proof}
Moreover, according to Selberg's central limit law, the curve $t\mapsto \zeta(\frac{1}{2}+it)$ has a preference to visit arbitrarily small neighbourhoods of zero than the ones of any other complex value.
\begin{theorem}\label{th:limitlawoncritline} 
Let $a\in\mathbb{C}$ and $0<\varepsilon<|a|$. Then,
$$
\lim_{T\rightarrow \infty}\frac{1}{T}\meas\left\{t\in (0,T]: \left| \zeta(\tfrac{1}{2}+it) -a \right|<\varepsilon \right\}  = 
\left\{\mbox{
\begin{tabular}{ll}
$0$ & if $a\neq 0$, \\ 
$\frac{1}{2}$ & if $a=0$.
\end{tabular}}\right.
$$
\end{theorem}
Comparing this with \eqref{eq:denseBohr}, we see once more that the value-distribution on the critical line is qualitatively different.
\begin{proof} According to Theorem \ref{th:measselberglimitlaw} (b), Selberg's central limit law implies that, for arbitrary $\varepsilon>0$, as $T\rightarrow\infty$,
$$ 
 \frac{1}{T} \meas\left\{t\in(0,T]\, : \, |\zeta({\textstyle\frac{1}{2}+it})|< \varepsilon \right\} = \tfrac{1}{2} + o(1).
$$
This proves the case $a=0$. Moreover, according to Theorem \ref{th:measselberglimitlaw} (b), Selberg's central limit law yields that, for arbitrary $a\in\C$, $a\neq 0$ and any $0<\varepsilon^*<|a|$,
$$
 \frac{1}{T} \meas\left\{t\in(0,T]\, : \, |a|-\varepsilon^*<|\zeta({\textstyle\frac{1}{2}+it})|\leq |a|+\varepsilon^* \right\}, = o(1),
$$
as $T\rightarrow\infty$. Due to the observation that, for any $0<\eps<\eps^*$ and any $T>0$,
\begin{align*}
0 \; \leq \; & 
\frac{1}{T} \meas\left\{t\in(0,T] : |\zeta({\textstyle\frac{1}{2}+it})-a|< \varepsilon \right\} \\
& \qquad \qquad \leq  \frac{1}{T} \meas\left\{t\in(0,T] :  |a|-\varepsilon^* <|\zeta({\tfrac{1}{2}+it})|\leq |a|+\varepsilon^* \right\} ,
\end{align*}
the statement for $a\neq 0$ follows. 
\end{proof}

\section{Approaching zero and infinity from different directions}\label{sec:appr}
Although the curve $t\mapsto \zeta(\tfrac{1}{2}+it)$ has a preference to be either close to zero or to infinity, it is neither known whether zero is an interior point of the set 
$$
\overline{V(\tfrac{1}{2})}=\overline{\left\{ \zeta(\tfrac{1}{2}+it) \, : \, t\in[1,\infty) \right\}} \subset\C
$$ 
nor whether zero is an interior point of the set
$$
\overline{\left\{ \zeta(\tfrac{1}{2}+it)^{-1} \, : \, t\in[1,\infty) \right\}}\subset\widehat{\C}.
$$ 
Relying on Theorem \ref{th:largesmall}, resp. Corollary \ref{cor:selbergsmalllarge}, we prove that we can approximate zero and infinity with non-zero values in $V(\frac{1}{2})$ from quite many directions.
\begin{theorem}\label{th:zeroasintpoint}
Let $\L\in\Sc^*$. Then, there exist real numbers $\theta_0,\theta_{\infty}\in[0,2\pi)$ with the following properties.
\begin{itemize}
\item[(a)] For every $\theta\in(\theta_0-\frac{\pi}{8},\theta_0+\frac{\pi}{8})$, there exist a sequence $(t_k)_k$ with $t_k\subset (1,\infty]$ and $\lim_{k\rightarrow\infty} t_k = \infty$ such that
$$
\L(\tfrac{1}{2}+it_k) \in e^{i\theta}\R^+ := \left\{re^{i\theta}\, :\, r\in\R^+ \right\}
$$
for $k\in\N$ and 
$$
\lim_{k\rightarrow\infty} \L(\tfrac{1}{2}+it_k) = 0.
$$
\item[(b)] For every $\theta\in(\theta_{\infty}-\frac{\pi}{8},\theta_{\infty}+\frac{\pi}{8})$, there exist a sequence $(t_k)_k$ with $t_k\subset (1,\infty]$ and $\lim_{k\rightarrow\infty} t_k = \infty$ such that
$$
\L(\tfrac{1}{2}+it_k) \in e^{i\theta}\R^+ 
$$
for $k\in\N$ and 
$$
\lim_{k\rightarrow\infty} \L(\tfrac{1}{2}+it_k) = \infty.
$$
\end{itemize}
\end{theorem}

\begin{proof} We only prove statement (a). Statement (b) can be proved  by essentially the same method.\par
For $\L\in\Sc^{*}$, there is an analogue of Hardy's $Z$-function at our disposal which allows us to write for sufficiently large real $t$, say $t\geq t_0$,
$$
\L(\tfrac{1}{2}+it) = Z_{\L}(t)\Delta_{\L}(\tfrac{1}{2}+it)^{1/2}.
$$
For the definition and basic properties of $Z_{\L}(t)$ the reader is referred to Section \ref{sec:classG}. Recall here, in particular, that for $t\geq t_0$,
$$
Z_{\L}(t)\in\R \qquad \mbox{ and }\qquad  \Delta_{\L}(\tfrac{1}{2}+it)^{1/2} \in\partial\D.
$$
We fix an arbitrary $0<q<\frac{1}{2}$ and set $\kappa_{\L}:=\frac{\pi}{2d_{\L}}$. Then, according to Theorem \ref{th:largesmall} (b), resp. Corollary \ref{cor:selbergsmalllarge} (b), we find a sequence $(\tau_k)_k$ with $\tau_k\in[t_0,\infty)$ and $\lim_{k\rightarrow\infty}\tau_k = \infty$ such that the functions 
\begin{align*}
\L_{\tau_k}(y)&:=\L\left(\tfrac{1}{2} + i \frac{\kappa_{\L}}{\log \tau _k}\, y +i\tau_k \right)\\
&= Z_{\L}\left(\frac{\kappa_{\L}}{\log \tau _k}\, y + \tau_k \right)\Delta_{\L}\left(\tfrac{1}{2}+i\frac{\kappa_{\L}}{\log \tau _k}\, y + i\tau_k\right)^{1/2}
\end{align*}
do not vanish on the interval $[-q,q]$ for $k\in\N$ and converge uniformly on $[-q,q]$ to zero, as $k\rightarrow\infty$.\par

As the functions $\L_{\tau_{k}}(y)$, $k\in\N$, are real and non-vanishing on the interval $[-q,q]$, we find, for every $k\in\N$, an integer $\eta_k\in\{-1,+1\}$ such that
$$
f_k(y):= \eta_k \cdot Z_{\L}\left(\frac{\kappa_{\L}}{\log \tau _k}\,y + \tau_{k} \right) >0 \qquad \mbox{ for } y\in\left[-q,q\right] .
$$
The factors $\eta_k$ assure that $f_k(y)=\left| \L_{\tau_k}(y) \right|$ for $k\in\N$ and $y\in[-q,q]$. Further, we observe that, uniformly for $y\in[-q,q]$, 
\begin{equation}\label{kw1}
 f_k(y) \rightarrow 0, \qquad \mbox{as }k\rightarrow\infty.
\end{equation}
Now, we set
$$
h_k(y):= \eta_k^{-1} \cdot \Delta_{\L}\left(\tfrac{1}{2}+i\frac{\kappa_{\L}}{\log \tau _k}\, y + i\tau_k\right)^{1/2} .
$$
Observe that $\left|h_k(y)\right|=1$ and that 
\begin{equation}\label{Zhf}
f_k(y)\cdot h_k(y)=\L_{\tau_k}(y)
\end{equation}
for $k\in\N$ and $y\in[-q,q]$. Since $h_k(0)\in\partial\D$, it follows from the compactness of $\partial\D$ that the set $\{h_k(0)\}_k$ has at least one accumulation point $e^{i\theta_0}\in\partial\D$ with $\theta_0\in [0,2\pi)$. Without loss of generality, we may suppose that 
$$
\lim_{k\rightarrow\infty} h_{k}(0) = e^{i\theta_0}.
$$ 
Otherwise, we work with a suitable subsequence of $(h_k(0))_k$. For $k\in\N$, we define $\theta_{k}\in[0,2\pi)$ such that $ e^{i\theta_{k}} = h_{k}(0)$. Then, according to Lemma \ref{lem:Delta_p_phi}, we have, uniformly for $y\in[-q,q]$,
\begin{align*}
h_{k}(y) &= h_k(0)\cdot \exp\left(- i\tfrac{1}{2}d_{\L}\kappa_{\L} y \right) \left(1 + O\left(\frac{1}{\log \tau_{k}} \right) \right)\\
&= \exp\left(i\theta_k - i\tfrac{\pi}{4} y \right) \left(1 + O\left(\frac{1}{\log \tau_{k}} \right) \right).
\end{align*}
as $k\rightarrow\infty$. Thus, for any $\theta\in (\theta_0-\frac{q\pi}{4}, \theta_0+\frac{q\pi}{4})$, we find for sufficiently large $k \in\N$, a real number $y_{k}\in [-q,q]$ with
\begin{equation}\label{kw2}
h_{k}(y_{k})= \exp(i\theta).
\end{equation}
Now, we set 
$$t_k := \tau_{k} + \frac{\kappa_{\L}}{\log \tau_k}\,y_k.$$ Then, according to \eqref{kw1}, \eqref{Zhf} and \eqref{kw2}, we have $\L(\tfrac{1}{2}+it_k ) \in e^{i\theta}\R^+$ for all sufficiently large $k\in\N$ and $\lim_{k\rightarrow\infty} \L(\tfrac{1}{2}+it_k) = 0$. As we can choose arbitrary $0<q<\frac{1}{2}$, the assertion is proved. 
\end{proof}
For the Riemann zeta-function, there are results of Kalpokas, Korolev \& Steuding \cite{kalpokaskorolevsteuding:2013} at our disposal which exceed the statement of Theorem \ref{th:zeroasintpoint} (b) by far. Kalpokas \& Steuding \cite{kalpokassteuding:2011} investigated intersection points of the curve $\R \ni t\mapsto \zeta(\frac{1}{2}+it)$ with straight lines $e^{i\theta}\R := \{re^{i\theta}\, : \, r\in\R\}$ through the origin. In particular, they observed that
$$
\zeta(\tfrac{1}{2}+it)\in e^{i\theta}\R \qquad \Longleftrightarrow \qquad \zeta(\tfrac{1}{2}+it) = 0 \quad \mbox{or}\quad \Delta_{\zeta}(\tfrac{1}{2}+it) = e^{i2\phi}.
$$
Their works were extended by Christ \& Kalpokas \cite{christkalpokas:2012, christkalpokas:2013} and Kalpokas, Korolev \& Steuding \cite{kalpokaskorolevsteuding:2013}. The latter showed that, for every $\theta\in[0,2\pi)$, there is a sequence $(t_k)_k$ with $t_k\in\R$ and $\lim_{k\rightarrow\infty} t_k = \infty$ such that, for all $k\in\N$,
$$
\zeta(\tfrac{1}{2}+it_k) \in e^{i\theta}\R^+ \qquad \mbox{and}\quad |\zeta(\tfrac{1}{2}+it_k)|\geq C (\log t_k)^{5/4}
$$
with some positive constant $C$. Thus, roughly speaking, the values of the Riemann zeta-function on the critical line expand in every direction. \par

Korolev showed in a talk at the ELAZ conference 2012 at Schloss Schney that there exists a sequence $(t_k)_k$ with $t_k\in \R$ and $\lim_{k\rightarrow\infty} t_k = \infty$ such that, for every $k\in\N$,
$$
\zeta(\tfrac{1}{2}+it_k) \in\R\setminus\{0\} \qquad \mbox{and} \qquad \lim_{k\rightarrow\infty} \zeta(\tfrac{1}{2}+it_k) = 0.
$$
Korolev's result is not published yet.

\section{Denseness results on curves approaching the critical line}\label{sec:curves}
In the following, let $\epsilon:[1,\infty)\rightarrow\R$ be a function with $\lim_{t\rightarrow\infty} \epsilon(t) = 0$. For a humble attack on Ramachandra's conjecture, we investigate the value-distribution of the Riemann zeta-function on curves $t\mapsto \frac{1}{2}+\epsilon(t)+it$ which approach the critical line asymptotically as $t\rightarrow\infty$.\par
If we could establish a denseness statement for the zeta-values on curves approaching the critical line sufficiently fast, then the denseness of the zeta-values on the critical line would follow; see the subsequent theorem. Recall the definition of the function $\theta_{\L}(\sigma)$ from Section \ref{sec:orderofgrowth}, which indicates the order of growth of a given function $\L\in\Sc^{\#}$.

\begin{theorem}\label{th:curvesmotivation}
Let $\L\in\Sc^{\#}$ and $n\in\N_0$. Let $\epsilon:[1,\infty) \rightarrow \R$ be a function such that $\epsilon(t)\ll t^{-\theta_{\L}(\frac{1}{2})-\delta}$, as $t\rightarrow\infty$, with some $\delta>0$. Then,
$$
\overline{\left\{ \Delta_n \L(\tfrac{1}{2}+it) \, : \, t\in[1,\infty) \right\}} = \C^{n+1}
$$
if and only if
$$
\overline{\left\{ \Delta_n \L(\tfrac{1}{2}+\epsilon(t)+it) \, : \, t\in[1,\infty) \right\}} = \C^{n+1}.
$$
\end{theorem}
\begin{proof}
Let $a\in\C^{n+1}$ and $\eps>0$. Let $\|\cdot\|$ denote the maximum-norm in the complex vector space $\mathbb{C}^{n+1}$. Then, according to Lemma \ref{lem:growth}, there exist a constant $t_0>1$ such that 
$$
\bigl\|\Delta_n \L(\tfrac{1}{2}+\epsilon(t)+it)-\Delta_n \L(\tfrac{1}{2}+it)\bigr\| < \eps/2
$$
for $t\geq t_0$. Assume that $\overline{\left\{ \Delta_n \L(\tfrac{1}{2}+\epsilon(t)+it) \, : \, t\in[1,\infty) \right\}} = \C^{n+1}$. Then, we find a $t^*\geq t_0$ such that
$$
\bigl\|\Delta_n\L(\tfrac{1}{2}+\epsilon(t^*)+it^*)-a\bigr\|<\eps/2.
$$
By means of the triangle inequality, we obtain that
$$
\bigl\|\Delta_n \L(\tfrac{1}{2}+it^*)-a\bigl\|<\eps.
$$
As we can choose $a\in\C^{n+1}$ and $\eps>0$ arbitrarily in the argumentation above, we deduce that
$$
\overline{\left\{ \Delta_n \L(\tfrac{1}{2}+it) \, : \, t\in[1,\infty) \right\}} = \C^{n+1}.
$$
The other implication of the theorem can be proved by the same argument. 
\end{proof}

{\bf A negative denseness results on curves approaching the critical.} The property of the Riemann zeta-function that a multidimensional denseness result does not hold on the critical line carries over to certain curves approaching the critical line.
\begin{corollary}\label{cor:nondensenesscurves}
Let $\L\in\Sc^{\#}$ and $n\in\N$. Let $\epsilon:[1,\infty) \rightarrow \R$ be a function such that $\epsilon(t)\ll t^{-\theta_{\L}(\frac{1}{2}) - \delta}$, as $t\rightarrow\infty$, with some $\delta>0$. Then,
$$
\overline{\left\{ \Delta_n \L(\tfrac{1}{2}+\epsilon(t)+it) \, : \, t\in[1,\infty) \right\}} \neq \C^{n+1}.
$$
\end{corollary}
\begin{proof}
Recalling that $\Sc^{\#}\subset \mathcal{G}$ and noticing that $n\neq 0$, the corollary follows directly by combining Theorem \ref{th:nondensenesscritlineG} with Theorem \ref{th:curvesmotivation}.  
\end{proof}

{\bf Denseness results on curves approaching the critical line.}
\begin{figure}
\centering
\includegraphics[width=0.9\textwidth]{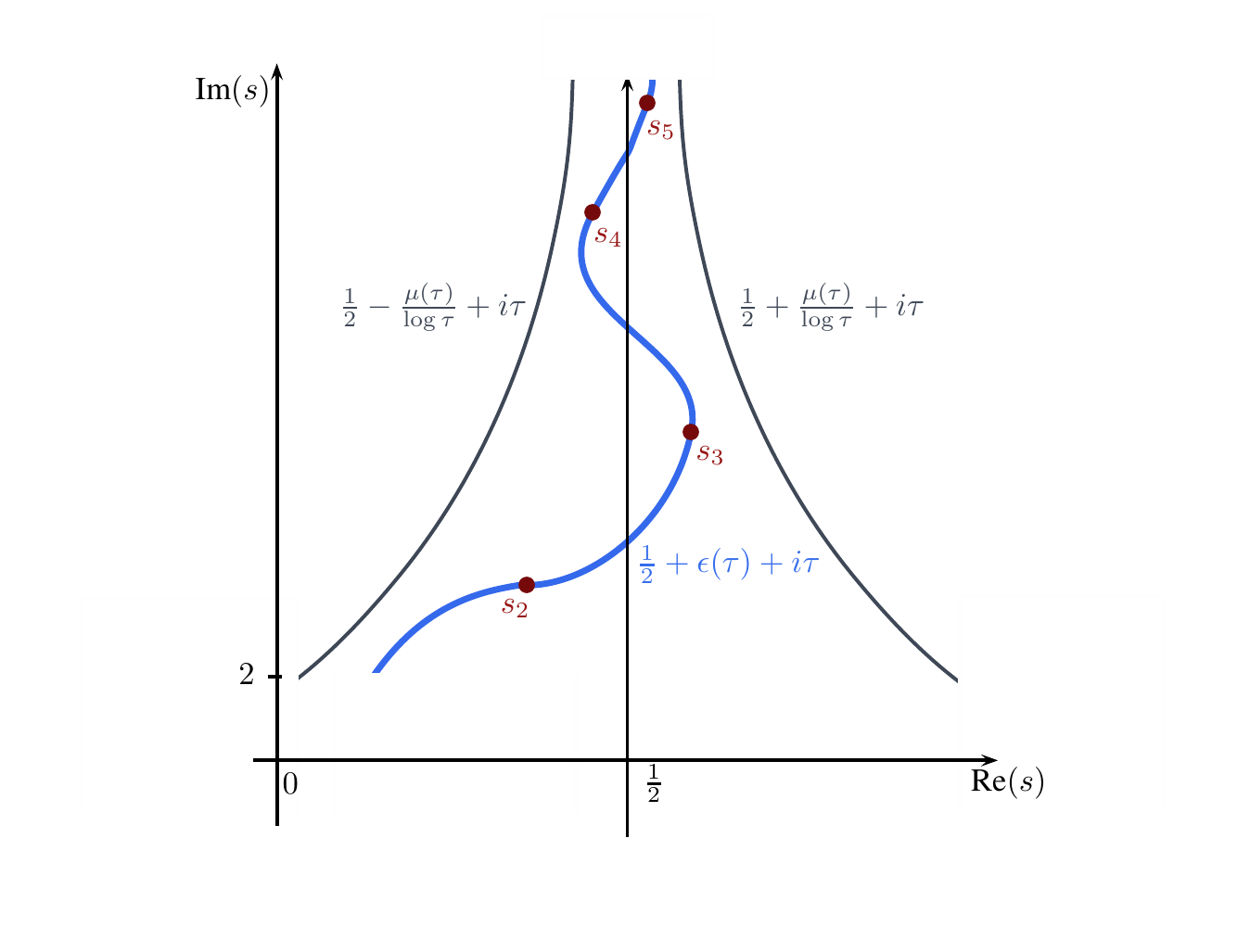}\\[-1cm]
\caption{The construction of the curve $t\mapsto \frac{1}{2}+\epsilon(t)+it$ in the proof of Theorem \ref{th:densenessavalues}.}
\label{fig:densecurve} 
\end{figure}
It seems difficult to obtain a denseness result on curves $t\mapsto \frac{1}{2}+\epsilon(t)+it$ by adapting Bohr's method. This is due to the fact that, according to a result of Laurin\v{c}ikas \cite{laurincikas:1992}, there is a quantitative swap in the asymptotic behaviour of the mean-square near the critical line. 
\begin{theorem}[Laurin\v{c}ikas, 1991]
Let $\mu:[2,\infty)\rightarrow \R$ be a non-negative, (not necessarily strictly) monotonically increasing or decreasing function satisfying $\lim_{T\rightarrow\infty}\frac{\mu(T)}{\log T} = 0$. 
\begin{itemize}
 \item[(a)] If $\lim_{T\rightarrow\infty} \mu(T) =\infty$, then, as $T\rightarrow\infty$,
$$
\frac{1}{T}\int_2^{T} \left|\zeta\left(\tfrac{1}{2}+\frac{\mu(T)}{\log T}+it\right) \right|^2 \d t \sim  \frac{1}{2\mu(T)} \log T.
$$
\item[(b)] If $\lim_{T\rightarrow\infty} \mu(T) =c$ with some $c>0$, then, as $T\rightarrow\infty$,
$$
\frac{1}{T}\int_2^{T} \left|\zeta\left(\tfrac{1}{2}+\frac{\mu(T)}{\log T}+it\right) \right|^2 \d t \sim \frac{1}{2c}\left(1-e^{-2c} \right)\log T. 
$$
\item[(c)] If $\lim_{T\rightarrow\infty} \mu(T) =0$, then, as $T\rightarrow\infty$,
$$
\frac{1}{T}\int_2^{T} \left|\zeta\left(\tfrac{1}{2}+\frac{\mu(T)}{\log T}+it\right) \right|^2 \d t \sim \log T.
$$
\end{itemize}
\end{theorem}
For a proof we refer to Laurin\v{c}ikas \cite[Theorems 1, 2 and 3]{laurincikas:1992}.\par
Nevertheless, we can use the $a$-point results from Chapter \ref{chapt:apoints} to deduce the existence of certain curves on which the values of the zeta-function lie dense in $\C$.
\begin{theorem}\label{th:densenessavalues}
There exists a continuous function $\epsilon:[2,\infty)\rightarrow \R$ with $\lim_{t\rightarrow\infty}\epsilon(t)=0$ such that
$$
\overline{\left\{ \zeta(\tfrac{1}{2}+\epsilon(t)+it) \, : \, t\in[2,\infty) \right\}}=\C.
$$
Here we can demand additionally that,
\begin{itemize}
 \item[(a)] unconditionally,
$$
|\epsilon(t)|\leq \frac{\mu(t)}{\log t}
$$ 
with any positive function $\mu$ satisfying $\lim_{t\rightarrow\infty}\mu(t)=\infty$.
 \item[(b)] by assuming the Riemann hypothesis,
$$
|\epsilon(t)| \leq \frac{\mu(t)(\log\log\log t)^3}{\log t \sqrt{\log\log t}}
$$
with any positive function $\mu$ satisfying $\lim_{t\rightarrow\infty}\mu(t)=\infty$.
\item[(c)] by assuming the generalized Riemann hypothesis for Dirichlet $L$-functions,\footnote{Very likely this can be proved by only assuming the Riemann hypothesis; see the remarks in Section \ref{sec:filling_zeta}.}
$$
|\epsilon(t)| \leq \frac{1}{(\log t)^{\delta}} 
$$
with any $\delta>0$.
 \item[(d)] by assuming the generalized Riemann hypothesis for Dirichlet $L$-functions,
$$
|\epsilon(t)| \leq \mu(t)\exp \left(-c_0 \frac{\log t}{\log\log t} \right)
$$
with any positive function $\mu$ satisfying $\lim_{t\rightarrow\infty}\mu(t)=\infty$ and any constant $0<c_0 < \frac{1}{\sqrt{2}}$. 
\end{itemize}
\end{theorem}

\begin{proof}
Let $(q_k)_k$ be an enumeration of $\mathbb{Q}+ i\mathbb{Q}$. According to Corollary \ref{cor:summaryapoints} (c), we find for every $q_k$, with at most one exception, infinitely many roots of the equation $\zeta(s)=q_k$ inside the strip $S$ defined by
$$
\tfrac{1}{2} - \frac{\mu(t)}{\log t} <\sigma < \tfrac{1}{2} + \frac{\mu(t)}{\log t}, \qquad t\geq 2,
$$
where $\mu$ is an arbitrarily positive function satisfying $\lim_{t\rightarrow\infty}\mu(t)$. Without loss of generality, we may suppose that the enumeration $(q_k)_k$ is arranged such that $q_1$ is the possibly existing value which is not assumed by the zeta-function in $S$. This agreement as well as the observation above assure that we find points $s_k =\sigma_k + it_k \in S$ such that $\zeta(s_k) = q_k$ and $t_k<t_{k+1}$ for every $k\in\N\setminus\{1\}$. By connecting the point $s_k$ with its successor $s_{k+1}$, respectively, by a straight line segment, we get a continuous function $\epsilon:[2,\infty)\rightarrow\R^+$ such that $|\epsilon(t)|\leq \frac{\mu(t)}{\log t}$ for all $t\in[2,\infty)$ and $\frac{1}{2}+\epsilon(t_k) + it_k = s_k$ for all $k\in\N$. Theorem \ref{th:densenessavalues} (a) follows then immediately as  $\mathbb{Q}+ i\mathbb{Q}$ (minus the possibly existent exceptional value $q_1$) is dense in $\C$.\par 
The statements (b), (c) and (d) follow from Corollary \ref{cor:summaryapoints} (d), (e) and (f) in an analogous manner.
\end{proof}
The curves of Theorem \ref{th:densenessavalues} can oscillate from the left to the right of the critical line. It would be nice to have a similar statement for curves which stay always either to the left or to the right of the critical line.\par

{\bf Curves approaching the critical line from the left.} From Selberg's conditional result concerning $a$-points to the left of the critical line, we deduce the following denseness statement.

\begin{theorem}
Assume that the Riemann hypothesis is true. Then, for arbitrary $0<c_1<c_2$, there exists a continuous function $\epsilon:[1,\infty)\rightarrow \R$ with 
$$
\frac{c_1 \sqrt{\log\log t}}{\log \log t} \leq \epsilon(t) \leq \frac{c_2 \sqrt{\log\log t}}{\log \log t}
$$
such that 
$$
\overline{\left\{ \zeta(\tfrac{1}{2}-\epsilon(t)+it) \, : \, t\in[1,\infty) \right\}}=\C.
$$
\end{theorem}
\begin{proof}
Relying on Selberg's result \eqref{eq:selberg_aleft}, the proof follows along the lines of the proof of Theorem \ref{th:densenessavalues}. 
\end{proof}

{\bf Curves approaching the critical line from the right.}
Relying on Bohr's, resp. Voronin's denseness result, it is possible to prove the following theorem.
\begin{theorem}[Christ \cite{christ:2012}, 2012]\label{th:enumerationbohr}
For every $n\in\N_0$, there is a positive, piecewise-constant function $\epsilon:[1,\infty) \rightarrow \mathbb{R}^+$ with $\lim_{t\rightarrow \infty} \epsilon(t)=0$ such that
$$\textstyle
\overline{\left\{\Delta_n \zeta\left(\frac{1}{2}+\epsilon(t)+it\right) \, :\,t\in[1,\infty) \right\}} = \mathbb{C}^{n+1}.
$$
\end{theorem}

\begin{proof} To prove Theorem \ref{th:enumerationbohr} we use Voronin's denseness result in combination with a certain enumeration method. Let $(\varepsilon_k)_{k}$ and $(\sigma_k)_{k}$ be sequences with $\varepsilon_k>0$ and $\sigma_k\in(0,\frac{1}{2})$ for $k\in\mathbb{N}$ that tend to zero as $k\rightarrow\infty$. Furthermore, for a given $n\in\N_0$, we define compact sets
$$
A_k := \{z\in\mathbb{C}^{n+1} \,: \, \|z\|\leq k\},\qquad k\in\mathbb{N}.
$$ 
The sets $A_k$ form a countable covering of the $(n+1)$-dimensional complex plane, i.e.
$$
\bigcup_{k\in \mathbb{N}} A_k = \mathbb{C}^{n+1}.
$$
We set $T_0=1$ and choose, inductively for $k\in\mathbb{N}$, a positive real number $T_k>T_{k-1}+1$  such that for all $a\in A_k$ there is a $\tau \in (T_{k-1}, T_k]$ with 
$$\textstyle
\left\|\Delta_n \zeta\left(\frac{1}{2}+\sigma_k + i\tau\right)-a\right\|<\varepsilon_k.
$$
The existence of such a number $T_k$ is assured by Voronin's denseness result applied to the vertical line $\sigma = \frac{1}{2}+\sigma_k$ and basic properties of compact sets. The piecewise-constant function $\epsilon:[1,\infty)\rightarrow \mathbb{R}^+$ defined by
$$
\epsilon(t)=\sigma_k \qquad \mbox{for }t\in(T_{k-1},T_k] \mbox{ with }k\in\mathbb{N},
$$
satisfies $\lim_{t\rightarrow\infty}\epsilon(t)=0$. The construction of the function $\epsilon$ yields that the curve $t\mapsto \frac{1}{2}+\epsilon(t)+it$ has the desired property. \par
\end{proof}
As we are very flexible in choosing the zero-sequences $(\varepsilon_k)_{k}$, $(\sigma_k)_{k}$ and a proper countable covering of the complex plane $(A_k)_{k\in\mathbb{N}}$, the proof of Theorem \ref{th:enumerationbohr} yields the existence of uncountable many curves with the desired property.\par
By adjusting the sets $A_k$ in a proper way, we can use a quantitative version of Voronin's denseness result (see Karatsuba \& Voronin \cite[Chapt. VIII]{karatsubavoronin:1992}) to get very rough estimates on how `fast' these curves approach the critical line.\par

With respect to Theorem \ref{cor:nondensenesscurves}, it would be very nice to characterize the changeover of $\{\Delta_2\zeta(\frac{1}{2}+\epsilon(t)+it)\, : \, t\in[1,\infty)\}$ from denseness to non-denseness in terms of the speed of convergence of $\epsilon(t)\rightarrow 0$, as $t\rightarrow\infty$. \par

{\bf Small and large values on curves approaching the critical line.}
\begin{figure}
\centering
\subfigure[$|\zeta(\sigma+it)|$]{
\centering
\includegraphics[width=0.45\textwidth]{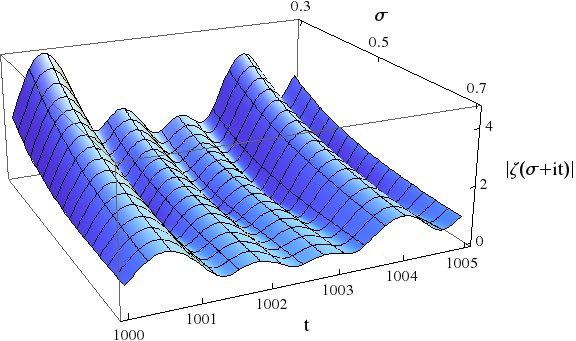}\label{fig:landscape1}%
}
\subfigure[$|\zeta(\sigma+it)|^{-1}$]{
\centering
\includegraphics[width=0.45\textwidth]{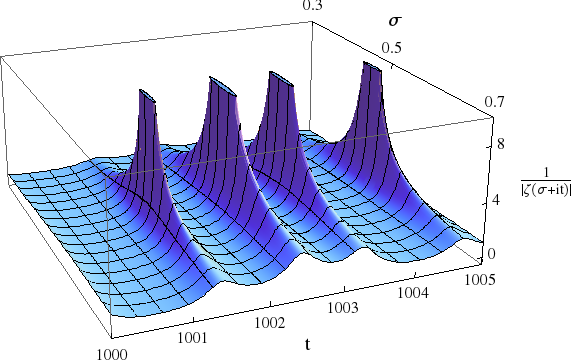}\label{fig:landscape2}%
}
\caption{The values $|\zeta(\sigma+it)|$ and $|\zeta(\sigma+it)|^{-1}$ on the rectangular region $0.3\leq \sigma \leq 0.7$ and $1000\leq t \leq 1005$.}
\end{figure}
From Corollary \ref{cor:selbergsmalllarge}, which we obtained by relying on Selberg's central limit law, we can deduce information on small and large values of the zeta-function on certain curves $t\mapsto \frac{1}{2}+\epsilon(t)+it$.
\begin{corollary}\label{cor:curvessmalllarge}
Let $c>0$ and $\epsilon:[1,\infty)\rightarrow\R$ be a function satisfying
$|\epsilon(t)|\leq \frac{c}{ \log t}$ for $t\in[1,\infty)$. 
Then, 
$$
\liminf_{t\rightarrow\infty }|\zeta(\tfrac{1}{2}+\epsilon(t)+it)| = 0 \qquad \mbox{ and }\qquad \limsup_{t\rightarrow\infty }|\zeta(\tfrac{1}{2}+\epsilon(t)+it)| = \infty.
$$
\end{corollary}
\begin{proof}
The corollary follows directly from Corollary \ref{cor:selbergsmalllarge} by working with the conformal map $\varphi_{\tau}(z)=\frac{1}{2}+\frac{\kappa_{\zeta}/4}{\log \tau} z +i\tau$ and the rectangular domain $\mathcal{R}(\frac{4c}{\kappa_{\zeta}},1)$.
\end{proof}
Roughly speaking, Corollary \ref{cor:curvessmalllarge} states that the zeta-function assumes both arbitrarily small and arbitrarily large values on every path to infinity which lies inside the strip
$$
\tfrac{1}{2}-\frac{c}{\log t}<\sigma < \tfrac{1}{2}+\frac{c}{ \log t}, \qquad t\geq 1,
$$ 
with arbitrary $c>0$.\par

\section{Denseness results on curves approaching the line \texorpdfstring{$\sigma=1$}{}}\label{sec:densesigma1}
 By a slight refinement of Bohr's method, we can show that the values of the zeta-function taken on any curve $[1,\infty)\ni t\mapsto 1+ \epsilon(t)+it$ with $\lim_{t\rightarrow\infty}\epsilon(t)=0$ are dense in $\mathbb{C}$.
\begin{theorem}[Christ \cite{christ:2012}, 2012]\label{th1}
Let $\epsilon:[1,\infty)\rightarrow\R$ be a function with $\lim_{t\rightarrow\infty}\epsilon(t)=0$. Then, for every $a\in\mathbb{C}$ and $\varepsilon>0$,
$$
\liminf_{T\rightarrow \infty}\frac{1}{T}\meas\left\{t\in (0,T]: \left| \zeta(1+\epsilon(t)+it) -a \right|<\varepsilon \right\}  >0.
$$
\end{theorem}
We choose the function $\epsilon$ in Theorem \ref{th1} such that the curve $t\mapsto 1+\epsilon(t)+it$ lies completely in the half-plane $\sigma >1$. In this case, Theorem \ref{th1} cannot be deduced from Voronin's universality theorem. 
Moreover, in Theorem \ref{th1}, there are no restrictions on how fast $\epsilon(t)$ tends to zero as $t\rightarrow \infty$.\par

\begin{proof} To prove Theorem \ref{th1}, we rely on the methods of Bohr \& Courant \cite{bohrcourant:1914} and Bohr \& Jessen \cite{bohrjessen:1932} who proved a corresponding result for vertical lines inside the strip $\frac{1}{2}<\sigma \leq 1$. We will refine their methods by adding a certain continuity argument. As Bohr and his collaborators did, we will prove the result not for $\zeta(s)$, but for $\log \zeta(s)$. The result for $\zeta(s)$ is then an immediate consequence therefrom. We define $\log\zeta(s)$ for $\sigma>\frac{1}{2}$ in a standard way; for details we refer to Steuding \cite[Chapt. 1.2]{steuding:2007}\par
Let $\varepsilon>0$ and $a\in\mathbb{C}$. Let $(p_n)_{n}$ be an enumeration of all prime numbers in ascending order. For a positive integer $N$, we define the truncated Euler product
$$
\zeta_N(s)=\prod_{n=1}^{N} (1-p_n^{-s})^{-1}.
$$
Note that $\zeta_N(s)$ defines an analytic and non-vanishing function in the half-plane $\sigma>0$. Bohr showed that there is positive real number $d$ and a positive integer $N_1$ such that, for all $N\geq N_1$, we find a subset $\mathcal{I}(N)\subset(0,\infty)$ of lower density 
$$
\liminf_{T\rightarrow\infty} \frac{1}{T}\meas \left(\mathcal{I}(N)\cap (0,T]\right) > d
$$
with the property that
$$
|\log \zeta_N(1+it) - a | < \varepsilon/3 \qquad \mbox{ for }t\in\mathcal{I}(N).
$$
Furthermore, Bohr proved that, for any $\varepsilon'>0$ and any $\frac{1}{2}<\sigma_0 < 1$, there is a positive real numbers $N_2$ such that, for every $N\geq N_2$ and every $\sigma\geq \sigma_0$,
$$
\lim_{T\rightarrow\infty} \frac{1}{T}\int_1^{T} \left| \frac{\zeta(\sigma+it)}{\zeta_N(\sigma+it)}-1\right| \d t < \varepsilon'.
$$
The bounded convergence theorem assures that
\begin{equation*}
\lim_{T\rightarrow\infty} \int_{1}^T \int_{\sigma_0}^{2} \left|\frac{\zeta(\sigma+it)}{\zeta_N(\sigma+it)}-1 \right|^2  \d \sigma \d t <  2\eps'
\end{equation*}
By carefully adopting Bohr's reasoning, we can deduce from the estimate above that there exists a positive integer $N_3$ such that, for all $N\geq N_3$, we find a subset $\mathcal{J}(N)\subset(0,\infty)$ of lower density 
$$
\liminf_{T\rightarrow\infty}\frac{1}{T}\meas \left( \mathcal{J}(N)\cap (0,T] \right)> 1-d
$$ 
with the property that
$$
|\log \zeta_N(1+\epsilon(t)+it)- \log\zeta(1+\epsilon(t)+it)|<\varepsilon/3 
\qquad\mbox{for } t\in\mathcal{J}(N).
$$
Now, we apply a certain continuity argument. We fix $N_0\geq \max\{N_1,N_3\}$ and choose a sufficiently small $\delta<\log (1- \frac{\varepsilon}{6N_0})^{-1}/(N_0\log 2)$. As $\lim_{t\rightarrow\infty}\epsilon(t)=0$, we find a positive number $T_0$ such that $\epsilon(t)<\delta$ for $t\geq T_0$. Then, for $t\geq T_0$,
$$
\left| \log \zeta_{N_0} (1+it) - \log \zeta_{N_0}(1+\epsilon(t)+it) \right|=\left| \sum_{n=1}^{N_0 }\log\left(1+\frac{1-p_n^{-\epsilon(t)}}{p_n^{1+it}-1} \right)\right| 
$$
$$
\leq  \sum_{n=1}^{N_0 } 2\left|\frac{1-p_n^{-\epsilon(t)}}{p_n^{1+it}-1} \right| \leq 2 N_0 (1-p_{N_0}^{-\epsilon(t)}) <2 N_0 (1-2^{-N_0\delta })<\varepsilon/3.
$$
Here, we used Betrand's postulate which states that $p_{n}\leq 2^{n}$. Although the estimation via Bertrand's postulate is quite rough, it is completely sufficient for our purpose.\par 
Altogether, we can deduce that the set $\mathcal{M}:=\mathcal{I}(N_0)\cap \mathcal{J}(N_0)\cap (T_0,\infty)$ has positive lower density, i.e.
$$
\liminf_{T\rightarrow\infty}\frac{1}{T}\meas \left(\mathcal{M}\cap (0,T]\right)> 0,
$$
and enjoys the property that 
\begin{eqnarray*}
|\log \zeta(1+\epsilon(t)+it)-a|&\leq & 
|\log \zeta(1+\epsilon(t)+it)-\log\zeta_{N_0}(1+\epsilon(t)+it)|\\
&  &+\, |\log \zeta_{N_0}(1+\epsilon(t)+it)-\log \zeta_{N_0}(1+it)|\\
&  &+\, |\log \zeta_{N_0}(1+it)-a|<\varepsilon\\
\end{eqnarray*}
for $t\in\mathcal{M}$. The statement of the theorem follows. 
\end{proof}

\section{A limiting process to the right of the critical line}\label{sec:universalityoncurves}

In this final section of Part I, we briefly discuss what happens to the limiting process of Section \ref{sec:shiftingshrinking} if we adjust the underlying conformal mappings such that they map the unit disc to discs which lie completely inside the strip $\frac{1}{2}<\sigma<1$, but arbitrarily close to the critical line. By relying on Voronin's universality, we can prove the following statement.

\begin{theorem}\label{th:univsliding}
Let $0\leq\eta\leq\frac{1}{4}$ and let $(\epsilon_k)_k$, $(\lambda_k)_k$ be sequences with $\epsilon_k,\lambda_k\in\R$,  $0<\lambda_k<\epsilon_k<\frac{1}{4}$ for $k\in\N$ and $\lim_{k\rightarrow\infty}\epsilon_k = \lim_{k\rightarrow\infty}\lambda_k = \eta$. Further, let 
$$
\zeta_{k,\tau}(z) :=\zeta(\tfrac{1}{2} +\epsilon_k +\lambda_kz+i\tau)\qquad \mbox{for }\tau\geq 1,\, k\in\N  \mbox{ and }z\in\D.
$$
Then, there is a sequence $(\tau_k)_k$ with $\tau_k\in[1, \infty)$ and $\lim_{k\rightarrow\infty} \tau_k = \infty$ such that, for every continuous and non-vanishing function $g$ on $\overline{\mathbb{D}}$ which is analytic in $\D$, there is a subsequence of $(\zeta_{k,\tau_k})_k$ which converges uniformly on $\overline{\D}$ to $g$. 
\end{theorem}

\begin{proof}[Sketch of the proof]
Due to the theorem of Mergelyan, see for example \cite{rudin:1966}, it is sufficient to establish the assertion of Theorem \ref{th:univsliding} for polynomial target functions $g$ which have rational coefficients and do not vanish on $\overline{\mathbb{D}}$. The proof follows then directly from Voronin's universality theorem in combination with a similar enumeration method as the one that we used in the proof of Theorem \ref{th:enumerationbohr}. We refer to Christ \cite{christ:2012} for more details.
\end{proof}

\part{Discrete and continuous moments}

$\mbox{ }$\vspace{5cm}

In part II we aim at extending a result of to Tanaka \cite{tanaka:2008} who established a weak version of the Lindel\"of hypothesis for the Riemann zeta-function.\par

Recall that, according to Hardy \& Littlewood \cite{hardylittlewood:1923}, the Lindel\"of hypothesis is equivalent to the statement that, for every $\sigma>\frac{1}{2}$ and $k\in\N$,
\begin{equation}\label{Lind1}
\lim_{T\rightarrow \infty}\frac{1}{T}\int_1^T \left|\zeta(\sigma+it) \right|^{2k} \d t =\sum_{n=1}^{\infty}\frac{d_k(n)^2}{n^{2\sigma}},
\end{equation}
where $d_k$ denotes the generalized divisor function appearing in the Dirichlet series expansion of $\zeta^{k}$. Tanaka showed that \eqref{Lind1} holds if one neglects a certain set $A\subset [1,\infty)$ of density zero from the path of integration.\par

In the sequel, let $\pmb{1}_X$ denote the indicator function of a set $X\subset \R$ and $X^c := \R\setminus X$ its complement. So far, formula \eqref{Lind1} is proved only in the cases $k=1,2$, due to classical works of Hardy \& Littlewood \cite{hardylittlewood:1922} and Ingham \cite{ingham:1926}. This is sufficient to derive the following boundedness property of the Riemann zeta-function in $\sigma>\frac{1}{2}$: \textit{for every $\eps>0$ and $\alpha>\frac{1}{2}$, there exist a constant $M_{\eps}>0$ and a subset $B\subset [1,\infty)$ of upper density
$$
\limsup_{T\rightarrow\infty} \frac{1}{T} \int_1^{T} \pmb{1}_B (t) \d t < \eps
$$
such that
\begin{equation} \label{e}
\left|\zeta(\sigma+it) \right|\leq M_{\eps} \qquad\mbox{for }\sigma\geq \alpha \mbox{ and } t\in B^c.
\end{equation}}This follows from a standard method which we shall explain later on in the proof of Lemma \ref{th:suffconditionsN} (b).\par
In particular, we deduce from \eqref{e} that, \textit{for every $\eps>0$ and $\alpha>\frac{1}{2}$, there exist a constant $M_{\eps}>0$ and a subset $B\subset [1,\infty)$ of upper density
$$
\limsup_{T\rightarrow\infty} \frac{1}{T} \int_1^{T} \pmb{1}_B (t) \d t < \eps
$$
such that, for every $\sigma\geq \alpha$ and $k\in\N$,
\begin{equation}\label{limT}
\limsup_{T\rightarrow \infty}\frac{1}{T}\int_1^T \left|\zeta(\sigma+it) \right|^{2k}\pmb{1}_{B^c}(t) \d t \leq M_{\eps}^{2k}.
\end{equation}}This provides already a weak version of Tanaka's result.\par
Tanaka used some additional ergodic theoretical reasoning  to control the limit in \eqref{limT} as $\eps\rightarrow 0$. He obtained that, \textit{for every $\eps>0$ and $\alpha>\frac{1}{2}$, there exist a subset $A\subset [1,\infty)$ of density
$$
\lim_{T\rightarrow\infty} \frac{1}{T} \int_1^{T} \pmb{1}_A (t) \d t =0
$$
such that, for every $\sigma\geq \alpha$ and $k\in\N$,
$$
\lim_{T\rightarrow \infty}\frac{1}{T}\int_1^T \left|\zeta(\sigma+it) \right|^{2k}\pmb{1}_{A^c}(t) \d t =\sum_{n=1}^{\infty}\frac{d_k(n)^2}{n^{2\sigma}}.
$$}\par We extend Tanaka's result to a large class of functions which we shall denote by $\mathcal{N}$. We rely here essentially on the ideas and methods developed by Tanaka \cite{tanaka:2008}.\par
Two features of a function $\L$ are crucial in order to obtain an asymptotic expansion for the second power moment $(k=1)$ of $\L$ in the sense of Tanaka:
\begin{itemize}
 \item[(i)] In the half-plane $\sigma>1$, the function $\L$ is represented by a Dirichlet series
$$
\L(s)=\sum_{n=1}^{\infty} \frac{a(n)}{n^s}, \qquad \sigma>1,
$$ 
with coefficients $a(n)\in\C$ satisfying
$$
\sum_{n=1}^{\infty} \frac{|a(n)|^2}{n^{\sigma}}<\infty, \qquad \mbox{ for }\sigma>1.
$$
\item[(ii)] The function $\L$ satisfies a certain normality feature in the half-plane $\sigma>\frac{1}{2}$ which we shall define later on in Section \ref{sec:classN}. This normality feature is more or less equivalent to the boundedness property \eqref{e} stated above for the Riemann zeta-function. 
\end{itemize}
In order to obtain asymptotic expansions for the $2k$-th moment for $\L$ in the sense of Tanaka with $k\in\N$, it is necessary that both $\L$ and its $k$-th powers $\L^k$ satisfy property (i). For this purpose, we study in Chapter \ref{ch:coeff} the Dirichlet series expansions of $\L^k$ and other functions related to a given Dirichlet series $\L$. Here, we mainly work with Dirichlet series that satisfy the Ramanujan hypothesis.\par

Dirichlet series can be modeled by an ergodic flow on the infinite dimensional torus. We outline this concept in Chapter \ref{ch:poly} and focus especially on the results which we shall need later on to prove our extended version of Tanaka's result in Chapter \ref{ch:probmom}. \par

In Chapter \ref{ch:classN} we define the normality feature stated in property (ii) above and set up the class $\No$. We investigate basic properties of functions in $\mathcal{N}$ which we shall need later on to prove our main result.\par

In Chapter \ref{ch:probmom} we state and prove our main theorem, i.e. an extended version of Tanaka's result.

\chapter{Arithmetic functions and Dirichlet series coefficients}\label{ch:coeff}
In this chapter we study the Dirichlet series expansions of certain functions related to a given Dirichlet series $\L$. We start with some basic properties of certain arithmetic functions.

\section{Arithmetic functions connected to the Riemann zeta-function}
A function $a:\N\rightarrow\C$ is said to be an {\it arithmetic function}. According to the fundamental theorem of arithmetic, for every $n\in\N$ and every prime number $p\in\mathbb{P}$ with $p|n$, there exist uniquely determined quantities $\nu(n;p)\in\N$ such that
\begin{equation}\label{eq:fundamentaltheoremarith}
n = \prod_{\begin{subarray}{c}p\in\mathbb{P}\\ p|n\end{subarray}} p^{\nu(n;p)}.
\end{equation} 
We call an arithmetic function $a:\N\rightarrow\C$ {\it multiplicative} if
$$
a(n)=\prod_{\begin{subarray}{c}p\in\mathbb{P}\\ p|n\end{subarray}}a( p^{\nu(n;p)}) \qquad \mbox{for }n\in\N.
$$
If $a$ satisfies the stronger property
$$
a(n)=\prod_{\begin{subarray}{c}p\in\mathbb{P}\\ p|n\end{subarray}}a( p)^{\nu(n;p)} \qquad \mbox{for }n\in\N,
$$
we call the function $a$ {\it completely multiplicative}. For two arithmetic functions $a,b:\N\rightarrow\C$, the arithmetic function $a*b:\N\rightarrow\C$ defined by
$$
(a*b)(n):= \sum_{\begin{subarray}{c}(n_1,n_2)\in\N^2 \\ n_1n_2 = n\end{subarray}} a(n_1)b(n_2) \qquad \mbox{for }n\in\N
$$
is said to be the {\it Dirichlet convolution of $a$ and $b$}.\par 
We associate a given arithmetic function $a:\N\rightarrow\C$, $n\mapsto a(n)$ with the formal Dirichlet series
$$
\L(s)=\sum_{n=1}^{\infty}\frac{a(n)}{n^s}, 
$$
and vice versa. We say that the coefficients of a given Dirichlet series are (completely) multiplicative if its associated arithmetic function is (completely) multiplicative. If the Dirichlet series 
$$
\L_1(s)=\sum_{n=1}^{\infty}\frac{a(n)}{n^s}\qquad \mbox{and}\qquad \L_2(s)=\sum_{n=1}^{\infty}\frac{b(n)}{n^s}
$$
converge absolutely for a given $s\in\C$, then their product $\L_1(s)\cdot\L_2(s)$ is also an absolutely convergent Dirichlet series given by
\begin{equation}\label{eq:l1l2}
\L_1 (s)\cdot \L_2(s) =\sum_{n=1}^{\infty}\frac{(a*b)(n)}{n^s}.
\end{equation}
{\bf The generalized divisor function.} We fix $\kappa\in\R$ and define, for every prime number $p\in\mathbb{P}$ and every $\nu\in\N_0$,
$$
d_{\kappa}(p^{\nu}):=\binom{\kappa+\nu-1}{\nu},
$$
where the generalized binomial coefficient is defined in a standard way:
$$
\binom{x}{0} := 1 \quad\mbox{and} \quad \binom{x}{\nu}:=  \frac{1}{\nu!} \prod_{j=0}^{\nu-1} (x-j) \quad \mbox{for } x\in\R \mbox{ and } \nu\in\N.
$$
From the functional equation of the Gamma-function, we derive that
$$
d_{\kappa}(p^{\nu}) = \frac{\Gamma(\kappa+\nu)}{\Gamma(\kappa) \nu!}.
$$ 
We set
$$
d_{\kappa}(n):= \prod_{\begin{subarray}{c}p\in\mathbb{P}\\ p|n\end{subarray}}d_{\kappa}( p^{\nu(n;p)}).
$$
for every $n\in\N$ and refer to the arithmetic function $d_{\kappa}:\N\rightarrow\N$ defined by $n\mapsto d_{\kappa}(n)$ as generalized divisor function with parameter $\kappa$. By definition, the function $d_{\kappa}$ is multiplicative. The Dirichlet series associated with $d_{\kappa}$ can be identified with the Dirichlet series expansion of $\zeta(s)^{\kappa}:= \exp(\kappa \log\zeta(s))$ in $\sigma>1$, where $\log \zeta(s)$ is defined in a standard way. For details we refer to Tenenbaum \cite[Chapt. II.5.1]{tenenbaum:1995} or Heath-Brown \cite{heathbrown:1981}. In the following lemma, we gather some well-known properties of the generalized divisor function if its allied parameter is an integer. In this case, $d_{\kappa}$ has an important number theoretical interpretation: for $k\in\N$, the quantity $d_k(n)$ counts the representations of $n\in\N$ as a product of $k$ natural numbers; additionally, we have $d_{-k}(n)=0$ if and only if there exist a prime number $p$ with $p|n$ such that $\nu(n;p)\geq k+1$.
\begin{lemma}\label{lem:divisor} Let $k\in\N$.
\item[(a)] For $n\in\N$,
$$
d_0(n)=\begin{cases} 1 & \mbox{if }n=1, \\ 0 & \mbox{if $n\neq 1$,}\end{cases}
\qquad\mbox{and}\qquad
d_1(n)=1. 
$$
\item[(b)] The function $d_k$ is the $k$-fold Dirichlet convolution of $d_1$. In particular, for $n\in\N$,
$$
d_k(n)=\sum_{\begin{subarray}{c} (n_1,...,n_k) \in\N^k \\ n_1\cdots n_k = n \end{subarray}} 1.
$$ 
\item[(c)] For sufficiently large $n\in\N$,
$$
1\leq d_k(n) \leq \exp\left( (k-1) \log 2 \frac{\log n}{\log\log n}\left(1 + O\left(\frac{\log\log\log n}{\log\log n}\right) \right) \right).
$$
\item[(d)] The function $d_{-1}$ is the classical M\"obius function $\mu$. In particular, for $n\in\N$,
$$
d_{-1}(n) = \mu(n) = 
\begin{cases}
(-1)^r & \mbox{if } n=p_1\cdots p_r \mbox{ with pairwise distinct } p_1,...,p_r\in\mathbb{P}, \\ 
0 & \mbox{otherwise.} 
\end{cases}
$$
\item[(e)] The function $d_{-k}$ is the $k$-fold Dirichlet convolution of $d_{-1}$. In particular, for $p\in\mathbb{P}$ and $\nu\in\N$ with $\nu\geq k+1$, 
$$
d_{-k}(p^{\nu}) = 0.
$$
\item[(f)] For $n\in\N$, we have $|d_{-k}(n)|\leq d_{k}(n)$.
\end{lemma}
\begin{proof}[Sketch of the proof:]
The statements (a) and (d) follow directly from the definition of $d_k$ and a short computation. As the Dirichlet series associated with $d_k$ can be identified as the Dirichlet series expansion of $\zeta(s)^k$ in $\sigma>1$, we obtain that
$$
\zeta(s)^k = \left(\sum_{n=1}^{\infty} \frac{1}{n^s} \right)^k = \sum_{n=1}^{\infty}\frac{d_k(n)}{n^s},\qquad \sigma>1.
$$
Therefrom, we deduce that
\begin{equation}\label{eq:kfoldDirichlet}
d_k = d_{k-1}* d_1=\underbrace{d_1 * ...* d_1}_{k\scalebox{0.8}{\mbox{-times}}}.
\end{equation}
Hence, $d_k$ is the $k$-fold Dirichlet convolution of $d_1$. This implies that 
\begin{equation}\label{repdk}
d_k(n) = \sum_{\begin{subarray}{c} (n_1,...,n_k) \in\N^k \\ n_1\cdots n_k = n \end{subarray}} 1
\qquad \mbox{for }n\in\N .
\end{equation}
Alternatively, the latter identity can be derived from the fundamental theorem of arithmetics and a combinatorial argument. Altogether, statement (b) follows. By standard estimates, we obtain that, for sufficiently large $n\in\N$,
\begin{equation}\label{eq:d2n}
d_2(n) \leq \exp\left( \log 2 \frac{\log n}{\log\log n}\left(1 + O\left(\frac{\log\log\log n}{\log\log n}\right) \right) \right);
\end{equation}
see, for example, Steuding \cite[Chapt. 2.3]{steuding:2007} or Hardy \& Wright \cite[Chapt. 18.1]{hardywright:1979}. It follows from \eqref{eq:kfoldDirichlet} that, for $n\in\N$,
$$
d_k(n) = \sum_{\begin{subarray}{c} (n_1,n_2) \in\N^2 \\ n_1\cdot n_2 = n \end{subarray}} d_{k-1}(n_1) d_1(n_2)\leq \, d_{k-1}(n)\cdot \hspace{-0.3cm}\sum_{\begin{subarray}{c} (n_1,n_2) \in\N^2 \\ n_1\cdot n_2 = n \end{subarray}} \hspace{-0.1cm}1 =  d_{k-1}(n)\cdot d_2(n).
$$
Iteratively, we obtain that
$$
d_k(n) \leq d_2(n)^{k-1}\qquad \mbox{for }n\in\N.
$$
Together with \eqref{eq:d2n}, this proves the upper bound for $d_k(n)$ in statement (c). The lower bound for $d_k(n)$, i.e. 
$$
d_k(n) \geq 1 \qquad \mbox{for }n\in\N,
$$
follows immediately from the definition of $d_k$. The identity
$$
\zeta(s)^{-k} = \left(\sum_{n=1}^{\infty} \frac{d_{-1}(n)}{n^s}  \right)^{k} =  \sum_{n=1}^{\infty} \frac{d_{-k}(n)}{n^s}, \qquad \sigma>1. 
$$
implies that $d_{-k}$ is the $k$-fold Dirichlet convolution of $d_{-1}$. Moreover, we deduce from the definition of $d_{\kappa}$ that, for any $\kappa\in\R$, $p\in\mathbb{P}$ and $\nu\in\N$,
\begin{equation}\label{product1}
d_{\kappa}(p^{\nu}) = \frac{1}{\nu!} \prod_{j=0}^{\nu-1} (\kappa-\nu + 1 -j) =  \frac{1}{\nu!} \prod_{j=0}^{\nu-1}  (\kappa + j).
\end{equation}
If we set $\kappa=-k$ in \eqref{product1}, then zero occurs as a factor in the products on the righthand-side, whenever $\nu\geq k+1$. Statement (e) follows. Moreover, we deduce from \eqref{product1} that, for $p\in\mathbb{P}$ and $\nu\in\N$,
$$
\left| d_{-k}(p^{\nu}) \right|  = \left|  \frac{1}{\nu!} \prod_{j=0}^{\nu-1}  (-k + j)\right| \leq  \frac{1}{\nu!} \prod_{j=0}^{\nu-1}  (k + j) = d_{k}(p^{\nu}).
$$
Statement (f) follows then from the multiplicity of $d_{-k}$ and $d_{k}$.
\end{proof}
Now, we state some basic properties of $d_{\kappa}$ in the general situation if its allied parameter $\kappa$ is an arbitrary real number.

\begin{lemma}\label{lem:gendivisor}
Let $\kappa\in\R$. Then, the following statements are true.
\begin{itemize}
\item[(a)] If $\kappa \geq 0$, then $d_{\kappa}(n)\geq 0$ for $n\in\N$. If $\kappa\geq 1$, then $d_{\kappa}(n)\geq 1$ for $n\in\N$.
\item[(b)] For $n\in\N$ and $K\geq |\kappa|$, we have $\left| d_{\kappa}(n) \right|  \leq d_{K}(n)$.
 \item[(c)] There is an absolute constant $c>0$ such that, for every $p\in\mathbb{P}$ and every $\nu\in\N_0$,
$$
d_{\kappa}(p^{\nu}) < c|\kappa| \nu^{\kappa-1}.
$$
\item[(d)] For sufficiently large $n\in\N$, 
$$
\left| d_{\kappa}(n) \right| \leq \exp\left( c(\kappa) \log 2 \frac{\log n}{\log\log n}\left(1 + O\left(\frac{\log\log\log n}{\log\log n}\right) \right) \right).
$$
where $c(\kappa):= \min\{n\in\N \, : \, n\geq |\kappa|-1\}$.
\item[(e)] For $m\in\N$, the function $d_{m\kappa}$ is the $m$-fold Dirichlet convolution of $d_{\kappa}$. In particular, for $n\in\N$,
\begin{align*}
d_{m\kappa}(n)&= \sum_{\begin{subarray}{c}(n_1,...,n_m)\in\N^m\\ n_1\cdots n_m = n\end{subarray}} d_{\kappa}(n_1)\cdots d_{\kappa}(n_2) 
=  \prod_{\begin{subarray}{c} p\in\mathbb{P}\\ p|n\end{subarray}} \;
\sum_{\begin{subarray}{c} k_1,...,k_m \in \N_0 \\ k_1+...+k_m =\nu(n;p) \end{subarray}}
\prod_{j=1}^m d_\kappa(p^{k_j}).
\end{align*}

\end{itemize}
\end{lemma}

\begin{proof}
Let $\kappa\in\R$. It follows immediately from \eqref{product1} and the multiplicativity of $d_{\kappa}$ that, for $n\in\N$,
$$
d_{\kappa}(n)\geq \begin{cases}0 & \mbox{if }\kappa\geq 0, \\ 1 & \mbox{if }\kappa\geq 1 .\end{cases}
$$
Statement (a) is proved. Similarly, we derive statement (b) from the observation that, for $p\in\mathbb{P}$, $\nu\in\N$ and any $K\geq |\kappa|$,
$$
\left|d_{\kappa}(p^{\nu})\right| \leq \frac{1}{\nu!}\prod_{j=0}^{\nu-1} (\left|\kappa\right| +j)\leq d_{K}(p^{\nu}).
$$
Now, we rewrite \eqref{product1} in the form
\begin{equation}\label{darstellung}
d_{\kappa}(p^{\nu}) = \frac{\kappa}{\nu}\cdot \prod_{j=1}^{\nu-1} \left( 1+\frac{\kappa}{j}\right).
\end{equation} 
By using the inequality $1+x \leq \exp(x)$, which is true for any real $x$, we obtain that, for $p\in\mathbb{P}$ and $\nu\in\N$,
$$
\left| d_{\kappa}(p^{\nu}) \right| \leq  \frac{|\kappa|}{\nu}\exp\left( \kappa \sum_{j=1}^{\nu-1}\frac{1}{j}\right).
$$
A standard asymptotic estimate for the partial sums of the harmonic series yields the existence of an absolute constant $c>0$ such that, for every $p\in\mathbb{P}$ and every $\nu\in\N_0$,
$$
d_{\kappa}(p^{\nu}) < c |\kappa| \nu^{\kappa-1}.
$$
Statement (c) is proved. Statement (d) follows by combining the estimates of statement (b) and Lemma \ref{lem:divisor} (c). The identity
$$
\sum_{n=1}^{\infty}\frac{d_{\kappa m}(n)}{n^s}=\zeta(s)^{\kappa m} = \left(\zeta(s)^{\kappa}\right)^m = \left(\sum_{n=1}^{\infty}\frac{d_{\kappa}(n)}{n^s}\right)^m
$$
implies that $d_{m\kappa}$ is the $m$-fold Dirichlet convolution of $d_{\kappa}$ and that
$$
d_{m\kappa}(n)=\sum_{\begin{subarray}{c}(n_1,...,n_m)\in\N^m\\ n_1\cdots n_m = n\end{subarray}} d_{\kappa}(n_1)\cdots d_{\kappa}(n_2) \qquad\mbox{for }n\in\N
$$
The multiplicativity of $d_{\kappa}$ assures that
$$
d_{m\kappa}(n)
=  \prod_{\begin{subarray}{c} p\in\mathbb{P}\\ p|n\end{subarray}} \;
\sum_{\begin{subarray}{c} k_1,...,k_m \in \N_0 \\ k_1+...+k_m =\nu(n;p) \end{subarray}}
\prod_{j=1}^m d_\kappa(p^{k_j})\qquad \mbox{for }n\in\N.
$$ 
Statement (e) is proved.
\end{proof}

{\bf Von Mangoldt function.} The von Mangoldt function $\Lambda:\N\rightarrow \R$ is defined by
$$
\Lambda(n)= 
\begin{dcases}
\log p & \mbox{if } n=p^{\nu} \mbox{with some $p\in\mathbb{P}$ and $\nu\in \N$},\\
\quad  0 & \mbox{otherwise}.
\end{dcases}
$$
The von Mangoldt function appears in the Dirichlet series expansion of the logarithmic derivative of the Riemann zeta-function. More precisely,
$$
\frac{\zeta'(s)}{\zeta(s)} = - \sum_{n=1}^{\infty} \frac{\Lambda(n)}{n^s},\qquad \sigma>1.
$$

\section{The coefficients of certain Dirichlet series}\label{sec:coeff}
Let 
\begin{equation}\label{dirichletseries}
\L(s)=\sum_{n=1}^{\infty}\frac{a(n)}{n^s}
\end{equation}
be a Dirichlet series with coefficients $a(n)\in \C$. In the sequel, we shall pose certain conditions on the coefficients of $\L(s)$. Here, we work basically with the Ramanujan hypothesis and a polynomial Euler product representation. For the convenience of the reader, we recall the definition of these two features of a Dirichlet series.\par
\textbf{\textit{Ramanujan hypothesis.}}
{\it \textbf{}} $\L(s)$ is said to satisfy the Ramanujan hypothesis if, for any $\eps>0$,
$$
a(n)\ll_{\eps} n^{\eps},
$$
as $n\rightarrow\infty$. Here, the constant inherent in the Vinogradov symbol may depend on $\eps$. \par
\textbf{\textit{Polynomial Euler product.}}
{\it \bf } $\L(s)$ is said to have a representation as a polynomial Euler product if there exist a positive integer $m$ and complex numbers $\alpha_1(p)$,...,$\alpha_m(p)$ such that
\begin{equation}\label{eulerproduct}
\L(s)=\prod_{p\in\mathbb{P}} \prod_{j=1}^m \left(1-\frac{\alpha_j(p)}{p^s}\right)^{-1}.
\end{equation}
We call the coefficients $\alpha_1(p)$,...,$\alpha_m(p)$ the local roots of $\L(s)$ at $p\in\mathbb{P}$.\par

The Ramanujan hypothesis regulates the growth behaviour of the coefficients $a(n)$, as $n\rightarrow\infty$. The polynomial Euler product representation implies a multiplicative structure of the coefficients $a(n)$.\par

In the following, we study the coefficients of Dirichlet series which are related to $\L(s)$ by means of certain analytic transformations. In particular, we investigate whether these related Dirichlet series satisfy the Ramanujan hypothesis under the presumption that the latter is true for $\L(s)$. We would like to stress that the theorems of this section are not to be considered as new. They occur in slightly modified versions in many papers and monographies dealing with Dirichlet series.\par

First, we observe that the set of Dirichlet series which converge absolutely in a given point $s\in\C$ and satisfy the Ramanujan hypothesis is closed under finite summation and multiplication.
\begin{theorem}\label{th:sumprodL} Let $k\in\N$ and 
$$
\L_j(s)=\sum_{n=1}^{\infty} \frac{a_j(n)}{n^s}, \qquad j=1,...,k
$$ 
be absolutely convergent Dirichlet series which satisfy the Ramanujan hypothesis. Then, the following statements are true.
\begin{itemize}
\item[(a)] The sum $\mathcal{B}(s):=\L_1(s) + \dots + \L_k(s)$ can be written as an absolutely convergent Dirichlet series which satisfies the Ramanujan hypothesis. In particular,
$$
\mathcal{B}(s)=\sum_{n=1}^{\infty} \frac{b(n)}{n^s}\qquad\mbox{with}\qquad
b(n)=\sum_{j=1}^k a_j(n).
$$
\item[(b)] The product $\mathcal{C}(s):=\L_1(s)\cdots \L_k(s)$ can be written as an absolutely convergent Dirichlet series which satisfies the Ramanujan hypothesis. In particular,
$$
\mathcal{C}(s)=\sum_{n=1}^{\infty} \frac{c(n)}{n^s} \qquad \mbox{with}\qquad c(n)= \sum_{\begin{subarray}{c} (n_1,...,n_k) \in\N^k \\ n_1\cdots n_k = n \end{subarray}} a_1(n_1)\cdots a_k(n_k ).
$$
\end{itemize}
\end{theorem}

\begin{proof}
Certainly, the sum of finitely many absolutely convergent Dirichlet series is again an absolutely convergent Dirichlet series. By Riemann's rearrangement theorem, we can write 
$$
\mathcal{B}(s)=\sum_{j=1}^k \sum_{n=1}^{\infty} \frac{a_j(n)}{n^s} = \sum_{n=1}^{\infty} \frac{b(n)}{n^s}\qquad\mbox{with}\qquad
b(n):=\sum_{j=1}^k a_j(n).
$$
As $\L_1(s),...,\L_k(s)$ satisfy the Ramanujan hypothesis for $j=1,...,k$, respectively, the same is true for $\mathcal{B}(s)$.\par

In a similar manner, we deduce that the product of finitely many absolutely convergent Dirichlet series is again an absolutely convergent Dirichlet series. By means of \eqref{eq:l1l2}, we obtain that
$$
\mathcal{C}(s)=\prod_{j=1}^k \L_j(s) = \sum \frac{c(n)}{n^s} \qquad \mbox{with}\qquad c(n)= \sum_{\begin{subarray}{c} (n_1,...,n_k) \in\N^k \\ n_1\cdots n_k = n \end{subarray}} a_1(n_1)\cdots a_k(n_k ).
$$
The Ramanujan hypothesis for the coefficients of the Dirichlet series $\L_1(s),...,\L_k(s)$ implies that, for any $\eps>0$,
$$
c(n)\ll n^{\eps}\cdot \sum_{\begin{subarray}{c} (n_1,...,n_k) \in\N^k \\ n_1\cdots n_k = n \end{subarray}} 1 = n^{\eps}\cdot d_k(n),
$$
as $n\rightarrow\infty$. It follows from the estimate for $d_k(n)$ in Lemma \ref{lem:gendivisor} (d) that the coefficients $\mathcal{C}(s)$ satisfies the Ramanujan hypothesis.
\end{proof}

Next, we study absolutely convergent Dirichlet series which possess a representation as a polynomial Euler product. Steuding \cite[Chapt. 2.3]{steuding:2007} revealed the following relation between the Dirichlet series coefficient $a(n)$ and the local roots $\alpha_1(p),...,\alpha_m(p)$ of $\L(s)$.

\begin{theorem}[Steuding, 2007]\label{coeffpoleuler}
Let $\L(s)$ be an absolutely convergent Dirichlet series of the form \eqref{dirichletseries} that can be written as a polynomial Euler product of the form \eqref{eulerproduct}. Then, the following statements are true.
\begin{itemize}
\item[(a)] The Dirichlet series coefficients of $\L(s)$ are multiplicative and satisfy
$$
a(1)=1 \qquad \mbox{ and }\qquad a(n) = \prod_{\begin{subarray}{c} p\in\mathbb{P}\\ p|n\end{subarray}} \quad
\sum_{\begin{subarray}{c} k_1,...,k_m \in \N_0 \\ k_1+...+k_m =\nu(n;p) \end{subarray}}
\prod_{j=1}^m \alpha_j(p)^{k_j}
$$
for $n\in\N\setminus\{1\}$, where $\alpha_1(p),...,\alpha_m(p)$ denote the local roots of $\L(s)$ at $p\in\mathbb{P}$. 
\item[(b)] $\L(s)$ satisfies the Ramanujan hypothesis if and only if
$$
\max_{j=1,...,m} |\alpha_j(p)| \leq 1 
$$
for every $p\in\mathbb{P}$.
\end{itemize}
\end{theorem}
For a proof, we refer to Steuding \cite[Chapt. 2.3]{steuding:2007}.\par

Now, we turn our attention to the analytic function that is described by a Dirichlet series for which the Ramanujan hypothesis is true.\par 
Suppose that $\L(s)$ satisfies the Ramanujan hypothesis. Then, $\L(s)$ is absolutely convergent in the half-plane $\sigma>1$ and defines there an analytic function $\L$. For $\ell\in\N$, let $\L^{(\ell)}$ denote the $\ell$-th derivative of $\L$ in $\sigma>1$. We shall see that $\L^{(\ell)}$ has a Dirichlet series expansion in $\sigma>1$ which is related to the one of $\L$.
\begin{theorem}\label{th:Dirichletderivative}
Let $\ell\in\N$ and $\L$ be an analytic function in $\sigma>1$ that is given by a Dirichlet series of the form \eqref{dirichletseries} for which the Ramanujan hypothesis is true. Then, $\L^{(\ell)}$ has Dirichlet series expansion 
$$
\L^{(\ell)}(s) = \sum_{n=1}^{\infty}\frac{a_{{(\ell)}}(n)}{n^s}, \qquad \sigma>1,
$$
where
$$
a_{{(\ell)}}(n)= (-1)^{\ell}a(n)(\log n)^{\ell}.
$$
In particular, the Dirichlet series representing $\L^{(\ell)}$ satisfies the Ramanujan hypothesis.
\end{theorem}
\begin{proof}
The statement follows from the observation that
$$
\frac{\d^{\ell}}{\d s^{\ell}} \frac{a(n)}{n^s} = \frac{a(n)(\log n)^{\ell}}{n^s}, \qquad n\in\N, 
$$
and basic convergence properties of series of analytic functions; see Busam \& Freitag \cite[Theorem III.1.6]{busamfreitag:2009}.
\end{proof}

Suppose that $\L(s)$ satisfies the Ramanujan hypothesis and has a polynomial Euler product representation. Then, $\L(s)$ defines an analytic, non-vanishing function $\L$ in $\sigma>1$. Thus, there exists an analytic logarithm of $\L$ in $\sigma>1$. In the further course of our investigations, we define $\log \L$ as follows. Due to the polynomial Euler product representation, the leading Dirichlet series coefficient of $\L(s)$ is given by $a(1)=1$. By the absolute convergence of $\L(s)$ in $\sigma>1$, we find a $\sigma_{0} \geq 1 $ such that 
$$
\left| \L(s) - 1 \right|  \leq \tfrac{3}{4}, \qquad \mbox{for }\sigma> \sigma_{0}.
$$
For $s\in\C$, let $\mbox{Log}\ s$ denote the principal branch of the logarithm. If we set 
$$
(\log\L) (s) := \mbox{Log} (\L(s))\qquad \mbox{ for }\sigma>\sigma_0,
$$
then $(\log\L) (s)$ defines a uniquely determined analytic logarithm of $\L$ in the half-plane $\sigma> \sigma_0$ with the property that
$$
\Im \left( \log \L(s) \right) \in [-\pi,\pi) \qquad \mbox{ for }\sigma> \sigma_{0}.
$$ 
For all other simply connected domains $\Omega\subset \C$ which do not contain any zero of $\L$ and have a non-empty intersection $\mathcal{I}$ with the half-plane $\sigma>\sigma_{0}$, we define $\log\L$ by extending $\log\L |_{\mathcal{I}}$ analytically to $\Omega$. In particular, we obtain in this manner a uniquely determined analytic logarithm of $\L$ in $\sigma>1$ that satisfies the normalization
$$
\lim_{\sigma\rightarrow+\infty}  \log \L(\sigma) = 0.
$$

\begin{theorem}\label{th:Dirichletlog}
Let $\L$ be an analytic function in $\sigma>1$ that is given by a Dirichlet series of the form \eqref{dirichletseries} which satisfies the Ramanujan hypothesis and can be written as a polynomial Euler product of the form \eqref{eulerproduct}. Then, $\log \L$ has Dirichlet series expansion 
$$
\log \L(s) = \sum_{n=1}^{\infty} \frac{a_{\log \L}(n)}{n^s}, \qquad \sigma>1,
$$
where 
$$
a_{ \log \L}(n) = 
\begin{dcases}
-\frac{1}{\nu} \sum_{j=1}^{m}  \alpha_{j} (p)^{\nu} & \mbox{if } n=p^{\nu} \mbox{with some $p\in\mathbb{P}$ and $\nu\in \N$},\\
\quad 0 & \mbox{otherwise}.
\end{dcases}
$$
In particular, the Dirichlet series representing $\log \L$ satisfies the Ramanujan hypothesis.
\end{theorem}
\begin{proof}
Due to Theorem \ref{coeffpoleuler} (b), the Ramanujan hypothesis implies that
\begin{equation}\label{log1}
\max_{j=1,...,m} |\alpha_j(p)| \leq 1
\end{equation}
for $p\in\mathbb{P}$.
Consequently, we get that, for $p\in\mathbb{P}$ and $\sigma>1$,
$$
\sum_{j=1}^m \log \left( 1- \frac{\alpha_j(p)}{p^s}\right)^{-1} =  -  \sum_{j=1}^m \sum_{\nu=1}^{\infty} \frac{\alpha_{j}(p)^{\nu}}{\nu p^{\nu s}}
= - \sum_{\nu=1}^{\infty} \frac{\frac{1}{\nu}\sum_{j=1}^m\alpha_{j}(p)^{\nu}}{ p^{\nu s}}.
$$
This implies that, for every $p\in\mathbb{P}$ and $\sigma>1$,
$$
\left|\sum_{j=1}^m \log \left( 1- \frac{\alpha_j(p)}{p^s}\right)^{-1} \right| \leq \frac{2m}{p^{\sigma}}.
$$
Hence, the series 
\begin{equation*}\label{logbranch}
f(s):= \sum_{p\in\mathbb{P}} \sum_{j=1}^m \log \left( 1- \frac{\alpha_j(p)}{p^s}\right)^{-1}
\end{equation*}
converges absolutely in $\sigma>1$. By means of the polynomial Euler product representation, $f(s)$ defines an analytic logarithm of $\L$ in $\sigma>1$. According to the observation that, uniformly for $t\in\R$,
$$
 f(\sigma + it) = o(1), \qquad \mbox{as }\sigma\rightarrow +\infty,
$$ 
the branch of the logarithm $f$ and the branch chosen for $\log \L$ coincide. Hence, we derive that, for $\sigma>1$,
$$
\log \L(s) =  - \sum_{\nu=1}^{\infty} \frac{\frac{1}{\nu}\sum_{j=1}^m\alpha_{j}(p)^{\nu}}{ p^{\nu s}}
= \sum_{n=1}^{\infty} \frac{a_{\log \L}(n)}{n^s}.
$$
Due to \eqref{log1}, we obtain that $|a_{\log \L}(n)|\leq m$ for $n\in\N$. Hence, the Dirichlet series representing $\log \L$ satisfies the Ramanujan hypothesis and the theorem is proved.
\end{proof}
Next, we consider the logarithmic derivative of $\L$. We shall see that there appears a generalized form of the von Mangoldt function in the Dirichlet series expansion of $\L'/\L$.
\begin{theorem}\label{th:dirichletlogderivative}
Let $\L$ be an analytic function in $\sigma>1$ that is given by a Dirichlet series of the form \eqref{dirichletseries} which satisfies the Ramanujan hypothesis and can be written as a polynomial Euler product of the form \eqref{eulerproduct}. Then, the logarithmic derivative $\L'/\L$ has Dirichlet series expansion 
$$
\frac{\L'(s)}{\L(s)} = -\sum_{n=1}^{\infty} \frac{\Lambda_{ \L}(n)}{n^s}, \qquad \sigma>1,
$$
where 
$$
\Lambda_{ \L}(n) = 
\begin{dcases}
\left( \sum_{j=1}^{m}  \alpha_{j} (p)^{\nu} \right) \log p & \mbox{if } n=p^{\nu} \mbox{with some $p\in\mathbb{P}$ and $\nu\in \N$},\\
\quad 0 & \mbox{otherwise}.
\end{dcases}
$$
In particular, the Dirichlet series representing $\L'/\L$ satisfies the Ramanujan hypothesis.
\end{theorem}
\begin{proof}
By combining Theorem \ref{th:Dirichletderivative} and Theorem \ref{th:Dirichletlog}, we get that $\L'/\L$ has the stated Dirichlet series expansion in $\sigma>1$. Theorem \ref{coeffpoleuler} (b) assures that, for $n\in\N$,
$$
\left| \Lambda_{\L}(n) \right| \leq m\log n.
$$
Hence, the Dirichlet series representing $\L'/\L$ satisfies the Ramanujan hypothesis.
\end{proof}
For $\kappa\in\R$, we define $\L^{\kappa}$ by
$$
\L^{\kappa}(s):=\L(s)^{\kappa}:= \exp\left(\kappa \log \L(s) \right), \qquad \sigma>1.
$$
The next theorem deals with the Dirichlet series expansion of $\L^{\kappa}$.
\begin{theorem}\label{th:Lkappa}
Let $\L$ be an analytic function in $\sigma>1$ that is given by a Dirichlet series of the form \eqref{dirichletseries} which satisfies the Ramanujan hypothesis and can be written as a polynomial Euler product of the form \eqref{eulerproduct}. Then, $\L^{\kappa}$ has Dirichlet series expansion
$$
\L^{\kappa}(s) = \sum_{n=1}^{\infty} \frac{a_{\kappa}(n)}{n^s}, \qquad \sigma>1,
$$
where
$$
a_{\kappa}(1)=1 \qquad \mbox{ and }\qquad a_{\kappa}(n) = \prod_{\begin{subarray}{c} p\in\mathbb{P}\\ p|n\end{subarray}} \quad
\sum_{\begin{subarray}{c} (k_1,...,k_m) \in \N_0^m \\ k_1+...+k_m =\nu(n;p) \end{subarray}}
\prod_{j=1}^m d_{\kappa}(p^{k_j})\alpha_j(p)^{k_j}
$$
for $n\in\N\setminus\{1\}$. In particular, the coefficients $a_{\kappa}(n)$ are multiplicative and the Dirichlet series representing $\L^{\kappa}$ satisfies the Ramanujan hypothesis.
\end{theorem}

\begin{proof}
Let $\kappa\in\R$. According to Theorem \ref{coeffpoleuler} (b), the Ramanujan hypothesis assures that, for $p\in\mathbb{P}$,
\begin{equation}\label{maxalphaj}
\max_{j=1,..,m}|\alpha_j(p)| \leq 1.
\end{equation}
The Taylor expansion 
$$
(1-z)^{-\kappa} = \sum_{\nu =0}^{\infty} \binom{\kappa+\nu-1}{\nu} z^{\nu} \qquad \mbox{for }z\in\C \mbox{ with }|z|<1,
$$
yields that, for $p\in\mathbb{P}$ and $\sigma>1$,
$$
\left(1-\frac{\alpha_j(p)}{p^s}\right)^{-\kappa}
=  1+\sum_{\nu=1}^{\infty}\frac{d_{\kappa}(p^{\nu})\alpha_j(p)^{\nu}}{p^{\nu s}}.
$$
From \eqref{maxalphaj} and the estimate of Lemma for $d_{\kappa}(n)$, we derive that the series
$$
\sum_{p\in\mathbb{P}} \sum_{j=1}^m \sum_{\nu=1}^{\infty}\frac{ d_{\kappa}(p^{\nu})\alpha_j(p)^{\nu}}{p^{\nu s}}
$$
converges absolutely in $\sigma>1$. Hence, we conclude that
\begin{equation*}
\L(s)^{\kappa}  = \prod_{p\in\mathbb{P}} \prod_{j=1}^m \left(1-\frac{\alpha_j(p)}{p^s}\right)^{-\kappa}
= \prod_{p\in\mathbb{P}} \prod_{j=1}^m \left(1+\sum_{\nu=1}^{\infty}\frac{d_{\kappa}(p^{\nu})\alpha_j(p)^k}{p^{\nu s}}\right),
\qquad \sigma>1,
\end{equation*}
where the infinite product and sum appearing in the latter expression converge absolutely in $\sigma>1$. By multiplying out and rearranging the terms, we obtain that
$$
\L(s)^{\kappa} = \sum_{n=1}^{\infty} \frac{a_{\kappa} (n)}{n^s}, \qquad \sigma>1,
$$
with $a_{\kappa}(1)=1$ and
$$
a_{\kappa}(n) = \prod_{\begin{subarray}{c} p\in\mathbb{P}\\ p|n\end{subarray}} \quad
\sum_{\begin{subarray}{c} (k_1,...,k_m) \in \N_0^m \\ k_1+...+k_m =\nu(n;p) \end{subarray}}
\prod_{j=1}^m d_{\kappa}(p^{k_j})\alpha_j(p)^{k_j} \qquad \mbox{for }n\in\N\setminus\{1\}.
$$
We deduce from \eqref{maxalphaj} and the properties of $d_{\kappa}(n)$ in Lemma \ref{lem:gendivisor} (b) and (e) that, for $n\in\N$,
$$
\left| a_{\kappa}(n) \right| \leq \prod_{\begin{subarray}{c} p\in\mathbb{P}\\ p|n\end{subarray}} \quad
\sum_{\begin{subarray}{c} (k_1,...,k_m) \in \N_0^m \\ k_1+...+k_m =\nu(n;p) \end{subarray}}
\prod_{j=1}^m d_{|\kappa|}(p^{k_j}) =d_{m|\kappa|}(n).
$$
Now, it follows from Lemma \ref{lem:gendivisor} (d) that the coefficients $a_{\kappa}(n)$ satisfy the Ramanujan hypothesis.
\end{proof}
Beyond our considerations, it would be interesting to investigate the situation if $\L(s)$ cannot be written as a polynomial Euler product representation, but as an Euler product of the general form used in the definition of the Selberg class. In this case, some additional obstacles occur. Especially, we get problems to control the growth behaviour of the coefficients appearing in the Dirichlet series expansion of $\zeta(s)^{\kappa}$, $\kappa< 0$, by means of the Ramanujan hypothesis; see de Roton \cite[Sect. 2]{deroton:2009} and
Kaczorowski \& Perelli \cite[Lemma 2]{kaczorowskiperelli:2003}.

\chapter{Dirichlet series and the infinite dimensional torus}\label{ch:poly}

It was an ingenious idea of Bohr \cite{bohr:1913-2} to model Dirichlet series as functions on the infinite dimensional torus. Meanwhile this approach was translated into the modern language of functional analysis and probability theory. Concerning the probabilistic approach, we refer to the pioneering work of Bagchi \cite{bagchi:1981} and the extensive work of Laurin\v{c}ikas, see for example \cite{laurincikas:1991-2}. Concerning the functional analytic point of view, we refer to the seminal papers of Helson \cite{helson:1967,helson:1969} and the recent works by Hedenmalm, Lindqvist \& Seip \cite{hedenmalmlindqvistseip:1997, hedenmalmlindqvistseip:1999} and Tanaka \cite{tanaka:2001, tanaka:2008}.\par

\section{The infinite dimensional torus, the compact group \texorpdfstring{$K$}{ } and a local product decomposition of \texorpdfstring{$K$}{ }}\label{sec:K}
{\bf The discrete group $\Gamma$.} Let $\Lambda$ be a countable set of real numbers which are linearly independent over $\Q$. Let $(\lambda_n)_{n\in\N}$ be a denumeration of the elements of $\Lambda$ in ascending order. Further, let $\Gamma$ be the additive subgroup of $\R$ that is generated by $\Lambda$. It follows from the linear independence of the elements in $\Lambda$ that, for every $\gamma\in\Gamma$, there exist uniquely determined quantities $\nu_n(\gamma) \in \Z$, indexed by $n\in\N$, such that
\begin{equation}\label{coordinates}
\gamma = \sum_{n=1}^{\infty} \nu_n(\gamma)\cdot \lambda_n.
\end{equation}
For given $\gamma\in\Gamma$, all but finitely many of the quantities $\nu_n(\gamma)$ are equal to zero. This implies, in particular, that the sum in \eqref{coordinates} is finite.\par

We endow $\Gamma$ with the discrete topology. In this way, $\Gamma$ becomes a locally compact abelian Hausdorff group (LCA-group). Moreover, as $\Gamma$ has only countable many elements, $\Gamma$ is separable as a topological space. For LCA-groups there is a generalized concept of Fourier analysis. We shall briefly sketch the very basics of this concepts. For details and more information the reader is referred to Deitmar \cite{deitmar:2002} and Hewitt \& Ross \cite{hewittross:1994}.\par

{\bf Excursus: abstract harmonic analysis.} Let ${\sf G}$ be an LCA-group and $\mathbb{T}:=\{z\in\C \, : \,|z|=1\}$ denote the circle group, endowed with the standard topology generated by open arcs. A continuous group homomorphism $\chi: {\sf G} \rightarrow \mathbb{T}$ is said to be a character of ${\sf G}$. Under pointwise multiplication, the set ${\sf G}^*$ of all characters of ${\sf G}$ forms a group, the so called dual group of ${\sf G}$. By endowing ${\sf G}^*$ with the compact-open topology, ${\sf G}^*$ becomes also an LCA-group. It is a fundamental observation that ${\sf G}^*$ is compact, if ${\sf G}$ is discrete and that ${\sf G}^*$ is discrete, if ${\sf G}$ is compact; see Deitmar \cite[Prop. 7.2.1]{deitmar:2002}. In fact, the Pontryagin duality theorem reveals that ${\sf G}$ can be identified group-theoretically and topologically with its bidual ${\sf G}^{**}$; see Hewitt \& Ross \cite[\S 24]{hewittross:1994}. A further observation that we shall use later on is that the dual group ${\sf G}^*$ of ${\sf G}$ is metrizable if ${\sf G}$ is a separable LCA-group. \par
On every LCA-group there exist a non-trivial, non-negative, regular and translation-invariant measure, called Haar-measure, which is unique up to scalar multiplication; see for example Hewitt \& Ross \cite[\S 15,16]{hewittross:1994}. Let ${\sf G}$ be an LCA-group and $\pmb{\sigma}$ a Haar-measure on ${\sf G}$, then we define, for $p\geq 1$, the space
$$
L^p_{\pmb{\sigma}}({\sf G}):= \left\{f:{\sf G}\rightarrow \C \, : \, \int_{{\sf G}} \left| f \right|^p \d \pmb{\sigma}<\infty \right\}.
$$
If $f\in L^1_{\pmb{\sigma}}({\sf G})$, we call $\hat{f}: {\sf G}^*\rightarrow \C$, defined by
\begin{equation}\label{fouriert}
\hat{f}(\chi) = \int_{{\sf G}} f\, \overline{\chi} \d\pmb{\sigma},
\end{equation}
where $\overline{\chi}$ denotes the complex conjugation of a character $\chi\in {\sf G}^*$, the Fourier transform of $f$. \par
The theorem of Plancherel connects the $L^2$-norm of $f$ with the $L^2$-norm of its Fourier transform; see Hewitt \& Ross \cite[\S 31]{hewittross:1994}. The theorem of Plancherel can be considered as an analogue of Parseval's theorem for Fourier series. For sake of simplicity, we assume that ${\sf G}$ is compact. Then, firstly, we find a uniquely determined Haar measure on ${\sf G}$ that satisfies the normalization $\pmb{\sigma}\left({\sf G}\right) =1$. Secondly, the Cauchy-Schwarz inequality implies that $L^2_{\pmb{\sigma}}({\sf G})\subset L^1_{\pmb{\sigma}}({\sf G})$. And thirdly, we know that the dual group ${\sf G}^*$ of ${\sf G}$ is discrete. In this special situation, the theorem of Plancherel states that
\begin{equation}\label{plancherel}
\int_{\sf G}\left| f \right|^2 \d \pmb{\sigma} = \sum_{\chi \in {\sf G}^*} \left| \hat{f}(\chi) \right|^2.
\end{equation}
Here, a central ingredient in the proof is the fact that, for every two characters $\chi,\psi\in {\sf G}^*$,
\begin{equation}\label{fourierorth}
\int_{{\sf G}} \chi\, \overline{\psi} \d \pmb{\sigma}= \begin{cases}1, & \mbox{if }\chi=\psi,\\ 0, &\mbox{otherwise.} \end{cases}
\end{equation}

{\bf The dual group $K$ of $\Gamma$.} In the following, let $K$ be the dual group of $\Gamma$. By the remarks above, we conclude that $K$ is a compact and metrizable group. Hence, there is a unique Haar-measure $\pmb{\sigma}$ on $K$ that satisfies the normalization $\pmb{\sigma}(K)=1$. In the sequel, we denote the elements of $K$, i.e. the characters of $\Gamma$, by
$$
x:\Gamma\rightarrow\mathbb{T}
$$
and the characters of $K$ by
$$
\chi : K \rightarrow \mathbb{T}.
$$

There is a natural identification of $K$ with the infinite-dimensional torus
$$
\mathbb{T}^{\infty}:= \mathbb{T}_1 \times \mathbb{T}_2 \times  ...,
$$
which is given as the direct product of countably many copies $\mathbb{T}_n$ of the unit circle $\mathbb{T}$. For given 
$$
\omega:=\left(e^{i\theta_n}\right)_{n\in\mathbb{N}} = \left(e^{i\theta_1},e^{i\theta_2},... \right)\in \mathbb{T}^{\infty},
$$
we define $x_{\omega}:\Gamma \rightarrow \mathbb{T}$ to be the character of $\Gamma$ that satisfies
$$
x_{\omega}(\lambda_n) = e^{i\theta_n}, \qquad n\in\N.
$$
As the elements of $\Lambda$ are both linearly independent over $\Q$ and generate the group $\Gamma$, the character $x_{\omega}$ is well-defined and uniquely determined. It is easy to see that the map
$$
h:\mathbb{T}^{\infty} \rightarrow K, \qquad \omega \mapsto x_{\omega},
$$ 
is a group isomorphism. If we endow $\mathbb{T}^{\infty}$ with the product topology, then $\mathbb{T}^{\infty}$ is compact, due to Tychonoff's theorem; see Loomis \cite{loomis:1953}. In this case, the map $h$ is also a homeomorphism between $K$ and $\mathbb{T}^{\infty}$ which allows us to identify $K$ with $\mathbb{T}^{\infty}$ in the sequel.\par 

The uniquely determined Haar-measure $\pmb{\sigma}'$ on $\mathbb{T}^{\infty}$ satisfying $\pmb{\sigma}'(\mathbb{T}^{\infty})=1$ coincides with the properly normalized product measure on $\mathbb{T}^{\infty}$: let $\lambda$ be the arc length measure on $\mathbb{T}$, normalized such that $\lambda(\mathbb{T})=1$. Then, for every set
$$
E:=E_1\times...\times E_{N}\times \mathbb{T} \times \mathbb{T}\times ... \subset\mathbb{T}^{\infty}
$$ 
with arbitrary Borel subsets $E_1,..., E_N\subset \mathbb{T}$, we have
$$
\pmb{\sigma}'(E) = \lambda(E_1)\cdots \lambda(E_N).
$$

{\bf A local product decomposition of $K$.} Local product decompositions of compact groups go back to Hoffman \cite{hoffman:1958}. They are important tools to study the structure of compact abelian groups which occur as dual groups of a subgroup of the discrete real line; see Gamlin \cite{gamelin:1969}. Roughly speaking, a local product decomposition of a compact group decomposes the latter into a compact subgroup and a real interval. Tanaka \cite{tanaka:2008} used a local product decomposition of $K$ to model the Riemann zeta-function in the right half of the critical strip. However, without this abstract background, this idea appeared already before in the theory of vertical limit functions for Dirichlet series.\par 
For $t\in\R$, let $e_t$ denote the element of $K$ which is given by
$$
e_t (\gamma) = e^{-it\gamma}, \qquad \gamma\in\Gamma.
$$
Let $\gamma\in\Gamma$. Then, we denote by $\chi_{\gamma}$ the character of $K$ that satisfies $\chi_{\gamma}(x)= x(\gamma)$ for $x\in K$. Let $l>0$ such that $\frac{2\pi}{l}\in\Gamma$. We define 
$$
K_{2\pi/l} := \left\{x\in K \, : \, \chi_{2\pi/l}(x)=1\right\}.
$$
It can be seen easily that $K_{2\pi/ l}$ is a compact subgroup of $K$. Let $\pmb{\tau}$ denote the uniquely determined Haar measure on $K_{2\pi/ l}$ that satisfies the normalization $\pmb{\tau}(K_{2\pi/l})=1$. Via the map
$$
h: K_{2\pi/ l}\times [0, l) \rightarrow K , \qquad (y,u)\mapsto y + e_u ,
$$ 
we can identify $K_{2\pi/ l}\times [0, l)$ with $K$ group theoretically, topologically (if we identify the left end point $0$ of the interval $[0,l)$ with $l$) and measure-theoretically. Here, the measure $\pmb{\sigma}$ on $K$ corresponds to the measure $\pmb{\tau}\times \frac{1}{l}\d t$ on $K_{2\pi/l}\times [0,l)$.

\section{An ergodic flow on \texorpdfstring{$K$}{} and a special version of the ergodic theorem}\label{sec:ergodicflow}
In this section we consider certain ergodic processes on $K_{2\pi/l}$ and $K$. Ergodic theory studies the long term average behaviour of dynamical systems. For basic definitions and results in ergodic theory, the reader is referred to Dajani \& Dirksin \cite{dajanidirksin:2008}, Steuding \cite{steuding:2010} and Cornfeld, Fomin \& Sinai \cite[Chapt. 3, \S 1]{cornfeldfominsinai:1982}. Here, we shall work essentially with the Birkhoff-Khinchine ergodic theorem.\par
We define the map $T:K_{2\pi/l}\rightarrow K_{2\pi/l}$ by
$$
Ty := y+e_l. 
$$
For $n\in \N$, we set 
$$
T^{n}y := \underbrace{(T\circ ... \circ T)}_{n\scalebox{0.7}{\mbox{-times}}} y = y +ne_l
$$
It is well-known that the system $(T,K_{2\pi/l})$ is uniquely ergodic where the unique $T$-invariant probability measure is given by $\pmb{\tau}$. For details, we refer to Cornfeld, Fomin \& Sinai \cite[Chapt. 3, \S 1]{cornfeldfominsinai:1982}. The proof relies essentially on a theorem of Kronecker which states the following:
\begin{theorem}[Theorem of Kronecker]\label{th:kronecker}
Let $\theta_1$,...,$\theta_M\in\R$ such that the numbers $1,\theta_1$,...,$\theta_M$ are linearly independent over $\Q$. Furthermore, let $\alpha_1$,...,$\alpha_M\in \R$, $\eps>0$ and $N\in\N$. Then, there exist $n, q_1,...,q_M\in\N$ with $n>N$ such that
$$
\left|n\theta_m - q_m -\alpha_m \right|< \eps, \qquad m=1,...,M.
$$
\end{theorem}
A proof of Kronecker's theorem can be found in Hardy \& Wright \cite[Chapt. 23]{hardywright:1979}. \par
Let $f\in L^1_{\pmb{\tau}}(K_{2\pi/l})$. The Birkhoff-Khinchine ergodic theorem  implies that, for $\pmb{\tau}$-almost every $y\in K_{2\pi/l}$,
\begin{equation}\label{ergodic}
\lim_{N\rightarrow\infty}\frac{1}{N}\sum_{n=1}^{N} f\left( T^n y \right) = \int_K f \d\pmb{\sigma};
\end{equation}
see Cornfeld, Fomin \& Sinai \cite[Chapt. 1, \S 2]{cornfeldfominsinai:1982}.
Since $T$ is {\it uniquely} ergodic, the formula \eqref{ergodic} holds even for every $y\in K_{2\pi/l}$, if $f$ is continuous on $K_{2\pi/l}$; see Cornfeld, Fomin \& Sinai \cite[Chapt. 1, \S 8]{cornfeldfominsinai:1982}.\par

For $t\in\R$, we define the map $T_t : K \rightarrow K$ by
$$
T_t x = x + e_t.
$$
The set $\{T_t\}_{t\in\R}$ forms a one-parameter group of homeomorphisms of $K$. It is well-known that the flow $(\{T_t\}_{t\in \R}, K)$ is uniquely ergodic. Its associated unique invariant probability measure is given by $\pmb{\sigma}$; see Cornfeld, Fomin \& Sinai \cite[Chapt. 3, \S 1]{cornfeldfominsinai:1982}. The Birkhoff-Khinchine ergodic theorem implies that
\begin{equation}\label{ergodic2}
\lim_{T\rightarrow\infty}\frac{1}{T}\int_0^T f\left( T_t x \right) \d t = \int_K f \d\pmb{\sigma} .
\end{equation}
holds for $\pmb{\sigma}$-almost every $x\in K$, if $f\in L^1_{\pmb{\sigma}}(K)$, and for every $x\in K$, if $f$ is continuous on $K$.\par
The maps $T$ and $T_t$ are strongly connected to one another. If we identify $K$ with $K_{2\pi/l}\times [0,l)$ as described in the preceeding section, the map $T_t$ is represented on $K_{2\pi/l}\times [0,l)$ by 
$$\mathcal{T}_t:K_{2\pi/l}\times[0,l)\rightarrow K_{2\pi/l}\times[0,l),
\qquad (y,u)\mapsto (T^{N_t}y, t+u-N_t l) $$ 
where
$$
N_t := \left[\frac{t+u}{l} \right]
$$
with $[x]$ denoting the largest integer not exceeding $x\in\R$.\par

For $y\in K_{2\pi/ l}$, $l>0$ and $J\subset \N$, we define
\begin{equation}\label{EJ}
E_{l,y}(J):= \overline{\left\{T^n y \, : \, n\in J\right\}} \subset K_{2\pi/ l},
\end{equation}
where the closure is taken with respect to the topology of $K_{2\pi/l}$.
As every closed set in a compact space is compact, we deduce immediately that $E_{l,y}(J)$ is compact. \par

Further, for a subset $J\subset \N$, we define its upper density by
$$
\dens^* J = \limsup_{N\rightarrow\infty} \frac{\# \left(J\cap [1,N] \right) }{N} 
$$
and its lower density by
$$
\dens_* J = \liminf_{N\rightarrow\infty} \frac{\# \left(J\cap [1,N] \right) }{N}.
$$
If $\dens^* J = \dens_* J=d$, we say that $J$ has density $d$ and write $\dens\ J := d$. \par
The following lemma is due to Tanaka \cite[Lemma 3.1]{tanaka:2008} and relies essentially on the ergodicity of $T$. 
\begin{lemma}[Tanaka, 2008]\label{lem:tanaka1}
Let $l>0$, $y\in K_{2\pi/l}$ and $J\subset \N$. Then,
$$
\dens^* (J) \leq \pmb{\tau}\left( E_{l,y}(J)\right).
$$
\end{lemma}
For a proof, we refer to Tanaka \cite[Lemma 3.1]{tanaka:2008}. The following refinement of Lemma \ref{lem:tanaka1} is also due to Tanaka \cite[Lemma 3.2]{tanaka:2008} and yields a modified version of the Birkhoff-Khinchine ergodic theorem.

\begin{lemma}[Tanaka, 2008]\label{lem:Tanakaergodic}
Let $l>0$, $y\in K_{2\pi/ l}$ and $J\subset \N$. Suppose that, for any $\eps>0$, there exists a subset $J_{\eps}\subset \N\setminus J$ such that
$$
\pmb{\tau} \left(E_{l,y}(J)\cap E_{l,y}(J_{\eps}) \right) = 0 \qquad \mbox{and }\qquad \dens^* (\N\setminus(J\cup J_{\eps})) < \eps.
$$
Let $p$ be a function on $K_{2\pi/ l}$ which is continuous on $E_{l,y}(J)$ and zero on $K_{2\pi/l}\setminus E_{l,y}(J)$. Then,
$$
\lim_{N\rightarrow\infty} \frac{1}{N}\sum_{n=0}^{N} p(T^n y) = \int_{E_{l,y}(J)} p(y)\d \pmb{\tau}
$$
and 
$$
\dens (J) = \pmb{\tau}(E_{l,y}(J)).
$$
\end{lemma}
For a proof, we refer again to Tanaka \cite[Lemma 3.2]{tanaka:2008}.

\section{Extension of Dirichlet series to functions on the infinite dimensional torus}
In this section we outline the basic principles to study Dirichlet series as a function on the infinite dimensional torus. In our notation, we follow mainly Tanaka \cite{tanaka:2008}.\par

Let $l>0$ be a fixed parameter and $\Lambda_P = \{\log p \, : \, p\in\mathbb{P}\}$. From now on, let $\Gamma$, $K$ and $K_{2\pi/l}$ be the groups of Section \ref{sec:K} that we obtain for the special choice of 
$$
\Lambda = 
\begin{dcases*}
\Lambda_P\cup\{\tfrac{2\pi}{l}\}   & if $l\notin \{2\pi k(\log \frac{n}{m})^{-1} \, : \, k,n,m\in\mathbb{N}, n\neq m\}$,\\
\Lambda_P & otherwise.
\end{dcases*}
$$
The fundamental theorem of arithmetic and the transcendence of $\pi$ assure that the elements of $\Lambda$ are linear independent over $\Q$.\par

In the sequel, let $\L(s)$ denote a Dirichlet series of the form 
\begin{equation}\label{dirichlet1}
\L(s) = \sum_{n=1}^{\infty} \frac{a(n)}{n^s}.
\end{equation}
To a given Dirichlet series $\L(s)$, we assign a set of allied series which we define formally by 
\begin{equation}\label{def:ext}
L(s,x):= \sum_{n=1}^{\infty} \frac{a(n)}{n^{s}}  \chi_{\log n} (x) \qquad \mbox{ with }x\in K.
\end{equation}
There are some fundamental relations between $\L(s)$ and $L(s,x)$. For $\sigma+it\in\C$ and $x\in K$, we have
$$
L(\sigma+it,x) = \sum_{n=1}^{\infty} \frac{a(n)}{n^{\sigma}} e^{-it\log n}  \chi_{\log n} (x)
= \sum_{n=1}^{\infty} \frac{a(n)}{n^{\sigma}}  \chi_{\log n}(e_t)\chi_{\log n} (x)
$$
$$
= \sum_{n=1}^{\infty} \frac{a(n)}{n^{\sigma}} \chi_{\log n}(x+e_t) = L(\sigma, x+e_t).
$$
Moreover, let $x_0\in K$ denote the principal character in $K$ which is given by
$$
x_0(\gamma)=1 \qquad  \mbox{ for }\gamma\in\Gamma.
$$
Then, the relation
$$
\L(\sigma+it) = L(\sigma+it,x_0) = L(\sigma,e_t)
$$
holds for every $\sigma+it\in\C$. Roughly speaking, these identities allow us to model
$$
\R\ni t\mapsto \L(\sigma+it),
$$
for a properly chosen $\sigma\in\R$, as an ergodic flow on $K$.\par
Formally, $L(s,x)$ defines a function on $\C\times K$ which we denote by $L$, i.e.
\begin{equation}\label{L}
L : (s,x) \mapsto L(s,x), \qquad (s,x)\in \C\times K.
\end{equation}
Later on, it will be convenient to fix $x\in K$ and to consider $L$ as a function on $\C$. For this purpose, we define formally the function $L_x$ on $\C$ via 
\begin{equation}\label{Lx}
L_x : s\mapsto L_x (s):=L(s,x), \qquad s\in \C,
\end{equation}
where we consider $x\in K$ as a fixed parameter. Similarly, it will be useful at some places to fix $s\in \C$ and to regard $L$ as a function on $K$. For this purpose, we define formally the function $L_s$ on $K$ by
\begin{equation}\label{Ls}
L_s : x\mapsto L_s (x):=L(s,x), \qquad x\in K,
\end{equation}
where $s\in\C$ is considered as a fixed parameter.\par

Firstly, we shall consider the Dirichlet series expansion of the functions $L_x$, $x\in K$, attached to given Dirichlet series $\L(s)$ by means of \eqref{Lx}.
\begin{lemma}\label{lem:ext1}
Let $\L(s)$ be a Dirichlet series, $\sigma_a$ its abscissa of absolute convergence and $\sigma_u$ its abscissa of uniform convergence. Then, the following statements are true.
\begin{itemize}
\item[(a)] For every $x\in K$, the abscissa of absolute convergence of the Dirichlet series expansion of $L_x$ coincides with $\sigma_a$.
\item[(b)] For every $x\in K$, the abscissa of uniform convergence of the Dirichlet series expansion of $L_x$ coincides with $\sigma_u$.
\end{itemize}
\end{lemma}
\begin{proof}[Sketch of the proof:]
Statement (a) follows directly from the observation that $|\chi_{\log n}(x)|=1$ for $x\in K$ and $n\in\N$. We continue to sketch a proof of statement (b). For any $\sigma_0>\sigma_u$ and any $\eps>0$, we find an integer $N_0\in\N$ such that the inequality
$$
\left| \L(s)-\sum_{n=1}^{N}\frac{a(n)}{n^s} \right| < \eps
$$ 
holds for every $\Re\ s\geq \sigma_0$ and arbitrary $N\geq N_0$. Let $p_1,...,p_m$ denote the prime numbers less than or equal to $N$ in ascending order. Further, let $\mathbb{T}^{m}$ denote the direct product of $m$ copies of the unit circle. Let $S_m:\mathbb{T}^{m}\times \C\rightarrow\C$ be defined by
\begin{equation}\label{defS}
S_m(\theta_1,...,\theta_m,s) := \sum_{n=1}^{N}\frac{a(n)}{n^s}\phi_{\log n}(\theta_1,...,\theta_m),
\end{equation}
where
$$
\phi_{\log n}(\theta_1,...,\theta_m) = \theta_1^{\nu(n;p_1)}\cdots \theta_m^{\nu(n;p_m)}
$$
with $\nu(n;p_j)$ being the uniquely determined quantities for which
$$
n= \prod_{j=1}^m p_j^{\nu(n;p_j)}.
$$
Certainly, $S_m$ is continuous on $\mathbb{T}^{m}\times \C$. According to the theorem of Kronecker (Theorem \ref{th:kronecker}), the set
$$
\{e^{-it\log p} \, : \, t\in\R\}
$$ 
is dense in $\mathbb{T}^m$. From this observation and the continuity of $S_m$, we deduce that, for any $\sigma\in\R$
$$
\sup \left\{\sum_{n=1}^{N}\frac{a(n)}{n^{\sigma+it}} \, : \, t\in\R \right\}
= \sup  \left\{\sum_{n=1}^{N}\frac{a(n)\phi_{\log n}(\pmb{\theta})}{n^{\sigma}} \, : \, \pmb{\theta}\in \mathbb{T}^m \right\}.
$$
The latter inequality and the topological correspondence of $\mathbb{T}^{\infty}$ and $K$ allow us to conclude that, for every $x\in K$, $N\geq N_0$ and $\sigma\geq \sigma_0$,
$$
\left| L_x(s)-\sum_{n=1}^{N}\frac{a(n)}{n^s} \right| \leq \eps.
$$
Statement (b) follows.
\end{proof}

The following lemma gathers some fundamental analytic properties of the functions $L$ and $L_x$, $x\in K$, attached to a given Dirichlet series $\L(s)$.

\begin{lemma}\label{lem:analyticpropertiesofLinU}
Let $\L(s)$ be a Dirichlet series and $U$ its half-plane of uniform convergence. Then, the following statements are true
\begin{itemize}
\item[(a)] For every $x\in K$, the function $L_x$ is an analytic function in $U$. 
\item[(b)] The function $L$ is continuous on $U\times K$.
\end{itemize}
\end{lemma}
\begin{proof}
Statement (a) follows immediately from \ref{lem:ext1} (a). Statement (b) can be deduced from the continuity of the function $S_m$ on $\mathbb{T}^m\times \C$, defined by  \eqref{defS}, and the topological correspondence of $\mathbb{T}^{\infty}$ and $K$.
\end{proof}

The analytic and probabilistic relevance of the functions $L_x$, $x\in K$, lies in the fact that they appear as vertical limit functions of $\L$ in its half-plane $U$ of uniform convergence. This observation dates back to Bohr \cite{bohr:1922}. In particular, we have
$$
\overline{\left\{L_{e_{\tau}} \, : \, \tau\in\R \right\}} = \{L_x \, : \, x\in K\} \subset \mathcal{H}(U)
$$
and, for arbitrary $l>0$,
$$
\overline{\left\{L_{ne_{l}} \, : \, n\in\N \right\}} = \{L_y \, : \, y\in K_{2\pi/l}\} \subset \mathcal{H}(U).
$$
Here, $\mathcal{H}(U)$ denotes the set of analytic functions on $U$ and the closures are taken with respect to the topology of uniform convergence on compact subsets of $U$, respectively. Observe further that the relations
$$
L_{e_{\tau}}(s)= \L(s+i\tau) \qquad \mbox{and} \qquad L_{ne_l}(s) = \L(s+inl)
$$
hold for $s\in U$.\par
In the half-plane $U$ the analytic behaviour of $\L$ and its allied functions $L_x$, $x\in K$, is quite well-understood. Things are getting more interesting if $\L$ can be continued analytically beyond $U$.

\section{The space \texorpdfstring{$\mathscr{H}^2$}{} of Dirichlet series with square summable coefficients}\label{sec:DirichletH2}
Hedenmalm, Lindqvist \& Seip \cite{hedenmalmlindqvistseip:1997} investigated a Hilbert space $\mathscr{H}^2$ of Dirichlet series which have square-summable coefficients. This space can be considered as an analogue for Dirichlet series of the Hardy space $H^2(\mathbb{T})$ for Fourier series.\par

Since we work with the Ramanujan hypothesis, we normalize the space $\mathscr{H}^2$ in a slightly different way than it was done by Hedenmalm, Lindqvist \& Seip \cite{hedenmalmlindqvistseip:1997}. We define that a Dirichlet series 
$$
\L(s)=\sum_{n=1}^{\infty} \frac{a(n)}{n^{s}}
$$
belongs to the space $\mathscr{H}^2$ if and only if
$$
 \sum_{n=1}^{\infty} \frac{|a(n)|^2}{n^{\sigma}} <\infty \qquad \mbox{ for every }\sigma>1.
$$ 
If $\L(s)\in\mathscr{H}^2$, then also the Dirichlet series expansions of the functions $L_x$, $x\in K$, attached to $\L(s)$ via \eqref{def:ext}, are elements of $\mathscr{H}^2$. The Ramanujan hypothesis is a sufficient condition for $\L(s)$ to lie in $\mathscr{H}^2$. If $\L(s)$ satisfies the Ramanujan hypothesis, we deduce from our considerations in Section \ref{sec:coeff} that several Dirichlet series related to $\L(s)$ lie also in $\mathscr{H}^2$, for example $\L(s)^k$ and $\L^{(\ell)}(s)$ with $k,\ell\in\N_0$. If $\L(s)$ satisfies the Ramanujan hypothesis and, additionally, $\L(s)$ can be written as a polynomial Euler product in $\sigma>1$, then we find also $\log \L(s)$ and $\L(s)^{\kappa}$ with $\kappa\in\R$ in $\mathscr{H}^2$.

Next, we shall consider analytic properties of the functions $L_x$, $x\in K$, attached to a given Dirichlet series in $\mathscr{H}^2$ by \eqref{Lx}. The subsequent lemma follows directly from of the H\"older inequality
$$
\sum_{n=1}^{\infty} \frac{|a(n)|}{n^{\sigma}} \leq \left( \sum_{n=1}^{\infty} \frac{|a(n)|^2}{n^{\sigma}} \right)^{1/2} \cdot \left(\sum_{n=1}^{\infty} \frac{1}{n^{\sigma}}\right)^{1/2}, \qquad \sigma>1.
$$
\begin{lemma}\label{lem:analyticH2}
Let $\L(s)\in\mathscr{H}^2$. Then, the following statements are true.
\begin{itemize}
\item[(a)] For $x\in K$, the Dirichlet series expansion of $L_x$ converges absolutely in $\sigma>1$.
\item[(b)] For $x\in K$, the function $L_x$ is analytic in $\sigma>1$.
\end{itemize}
\end{lemma}
The functions $L_{\sigma}$ with $\sigma>\frac{1}{2}$ which are attached to a given Dirichlet series $\L(s)\in\mathscr{H}^2$ by means of \eqref{Ls} lie in the space $L^2_{\pmb{\sigma}}(K)$ of the compact group $K$.\footnote{More precisely, the functions $L_{\sigma}$ with $\sigma>\frac{1}{2}$, are contained in the Hardy space $H^2_{\pmb{\sigma}}(K)\subset L^2_{\pmb{\sigma}}(K)$ of the compact group $K$; for a definition of Hardy spaces of compact groups with ordered duals, we refer to Tanaka \cite{tanaka:2008}.}
According to Plancherel's theorem we obtain that, for $\sigma>\frac{1}{2}$,
\begin{equation*}\label{plan}
\int_K \left|L_{\sigma}(x) \right|^2 \d\pmb{\sigma} = \sum_{n=1}^{\infty}\frac{|a(n)|^2}{n^{2\sigma}}.
\end{equation*}
By the unique ergodicity of the flow $\{T_t\}_{t\in\R}$ and the Birkhoff-Khinchine ergodic theorem, we get that, for every $\sigma>\frac{1}{2}$ and $\pmb{\sigma}$-almost every $x\in K$,
$$
\lim_{T\rightarrow\infty} \frac{1}{2T}\int^T_{-T} \left| L_{x}(\sigma+it)\right|^2 \d t = \int_K \left|L_{\sigma}(x) \right|^2 \d\pmb{\sigma}.
$$
These observations allow to retrieve information on the $\pmb{\sigma}$-almost sure behaviour of the functions $L_x$, $x\in K$, in the half-plane $\sigma>\frac{1}{2}$. The next theorem gathers fundamental results in this direction.

\begin{theorem}[Helson, Steuding]\label{th:almostsurebehaviour}
Let $\L(s)\in\mathscr{H}^2$. Then, there are subsets $E_1,E_2,E_3\subset K$ with $\pmb{\sigma}(E_1) = \pmb{\sigma}(E_2) = \pmb{\sigma}(E_3) = 1$ such that the following statements hold:
\begin{itemize}
 \item[(a)] For $x\in E_1$, the Dirichlet series expansion of $L_{x}$ converges in the half-plane $\sigma>\frac{1}{2}$.
 \item[(b)] For $x\in E_2$, the function $L_x$ can be continued analytically to the half-plane $\sigma>\frac{1}{2}$.
 \item[(c)] For $x\in E_3$, the mean-square of $L_x$ is given by
$$
\lim_{T\rightarrow\infty}\frac{1}{2T} \int_{-T}^T \left|L_x(\sigma+it) \right|^{2}\d t = \sum_{n=1}^{\infty} \frac{|a(n)|^2}{n^{2\sigma}} \qquad \mbox{for every }\sigma>\tfrac{1}{2}.
$$
\end{itemize}
Suppose, additionally, that $\L(s)\in\mathscr{H}^2$ satisfies the Ramanujan hypothesis. Then, there exists a subset $E_4 \subset K$ with $\pmb{\sigma}(E_4)=1$ such that the following statements hold:
\begin{itemize}
 \item[(d)] For $x\in E_4$,
$$
\lim_{T\rightarrow\infty}\frac{1}{2T} \int_{-T}^T \left|L_x(\sigma+it) \right|^{2k}\d t = \sum_{n=1}^{\infty} \frac{|a_k(n)|^2}{n^{2\sigma}} \qquad \mbox{for every }\sigma>\tfrac{1}{2} \mbox{ and }k\in\N ,
$$
where the $a_k(n)$ denote the coefficients of the Dirichlet series expansion of $\L^k$ in $\sigma>1$. 
\item[(e)] For $x\in E_4$ and $\sigma>\frac{1}{2}$, the asymptotic estimate  
$$
L_x(\sigma+it) \ll_{\sigma,\eps} |t|^{\eps}
$$
is true for any $\eps>0$, as $|t|\rightarrow\infty$.
\end{itemize}
Suppose that $\L(s)\in\mathscr{H}^2$ satisfies the Ramanujan hypothesis and can be written as a polynomial Euler product. Then, there exist subsets $E_5, E_6, E_7 \subset K$ with $\pmb{\sigma}(E_5) = \pmb{\sigma}(E_6)=1$ such that the following assertions hold.
\begin{itemize}
 \item[(f)] For $x\in E_5$, the function $L_x$ is analytic and non-vanishing in $\sigma>\frac{1}{2}$.
 \item[(g)] For $x\in E_6$ and $\kappa\in\R$, 
$$
\lim_{T\rightarrow\infty}\frac{1}{T} \int_{0}^T \left|L_x(\sigma+it) \right|^{2\kappa}\d t = \sum_{n=1}^{\infty} \frac{|a_{\kappa}(n)|^2}{n^{2\sigma}} \qquad \mbox{for every }\sigma>\tfrac{1}{2},
$$
where $a_{\kappa}(n)$ denote the coefficients of the Dirichlet series expansion of $\L^{\kappa}$ in $\sigma>1$. 
\end{itemize}
Suppose that $\L(s)\in\mathscr{H}^2$ satisfies the Ramanujan hypothesis, can be written as a polynomial Euler product and satisfies the prime mean-square condition (S.6). Then, there exist subsets $E_7, E_8 \subset K$ with $\pmb{\sigma}(E_7) = \pmb{\sigma}(E_8)=1$ such that the following holds.
\begin{itemize}
\item[(h)] For $x\in E_7$, the function $L_x$ is analytic in $\sigma>\frac{1}{2}$ and universal in the sense of Voronin inside the strip $\frac{1}{2}<\sigma<1$.
\item[(i)] For $x\in E_8$, the function $L_x$ has a convergent Dirichlet series expansion in $\sigma>\frac{1}{2}$ and the set $\{L_x \, : \, x\in E_8 \}$ lies dense in the set of all non-vanishing analytic functions in $\frac{1}{2}<\sigma<1$, with repsect to the topology of uniform convergence on compact subsets.
\end{itemize}
\end{theorem}
Statements (a)-(c) follow directly from Helson \cite{helson:1969}. We refer also to Hedenmalm, Lindqvist \& Seip \cite{hedenmalmlindqvistseip:1997} for a slightly different proof of statement (c). Statement (d) can be deduced from (c) by observing that $\L(s)^k\in\mathscr{H}^2$ for $k\in\N$, if $\L(s)$ satisfies the Ramanujan hypothesis (see Theorem \ref{th:sumprodL}), and that the union of countable many sets $E\subset K$ with $\pmb{\sigma}(E)=0$ is again a set of $\pmb{\sigma}$-measure zero. Statement (e) can be deduced from (d) by standard methods; see Titchmarsh \cite[\S 13.1]{titchmarsh:1986}. If $\L(s)$ satisfies the Ramanujan hypothesis and can be written as a polynomial Euler product, then both $\L(s)$ and $\L(s)^{-1}$ lie in $\mathscr{H}^2$; see Theorem \ref{th:Lkappa}. This together with (b) yields statement (f). Statement (g) follows from (c) and Theorem \ref{th:Lkappa}. Statements (h) and (i) were proved by Steuding \cite[Chapt. 5]{steuding:2007}.

We can interpret statement (e) and (f) as follows: under quite general assumptions on a Dirichlet series $\L(s)\in\mathscr{H}^2$, almost every of its attached functions $L_x$, $x\in K$, satisfy an analogue of the Lindel\"of hypothesis or an analogue of the Riemann hypothesis. However, for a particular function $L_x$, it seems very difficult to decide whether it lies in the exceptional zero-sets of Theorem \ref{th:almostsurebehaviour} or not. In fact, the zero-sets are not negligible at all: in the case of the Riemann zeta-function Tanaka \cite[\S 2]{tanaka:2008} showed that one can find $x\in K$ such that $\zeta_x$ has zeros and poles at prescribed points in $\frac{1}{2}<\sigma<1$. Moreover, for any $\frac{1}{2}<\sigma_0<1$, there exist $x\in K$ such that $\zeta_x$ is meromorphic in $\sigma>\sigma_0$ but does not extend meromorphically to a larger half-plane. \par

In the next section we shall transfer almost-sure properties of the family $\{L_x\}_{x\in K}$ to special functions in $\{L_x\}_{x\in K}$.\par
 
We conclude with a further observation. Let $\L(s)\in\mathscr{H}^2$. If $$l\notin\Gamma_{P}:= \{2\pi k(\log\tfrac{n}{m})^{-1} \, : \, k,n,m\in\N, n\neq m\},$$ then Plancherel's theorem \eqref{plancherel} together with \eqref{fouriert} and \eqref{fourierorth} yields that
\begin{equation}\label{plan1}
\int_{K_{2\pi/l}} \left|L_{\sigma}(y) \right|^2 \d\pmb{\tau} = \int_K \left|L_{\sigma}(y) \right|^2 \d\pmb{\sigma} =  \sum_{n=1}^{\infty}\frac{|a(n)|^2}{n^{2\sigma}}
\end{equation}
and 
\begin{equation}\label{plan2}
\int_{K_ {2\pi/l}} L_{\sigma}(y)  \d\pmb{\tau} = \int_K L_{\sigma}(y)  \d\pmb{\sigma} =  a(1),
\end{equation}
where the $a(n)$ are the coefficients of the Dirichlet series expansion of $\L$.
It is slightly more delicate to evaluate the integrals above if $l\in \Gamma_P$.  In the next lemma, we discuss the case if $l\in \Gamma_P$ has a very simple form. 
The works of Reich \cite{reich:1980-2} and Good \cite{good:1978} offer methods to handle the case of arbitrary $l\in \Gamma_P$.
\begin{lemma}\label{lem:lLambda}
Let $\L(s)$ be a Dirichlet series which satisfies the Ramanujan hypothesis and can be written as a polynomial Euler product. Let $l>0$ be of the form
$$
l= \frac{2\pi k}{\log p} \qquad \mbox{ with }p\in\mathbb{P} \mbox{ and } k\in\N .
$$
Then, for $\sigma>\frac{1}{2}$ and $\kappa\in\R$,
$$
\int_{K_{2\pi/l}}  \left|L_{\sigma}(y) \right|^{2\kappa} \d\ptau =
\left| \prod_{j=1}^m \left(1-\frac{\alpha_j(p)}{p^{\sigma}}\right)^{-2\kappa}   \right| \cdot
 \sum_{\begin{subarray}{c}n\in \N\\ p\nmid n \end{subarray}} \frac{|a_{\kappa}(n)|^2}{n^{2\sigma}}
$$
and 
$$
\int_{K_{2\pi/l}} L_{\sigma}(y)^{\kappa} \d\ptau = a_{\kappa}(1) \cdot
 \prod_{j=1}^m \left(1-\frac{\alpha_j(p)}{p^{\sigma}}\right)^{-\kappa},
$$
where the $a_{\kappa}(n)$ denote the Dirichlet series coefficients of $\L^{\kappa}$ and $\alpha_j(p)$ the local roots of the polynomial Euler product of $\L$.
\end{lemma}

\begin{proof}
By the definition of $K_{2\pi/l}$, we get that $\chi_{\log p}(y) = 1$ for $y\in K_{2\pi/l}$. Hence, due to the Euler product representation, we write
$$
L_{\sigma}(y) = \prod_{j=1}^m \left(1-\frac{\alpha_j(p)}{p^{\sigma}}\right)^{-1}\cdot
\sum_{\begin{subarray}{c}n\in \N\\ p\nmid n \end{subarray}} \frac{a(n)\chi_{\log n}(y)}{n^{\sigma}}, \qquad \sigma>1
$$
where the $a(n)$ denote the Dirichlet series coefficients of $\L$. Thus, in particular,
$$
L_{\sigma}^{\kappa}(y) = \prod_{j=1}^m \left(1-\frac{\alpha_j(p)}{p^{\sigma}}\right)^{-\kappa}\cdot
\sum_{\begin{subarray}{c}n\in \N\\ p\nmid n \end{subarray}} \frac{a_{\kappa}(n)\chi_{\log n}(y)}{n^{\sigma}}, \qquad \sigma>1. 
$$
By Theorem \ref{th:Lkappa} and Plancherel's theorem, we obtain that
$$
\int_{K_{2\pi/l}}  \left| \prod_{j=1}^m \left(1-\frac{\alpha_j(p)}{p^{\sigma}}\right)^{\kappa} \cdot L_{\sigma}(y)^{\kappa} \right|^2 \d\ptau = \sum_{\begin{subarray}{c}n\in \N\\ p\nmid n \end{subarray}} \frac{|a_{\kappa}(n)|^2}{n^{2\sigma}}, \qquad\sigma>\tfrac{1}{2},
$$
and
$$
\int_{K_{2\pi/l}}  \left( \prod_{j=1}^m \left(1-\frac{\alpha_j(p)}{p^{\sigma}}\right)^{\kappa} \cdot L_{\sigma}(y)^{\kappa} \right) \d\ptau = a_{\kappa}(1), \qquad\sigma>\tfrac{1}{2}.
$$
As the factor $\prod_{j=1}^m \left(1-\frac{\alpha_j(p)}{p^{\sigma}}\right)^{\kappa}$ does not depend on $y$ the statement of the lemma follows.
\end{proof}

\chapter{The class \texorpdfstring{$\No$}{} and vertical limit functions}\label{ch:classN}

In this section we define the class $\mathcal{N}(u)$ with $u\in[\frac{1}{2},1)$. Roughly speaking, the class $\No(u)$ gathers all function which have a Dirichlet series expansion in $\mathscr{H}^2$ and satisfy a certain normality feature in the half-plane $\sigma>u$. In Chapter \ref{ch:probmom} we shall see that, for every function $\L\in\No(u)$, its mean-square exists in a certain measure-theoretical sense on vertical lines in the half-plane $\sigma>u$.

\section{The class \texorpdfstring{$\No$}{} and its elements}\label{sec:classN}

In the sequel, we shall work with certain half-strips $Q_n(\alpha,l)$ and certain compact rectangular regions $\mathcal{R}_n(\alpha,l)$. 
For $n\in\Z$ and $\alpha,l \in\R$ with $\alpha,l>0$, we define $Q_n(\alpha,l)\subset \C$ to be the open horizontal half-strip
$$
Q_n(\alpha, l):= \left\{ \sigma + it \in \C \, : \, \sigma>\alpha, \, (n-1)l < t < (n+2) l   \right\}
$$
and $\mathcal{R}_n(\alpha, l)\subset \C$ to be the compact rectangular region
\begin{equation}\label{def:Rn}
\mathcal{R}_n(\alpha, l):= \left\{ \sigma + it \in \C \, : \, \alpha\leq \sigma \leq 2, \, nl \leq t \leq  (n+1)l    \right\}.
\end{equation}
Furthermore, we set
\begin{equation}\label{def:Q}
Q(\alpha, l ):= Q_0(\alpha,l) = \left\{ \sigma + it \in \C \, : \, \sigma>\alpha, \, -l < t < 2 l   \right\}.
\end{equation}
and
\begin{equation}\label{def:R}
\mathcal{R}(\alpha,l):=\mathcal{R}_0(\alpha, l):= \left\{ \sigma + it \in \C \, : \, \alpha\leq \sigma \leq 2, \, 0 \leq t \leq l    \right\}.
\end{equation}
For an illustration, we refer to Figure \ref{fig:RQ}

\begin{figure}\centering
\includegraphics[width=0.7\textwidth]{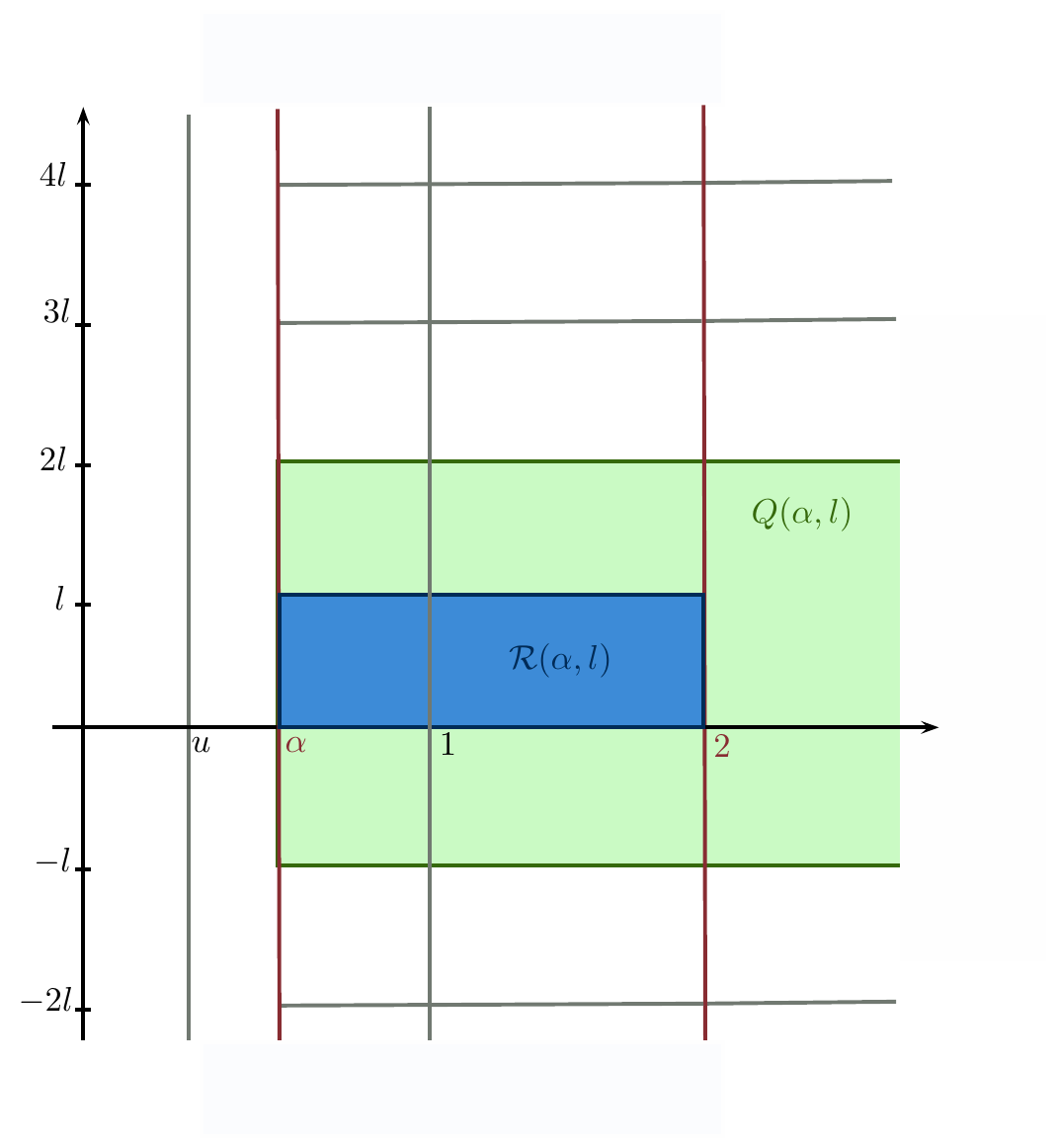}
\caption{The half-strip $Q(\alpha,l)$ and the compact rectangular set $\mathcal{R}(\alpha,l)$.}
\label{fig:RQ} 
\end{figure}

{\bf Definition of the class $\mathcal{N}(u)$.} Let $u\in[\frac{1}{2},1)$ and $\mathbb{H}_1$ denote the half-plane $\sigma>1$. A function $\L:\mathbb{H}_1\rightarrow \C$ belongs to the class $\mathcal{N}(u)$ if it satisfies the properties (N.1) and (N.2) stated below.
\begin{itemize}
 \item[(N.1)] {\it Dirichlet series expansion in $\mathscr{H}^2$.} In the half-plane $\sigma>1$, the function $\L$ has a Dirichlet series expansion that is an element of $\mathscr{H}^2$, i.e.
$$
\L(s) = \sum_{n=1}^{\infty} \frac{a(n)}{n^s}, \qquad \sigma>1.
$$
with coefficients $a(n)\in\C$ satisfying
$$
\sum_{n=1}^{\infty} \frac{|a(n)|^2}{n^{\sigma}}< \infty\qquad \mbox{for}\qquad \sigma>1.
$$

\item[(N.2)] {\it Analytic continuation and normality.} Let $L_x$, $x\in K$, denote the functions associated to $\L$ by means of \eqref{Lx}. For any real numbers $\alpha$, $l$ and $\eps$ with  $\alpha\in(u,1]$ and $\eps,l>0$, there exists a subset $J_{\eps}:=J(\alpha,l,\eps,\L)\subset\N$ with $\dens_* J_{\eps}>1-\eps$ such that the following holds.
\begin{itemize}
\item[(N.2a)] For every $n\in J_{\eps}$, the function $L_{ne_l}$ extends to an analytic function on the domain $ Q(\alpha,l)$.
\item[(N.2b)] The family $\{L_{ne_l}\}_{n\in J_{\eps}}$ is normal in $Q(\alpha,l)$.
\end{itemize}
\end{itemize}

Let $\L\in\No(u)$ with $u\in[\frac{1}{2},1)$. We start with some remarks on the analytic character of $\L$. By Lemma \ref{lem:analyticH2} (b), property (N.1) assures that $\L$ is analytic in the half-plane $\sigma>1$.\par
Property (N.2) implies that $\L$ can be continued analytically to a larger domain: for appropriately fixed $\alpha,l>0$, let $J_{\eps}\subset \N$ with $\eps>0$ be the sets defined in (N.2). We set $I:=\bigcup_{\eps>0} J_{\eps}$. Then, $\dens\ I = 1$ and, according to (N.2a), the functions $L_{ne_l}$ are analytic on $Q(\alpha,l)$ for every $n\in I$. From the relation 
\begin{equation*}
L_{ne_l}(s) = \L(s+inl),
\end{equation*}
which holds for $n\in\N$, $l>0$ and $s\in\C$, provided that $\L(s+inl)$ is well-defined, we deduce that $\L$ can be continued analytically to the domain 
$$
\bigcup_{n\in I} Q_n(\alpha,l)\cup \mathbb{H}_1.
$$
We proceed with some remarks on the normality feature. Roughly speaking, the normality feature of $\L$ assures that, for $\pmb{\sigma}$-almost every $x\in K$, the function $L_x$ appears as a vertical limit function of $\L$ in $\sigma>u$; see Section \ref{sec:verticallimit} for details. Due to this, certain properties which hold $\pmb{\sigma}$-almost surely for $L_x$, $x\in K$, in the half-plane $\sigma>\frac{1}{2}$ pass over to $\L$ in a certain measure-theoretical sense in the half-plane $\sigma>u$. \par
According to (N.2b), the families $\{L_{ne_l}\}_{n\in J_{\eps}}$ are normal in $Q(\alpha,l)$ for every $\eps>0$. Note, however, that the family $\{L_{ne_l}\}_{n\in I}$ with $I=\bigcup_{\eps>0} J_{\eps}$ is not necessarily normal in $Q(\alpha,l)$.\par
 
The Dirichlet series expansion of $\L$ in $\sigma>1$ assures that property (N.2b)  is equivalent to the local boundedness of $\{L_{ne_l}\}_{n\in J_{\eps}}$ in $Q(\alpha,l)$. This observation is fundamental for our further considerations and follows from the next lemma.
\begin{lemma}\label{lem:locboundednessNo}
Let $\L:\mathbb{H}_1\rightarrow \C$ be a function which satisfies (N.1). Let $\alpha<1$, $l>0$ and $Q:=Q(\alpha,l)$ be defined by \eqref{def:Q}. Suppose that $J\subset \N$ is such that, for $n\in J$, the functions $L_{ne_l}$ are analytic on $Q$. Then,  $\{L_{ne_l}\}_{n\in J}$ is normal in $Q$ if and only if  $\{L_{ne_l}\}_{n\in J}$ is locally bounded in $Q$.
\end{lemma}
\begin{proof}
If  $\{L_{ne_l}\}_{n\in J}$ is locally bounded in $Q$, then $\{L_{ne_l}\}_{n\in J}$ is normal in $Q$ due to Montel's theorem. In order to prove the converse, suppose that $\{L_{ne_l}\}_{n\in J}$ is normal in $Q$. In the half-plane $\sigma>1$, $\L$ can be written as an absolutely convergent Dirichlet series
$$
\L(s)=\sum_{n=1}^{\infty}\frac{a(n)}{n^s}, \qquad \sigma>1;
$$
see Lemma \ref{lem:analyticH2} (a). We set 
$$
M:= \sum_{n=1}^{\infty}\frac{|a(n)|}{n^{2}} <\infty.
$$
By observing that $2\in Q$ and that
$$
\left|L_{ne_l} (2)\right| = \left|\L(2+inl) \right|\leq M\qquad\mbox{ for every }n\in\N,
$$ 
the local boundedness follows immediately by means of Montel's theorem (see also the remark after Theorem \ref{th:montel1} in the appendix).
\end{proof}

\par

{\bf Sufficient conditions for the normality feature.} The following theorem provides sufficient conditions for a function $\L$ with property (N.1) to satisfy the normality feature (N.2).

\begin{theorem}\label{th:suffconditionsN}
Let $\L:\mathbb{H}_1\rightarrow \C$ be a function which satisfies property (N.1). Let $u\in[\frac{1}{2},1)$ and suppose that $\L$ can be continued meromorphically to the half-plane $\sigma>u$ with at most finitely many poles. Then, $\L\in \No(u)$ if at least one of the following conditions is true.
 \begin{itemize}
\item[(a)] \textbf{\textit{Boundedness.}} For every $\alpha>u$, there is a constant $M>0$ such that
$$
\left|\L(s) \right| \leq M \qquad \mbox{ for } \sigma\geq \alpha.
$$
\item[(b)] \textbf{\textit{Existence of the mean-square.}} The function $\L$ has finite growth in $\sigma>u$ and satisfies, for every $\sigma>u$,
$$
\limsup_{T\rightarrow\infty} \frac{1}{2T} \int_{-T}^T \left|\L(\sigma+it) \right|^2 \d t <\infty.
$$
\item[(c)] \textbf{\textit{$a$-point density estimate.}} There exist two distinct points $a,b\in\C$ such that, for $\sigma> u$, as $T\rightarrow\infty$,
$$
N_a(\sigma,T) = O_{\sigma}(T) \qquad \mbox{ and }\qquad N_b(\sigma, T) = o_{\sigma}(T).
$$
Here, $N_a(\sigma, T)$ denotes, as usual, the number of $a$-points $\rho_a=\beta_a + i\gamma_a$ of $\L$ with imaginary part $0<\gamma_a \leq T$ and real part $\beta_a > \sigma$.
\end{itemize}
\end{theorem}
It is an immediate consequence of Montel's theorem that condition (a) implies that $\L\in\No(u)$. In fact, condition (a) implies even more. We deduce from (a) that $\L$ is analytic in the half-plane $\sigma>u$ and that the Dirichlet series representing $\L$ converges uniformly in any half-plane $\sigma\geq \alpha >u$. In particular, we obtain that $\sigma_{unif,\L} \leq u$, where $\sigma_{unif,\L}$ denotes the abscissa of uniform convergence of the Dirichlet series connected to $\L$. Most of our subsequent results are trivial in this case. We included condition (a) for sake of completeness. In the following, however, we shall focus on the more interesting situation if $\L\in\No(u)$ and $u<\sigma_{unif,\L}$.\par
It follows essentially from an integrated version of Cauchy's integral formula (see Titchmarsh \cite[\S 11.8]{titchmarsh:1986} or Theorem \ref{th:sufflocalbound} in the appendix) that the mean-square condition (b) is sufficient for $\L$ to lie in $\No(u)$. In fact, the existence of the mean-square is a standard tool used in the theory of vertical limit functions of Dirichlet series and appears already in the works of Bohr; see for example Bohr \cite{bohr:1913-2} and Bohr \& Jessen \cite{bohrjessen:1932}. \par

By means of a generalized version of Montel's fundamental normality test (Theorem \ref{th:FNTextension}), we deduce that the $a$-point density estimate (c) implies that $\L\in\No(u)$. As far as the author knows, except for Lee \cite{lee:2012} who derived universality for Hecke $L$-functions in $\sigma>\frac{1}{2}$ by assuming the truth of Selberg's zero-density hypothesis, condition (c) or something similar did not appear in the context of vertical limit functions yet. In particular, condition (c) is one reason why we set up the class $\No(u)$ by means of the normality feature (N.2) and not, as common, by demanding finite growth for $\L$ and the existence of the mean-square value in $\sigma>u$.  \par

Conditions (a), (b) and (c) are not completely independent from one another. If $\L$ satisfies (a), then also (b) is true for $\L$. Moreover, if $\L$ satisfies (b), we deduce that, for every $a\in\C$ and $\sigma>u$, as $T\rightarrow\infty$,
$$
N_a(\sigma, T) = O_{a,\sigma}(T);
$$
see Section \ref{subsec:meansquare}. However, (b) does not necessarily imply (c) and the latter not necessarily (b). In Chapter \ref{ch:probmom}, we shall see that (c) implies (b) in a certain measure-theoretical sense. We shall now start to prove Theorem \ref{th:suffconditionsN}  \par

\begin{proof}
We fix arbitrary real numbers $l,\eps>0$ and $\alpha\in(u,1]$. Let $Q:=Q(\alpha, l)$ be defined by \eqref{def:Q}. 

(a): From condition (a), we deduce immediately that, for every $n\in\N$, the functions $L_{ne_l}$ are analytic on $Q$ and that the family $\mathcal{F}:=\{L_{ne_l}\}_{n\in\N}$ is bounded on $Q$. According to Montel's theorem, $\mathcal{F}$ is normal in $Q$. Altogether, as $\alpha\in(u,1]$ and $l>0$ were chosen arbitrarily, we conclude that $\L$ satisfies property (N.2).\par
(b): We follow Tanaka \cite[\S 4]{tanaka:2008} to prove the sufficiency of condition (b). According to Carlson's theorem (Theorem \ref{th:carlson}), we have for $\sigma>u$
$$
\lim_{T\rightarrow\infty} \frac{1}{2T} \int_{-T}^T \left|\L(\sigma+it) \right|^2 \d t = \sum_{n=1}^{\infty} \frac{|a(n)|^2}{n^{2\sigma}} =: f(\sigma).
$$
Due to the Ramanujan hypothesis, the Dirichlet series $f(\sigma)$ is absolutely convergent for $\sigma>u\geq \frac{1}{2}$. Moreover, $f(\sigma)$ defines a positive, monotonically decreasing, continuous function on the interval $(u,\infty)$. The bounded convergence theorem assures that
\begin{equation}\label{eq1}
\lim_{T\rightarrow\infty} \int_{\alpha}^{2}\left(  \frac{1}{2T} \int_{-T}^T  \left|\L(\sigma+it) \right|^2  \d t \right) \d \sigma
=  \int_{\alpha}^{2} f(\sigma) \d \sigma =: L < \infty.
\end{equation}
Let $\mathcal{R}:=\mathcal{R}(\alpha,l)$ be defined by \eqref{def:R}. We set
$$
B(n) := \iint_{\mathcal{R}} \left|L_{ne_l}(\sigma+it) \right|^2  \mbox{d} \sigma \mbox{d} t .
$$
By Fubini's theorem and the identity $L_{ne_l}(s)=\L(s+inl)$, we can write
$$
\frac{1}{2N l}\sum_{n=-N}^N B(n) = \frac{1}{2N l}  \int_{\alpha}^{2}\left(  \int_{-N l}^{N l}  \left|\L(\sigma+it) \right|^2  \d t \right) \d \sigma .
$$ 
Thus, we obtain by \eqref{eq1} that
\begin{equation}\label{eq2}
\lim_{N\rightarrow\infty}\frac{1}{2N l}\sum_{n=-N}^N B(n) = L
\end{equation}
We set $L' := 6L/\eps$ and define 
$$
I : = \left\{ n\in\N \, : \, B(n) \leq L' \right\}\qquad \mbox{and}\qquad
I^c : = \N\setminus I =  \left\{ n\in\N \, : \, B(n) > L' \right\}.
$$
We deduce from \eqref{eq2} that 
$\dens^* I^c< \frac{\eps}{3}$ and obtain consequently that $\dens_* I >1-\frac{\eps}{3}$.
Now, we set
$$
J : = \left\{ n\in\N \, : \, \max\{B(n-1),B(n),B(n+1)\} \leq L' \right\} \quad \mbox{ and } \quad
J^c : = \N\setminus J
$$
As $\dens^* J^c \leq 3\cdot \dens^* I^c$, we get that
$$
\dens^* J^c < \eps \qquad \mbox{and} \qquad \dens_* J > 1- \eps.
$$
We define $Q'$ to be the rectangular domain which consists of all points that lie in $Q$ but not in $\sigma\geq 2$, i.e.
$$
Q' := \left\{\sigma + it \in\C \, : \, \alpha < \sigma < 2 \, , - l < t < 2 l \right\}.
$$
Then, the construction of the sets $J$ and $Q'$ assures that, for every $n\in J$,
$$
\iint_{Q'} \left|L_{ne_l}(\sigma+it) \right|^2  \mbox{d} \sigma \mbox{d} t \leq 3 L'.
$$
This implies that the family $\{L_{ne_l}\}_{n\in J}$ is locally bounded on $Q'$ (see Theorem \ref{th:sufflocalbound}). Moreover, it follows from the Dirichlet series expansion of $\L$ in $\sigma>1$ that $\L$ is bounded in the half-plane $\sigma\geq \frac{3}{2}$. Altogether, we obtain that $\{L_{ne_l}\}_{n\in J}$ is locally bounded on $Q$. Montel's theorem implies that $\{L_{ne_l}\}_{n\in J}$ is normal in $Q$. As $\alpha\in(u,1]$ and $l,\eps>0$ can be chosen arbitrarily, we conclude that $\L$ satisfies property (N.2)..\par
(c): For $c\in\C$, let $D_c(n)$ denote the number of $c$-points of $L_{ne_l}$ in 
$$
\mathcal{R}' := \left\{\sigma + it \in\C \, : \,  \sigma \geq \alpha, \, 0 \leq t <  l \right\}.
$$
According to our assumption, there are two distinct $a,b\in\C$ and a constant $L>0$ such that
\begin{equation}\label{eq:densitiesab}
\limsup_{T\rightarrow\infty} \frac{1}{T}N_a(\alpha,T) \leq L \qquad \mbox{and }\qquad \limsup_{T\rightarrow\infty} \frac{1}{T}N_b(\alpha,T) = 0.
\end{equation}
We set $L':= 3L/\eps$ and define
$$
I_a:= \left\{n\in\N \, : \,  D_a(n) \leq L' \right\}
\qquad \mbox{and}\qquad
I_b :=  \left\{n\in\N \, : \,  D_b(n) =0 \right\}.
$$
It follows from \eqref{eq:densitiesab} that
\begin{equation*}\label{IaIb}
\dens_* I_a > 1 - \frac{\eps}{3} \qquad \mbox{ and } \qquad \dens\ I_b =1.
\end{equation*}
Now, let 
$$
J_a := \left\{n\in\N \, : \,  \max\{D_a(n-1),D_a(n),D_a(n+1)\} \leq L' \right\}
$$
and
$$
J_b := \left\{n\in\N \, : \,  D_b(n-1) = D_b(n) = D_b(n+1) = 0\} \leq L' \right\}.
$$
By a similar argumentation as in (b), we deduce from \eqref{IaIb} that
$$
\dens_* J_a > 1 - \eps \qquad \mbox{ and } \qquad \dens\ J_b =1.
$$
As $\L$ has only finitely many poles in the half-plane $\sigma>u$, we find a positive integer $N_{p}$ such that, for every $n\in J_{p}=\N \setminus\{1,...,N_{\infty}\}$, the function $L_{ne_l}$ is analytic in $Q$. \par
We set $J:= J_a \cap J_b \cap J_{p}$. Then, by the construction of $J$, we get that $\dens_* J >1-\eps$ and that the functions in the family $\mathcal{F}:=\{L_{ne_l}\}_{n\in J}$ are analytic on $Q$, omit the value $b$ on $Q$ and do not assume the value $a\in\C\setminus\{b\}$ at more than $L'$ points of $Q$. An extension of Montel's fundamental normality test (Theorem \ref{th:FNTextension} (b)) yields that the family $\{L_{ne_l}\}_{n\in J}$ is normal in $Q$. Since $\alpha\in(u,1]$ and $l,\eps>0$ can be chosen arbitrarily, we conclude that $\L$ satisfies property (N.2).
\end{proof}

{\bf Basic structure of the class $\No$(u).} For distinct $u_1,u_2\in[\frac{1}{2},1)$, the classes $\No(u_1)$ and $\No(u_2)$ are related as follows.
\begin{lemma}
Let $u_1,u_2 \in[\frac{1}{2},1)$ with $u_1 \leq u_2$. Then,
$$
\No(u_1) \subset \No(u_2).
$$ 
\end{lemma}
In the sequel, we set
$$
\No := \bigcup_{u\in[\frac{1}{2},1)} \No(u).
$$

{\bf Elements of the class $\mathcal{N}$.} The class $\mathcal{N}$ contains functions from the extended Selberg class. Suppose that $\L\in\Sc^{\#}$ has degree $d_{\L}=0$. Then, according to Kaczorowski \& Perelli \cite{kaczorowskiperelli:1999}, $\L$ is given by a Dirichlet polynomial and we conclude immediately by Theorem \ref{th:suffconditionsN} (a) that
$$
\L\in\No(\tfrac{1}{2}).
$$
Now suppose that $\L\in\Sc^{\#}$ has degree $d_{\L}>0$ and satisfies the Ramanujan hypothesis. Then, according to Theorem \ref{th:suffconditionsN} (b), 
$$
\L\in \No(u_m) \qquad \mbox{where}\qquad u_m:= \max\{\tfrac{1}{2},\sigma_m\}
$$
and $\sigma_m$ denotes, as usual, the abscissa of bounded mean-square of $\L$. Recall that, for any $\L\in\Sc^{\#}$ which satisfies the Ramanujan hypothesis and has degree $d_{\L}>0$,
$$
\sigma_m \leq \max\{ \tfrac{1}{2},\tfrac{1}{2} - \tfrac{1}{d_{\L}}\}<1.
$$
If $\L\in\Sc$ satisfies the Lindel\"of hypothesis or the Riemann hypothesis, we know that $\sigma_m \leq \tfrac{1}{2};$ see Section \ref{subsec:meansquare} for details. In particular, the truth of the Grand Lindel\"of hypothesis or the Grand Riemann hypothesis implies that $\Sc\subset \No(\frac{1}{2})$. However, we would like to stress that a given function $\L\in\Sc$ does not necessarily have to satisfy the Lindel\"of hypothesis or the Riemann hypothesis necessarily, in order to lie in $\No(\frac{1}{2})$. According to Theorem \ref{th:suffconditionsN} (b) and (c), it is sufficient that $\L\in\Sc$ fulfills the weaker condition $\sigma_m\leq \frac{1}{2}$ or the density estimates
$$
N_a(\sigma,T) = O_{\sigma}(T) \qquad \mbox{ and }\qquad N_0(\sigma, T) = o_{\sigma}(T),
$$
for every $\sigma>\frac{1}{2}$ and a particular $a\in\C\setminus\{0\}$, as $T\rightarrow\infty$.\par
Besides, the class $\mathcal{N}$ contains many functions that do not lie in $\mathcal{S}^{\#}$ at all. Dirichlet $L$-functions attached to non-primitive characters, for example, lie in $\No(\frac{1}{2})$, but they are not contained in $\mathcal{S}^{\#}$ as they lack an appropriate functional equation. Note further that a function $\L\in\No$ does not necessarily extend to a meromorphic function on the whole complex plane. In the next lemma we shall see that, if $\L$ lies in $\No(u)$ and its Dirichlet series expansion satisfies the Ramanujan hypothesis, then many functions related to $\L$ are also elements of $\No(u)$. 

\begin{lemma}\label{lem:classNrelfct}
Let $u\in[\frac{1}{2},1)$. Suppose that $\L,\L_1,...,\L_n\in\No(u)$ and that the Dirichlet series expansions of $\L,\L_1,...,\L_n$ satisfy the Ramanujan hypothesis, respectively. Then,
\begin{itemize}
 \item[(a)] $\L_1+...+\L_n\in\No(u)$.
 \item[(b)] $\L_1\cdots\L_n\in\No(u)$.
 \item[(c)] $\L^k \in \No(u)$ for any $k\in\N_0$.
 \item[(d)] $\L^{(\ell)}\in\No(u)$ for any $\ell\in\N_0$.
\end{itemize}
\end{lemma}

\begin{proof} Let $k,\ell\in\N_0$, $u\in[\frac{1}{2},1)$ and $\L,\L_1,...,\L_n\in\No(u)$. We observe the following:
\begin{itemize}
\item[(i)]  In the half-plane $\sigma>1$, the functions $\L,\L_1,...,\L_n$ are given by a Dirichlet series which satisfies the Ramanujan hypothesis. It follows from Theorems \ref{th:sumprodL} and \ref{th:Dirichletderivative} that, in the half-plane $\sigma>1$, each of the functions $\sum_{j=1}^n \L_j$, $\prod_{j=1}^n \L_j$, $\L^k$ and $\L^{(\ell)}$ is also given by a Dirichlet series which satisfies the Ramanujan hypothesis; thus, in particular by a Dirichlet series that is an element of $\mathscr{H}^2$.
\item[(ii)] If the functions $\L,\L_1,...,\L_n$ are analytic on a domain $Q\subset\C$, then the functions $\sum_{j=1}^n \L_j$, $\prod_{j=1}^n \L_j$, $\L^k$ and $\L^{(\ell)}$ are also analytic on $Q$.
\item[(iii)] Suppose that the families $\mathcal{F}, \mathcal{F}_1,..., \mathcal{F}_n$  of analytic functions on a domain $Q\subset\C$ are locally bounded on $Q$. Then, we deduce immediately that the families
$$\textstyle
\mathcal{F}_+ := \left\{\sum_{j=1}^n f_j \, : \, f_j\in \mathcal{F}_j\right\}, \qquad
\mathcal{F}_{\times} := \left\{\prod_{j=1}^n f_j \, : \, f_j\in \mathcal{F}_j\right\},
$$
$$
\mathcal{F}^{k} := \left\{f^k \, : \, f\in \mathcal{F}\right\}\qquad \mbox{and}\qquad
\mathcal{F}^{(\ell)} := \left\{f^{(\ell)} \, : \, f\in \mathcal{F}\right\}
$$
are also locally bounded on $Q$. Here, the statement for $\mathcal{F}^{(\ell)}$ follows from Cauchy's integral formula. 
\item[(iv)] Let $\eps>0$ and $J_1,...,J_n\subset \N$ with $\dens_* J_1, ..., \dens_*J_n> 1- \eps$. Then, 
$$
\dens_* \left(J_1\cap ... \cap J_n \right) > 1 - n\eps.
$$
\end{itemize}
The statement of the lemma follows from (i)-(iv), the definition of $\No(u)$ and Lemma \ref{lem:locboundednessNo}.
\end{proof}

In fact, Lemma \ref{lem:classNrelfct} was another motivation for us to define the class $\No(u)$ by the normality feature (N.2) and not by the mean-square condition of Theorem \ref{th:suffconditionsN} (b): we know that the Riemann zeta-function is an element of $\No(\tfrac{1}{2})$ and satisfies the Ramanujan hypothesis. By means of Lemma \ref{lem:classNrelfct} (c), we conclude immediately, that
$$
\zeta^k \in\No(\tfrac{1}{2}) \qquad \mbox{ for }k\in\N.
$$
However, as $\zeta^k$ is an element of the Selberg class of degree $k$, we only know for $\sigma>\max\{\frac{1}{2},1-\frac{1}{k}\}$, that
$$
\limsup_{T\rightarrow\infty} \frac{1}{2T} \int_{-T}^T \left|\zeta^k(\sigma+it) \right|^2 \d t <\infty.
$$
Thus, we would only get $\zeta^k \in \mathcal{N}(\max\{\frac{1}{2},1-\frac{1}{k}\})$ for $k\in\N$, if we replace  the normality feature (N.2) in the definition of $\mathcal{N}(u)$ by the mean-square condition of Theorem \ref{th:suffconditionsN} (b).\par
Let $\zeta_K$ be the Dedekind zeta-function of an abelian number field $K$. Then, $\zeta_K$ can be written as a finite product of Dirichlet $L$-functions; see for example Neukirch \cite[Chapt. VII, \S 5]{neukirch:2010}. As every Dirichlet $L$-functions is an element of $\No(\frac{1}{2})$ and satisfies the Ramanujan hypothesis, we deduce from Lemma \ref{lem:classNrelfct} (b) that
$$
\zeta_K \in \No(\tfrac{1}{2}).
$$ 
We mention a further implication of Lemma \ref{lem:classNrelfct}. Let $u\in[\frac{1}{2},1)$ and $\L\in\No(u)$. Suppose that $\L$ has Dirichlet series expansion
$$
\L(s)=\sum_{n=1}^{\infty} \frac{a(n)}{n^s}, \qquad \sigma>1.
$$
Then, it follows from Lemma \ref{lem:classNrelfct} (a) that, for $N\in\N$, the function $f_N$ defined by
$$
f_N(s):= \L(s) - \sum_{n=1}^{N} \frac{a(n)}{n^s}, \qquad \sigma>1,
$$
lies also in $\No(u)$.\par
We conclude with a possible extension of Lemma \ref{lem:classNrelfct} which we postpone to future works. Let $u\in[\frac{1}{2},1)$ and $\L\in\No(u)$. Further, let $\mathcal{H}(\mathbb{H}_1)$ denote the set of all analytic functions in the half-plane $\sigma>1$. Is there a nice way to describe axiomatically the set of all operators $T:\mathcal{H}(\mathbb{H}_1) \rightarrow \mathcal{H}(\mathbb{H}_1)$ for which $T(\L) \in \No(u)$.

\section{Normal families related to a function of the class \texorpdfstring{$\No$}{}}
In this section we discuss fundamental properties of the families $\{L_{ne_l}\}_{n\in J_{\eps}}$ related to $\L\in\No(u)$ by means of the normality feature (N.2). Several ideas that we use appear in Tanaka \cite{tanaka:2008} and in a slightly different language in the theory of probabilistic limit theorems, see for example Laurin\v{c}ikas \cite{laurincikas:1991-2}. It is our claim to rely  strictly on the normality feature (N.2) in our argumentation and not to use the mean-square condition of Theorem \ref{th:suffconditionsN} directly. We shall need the following lemmas to establish our main theorem in Chapter \ref{ch:probmom}. \par

\begin{lemma}\label{lem:normalityNo}
Let $u\in[\frac{1}{2},1)$ and $\L\in\No(u)$. Let $\alpha\in(u,1]$, $l>0$ and $Q:=Q(\alpha,l)$ be defined by \eqref{def:Q}. Further, let $J\subset \N$. 
\begin{itemize}
 \item[(a)] Suppose that $d_*:= \dens_* J>0$. Then, for every $\eps>0$, there is a subset $I_{\eps}\subset J$ with $\dens_* I_{\eps}>d_* -\eps$ such that the functions $L_{ne_l}$ are analytic on $Q$ for $n\in I_{\eps}$ and the family $\{L_{ne_l}\}_{n\in I_{\eps}}$ is normal in $Q$.
 \item[(b)] Suppose that $d^*:= \dens^* J>0$. Then, for every $\eps>0$, there is a subset $I_{\eps}\subset J$ with $\dens^* I_{\eps}>d^* -\eps$ such that the functions $L_{ne_l}$ are analytic on $Q$ for $n\in I_{\eps}$ and the family $\{L_{ne_l}\}_{n\in I_{\eps}}$ is normal in $Q$.
\end{itemize}
\end{lemma}

\begin{proof}
Suppose that $J\subset \N$. Let $\eps>0$. Then, due to the normality feature of $\L$, we find a subset $J_{\eps}\subset \N$ with $\dens_* J_{\eps} > 1- \eps$ such that the functions $L_{ne_l}$ are analytic on $Q$ for $n\in J_{\eps}$ and the family $\{L_{ne_l}\}_{n\in J_{\eps}}$ is normal in $Q$. We set $I_{\eps}:= J\cap J_{\eps}$. Suppose that $d_*>0$. Then, we find immediately that 
$$
\dens_* I_{\eps} \geq 1- \dens^* (\N\setminus I_{\eps}) >  d_* -\eps. 
$$
Suppose that $d^*<0$. In view of
$$
d^* \leq \dens^* I_{\eps} + \dens^* (J \cap (\N\setminus J_{\eps})) 
\leq \dens^* I_{\eps} + (1 - \dens_* J_{\eps})
$$
it follows that
$$
\dens^* I_{\eps} > d^*-  \eps. 
$$ 

\end{proof}
The next lemma reflects the fact that a family $\{L_{ne_l}\}_{n\in J}$ is normal on a set $Q(\alpha,l)$ if and only if $\{L_{ne_l}\}_{n\in J}$ is locally bounded on $Q(\alpha,l)$; see Lemma \ref{lem:locboundednessNo}.

\begin{lemma}\label{lem:JM}
Let $u\in[\frac{1}{2},1)$ and $\L\in\No(u)$. Let $\alpha\in(u,1]$, $l>0$ and $Q:=Q(\alpha,l)$ be defined by \eqref{def:Q}. Let $\mathcal{R}$ be a compact subset of $Q$. For $M>0$, let $J(M):=J(M,l,\alpha,\mathcal{R},\L)$ be the set of all $n\in\N$ that satisfy the following properties:
\begin{itemize}
\item[(i)] The function $L_{ne_l}$ is analytic on $Q$.
\item[(ii)] The inequality $\max_{s\in\mathcal{R}} \left|L_{ne_l}(s) \right| \leq M$ holds.
\end{itemize}
Then, for any $\eps>0$, there exists a constant $M_{\eps}>0$ such that
$$
\dens_* J(M_{\eps}) > 1 - \eps.
$$

\end{lemma}

\begin{proof}
The normality feature of $\L$ assures that, for any $\eps>0$, we find a subset $J_{\eps}\subset \N$ with $\dens_* J_{\eps} > 1- \eps$ such that, for $n\in J_{\eps}$, the functions $L_{ne_l}$ are analytic on $Q$ and, additionally, the family $\{L_{ne_l}\}_{n\in J_{\eps}}$ is normal in $Q$. According to Lemma \ref{lem:locboundednessNo}, the family $\{L_{ne_l}\}_{n\in J_{\eps}}$ is locally bounded on $Q$. As $\mathcal{R}$ is a compact subset of $Q$, there exist a constant $M_{\eps}>0$ such that
$$
\max_{s\in\mathcal{R}}|L_{nl}(s)|\leq M_{\eps} \qquad \mbox{ for } n\in J_{\eps}.
$$
Thus, the statement of the lemma follows by setting $J(M_{\eps}):=J_{\eps}$. 
\end{proof}
For functions which satisfy the mean-square condition of Theorem \ref{th:suffconditionsN} (b), the statement of Lemma \ref{lem:JM} was essentially already known to Bohr (see for example Bohr \cite{bohr:1915}) and can also be found in Tanaka \cite[Lemma 4.1]{tanaka:2008}. Our motivation was to work out that the statement of Lemma \ref{lem:JM} does not only hold for functions which satisfy the mean-square condition but in general for those which satisfy the normality feature (N.2); thus, for example, for functions which satisfy the $a$-density estimate of Theorem \ref{th:suffconditionsN} (c).\par

In the subsequent lemma, we gather important properties of a family $\{L_{ne_l}\}_{n\in J}$ if it is normal in a certain half-strip.\par
For a subset $J\subset \N$ and a fixed $l>0$, let $E(J)\subset K_{2\pi/l}$ denote from now on the closure of the set $\{ne_l \, : \, n\in J \}$, with repsect to the topology of $K_{2\pi/l}$.

\begin{lemma}\label{lem:normalNproperties}
Let $u\in[\frac{1}{2},1)$ and $\L\in\No(u)$. Let $\ell,l>0$, $\alpha\in(u,1]$  and $Q:=Q(\alpha,\ell)$ be defined by \eqref{def:Q}. Let the set $\mathcal{H}(Q)$ of all analytic functions on $Q$ be endowed with the topology of uniform convergence on compact subsets of $Q$. Suppose that $J$ is a subset of $\N$ such that the functions $L_{ne_l}$ are analytic on $Q$ for $n\in J$ and, additionally, the family $\{L_{ne_l}\}_{n\in J}$ is normal in $Q$. Then, the following statements are true.
\begin{itemize}
\item[(a)] For every $y\in E(J)$, the function $L_y$ is analytic on $Q$. Moreover,
$$
\overline{\{L_{ne_l}\,:\, n\in J  \}} = \{L_y\, : \, y\in E(J)\} \subset \mathcal{H}(Q),
$$
where we regard $\{L_{ne_l}\,:\, n\in J  \}$ and $\{L_y\, : \, y\in E(J)\}$ as subsets of $\mathcal{H}(Q)$ and take the closure of $\{L_{ne_l}\,:\, n\in J  \}$ with respect to the topology of $\mathcal{H}(Q)$ chosen above.
\item[(b)] The function $L:(s,y)\mapsto L_y(s)$ is continuous on $Q\times E(J)$.

\item[(c)] Let $\mathcal{R}$ be a compact subset of $Q$ and suppose that there are constants $m,M>0$ such that
$$
m \leq \max_{s\in\mathcal{R}} \left| L_{ne_l}(s) \right| \leq M \qquad \mbox{for every }n\in J.
$$
Then, 
$$
m \leq \max_{s\in\mathcal{R}} \left| L_{y}(s) \right| \leq M \qquad \mbox{for every }y\in E(J).
$$
\item[(d)] Let $\mathcal{R}$ be a compact subset of $Q$ and suppose that, for every $n\in J$, the function $L_{ne_l}$ has at least one zero in $\mathcal{R}$. Then, for every $y\in E(J)$, the function $L_{y}$ has at least one zero in $\mathcal{R}$.
\end{itemize}
\end{lemma}

\begin{proof}
(a): Let $y\in E(J)$. Then, according to Lemma \ref{lem:analyticH2} (b) and property (N.1) of $\L$, the function $L_y$ is analytic on the intersection $Q(1,l)$ of the domain $Q$ with the half-plane $\sigma>1$. Moreover, due to the definition of $E(J)$, we find natural numbers $n_k \in J$, indexed by $k\in\N$, such that the sequence $(n_ke_l)_k$ converges to $y$, with respect to the topology of $K_{2\pi/l}$. According to Lemma \ref{lem:analyticpropertiesofLinU} (b) and Lemma \ref{lem:analyticH2} (a), the function 
$$
L: (s,y') \mapsto L_{y'}(s)
$$ 
is continuous on $Q(1,l)\times K_{2\pi/l}$. Thus, for every $s\in Q(1,l)$, we have
$$
\lim_{k\rightarrow\infty} L_{n_{k}e_l}(s) = \lim_{k\rightarrow\infty} L(s,n_{k}e_l) = L(s,y)=L_{y}(s).
$$ 
Since the family $\mathcal{F}:=\{L_{ne_l}\}_{n\in J}$ is normal in $Q(\alpha,l)$, there exists a subsequence of $(L_{n_ke_l})_{k}$ which converges locally uniformly on $Q(\alpha,l)$ to a function $f\in\mathcal{H}(Q)$. It follows from the uniqueness of the limit function that $f$ is the analytic continuation of $L_{y}$ to $Q(\alpha,l)$. We have proved that $L_y$ is analytic on $Q$ and that 
$$\{L_y\, : \, y\in E(J)\} \subset \overline{\{L_{ne_l}\,:\, n\in J  \}}.$$ 
Now, suppose that the natural numbers $n_k\in J$ are chosen such that the sequence $(L_{n_ke_l})_k$ converges locally uniformly on $Q$ to a function $f\in\mathcal{H}(Q)$. We note that the set $E(J)\subset K_{2\pi/ l}$ is compact. Thus, we find a subsequence of $(n_ke_l)_{k}$ which converges to a certain element $y\in E(J)$. By the continuity of $L:(s,y')\mapsto L_{y'}(s)$ on $Q(1,l)\times K_{2\pi/l}$, we obtain that
$$
f(s)=L_y(s) \qquad \mbox{for }s\in Q(1,l).
$$
The uniqueness of the limit function implies that $f$ is the analytic continuation of $L_y$ to $Q(\alpha,l)$. This proves that $$\overline{\{L_{ne_l}\,:\, n\in J  \}} \subset \{L_y\, : \, y\in E(J)\}.$$ \par
(b): Let $(s_0,y_0)\in Q\times E(J)$ and $(s_k,y_k)_k$ be a sequence of points $(s_k,y_k)\in Q\times E(J)$ with 
$$
\lim_{k\rightarrow\infty} (s_k,y_k) = (s_0,y_0).
$$ 
We choose an arbitrary $\eps>0$. It follows from (a) that the function $L_{y_0}$ is analytic and, thus, continuous on $Q$. Hence, we find a disc $D_{\delta}(s_0)$ with $\delta>0$ such that $$\overline{D_{\delta}(s_0)}\subset Q$$ and
\begin{equation}\label{C1}
\left|L_{y_0}(s)-L_{y_0}(s_0) \right|<\frac{\eps}{2} \qquad \mbox{for }s\in D_{\delta}(s_0).
\end{equation}
Statement (a) together with Lemma \ref{lem:locboundednessNo} implies that the family $\{L_y\}_{y\in E(J)}$ is locally bounded on $Q$ and, in particular, normal in $Q$. Similarly as in the proof of statement (a), we deduce from the continuity of $L$ in $Q(1,l)\times E(J)$, that the sequence $(L_{y_k})_k$ converges locally uniformly on $Q$ to $L_{y_0}$. Hence, for sufficiently large $k$,
\begin{equation}\label{C2}
\left|L_{y_k}(s)-L_{y_0}(s) \right|<\frac{\eps}{2}\qquad \mbox{for }s\in D_{\delta}(s_0).
\end{equation}
By combining \eqref{C1} and \eqref{C2}, we obtain that, for sufficiently large $k$,
$$
\left|L(s_k,y_k)-L(s_0,y_0) \right| <\eps.
$$
The assertion follows.\par
(c): Statement (c) follows directly from (a).\par
(d): Statement (d) follows from (a) by means of the theorem of Hurwitz (Theorem \ref{th:hurwitz}).\par
\end{proof}

\section{The class \texorpdfstring{$\No$}{} and a polynomial Euler product representation}
In this section we suppose that $\L\in\No(u)$ satisfies the Ramanujan hypothesis and can be written as a polynomial Euler product of the form \eqref{eulerproduct}. We proceed to investigate the properties of the families $\{L_{ne_l}\}_{n\in J_{\eps}}$ attached to $\L\in\No(u)$ by means of the normality feature (N.2). Our main aim of this section is to establish Lemma \ref{lem:classNrelfct2} which states, roughly speaking, that the Ramanujan hypothesis together with a polynomial Euler product representation for $\L\in\No(u)$ implies that also $\log\L$ and $\L^{\kappa}$ with $\kappa\in\R$ are elements of $\No(u)$. \par

If $\L\in\No(u)$ satisfies the Ramanujan hypothesis and has a polynomial Euler product representation in $\sigma>1$, then $\L$ is free of zeros in $\sigma>1$. The next lemma states that possible zeros of $\L$ in the strip $u<\sigma\leq 1$ cannot lie too dense. For functions $\L\in\No(u)$ which satisfy the mean-square condition of Theorem \ref{th:suffconditionsN} (b), the following lemma is well-known and can be found in a slightly modified version for the peculiar case of the Riemann zeta-function, for example, in the paper of Tanaka \cite[Proposition 2.1]{tanaka:2008}.

\begin{lemma}\label{lem:Jzero}
Let $u\in[\frac{1}{2},1)$ and $\L\in\No(u)$. Suppose that $\L$ satisfies the Ramanujan hypothesis and can be written as a polynomial Euler product in $\sigma>1$. Let $\alpha\in[u,1)$, $l>0$ and $\mathcal{R}$ be a compact subset of the half-strip $Q:=Q(\alpha,l)$ defined by $\eqref{def:Q}$. Then, there exists a subset $J_{zf}:=J_{zf}(\alpha,l,\mathcal{R},\L)\subset \N$ with the following properties:
\begin{itemize}
 \item[(i)] $L_{ne_l}$ is analytic in $Q$ for $n\in J_{zf}$.
 \item[(ii)] $L_{ne_l}$ non-vanishing in $\mathcal{R}$ for $n\in J_{zf}$.
 \item[(iii)] The set $J_{zf}$ has density $\dens \ J_{zf} = 1$.
\end{itemize}
\end{lemma}
\begin{proof}
According to the normality feature of $\L$, we find a subset $J\subset \N$ with $\dens \ J =1$ such that $L_{ne_l}$ is analytic on $Q$ for $n\in J$. Suppose that there is a subset $J^c_{zf} \subset J$ with $d:=\dens^* J^c_{zf}>0$ such that, for every $n\in J_{zf}^c$, the function $L_{ne_l}$ has at least one zero in $\mathcal{R}\subset Q$. We choose $0<\eps<d$. Then, according to Lemma \ref{lem:normalityNo}, there is a subset $I_{\eps}\subset J_{zf}$ with $\dens^*>1-\eps$ such that the family $\{L_{ne_l}\}_{n\in I_{\eps}}$ is normal in $Q$. It follows from Lemma \ref{lem:normalNproperties} (d) that every function of the family $\{L_{y}\}_{y\in E(I_{\eps})}$ is analytic on $Q$ and has at least one zero in $\mathcal{R}$. Moreover, due to Lemma \ref{lem:tanaka1}, we have
\begin{equation} \label{Jzf}
\ptau(E(I_{\eps})) \geq d-\eps >0.
\end{equation}
Since $\mathcal{R}$ is a compact subset of $Q$, we find an open set $Q'$ with
$$
\mathcal{R} \subset Q' \qquad \mbox{and} \qquad \overline{Q'} \subset Q.
$$
Moreover, due to the relation
$$
L_{y+e_{\tau}}(s)=L_{y}(s+i\tau),
$$
there exists an interval $\mathcal{I}\subset [0,l)$ of positive Lebesgue measure  such that, for $(y,u)\in E(I_{\eps})\times\mathcal{I}$, the function $L_{y+e_u}$ is analytic on $Q'$ and has at least one zero in $Q'$. By our identification of $K_{2\pi/l}\times [0,l)$ with $K$, this and \eqref{Jzf} imply that there is a subset $G\subset K$ with $\pmb{\sigma}(G)>0$ such that, for $x\in G$, the function $L_{x}$ is analytic on $Q'$ and has zeros in $Q'$. This is in contradiction to statement (f) of Theorem \ref{th:almostsurebehaviour}. Thus, $\dens\ J^c_{zf} = 0$ and, consequently, $\dens\ J_{zf}=1$. 
\end{proof}

The next lemma deals with the reciprocals $L_{ne_l}^{-1}$ of the functions $L_{ne_l}$ attached to $\L$.

\begin{lemma}\label{lem:boundIM}
Let $u\in[\frac{1}{2},1)$ and $\L\in\No(u)$. Suppose that $\L$ satisfies the Ramanujan hypothesis and can be written as a polynomial Euler product in $\sigma>1$. Let $\alpha\in[u,1)$, $l>0$ and $Q:=Q(\alpha,l)$ be defined by \eqref{def:Q}. Then, for any $\eps>0$, there exists a subset $J_{\eps}\subset \N$ with $\dens_* J_{\eps} >1-\eps$ such that the functions $L_{ne_l}^{-1}$ are analytic in $Q$ for $n\in J_{\eps}$ and the family $\{L_{ne_l}\}_{n\in J_{\eps}}$ is locally bounded in $Q$. 
\end{lemma}
\begin{proof}
Due to the polynomial Euler product and the Dirichlet series representation of $\L$ in $\sigma>1$, the functions $L_{ne_l}$ are analytic and non-vanishing in $\sigma>1$ for every $n\in\N$; see Theorem \ref{th:Lkappa}. Moreover, for any $\sigma_0>1$, we find a constant $M\geq 1$ such that, for every $n\in\N$,
\begin{equation}\label{Lnelbound}
 |L_{ne_l}(\sigma+it)| \geq \frac{1}{M} \qquad \mbox{and} \qquad
 |L_{ne_l}(\sigma+it)| \leq M \qquad \mbox{for }\sigma\geq\sigma_0. 
\end{equation}
We choose $\alpha^*\in (u,1]$ with $\alpha^*<\alpha$ and set $Q^*:=Q(\alpha^*,l)$. Furthermore, let $Q'$ denote the half-strip
$$
Q' := Q(\alpha, \tfrac{3}{4}l) \subset Q^*
$$
and $\mathcal{R}'$ its closure, i.e.
$$
\mathcal{R}'=  \left\{\sigma+it\in\C \, : \, \sigma\geq \alpha,\ -\tfrac{3}{4} l \leq t \leq \tfrac{3}{2} l\right\} \subset Q^*.
$$
It follows from Lemma \ref{lem:normalityNo}, Lemma \ref{lem:Jzero} and \eqref{Lnelbound} that, for any $\eps>0$, there exist a subset $I_{\eps}\subset \N$ with $\dens_* I_{\eps}>1-\frac{\eps}{3}$ such that the following holds:
\begin{itemize}
 \item[(i)] For $n\in I_{\eps}$, the functions $L_{ne_l}$ are analytic in $Q^*$ and non-vanishing in $\mathcal{R}'\subset Q^*$.
 \item[(ii)] The family $\{L_{ne_l}\}_{n\in I_{\eps}}$ is normal in $Q^*$.
\end{itemize}
From $(i)$ we deduce immediately that the functions $L_{ne_l}^{-1}$ are analytic on the domain $Q'$ for $n\in I_{\eps}$.\par
Moreover, the observation \eqref{Lnelbound} implies that the family $\{L_{ne_l}\}_{n\in I_{\eps}}$ contains no sequence that converges locally uniformly on $Q'$ to $f\equiv 0$. With respect to this, it follows from the theorem of Hurwitz and the non-vanishing property (i) that, for every compact subset $\mathcal{K}\subset Q'$, we find a constant $M>0$ such that
$$
\min_{s\in\mathcal{K}}|L_{ne_l}(s)| \geq \frac{1}{M}.
$$
Hence, the family $\{L_{ne_l}^{-1}\}_{n\in I_{\eps}}$ is locally bounded in $Q'$. Let $J_{\eps}$ be the set of all $n\in I_{\eps}$ for which $\{n-1,n,n+1\}\subset I_{\eps}$. It is easy to show that $\dens_* J_{\eps} > 1-\eps$ and that $J_{\eps}$ fulfills the assertions of the lemma.
\end{proof}

From Lemma \ref{lem:boundIM} we derive that, if $\L\in\No(u)$ has a polynomial Euler product representation in $\sigma>1$, many functions related to $\L$ lie also in $\No(u)$, in addition to the ones provided by Lemma \ref{lem:classNrelfct}.

\begin{lemma}\label{lem:classNrelfct2}
Let $u\in[\frac{1}{2},1)$ and $\L\in\No(u)$. Suppose that $\L$ satisfies the Ramanujan hypothesis and can be written as a polynomial Euler product in $\sigma>1$. Then,
\begin{itemize}
 \item[(a)] $\L^{\kappa} \in \No(u)$ for any $\kappa\in\R$.
 \item[(b)] $\log \L \in \No(u)$.
 \item[(b)] $\L'/\L\in\No(u)$.
\end{itemize}
\end{lemma}
\begin{proof}
Let $\kappa\in\R$, $u\in[\frac{1}{2},1)$ and $\L\in\No(u)$. We observe the following:
\begin{itemize}
\item[(i)]  In the half-plane $\sigma>1$, the function $\L$ is given by a Dirichlet series which satisfies the Ramanujan hypothesis and can be written as a polynomial Euler product. It follows from Theorem \ref{th:Lkappa}, \ref{th:Dirichletlog} and \ref{th:dirichletlogderivative} that, in the half-plane $\sigma>1$, the functions $\L^{\kappa}$, $\log \L$ and $\L'/\L$ are also given by Dirichlet series which satisfy the Ramanujan hypothesis and, thus, lie in $\mathscr{H}^2$.
\item[(ii)] Let $\alpha\in[u,1)$, $l,\eps>0$ and $Q:=Q(\alpha,l)$. For $n\in\N$, we define $\log L_{ne_l}$ by
$$
\log L_{ne_l}(s):= \log \L(s+inl), \qquad \sigma>1.
$$ 
This choice assures that, uniformly for $n\in\N$, 
\begin{equation}\label{log}
\lim_{\sigma\rightarrow\infty} \log L_{ne_l}(\sigma)=0.
\end{equation}
Lemma \ref{lem:boundIM} implies that, for any $\eps>0$, there is a set $J_{\eps}\subset \N$ with $\dens_* J_{\eps} >1-\eps$ such that, for every $n\in J$, the functions $\log L_{ne_l}$, $L^{\kappa}_{ne_l}$ and $L'_{ne_l}/ L_{ne_l}$ are well-defined and analytic on $Q$ and such that the families
$$
\mathcal{F}_{\log} := \left\{\log L_{ne_l}\right\}_{n\in J_{\eps}}, \qquad
\mathcal{F}_{\kappa} := \left\{ L^{\kappa}_{ne_l}\right\}_{n\in J_{\eps}}
$$
and
$$
\mathcal{F}_{L'/L} := \left\{ \frac{L'_{ne_l}}{L_{ne_l}}\right\}_{n\in J_{\eps}}
$$
are locally bounded on $Q$. Note that the local boundedness of $\mathcal{F}_{\log}$ can be deduced from the local boundedness of $\{L_{ne_l}\}_{n\in J_{\eps}}$ on $Q$ by means of \eqref{log} and the Borel-Carath\'{e}dory theorem; see Titchmarsch \cite[\S 5.5]{titchmarsh:1939}.
\end{itemize}
The statement of the lemma follows from (i), (ii) and the definition of $\No(u)$.
\end{proof}
We know that the Riemann zeta-function is an element of $\No(\frac{1}{2})$. By means of Lemma \ref{lem:classNrelfct2} we obtain that also its logarithm $\log \zeta$, its logarithmic derivative $\zeta'/\zeta$ and its $\kappa$-th power $\zeta^{\kappa}$ with any $\kappa\in\R$ lie in $\No(\frac{1}{2})$.

\section{Vertical limit functions }\label{sec:verticallimit} 
Let $\L\in\No(u)$ with $u\in[\frac{1}{2},)$. By our considerations of the preceeding section, we find that, for $\pmb{\sigma}$-almost every $x\in K$, the function $L_x$ occurs as a vertical limit functions of $\L$ in $\sigma>u$. 
\begin{corollary}\label{cor:limitfctu}
Let $u\in[\frac{1}{2},1)$ and $\L\in\No(u)$. Let $\alpha\in(u, 1]$, $l>0$ and $Q:=Q(\alpha,l)$ be defined by \eqref{def:Q}. Then, there is a subset $G\subset K$ with $\pmb{\sigma}(G)=1$ and a subset $A\subset \R^+$ such that
\begin{equation}\tag{$a$}
\{L_x \, : \, x\in G\} \subset \overline{\left\{L_{e_{\tau}} \, : \, \tau\in A\right\}}  \subset \mathcal{H}(Q)
\end{equation}
Moreover, there is a subset $E\subset K$ with $\pmb{\tau}(E)=1$ and a subset $J\subset \N$ such that
\begin{equation}\tag{$b$}
\{L_y \, : \, y\in K_{2\pi/l}\} \subset \overline{\left\{L_{ne_{l}} \, : \, n\in J \right\}} \subset  \mathcal{H}(Q).
\end{equation}
Here, $\mathcal{H}(Q)$ denotes the set of analytic functions on $Q$ and the closures above are taken with respect to the topology of uniform convergence on compact subsets of $Q$.
\end{corollary}
If $\L$ satisfies the Ramanujan hypothesis and the mean-square condition of Theorem \ref{th:suffconditionsN} (b) in $\sigma>u$, the statement of the lemma is well-known both for $\L$ and several functions related to $\L$; for example for $\L^k$ with $k\in\N$,  $\L^{(\ell)}$ with $\ell\in\N$ and, if $\L$ has a polynomial Euler product, also for $\log \L$ and $\L^{-1}$. By the definition of the normality feature, it is natural that the statement persists for all functions in $\No(u)$. Here, as far as the author knows, it may be considered as new that the statement of Corollary \ref{cor:limitfctu} holds for functions $\L$ which satisfy the $a$-point density estimate of Theorem \ref{th:suffconditionsN} (c).
\begin{proof}[Proof of Corollary \ref{cor:limitfctu}]
Statement (a) follows directly from Lemma \ref{lem:tanaka1}, Lemma \ref{lem:normalityNo} and Lemma \ref{lem:normalNproperties} (a). Statement (b) follows from (a) by observing that we can identify every $x\in K$ with an element $(y,u)\in K_{2\pi/l}\times [0,l)$.
\end{proof}

\chapter[Discrete and continuous moments to the right of the critical line]{Discrete and continuous moments  }\label{ch:probmom}

\section{An extension of a theorem due to Tanaka to the class \texorpdfstring{$\No(u)$}{ }}\label{sec:probmom}

In this chapter we extend a result of Tanaka \cite{tanaka:2008}, which he obtained for the Riemann zeta-function, to functions in the class $\No$. Building on the preliminary works of the preceeding sections, we strongly rely on his methods and ideas to prove our result.\par

We introduce the following notation. For a given $l>0$, a set $A\subset [1,\infty)$ is said to be an {\it $l$-set of density zero} if there exists a subset $J\subset \N$ with $\dens \ J = 0$ such that
$$
A= \bigcup_{n\in J} [nl,(n+1)l).
$$
It is easy to see that an $l$-set $A$ of density zero satisfies
\begin{equation*}\label{eq:setA}
\lim_{T\rightarrow\infty} \frac{1}{T} \int_1^T \pmb{1}_A(t) \d t = 0\qquad \mbox{and}\qquad\lim_{T\rightarrow\infty} \frac{1}{T} \int_1^T \pmb{1}_{A^c}(t) \d t = 1,
\end{equation*}
where $\pmb{1}_X$ denotes the indicator function of a set $X\subset \R$ and $X^c:=\R\setminus X$ its complement.\par 

Let $u\in[\frac{1}{2},1)$ and $\L\in\No(u)$. Further, let $p:\C\rightarrow\C$ be a continuous function with $p(z)\ll |z|^2$, as $|z|\rightarrow\infty$. We shall establish asymptotic formulas for moments of the form
$$
\frac{1}{T} \int_1^T p\left( \L(\sigma + it) \right) \pmb{1}_{A^c}(t) \d t, \qquad \sigma>u, \qquad \mbox{as }T\rightarrow\infty,
$$
where we omit a certain $l$-set $A\subset [1,\infty)$ of density zero from the path of integration. Moreover, for $l>0$, we shall derive asymptotic formulas for discrete moments of the form
$$
\frac{1}{N} \sum_{n=1}^N p\left(\L(\sigma+i\lambda + inl) \right)\pmb{1}_{A^c}(nl) , 
\qquad \sigma>u, \qquad 0\leq\lambda\leq l,\qquad \mbox{as }N\rightarrow\infty, 
$$\par
where we neglect, by the definition of $A$, a certain set $J\subset\N$ of density zero in the summation. \par

The next theorem is the main theorem of Part II of this thesis.
\begin{theorem}\label{th:probmom}
Let $u\in[\frac{1}{2},1)$. Let $(\L_j,p_j)_j$ be a sequence of pairs which consist of a function $\L_j\in\No(u)$ and a continuous function $p_j:\C\rightarrow\C$ satisfying
\begin{equation}\label{condC}\tag{C}
p_j(z) \ll_j |z|^2 ,\qquad \mbox{ as } |z|\rightarrow\infty.
\end{equation}
Let $L_{j}$ denote the extension of $\L_j$ to $\C\times K$ defined by \eqref{Ls}. Then, for any $\alpha\in(u, 1]$ and $l>0$, there exist an $l$-set $A\subset [1,\infty)$ of density zero and a sequence $(N_j)_j$ of positive integers such that the following holds: 
\begin{itemize}
\item[(i)] For every $j\in\N$, as $T\rightarrow\infty$,
$$
\frac{1}{T}  \int_{N_j}^{T}\int_2^{\alpha} p_j\bigl(\L_j(\sigma+it)\bigr)\cdot \pmb{1}_{A^c}(t) \d\sigma \d t = \int_{\alpha}^{2}\int_{K}  p_j\bigl(L_{j}(\sigma,x) \bigr)  \d \pmb{\sigma} \d\sigma + o_{j}(1).
$$
\item[(ii)] For every $j\in\N$, uniformly for $\alpha\leq \sigma\leq 2$, as $T\rightarrow\infty$,
\begin{equation*}\label{tanakastar}
\frac{1}{T}  \int_{N_j}^{T}  p_j\bigl(\L_j(\sigma+it)\bigr) \cdot \pmb{1}_{A^c}(t)\d t = \int_{K}  p_j\bigl(L_{j}(\sigma,x) \bigr)  \d \pmb{\sigma} + o_{j}(1).
\end{equation*}
\item[(iii)] Suppose that $l\notin\Gamma_{P}:= \{2\pi k(\log\frac{n}{m})^{-1} \, : \, k,n,m\in\N, n\neq m\}$. Then, for every $j\in\N$, uniformly for $\alpha\leq \sigma\leq 2$ and $0\leq \lambda \leq l$, as $N\rightarrow\infty$,
$$
 \frac{1}{N }  \sum_{n=N_j}^N \, p_j\bigl( \L_j(\sigma+i\lambda+inl) \bigr)\cdot \pmb{1}_{A^c}(nl)  = \int_{K_{2\pi/l}}  p_j\bigl(L_{j}(\sigma,x) \bigr)  \d \pmb{\tau} + o_{j}(1).
$$
 \end{itemize}
\end{theorem}

\par
Tanaka established statement (ii) of Theorem \ref{th:probmom} for the Riemann zeta-function $\zeta$, its $\kappa$-th power $\zeta^{\kappa}$, $\kappa\in\R$, with the special choices of $p$ given by $p(z)=z$ and $p(z)=|z|^2$. We extend Tanaka's result to the quite general class $\No$. Moreover, we provide with (i) and (ii) an integrated and a discrete version of (ii). Some remarks to Theorem \ref{th:probmom} are in order: 
\begin{itemize}
\item[1.] Suppose that $\L\in\No(u)$ has Dirichlet expansion
$$
\L(s)=\sum_{n=1}^{\infty} \frac{a(n)}{n^s} , \qquad \sigma>1
$$
and that $l\notin \Gamma_P=\{2\pi k(\log\frac{n}{m})^{-1} \, : \, k,n,m\in\N\}$. Then, it follows from our considerations in Section \ref{sec:DirichletH2}, in particular from the identities \eqref{plan1} and \eqref{plan2}, that
$$
\int_{K_ {2\pi/l}} \left| L(\sigma,y) \right|^2 \d\pmb{\tau} =  \int_{K}  \left| L(\sigma,x) \right|^2  \d \pmb{\sigma} = \sum_{n=1}^{\infty} \frac{|a(n)|^2}{n^{2\sigma}}, \qquad \sigma>u,
$$
and
$$
\int_{K_ {2\pi/l}} L(\sigma,y)  \d\pmb{\tau} = \int_{K}   L(\sigma,x)   \d \pmb{\sigma} = a(1), \qquad \sigma>u.
$$
\item[2.] The proof of Theorem \ref{th:probmom} shall show that statement (iii) remains valid for $l\in \Gamma_{P}$, if the function defined by
$$
f(\sigma)=\int_{K_{2\pi/l}}  \bigl|L_{j}(\sigma,x) \bigr|^2  \d \pmb{\tau} 
$$
is continuous for $\sigma>u$. The results of Reich \cite{reich:1980-2} and Lemma \ref{lem:lLambda} assert that this is the case, if $\L\in\No(u)$ has a polynomial Euler product representation. 
\item[3.] Let $\L_j$, $u$, $p_j$, $N_j$ and $A$ be as in Theorem \ref{th:probmom}. Let $\overline{A^c}$ denote the closure of $A^c = \R\setminus A$. If $\L_j$ is analytic in the region
$$
\left\{ s\in\C \, : \, u<\sigma\leq 2, \, t\in \overline{A^c}\cap [0,N_j] \right\},
$$
then the limits in (i), (ii) and (iii) are not affected by replacing $N_j$ by $1$.

\item[4.] The statements of Theorem \ref{th:probmom} can be formulated in an analogous manner for the lower half-plane: there exists an $l$-set $A\subset[-1,-\infty)$ of density zero and a sequence of negative integers $(N_j)_j$ such that the statement (i)-(iii) hold as $T\rightarrow -\infty$.
\item[5.] By a diagonal argument (see Tanaka \cite[\S 5]{tanaka:2008}), similarly to the one that we shall use in the last step of our proof, we find a common $l$-set $A$ of density zero in Theorem \ref{th:probmom} such that the limits of statements (i)-(iii) hold for every $\sigma>u$. In this case, however, we loose the uniformity in $\sigma$, resp. the uniformity in $\sigma$ and $\lambda$.

\item[6.] In the half-plane where the Dirichlet series expansion of $\L\in\No$ converges uniformly, the statements (i)-(iii) hold trivially for $\L$ with $A=\emptyset$. This follows essentially from its almost periodic behaviour. Thus, the statements of Theorem \ref{th:probmom} are especially of interest if $\alpha$ is less than the abscissa of uniform convergence $\sigma_u$ of $\L$ or if the exact value of $\sigma_u$ is not known. We recall here the difficulties to determine $\sigma_u$ for functions in the extended Selberg class; see Section \ref{sec:charconvabs}. 

\item[7.] In the mean-square half-plane $\sigma>\sigma_m$ of $\L\in\No$, i.e. in the half-plane where the classical continuous mean-square of $\L$ is bounded, statement (ii) of Theorem \ref{th:probmom} holds trivially with $A=\emptyset$ and $p(z)=|z|^2$, due to Carlson's theorem (Theorem \ref{th:carlson}); for any other admissible choice of $p$ the limit superior of the left-hand side of (ii) is at least bounded. By the dominate convergence theorem, the same applies to (i). Additionally, in the half-plane $\sigma>\sigma_m$, there are methods available that allow to establish asymptotic expansions for discrete mean-values in (iii) with $A=\emptyset$; see, for example, Montgomery \cite[Chapt. 1]{montgomery:1971}, Reich \cite{reich:1980-2} and Good \cite{good:1978}. Thus, statements (i)-(iii) are especially of interest if $\alpha<\sigma_m$ or if the exact value of $\sigma_m$ is not known.

\item[8.] Tanaka's method is strongly related to a method of Reich \cite{reich:1980-2}. For functions $\L$ with polynomial Euler product representation of order two, Reich showed that the discrete and continuous mean-square value of $\L$ coincide in its mean-square half-plane, provided that $l\notin\Gamma_P$, i.e., for $l\notin\Gamma_P$, $\sigma>\sigma_m$,
$$
\lim_{N\rightarrow\infty} \frac{1}{N} \sum_{n=1}^N\left|\L(\sigma+inl)\right|^2
= \lim_{T\rightarrow\infty}\frac{1}{T}\int_0^{T} \left| \L(\sigma+it)\right|^2 \d t =
\sum_{n=1}^{\infty} \frac{|a(n)|^2}{n^{2\sigma}}.
$$
Reich derived also asymptotic expansions for the case $l\in\Gamma_P$. Reich relied on a uniform distribution result
and the existence of the classical continuous square-mean of $\L$. By loss of a set of density zero, Tanaka used the {\it uniqueness} of the ergodic system $(K_{2\pi/l}, T_t)$ and, instead of working with the continuous mean-square value directly, relied on a property of $\L$ which we revealed as the normality feature (N.2) in Section \ref{sec:classN}.

\item[9.] By definition, an $l$-set $A\subset [1,\infty)$ of density zero satisfies
\begin{equation*}\label{erg}
\frac{1}{T}\int_1^{\infty} \pmb{1}_{A^c}(t) \d t = o(1), \qquad \mbox{as }T\rightarrow\infty.
\end{equation*}
It would be interesting if the statements of Theorem \ref{th:probmom} are also true for $l$-sets $A$ whose density can be bounded asymptotically in a better way than above. Here, however, some additional reasoning seems to be necessary. 
\end{itemize}

We state some immediate corollaries of Theorem \ref{th:probmom}. 
\begin{corollary}\label{cor1}
Let $u\in[\frac{1}{2},1)$ and $\L\in \No(u)$. Suppose that the Dirichlet series expansion of $\L$ satisfies the Ramanujan hypothesis. Let $\alpha\in(u,1]$ and $l>0$. Then, there exist an $l$-set $A\subset[1,\infty)$ of density zero such that, for every $k\in\N$ and uniformly for $\sigma\in[ \alpha,2]$,
$$
\lim_{T\rightarrow\infty} \frac{1}{T} \int_{1}^T \left|\L(\sigma + it) \right|^{2k} \pmb{1}_{A^c} (t) \d t = \sum_{n=1}^{\infty} \frac{|a_k(n)|^2}{n^{2\sigma}}
$$
and 
$$
\lim_{T\rightarrow\infty} \frac{1}{T} \int_{1}^T \L^k(\sigma + it) \ \pmb{1}_{A^c} (t) \d t = a_k(1),
$$
where the $a_k(n)$ denote the coefficients of the Dirichlet series expansion of $\L^k$.
If $\L$ can be written additionally as a polynomial Euler product in $\sigma>1$, then we find an $l$-set $A\subset[1,\infty)$ of density zero such that, for every $k\in\N$, uniformly for $\sigma\in[\alpha,2]$,
\begin{equation}\label{ww}
\lim_{T\rightarrow\infty} \frac{1}{T} \int_{1}^T \left| \L(\sigma+it)\right|^{-2k} \pmb{1}_{A^c}(t) \d t = \sum_{n=1}^{\infty} \frac{|a_{-k}(n)|^2}{n^{2\sigma}},
\end{equation}
\begin{equation}\label{dw}
\lim_{T\rightarrow\infty} \frac{1}{T}\int_1^T \left| \log \L(\sigma+it)\right|^{2} \pmb{1}_{A^c}(t) \d t = \sum_{n=1}^{\infty} \frac{|a_{\log\L}(n)|^2}{n^{2\sigma}}
\end{equation}
and
\begin{equation}\label{www}
\lim_{T\rightarrow\infty} \frac{1}{T}\int_1^T \left| \frac{\L'(\sigma+it)}{\L(\sigma+it)}\right|^2 \pmb{1}_{A^c}(t) \d t = \sum_{n=1}^{\infty} \frac{|\Lambda_{\L}(n)|^2}{n^{2\sigma}},
\end{equation}
where the $a_{-k}(n)$, $a_{\log\L}(n)$ and $\Lambda_{\L}(n)$ denote the coefficients of the Dirichlet series expansion of $\L^{-k}$, $\log\L$ and $\L'/\L$, respectively.
\end{corollary}
\begin{proof}
According to Lemma \ref{lem:classNrelfct}, we have that $\L^k\in \No(u)$ for any $k\in \N$. It follows from Lemma \ref{lem:classNrelfct2} that $\L^{-k}$ with $k\in\N$, $\log \L$ and $\L'/\L$ are elements of $\No(u)$, if $\L$ can be written additionally as a polynomial Euler product in $\sigma>1$. The statement follows directly from Theorem \ref{th:probmom} by respecting Remark 1 and 3 stated after Theorem \ref{th:probmom}. 
\end{proof}
Let $\L\in\Sc$ have positive degree $d_{\L}$. The function $\L$ satisfies the Lindel\"of hypothesis if and only if, for every $\sigma>\frac{1}{2}$ and $k\in \N$,
\begin{equation}\label{int}
\lim_{T\rightarrow\infty} \frac{1}{2T}\int_{-T}^{T} \left|\L(\sigma+it) \right|^{2k} \d t= \sum_{n=1}^{\infty}\frac{|a_k(n)|^2}{n^{2\sigma}},
\end{equation}
where the $a_k(n)$ denote the Dirichlet series coefficients of $\L^k$. This follows essentially from classical methods due to Hardy \& Littlewood \cite{hardylittlewood:1923}, who settled the case of the Riemann zeta-function; see also Steuding \cite[Chapt. 6]{steuding:2007}. For given $k\in\N$, we know so far only for
\begin{equation}\label{q}
 \sigma> \max\{\tfrac{1}{2},1-\tfrac{1}{kd_{\L}}\}
\end{equation}
that \eqref{int} is true; see Section \ref{subsec:meansquare} for details. Let $\L_1,...,\L_n$ be primitive functions in the Selberg class of degree $d_{\L_1},...,d_{\L_n}$ such that
$$
\L=\L_1 \cdot \dots \cdot \L_n.
$$ 
We set $d_{\L}^*=\max\{d_{\L_1},...,d_{\L_n}\}$. We deduce from Lemma \ref{lem:classNrelfct} and Theorem \ref{th:suffconditionsN} (b) that 
$$\L\in \No(\max\{\tfrac{1}{2}, 1- \tfrac{1}{d_{\L}^*}\}).$$ 
By Theorem \ref{th:probmom} and Remark 5 after Theorem \ref{th:probmom}, we know that, for any $l>0$, there is an $l$-set $A\subset [1,\infty)$ of density zero such that, for every  $k\in\N$ and $\sigma>\max\{\tfrac{1}{2}, 1- \tfrac{1}{d_{\L}^*}\}$,
$$
\lim_{T\rightarrow\infty} \frac{1}{T} \int_{1}^{T} \left|\L(\sigma+it) \right|^{2k}\pmb{1}_{A^c}(t)\d t = \sum_{n=1}^{\infty}\frac{|a_k(n)|^2}{n^{2\sigma}}.
$$
Thus, in a certain measure-theoretical sense, \eqref{int} is true in the half-plane
$$\sigma>\max\{\tfrac{1}{2}, 1- \tfrac{1}{d_{\L}^*}\}.$$
Let $\L\in\Sc$. Due to possible zeros of $\L$ in $\sigma>\frac{1}{2}$, it is difficult to obtain unconditional asymptotic expansions for the moments in \eqref{ww}, \eqref{dw} and \eqref{www} with $A=\emptyset$. We refer to Selberg \cite{selberg:1992} for certain conditional results.  \par

Next, we state a discrete version of Corollary \ref{cor1}.

\begin{corollary}
Let $u\in[\frac{1}{2},1)$ and $\L\in \No(u)$. Suppose that the Dirichlet series expansion of $\L$ satisfies the Ramanujan hypothesis. Let $\alpha\in(u,1]$ and $l>0$. 
\begin{itemize}
\item[(a)] If $l\notin\Gamma_{P}:= \{2\pi k(\log\frac{n}{m})^{-1} \, : \, k,n,m\in\N, n\neq m\}$, then there exist an $l$-set $A\subset[1,\infty)$ of density zero such that, for every $k\in\N$, uniformly for $\sigma\in[ \alpha,2]$ and $\lambda\in[0,l]$,
\begin{equation}\label{dis}
\lim_{N\rightarrow\infty}\frac{1}{N} \sum_{n=1}^N  \bigl|\L(\sigma + i\lambda + inl) \bigr|^{2k} \pmb{1}_{A^c} (nl)  = \sum_{n=1}^{\infty} \frac{|a_k(n)|^2}{n^{2\sigma}}
\end{equation}
and 
$$
\lim_{N\rightarrow\infty}\frac{1}{N} \sum_{n=1}^N \L(\sigma+ i\lambda + inl)^k \, \pmb{1}_{A^c} (nl)  = a_k(1).
$$
\item[(b)] Suppose additionally that $\L$ can be written as a polynomial Euler product. If $l=2\pi k / \log p$ with some $k\in\N$ and $p\in\mathbb{P}$, then there exist an $l$-set $A\subset[1,\infty)$ of density zero such that, for every $k\in\N$, uniformly for $\sigma\in[ \alpha,2]$ and $\lambda\in[0,l]$,
$$
\lim_{N\rightarrow\infty}\frac{1}{N} \sum_{n=1}^N  \bigl|\L(\sigma + i\lambda + inl) \bigr|^{2k} \cdot \pmb{1}_{A^c} (nl)  = \left| \prod_{j=1}^m \left(1-\frac{\alpha_j(p)}{p^{\sigma}}\right)^{-2k}   \right| \cdot
 \sum_{\begin{subarray}{c}n\in \N\\ p\nmid n \end{subarray}} \frac{|a_{k}(n)|^2}{n^{2\sigma}}
$$
and 
$$
\lim_{N\rightarrow\infty}\frac{1}{N} \sum_{n=1}^N \L(\sigma+ i\lambda + inl)^k \, \pmb{1}_{A^c} (nl)  = a_{k}(1) \cdot
 \prod_{j=1}^m \left(1-\frac{\alpha_j(p)}{p^{\sigma}}\right)^{-k}.
$$
\end{itemize}
Here, the $a_k(n)$ denote the coefficient of the Dirichlet series expansion of $\L^k$ and $\alpha_j(p)$ the local roots of the polynomial Euler product representation of $\L$.
\end{corollary}

\begin{proof}
The statement follows directly from Lemma \ref{lem:classNrelfct} and Theorem \ref{th:probmom} (ii) by respecting the Remarks 1,2,3 and Lemma \ref{lem:lLambda}.
\end{proof}

For certain functions $\L$ with polynomial Euler product of order two, Reich \cite{reich:1980-2} proved that \eqref{dis} holds for $k=2$ and $A=\emptyset$ in the mean-square half-plane of $\L$. 
Good \cite{good:1978} used a different method to establish \eqref{dis} for certain functions in their mean-square half-plane with $A=\emptyset$. Besides a polynomial Euler product of order two, he assumed additionally the existence of an approximate functional equation and got better bounds in the asymptotic expansion \eqref{dis} than the ones provided by Reich \cite{reich:1980-2}. Both Reich \cite{reich:1980-2} and Good \cite{good:1978} studied also the case $l\in \Gamma_P$.\par

The next corollary shows that the $k$-th power $\L^k$ of a function $\L\in\No(u)$ which satisfies the Ramanujan hypothesis can be approximated in mean-square by certain Dirichlet polynomials in $\sigma>u$. 
\begin{corollary}\label{cor:Dirichletpol}
Let $u\in[\frac{1}{2},1)$ and $\L\in \No(u)$. Suppose that $\L$ satisfies the Lindel\"of hypothesis. Let $\alpha\in(u,1]$ and $l>0$. Then, there exist an $l$-set $A\subset[1,\infty)$ of density zero such that, for every $k\in\N$,
$$
\lim_{N\rightarrow\infty}\lim_{T\rightarrow\infty} \frac{1}{T} \int_1^T \int_{\alpha}^2 \left|\L^k(\sigma + it) - \sum_{n=1}^{N} \frac{a_k(n)}{n^{\sigma+it}}\right|^{2} \pmb{1}_A(t) \d t \d \sigma = 0.
$$
Here, the $a_k(n)$ denote the coefficient in the Dirichlet series expansion of $\L^k$.
\end{corollary}
\begin{proof}
Let $\L\in\No(u)$. Then, according to Lemma \ref{lem:classNrelfct}, the function defined by
$$
\L^k_N(s):= \L(s)- \sum_{n=1}^{N} \frac{a_k(n)}{n^{s}} =\sum_{n=N+1}^{\infty}\frac{a_k(n)}{n^{s}}, \qquad \sigma>1,
$$
with $k, N\in \N$, lies also in $\No(u)$. By observing that
$$
\lim_{N\rightarrow\infty} \sum_{n=N+1}^{\infty}\frac{|a_k(n)|^2}{n^{2\sigma}} = 0 ,\qquad \sigma>\frac{1}{2},
$$
the assertion can be derived from Theorem \ref{th:probmom} (a).
\end{proof}

In the mean-square half-plane of $\L^k$ the statement of Corollary \ref{cor:Dirichletpol} can be established for $A=\emptyset$ by standard methods relying on the residue theorem; see for example Steuding \cite[Chapt. 4.4]{steuding:2007}. We refer here also to Lee \cite{lee:2012} who proved that the logarithm of Hecke $\L$-functions can be approximated by certain Dirichlet polynomials in $\sigma>\frac{1}{2}$, under the assumption of a certain zero-density conjecture.\par

\section{Proof of the main theorem}
\subsection*{Auxiliary lemmas}
We start with some lemmas.
\begin{lemma}\label{lem:arithmeticmean}
Let $I\subset \N$ and $(a_n)_n$ be a sequence of complex numbers such that
$$
\lim_{N\rightarrow\infty} \frac{1}{N}\sum_{\begin{subarray}{c} n\in I \\ n\leq N \end{subarray}} a_n = a 
$$
with some $a\in\C$. 
\begin{itemize}
 \item[(a)] Suppose that, for $n\in\N$, the quantities $a_n$ are non-negative real number. Then, the limit $a$ is real and, for any $\delta>0$, there exist an integer $N_{\delta}\in\N$ such that, for every $N\in \N$ and every set $J \subset \N$ with $\{1,...,N_{\delta}\}\subset J$, the inequality
$$
\frac{1}{N}\sum_{\begin{subarray}{c} n\in I\setminus J  \\ n\leq N \end{subarray}} | a_n | < a +\delta
$$
is true.
 \item[(b)]  Suppose that there is a constant $C>0$ such that $|a_n|\leq C$ for $n\in\N$. Then, for any subset $J\subset I$ with $\dens \ J = 0$, 
$$
\lim_{N\rightarrow \infty} \frac{1}{N}\sum_{\begin{subarray}{c} n\in I\setminus J \\ n\leq N \end{subarray}} a_n = a.
$$
\end{itemize}
\end{lemma}
\begin{proof}
The assertions of the lemma follow by standard convergence arguments, respecting the conditions posed on $a_n$ and $J$, respectively.
\end{proof}

The next lemma is crucial for the proof of theorem and extends a lemma of Tanaka \cite[Lemma 5.2]{tanaka:2008}.

\begin{lemma}\label{lem:probmom}
Let $u\in[\frac{1}{2},1)$ and $\L\in\No(u)$. Let $L$ be the function connected to $\L$ by means of \eqref{L}. Furthermore, let $\alpha'',\alpha',\alpha\in(u,1]$ with $\alpha''<\alpha'<\alpha$ and $l>0$. For $M>0$, let $J(M):=J(M,l,\alpha'',\mathcal{R}',\L)\subset \N$ be defined as in Lemma \ref{lem:JM}, where we choose $\mathcal{R}'$ to be the compact rectangular set 
$$
\mathcal{R}':=\left\{\sigma+it\in\C \, : \, \alpha' \leq \sigma \leq 2, \, -\tfrac{3}{4}l \leq t \leq \tfrac{3}{2}l  \right\}.
$$ 
Then, either statement (A) or statement (B) is true:
\begin{itemize}
 \item[(A)] There is a real number $M_1\geq 1$ such that $\ptau(E(J(M_{1}))) = 1$. In this case, we set $\varTheta=\{1\}$.
 \item[(B)] There are real numbers $M_k\geq 1$ with $k\in \N$ such that
\begin{itemize}
\item[ ]$\qquad \pmb{\tau}(E(J(M_1)))>0$,\vspace{0.1cm}
\item[ ]$\qquad \pmb{\tau}(E(J(M_k))) < \pmb{\tau} (E(J(M_{k+1}))) \mbox{ for }k\in\N 
$\vspace{0.1cm}
\item[ ]$\qquad \displaystyle
\mbox{and }\lim_{k\rightarrow\infty} \pmb{\tau} (E(J(M_k))) =1.$
\end{itemize}
In this case, we set $\varTheta=\N$.
\end{itemize}
In both cases, the following holds:
\begin{itemize}
\item[(i)] Let $M_0:=0$ and $I_k:= J(M_k)\setminus J(M_{k-1})$ for $k\in\varTheta$. Then, for any $j,k\in\varTheta$ with $j\neq k$, 
$$
\pmb{\tau}(E(I_j)\cap E(I_k))) = 0.
$$
\item[(ii)] Let $p:\C\rightarrow\R^+_0$ be a non-negative, continuous function, $k\in\varTheta$ and
$$
G_k := \left\{ y+e_t \, : \, (y,t)\in E(I_k)\times [0,l) \right\} \subset K.
$$
Then, for any $\delta>0$, there exists a finite subset $\Delta_k\subset I_k$ such that, for every $\sigma\in [\alpha,2]$, every $\lambda\in[0,l]$, every $N\in\N$ and every $J \subset \N$ with $\Delta_k \subset J$,
\begin{equation}\tag{$\spadesuit$}\label{spade1}
\frac{1}{N l} \sum_{\begin{subarray}{c} n\in I_k\setminus J \\ n\le N  \end{subarray}} 
\int_{n l}^{(n+1) l} \int_{\alpha}^2  p\bigl(\L(\sigma+it)\bigr)  d t \ d \sigma \leq \int_{\alpha}^2 \int_{G_k} p\bigl(L(\sigma,x)\bigr)  d \pmb{\sigma} \d \sigma + \delta,
\end{equation}
\begin{equation}\tag{$\diamondsuit$}\label{diamond1}
\frac{1}{N l}  \sum_{\begin{subarray}{c} n\in I_k\setminus J \\ n\le N  \end{subarray}} 
\int_{n l}^{(n+1) l} p\bigl(\L(\sigma+it)\bigr)  \d t \leq \int_{G_k} p\bigl(L(\sigma,x)\bigr)  \d \pmb{\sigma} 
+ \delta
\end{equation}
and 
\begin{equation}\tag{$\divideontimes$}\label{divide1}
\frac{1}{N}\sum_{\begin{subarray}{c} n\in I_k\setminus J \\ n\le N  \end{subarray}}
 p\bigl(\L(\sigma+i\lambda+inl)\bigr) 
\leq \int_{E(I_k)} p\bigl(L(\sigma,y)\bigr)  \d \pmb{\tau} + \delta.
\end{equation}

\item[(iii)] Let $p:\C\rightarrow\C$ be a continuous function, $k\in\varTheta$ and $G_k$ be defined as above. Further, let $J\subset\N$ with $\dens\ J = 0$. Then, uniformly for $\sigma\in[\alpha,2]$ and $\lambda\in[0,l]$, as $N\rightarrow\infty$,
\begin{equation}\tag{$\clubsuit$}\label{spade2}
\frac{1}{N l} \sum_{\begin{subarray}{c} n\in I_k\setminus J \\ n\le N  \end{subarray}} 
\int_{0}^{ l} \int_{\alpha}^2  p\bigl(\L(\sigma+it)\bigr) \d t \d \sigma = \int_{\alpha}^2 \int_{G_k} p\bigl(L(\sigma,x)\bigr) \d \pmb{\sigma} \d \sigma + o(1),
\end{equation}
\begin{equation}\tag{$\triangledown$}\label{diamond2}
\frac{1}{N l}  \sum_{\begin{subarray}{c} n\in I_k\setminus J \\ n\le N  \end{subarray}} \int_{0}^{ l}  p\bigl(\L(\sigma+it)\bigr) \d t = \int_{G_k} p\bigl(L(\sigma,x)\bigr) \d \pmb{\sigma} + o(1)
\end{equation}
and
\begin{equation}\tag{$\times$}\label{divide2}
\frac{1}{N}\sum_{\begin{subarray}{c} n\in I_k\setminus J \\ n\le N  \end{subarray}}  p\bigl(\L(\sigma+i\lambda+inl)\bigr) = \int_{E(I_k)} p\bigl(L(\sigma,y)\bigr)  \d \pmb{\tau} + o(1).
\end{equation}
\end{itemize}
\end{lemma}

\begin{proof}[Proof of Lemma \ref{lem:probmom}] We divide the proof into several steps. First we shall figure out that either statement (A) or statement (B) is true.\par
{\bf The behaviour of $\pmb{\tau}(E(J(M)))$ as $M\rightarrow \infty$.} We consider the function $F:\R^+\rightarrow [0,1]$ defined by
$$
F(M):= \pmb{\tau}(E(J(M))) \qquad\mbox{ for } M>0.
$$
It follows from the definition of $J(M)$ that
$$
J(M) \subset J(M') \qquad \mbox{ for }0<M<M'
$$ 
and, consequently, that
$$
E(J(M)) \subset E(J(M'))\qquad \mbox{ for }0<M<M'.
$$
This implies that the function $F$ is monotonically increasing. Lemma \ref{lem:JM} together with Lemma \ref{lem:tanaka1} yields that 
\begin{equation}\label{pmf1}
\lim_{M\rightarrow\infty} F(M)=1.
\end{equation}
If we find a real number $M_1>0$ such that 
$$
F(M)=1\qquad \mbox{ for } M\geq M_1,
$$ 
then statement (A) is true. Otherwise, if $F(M)<1$ for every $M>0$, we find, according to \eqref{pmf1}, a sequence $(M_k)_k$ of real numbers $M_k\geq 1$ such that 
$$
F(M_1)>0, \qquad F(M_k)<F(M_{k+1}) \; \mbox{ for }k\in\N \qquad \mbox{ and } \qquad \lim_{k\rightarrow\infty} M_k =\infty.
$$ 
In this case, statement (B) is true. In the following, we focus on situation (B). In fact, if statement (A) is true, then the assertions (i), (ii) and (iii) follow easily from the subsequent consideration by just regarding the case $k=1$.\par

{\bf Properties of the sets $I_k\subset \N$.} Since $F$ is a monotonically increasing function, $F$ is discontinuous in at most countably many points. This observation allows us to adjust the sequence $(M_k)_k$ such that $F$ is continuous at every point $M_k$ with $k\in\N$. We set $M_0:=0$ and define $I_k:= J(M_k)\setminus J(M_{k-1})$ for $k\in\N$. Observe that the sets $I_1,...,I_k$ provide a disjoint decomposition of $J(M_k)$. Moreover, it follows immediately from the definitions of $I_k$ and $J(M_k)$, that for every $n\in I_k$,
\begin{equation}\label{1}
M_{k-1} \leq \max_{s\in\mathcal{R}'}\left|L_{ne_l}(s) \right| \leq M_k.
\end{equation}
We deduce from Montel's theorem and the local boundedness of the functions $L_{ne_l}$, $n\in\N$, in $\sigma>1$ that the family $\{L_{ne_l}\}_{n\in I_k}$ is normal in the half-strip $Q':=Q(\alpha',\frac{3}{4}l)$ defined by \eqref{def:Q}. The compact set $\mathcal{R}:=\mathcal{R}(\alpha,l)$ defined by \eqref{def:R} is a subset both of $Q'$ and $\mathcal{R}'$.
According to Lemma \ref{lem:normalNproperties} (c) and \eqref{1}, this implies that, for every $y\in E(I_k)$,
\begin{equation}\label{eq:boundLys}
M_{k-1} \leq \max_{s\in\mathcal{R}}\left|L_y(s) \right| \leq M_k.
\end{equation}
Consequently, $E(I_j)\cap E(I_k)=\emptyset$ for $j=1,...,k-2$. It remains to consider the case $j=k-1$. For any $\delta>0$, we derive that
$$
E(I_k)\cap E(I_{k-1}) \subset E(J(M_{k-1} + \delta))\setminus E(J(M_{k-1} - \delta)).
$$
By the additivity of the measure $\ptau$, this implies that
$$
\ptau\left( E(I_k)\cap E(I_{k-1}) \right) \leq F(M_k+\delta) - F(M_k - \delta).
$$
The continuity of $F$ at $M_k$ assures that 
$$
\lim_{\delta\rightarrow 0+} F(M_k + \delta) - F(M_k - \delta) =0.
$$
Consequently, we obtain that $\ptau\left( E(I_k)\cap E(I_{k-1}) \right)  =0$. Altogether, we proved that
$$ 
\ptau(E(I_k)\cap E(I_{j}) ) = 0 \qquad \mbox{ for }j,k\in \N \mbox{ with }j\neq k. 
$$ 
Statement (i) follows.\par

{\bf Applicability of Tanaka's ergodic theorem.} Now, we shall figure out that the sets $E(I_k)$ are constructed in a suitable way such that Tanaka's modified version of the ergodic theorem (Lemma \ref{lem:Tanakaergodic}) can be applied. We fix an arbitrary $k\in\N$. For any given $\eps>0$, we find according to Lemma \ref{lem:JM} an integer $\nu_0\in\N$ such that 
$$
\dens_* (J(M_{\nu_0})) > 1- \eps.
$$
We set
$$
S:= \bigcup_{\begin{subarray}{c} \nu=1,...,\nu_0,\\ \nu\neq k \end{subarray}} I_{\nu}.
$$
From the observation that
$$
S\cup I_k =  \bigcup_{\nu=1}^{\nu_0} I_{\nu} = J(M_{\nu_0}),
$$
we deduce that 
$$
\dens^* \left(\N\setminus(S_{\eps}\cup I_k) \right) < \eps.
$$
Moreover, by statement (i), we have 
$$
\pmb{\tau}(E(S_{\eps}) \cap E(I_k))=0.
$$ 
It follows from Lemma \ref{lem:Tanakaergodic} that, for any continuous function $p^{\triangledown}:E(I_k)\rightarrow\C$,
\begin{equation}\label{eq:tanakapmb}
\lim_{N\rightarrow\infty} \frac{1}{N}  \sum_{\begin{subarray}{c} n\in I_k \\ n\le N  \end{subarray}} p^{\triangledown}(ne_l) = \int_{E(I_k)} p^{\triangledown}(y) \d \pmb{\tau}
\end{equation}
This observation is quite central in our further considerations. In fact, by choosing $p^{\triangledown}$ in a proper way, the statement above implies already a weak version of (ii) and (iii). However, some further work is necessary to establish (ii) and (iii) in full extent. We proceed with a continuity consideration. \par

From now on, let $p:\C\rightarrow\C$ be a continuous function. We keep $k\in\N$ fixed and choose an arbitrary $\delta>0$. Moreover, we set
$$
V_k := \max_{\begin{subarray}{c} z\in \C \\ |z|\leq M_k \end{subarray}} \left|p(z) \right|.
$$

{\bf A uniform continuity argument.} We define the function $p^*: \mathcal{R}(\alpha,l)\times E(I_k)\rightarrow\C$ by
$$
p^* (s,y) := p\bigl( L(s,y)\bigr).
$$
According to Lemma \ref{lem:normalNproperties} (b), the function $p^*$ is continuous on  $H:=\mathcal{R}(\alpha,l)\times E(I_k)$. The compactness of $E(I_k)\subset K_{2\pi/l}$ and $\mathcal{R}\subset \C$ imply that $H\subset \C \times K_{2\pi/l}$ is compact. Thus, $p^*$ is uniformly continuous on $H$. Consequently, we find a partition 
$$
\alpha=\sigma_1<\sigma_2< ... <\sigma_M=2
$$
of the interval $[\alpha,2]$ and a partition
$$
0=\lambda_1 < \lambda_2 < ... < \lambda_D = l
$$
of the interval $[0,l]$ with the following properties: 
\begin{itemize}
\item[\textreferencemark] For any $\sigma\in[\alpha,2]$, there is an $m\in\{1,...,M\}$ such that, for every $(y,t)\in E(I_k)\times [0,l]$, 
\begin{equation}\label{eq:uc1}
\left| p^{*}(\sigma +it,y) - p^{*}(\sigma_m + it ,y) \right| < \frac{\delta}{3}.
\end{equation}
\item[\textreferencemark] For any $(\sigma,t)\in[\alpha,2]\times [0,l]$, there is an $(m,d)\in\{1,...,M\}\times\{1,...,D\}$ such that, for every $y\in E(I_k)$,
\begin{equation}\label{eq:uc2}
\left| p^{*}(\sigma +it,y) - p^{*}(\sigma_m + i\lambda_d ,y) \right| < \frac{\delta}{3}.
\end{equation}
\end{itemize}
Observe further that, according to \eqref{eq:boundLys}, we have
\begin{equation}\label{pstar}
\max_{(s,y)\in H} \left| p^*(s,y)\right| \leq V_k.
\end{equation}
We are ready to establish statement (ii) and (iii).\par
{\bf Continuous functions on $E(I_k)$.} For $\sigma\in[\alpha,2]$ and $\lambda\in[0,l]$, the functions $p^{\spadesuit},p_{\sigma}^{\diamondsuit},p_{\sigma,\lambda}^{\divideontimes}:E(I_k)\rightarrow\C$ defined by
$$
\begin{array}{cc}
\displaystyle{p^{\spadesuit}(y)}:= \frac{1}{l} \int_0^l\int_{\alpha}^2  p(L(\sigma+it,y)) \d\sigma \d t, 
\\[2em]

\displaystyle{ p_{\sigma}^{\diamondsuit}(y):=\frac{1}{l}\int_0^l  p(L(\sigma+it,y)) \d t  },
\\[2em]

\displaystyle{p_{\sigma,\lambda}^{\divideontimes}(y):=  p(L(\sigma+i\lambda,y))}
\end{array}
$$
are continuous on $E(I_k)$. Thus, \eqref{eq:tanakapmb} applies to them. In fact, the functions above are special cases of functions $P_{\pmb{\mu}}:E(I_k)\rightarrow\C$ defined by
$$
P_{\pmb{\mu}}(y):= \int_{[0,l]\times[\alpha,2]}  p(L(\sigma+it,y)) \d\pmb{\mu}(\sigma,t),
$$
where $\pmb{\mu}$ is an appropriate measure on $[0,l]\times[\alpha,2]$. It might be reasonable to establish Lemma \ref{lem:probmom} and Theorem \ref{th:probmom} for $P_{\pmb{\mu}}$. However, this general approach bears some further technical obstacles which we want to omit here.\par

{\bf Statement (ii).} First, we establish statement (ii). We suppose that $p$ is non-negative. Consequently, 
$$
p^{\spadesuit}(y)\geq 0,\qquad p^{\diamondsuit}(y)\geq 0, \qquad p^{\divideontimes}(y)\geq 0
$$ 
for every $y\in E(I_k)$. By means of \eqref{eq:tanakapmb} and Lemma \ref{lem:arithmeticmean} (a), we find a finite subset $\Delta_k\subset I_k$ such that, for every $J\subset \N$ with $\Delta_k\subset J$ and every $N\in\N$, the inequality
\begin{equation}\label{spadsuit1}
\frac{1}{N}  
\sum_{\begin{subarray}{c} n\in I_k\setminus J \\ n\le N  \end{subarray}} p^{\spadesuit}(ne_l) 
\leq \int_{E(I_k)} p^{\spadesuit}(y) \d \pmb{\tau}+\delta
\end{equation}
holds. First, we consider the right-hand side of the inequality \eqref{spadsuit1}. According to \eqref{pstar}, the function $p^*$ is bounded on $\mathcal{R}(\alpha,l)\times E(I_k)$. By a general version of Fubini's theorem (see for example Deitmar \cite[\S 8.2]{deitmar:2002}) and the topological equivalence of $K_{2\pi/l}\times[0,l)$ and $K$, we obtain that
\begin{align}\label{spadsuit2}
\int_{E(I_k)} p^{\spadesuit}(y) \d \pmb{\tau} 
&= \frac{1}{l} \int_{E(I_k)}  \int_0^l\int_{\alpha}^2  p(L(\sigma+it,y)) \d\sigma \mbox{d} t \mbox{d} \pmb{\tau}\\
&= \int_{\alpha}^2 \int_{E(I_k)\times[0,l)}  p(L(\sigma,y+e_t)) \, (\mbox{d}\ptau\times\tfrac{1}{l}\mbox{d} t) \d \sigma \notag \\
& =\int_{\alpha}^2 \int_{G_k}  p(L(\sigma,x)) \d\pmb{\sigma} \d \sigma, \notag
\end{align}
where $G_k:=\{y+e_t \, : \, (y,t)\in E(I_k)\times [0,l)\}\subset K$. Next, we regard the left-hand side of \eqref{spadsuit1}. We get that
\begin{align}\label{spadsuit3}
\frac{1}{N}  
\sum_{\begin{subarray}{c} n\in I_k\setminus J \\ n\le N  \end{subarray}} p^{\spadesuit}(ne_l) 
&=\frac{1}{Nl}  
\sum_{\begin{subarray}{c} n\in I_k\setminus J \\ n\le N  \end{subarray}}
\int_0^l\int_{\alpha}^2  p\bigl(\L(\sigma+it+inl)\bigr) \d\sigma \d t\\
&=\frac{1}{Nl}  
\sum_{\begin{subarray}{c} n\in I_k\setminus J \\ n\le N  \end{subarray}}
\int_{nl}^{(n+1)l}\int_{\alpha}^2 p\bigl(\L(\sigma+it)\bigr) \d\sigma \d t \notag
\end{align}
Now, the inequality $(\spadesuit)$ of statement (ii) follows by combining \eqref{spadsuit1}, \eqref{spadsuit2} and \eqref{spadsuit3}.\par

To establish $(\diamondsuit)$ of statement (ii), we proceed in an analogous manner. Here, however, we use additionally that $p^*$ is uniform continuous on $\mathcal{R}(\alpha,l)\times E(I_k)$. Again, \eqref{eq:tanakapmb} and Lemma \ref{lem:arithmeticmean} (a) assure that we find a finite subset $\Delta'_k\subset I_k$ such that, for every $m\in\{1,...,M\}$ and every $N\in\N$, 
\begin{equation}\label{diamondsuit1}
\frac{1}{Nl}  
\sum_{\begin{subarray}{c} n\in I_k\setminus J \\ n\le N  \end{subarray}} p_{\sigma_m}^{\diamondsuit}(ne_l) 
\leq \int_{E(I_k)} p_{\sigma_m}^{\diamondsuit}(y) \d \pmb{\tau}+\frac{\delta}{3},
\end{equation}
where $J$ is an arbitrary subset of $\N$ with $\Delta'_k\subset J$. According to the choice of the partition $\alpha=\sigma_1<...<\sigma_M=2$, for every $\sigma\in[\alpha,2]$, we find an $m'\in\{1,...,M\}$ such that \eqref{eq:uc1} holds. Then, we deduce by means of the triangle inequality that, for every $N\in\N$ 
\begin{equation}\label{diamondestimate1}
 \frac{1}{Nl}  
\sum_{\begin{subarray}{c} n\in I_k\setminus J \\ n\le N  \end{subarray}} p_{\sigma}^{\diamondsuit}(ne_l)  \leq  \frac{1}{Nl}  
\sum_{\begin{subarray}{c} n\in I_k\setminus J \\ n\le N  \end{subarray}} p_{\sigma_{m'}}^{\diamondsuit}(ne_l)  +\frac{\delta}{3}
\end{equation}
and
\begin{equation}\label{diamondestimate2}
\int_{E(I_k)} p_{\sigma_{m'}}^{\diamondsuit}(y) \d \pmb{\tau} \leq \int_{E(I_k)} p_{\sigma}^{\diamondsuit}(y) \d \pmb{\tau}  + \frac{\delta}{3} \pmb{\tau}(E(I_k)).
\end{equation}
Statement $(\diamondsuit)$ follows by combining \eqref{diamondsuit1}, \eqref{diamondestimate1} and \eqref{diamondestimate2} and by rewriting the appearing sums and integrals in an appropriate way as described in details for $p^{\spadesuit}$.\par
We can easily repeat the arguments above to prove \eqref{divide1}: firstly, by means of \eqref{eq:tanakapmb} and Lemma \ref{lem:arithmeticmean} (a) and (b), we establish \eqref{divide1} for $p_{\sigma_m,\lambda_m}^{\divideontimes}$ with $(m,d)\in\{1,...,M\}\times\{1,...,D\}$. Then, we argue via uniform continuity. We omit further details here.\par

{\bf Statement (iii).} We proceed to establish statement (iii). Thus, we drop the restriction that $p$ is non-negative. Observe that 
$$
|p^{\spadesuit}(y)|\leq V_k, \qquad |p^{\diamondsuit}(y)|\leq V_k, \qquad 
|p^{\divideontimes}(y)|\leq V_k
$$ 
for $y\in E(I_k)$. By \eqref{eq:tanakapmb} and Lemma \ref{lem:arithmeticmean} (b), we get that, for any $J\subset \N$ with $\dens\ J=0$, 
$$
\frac{1}{N}  
\sum_{\begin{subarray}{c} n\in I_k\setminus J \\ n\le N  \end{subarray}} p^{\spadesuit}(ne_l) =
\int_{E(I_k)} p^{\spadesuit}(y) \d \pmb{\tau} + o(1)
$$ 
as $N\rightarrow\infty$. By similar arguments as above, this translates to
$$
\frac{1}{Nl}  
\sum_{\begin{subarray}{c} n\in I_k\setminus J \\ n\le N  \end{subarray}}
\int_{nl}^{(n+1)l}\int_{\alpha}^2  p\bigl(\L(\sigma+it)\bigr) \d\sigma \d t 
= \int_{\alpha}^2 \int_{G_k}  p\bigl( L(\sigma,x) \bigr)  \d\pmb{\sigma} \d \sigma + o(1),
$$
as $N\rightarrow\infty$. We proved $(\clubsuit)$ of statement (iii).\par

Moreover, for any $J\subset \N$ with $\dens \ J = 0$, we find according to \eqref{eq:tanakapmb} and Lemma \ref{lem:arithmeticmean} (b) a positive integer $N_0$ such that 
\begin{equation}\label{heartestimate1}
\Biggl| \frac{1}{Nl}  
\sum_{\begin{subarray}{c} n\in I_k\setminus J \\ n\le N  \end{subarray}} p_{\sigma_m}^{\diamondsuit}(ne_l)  - 
\int_{E(I_k)} p_{\sigma_m}^{\diamondsuit}(y) \d \pmb{\tau} 
\Biggr|  < \frac{\delta}{3}
\end{equation}
holds for every $m\in\{1,...,M\}$ and $N\geq N_0$. By means of our choice of the partition $\alpha=\sigma_1<...<\sigma_M$, a simple application of the triangle inequality yields that, for every $\sigma\in[\alpha,2]$ and $N\geq N_0$,
\begin{equation}\label{heartestimate2}
\Biggl| \frac{1}{Nl}  
\sum_{\begin{subarray}{c} n\in I_k\setminus J \\ n\le N  \end{subarray}} p_{\sigma}^{\diamondsuit}(ne_l)  - 
\int_{E(I_k)} p_{\sigma}^{\diamondsuit}(y) \d \pmb{\tau} 
 \Biggr|  < \delta
\end{equation}
holds. Since
$$
\int_{E(I_k)} p_{\sigma}^{\diamondsuit}(y) \d \pmb{\tau} 
= \int_{G_k} p\left( L(\sigma,y)\right) \d \pmb{\sigma} 
$$
and our choice of $\delta>0$ was arbitrary, statement \eqref{diamond2} is proved.\par 
It is now clear how to prove \eqref{divide2}. First, by means of \eqref{eq:tanakapmb} and Lemma \ref{lem:arithmeticmean} (a) and (b), we establish \eqref{divide2} for $p_{\sigma_m,\lambda_d}^{\divideontimes}$ with $(m,d)\in\{1,...,M\}\times\{1,...,D\}$. Then, we argue via uniform continuity. Again, we omit further details here.
\end{proof}

\subsection*{Proof of Theorem \ref{th:probmom}}
We shall now derive Theorem \ref{th:probmom} from Lemma \ref{lem:probmom}. 
First, we establish Theorem \ref{th:probmom} for a single function $\L$ and a single function $p$.\par

{\bf Theorem \ref{th:probmom} for a single pair $(\L,p)$.} Let $u\in[\frac{1}{2},1)$, $\L\in \No(u)$ with Dirichlet series expansion
$$
\L(s) = \sum_{n=1}^{\infty} \frac{a(n)}{n^s} \in \mathscr{H}^2
$$
and $p:\C\rightarrow \C$ be a continuous function such that condition \eqref{condC} of Theorem \ref{th:probmom} is satisfied. The latter implies that we find constants $c_1,c_2 \geq 0$ such that
\begin{equation}\label{qqq}
\left| p(z) \right| \leq c_1 |z|^2 + c_2, \qquad z\in\C.
\end{equation}
We define $q:\C\rightarrow \R_0^+$ by
$$
q(z)= c_1 |z|^2 + c_2, \qquad z\in \C.
$$
Observe that $q$ is a non-negative, continuous function on $\C$. The theorem of Plancherel, in particular \eqref{plan1}, asserts that, for $\sigma>\frac{1}{2}$,
$$
\mathcal{Z}(\sigma):=\int_K q\left( L(\sigma, x) \right) \d \pmb{\sigma} 
= c_1\cdot\sum_{n=1}^{\infty} \frac{|a(n)|^{2}}{n^{2\sigma}} + c_2  < \infty
$$
Hence, due to \eqref{qqq}, we have, for $\sigma>\frac{1}{2}$,
$$
\left| \int_K p\left( L(\sigma, x) \right) \d \pmb{\sigma} \right| 
\leq \mathcal{Z}(\sigma)  < \infty
$$
and, as $l\in\Gamma_P$,
$$
\left| \int_{K_{2\pi/l}} p\left( L(\sigma, x) \right) \d \pmb{\tau} \right|
\leq   \int_{K_{2\pi/l}} q\left( L(\sigma, x) \right)  \d \pmb{\tau} =   \mathcal{Z}(\sigma)  < \infty.
$$

Let $(M_k)_{k\in\varTheta}$ be a sequence of positive real numbers $M_k\geq 1$ such that statements of Lemma \ref{lem:probmom} hold for $\L$ and the continuous, non-negative function $q$. Again, as in the proof of Lemma \ref{lem:probmom}, we may suppose without loss of generality that $\varTheta=\N$. For $M_k$, let $I_k$ denote the corresponding subset of $\N$ as defined in Lemma \ref{lem:probmom}. We fix positive numbers $\delta_k$, $k\in \N$, such that
$$
\sum_{k=1}^{\infty} \delta_k <\infty.
$$
For $k\in\N$ and fixed $\alpha\in(u,1]$, let $\Delta_k$ be a finite subset of $I_k$ such that the statements of Lemma \ref{lem:probmom} (ii) hold with the special choice of $\delta = \delta_k$, respectively. We set 
$$
\Upsilon:= \left( \bigcup_{k=1}^{\infty} \Delta_k \right) \cup \left(\N \setminus \bigcup_{k=1}^{\infty} I_k \right).
$$
Let $J(M)$ be defined as in Lemma \eqref{lem:probmom}. Then, according to Lemma \ref{lem:JM}, for any $\eps>0$, we find a natural number $\nu$ such that 
$$
\dens_* J(M_{\nu})>1-\eps
$$
and, consequently,
\begin{equation}\label{2}
\dens^* \left( \N\setminus J(M_{\nu}) \right) \leq \eps.
\end{equation}
By the definition of $I_k$, we obtain that the set $\N\setminus J(M_{\nu})$ contains both $\bigcup_{k=\nu+1}^{\infty} I_k$ and $\N \setminus \bigcup_{k=1}^{\infty} I_k$. Since the set $\bigcup_{k=1}^{\nu} \Delta_k$ has only finitely many elements, we conclude by means of \eqref{2} that
$$
\dens^* \Upsilon \leq \eps.
$$
Since this holds for any $\eps>0$, we derive that
$$
\dens \ \Upsilon = 0.
$$
From now on, let $J$ be an arbitrary subset of $\N$ with 
$$\dens\ J = 0 \qquad \mbox{and} \qquad \Upsilon \subset J.$$
Due to $\Upsilon \subset J$, we have in particular that $\Delta_k\subset J$ for $k\in\N$. We shall see later on why it is reasonable not to work with $\Upsilon$ but with an arbitrary set $J$ with the properties stated above. Roughly speaking, the set $J$ shall enable us to establish Theorem \ref{th:probmom} for countable many pairs $(\L_j,p_j)$, $j\in\N$, of functions $\L_j$ and $p_j$ simultaneously.\par
First, we establish statement (i) of Theorem \ref{th:probmom} for $(\L,p)$. We set 
$$
A_k := \int_{\alpha}^2 \int_{G_k} p(L(\sigma,x)) \d\pmb{\sigma} \d \sigma,
\qquad \qquad
A_k' := \int_{\alpha}^2 \int_{G_k} q(L(\sigma,x))  \d\pmb{\sigma} \d \sigma
$$
$$
Z:= \int_{\alpha}^2 \int_{K} p(L(\sigma,x)) \d\pmb{\sigma} \d \sigma
\qquad \mbox{and} \qquad
Z':= \int_{\alpha}^2 \int_{K} q(L(\sigma,x)) \d\pmb{\sigma} \d \sigma
$$
where $G_k$ is defined as in Lemma \ref{lem:probmom}. Relation \eqref{qqq} assures that 
$$\left| Z\right| \leq Z' = \int_{\alpha}^2 \mathcal{Z}(\sigma) \d \sigma <\infty.$$ 
From the properties of $E(I_k)$ in Lemma \ref{lem:probmom}, we deduce that
$$
\pmb{\tau}\left( \bigcup_{k=1}^{\infty} E(I_k)\right) = \lim_{k\rightarrow\infty} \ptau (E(I_k)) =1.
$$
This implies that $$\pmb{\sigma}\left(\bigcup_{k=1}^{\infty} G_k\right) =1.$$
Consequently, we get that
\begin{equation}\label{AAA}
\sum_{k=1}^{\infty}A_k = Z \qquad \mbox{and} \qquad \sum_{k=1}^{\infty}A'_k = Z'.
\end{equation}
We observe that
\begin{equation}\label{A}
\sum_{k=1}^{\infty}|A_k| \leq \sum_{k=1}^{\infty}A'_k = Z' < \infty.
\end{equation}
Furthermore, we set, for $n\in \N\setminus J$,
$$
B_p(n):= \frac{1}{l}\int_{\alpha}^2\int_{n l}^{(n+1) l}  p\bigl(\L(\sigma+it)\bigr)  \d t\d\sigma
$$
and
$$
B_q(n): =\frac{1}{l}\int_{\alpha}^2\int_{n l}^{(n+1) l}  q\bigl(\L(\sigma+it)\bigr)  \d t\d\sigma
$$
It is immediately clear that $|B_p(n)|\leq B_q(n)$ for $n\in\N\setminus J$. Since $\Delta_k\subset J$ for $k\in\N$, we get by applying statement \eqref{spade1} of Lemma \ref{lem:probmom} (ii) to $q$ that, for $k\in \N$,
\begin{equation}\label{BpBq}
|B_p(n)|\leq B_q(n) \leq A_k+\delta_k,
\end{equation}
provided that $n\in I_k\setminus J$. Hence, for every $N\in \N$,
\begin{equation}\label{Bp}
\frac{1}{N l} \sum_{k=1}^{\infty} \, \, \sum_{\begin{subarray}{c} n\in I_k\setminus J \\ n\le N  \end{subarray}} \left| B_p(n)\right| \leq 
\frac{1}{N l} \sum_{k=1}^{\infty} \, \, \sum_{\begin{subarray}{c} n\in I_k\setminus J \\ n\le N  \end{subarray}}  B_q(n) \leq 
\sum_{k=1}^{\infty} A_k' + \sum_{k=1}^{\infty}\delta_k < \infty.
\end{equation}
The estimates \eqref{A} and \eqref{Bp} imply, in particular, that both the double series 
$$
\frac{1}{N } \sum_{k=1}^{\infty} \, \, \sum_{\begin{subarray}{c} n\in I_k\setminus J \\ n\le N  \end{subarray}} B(n) \qquad \mbox{and} \qquad
\sum_{k=1}^{\infty} A_k
$$
converge absolutely for every $N\in\N$. Thus, we may rearrange their terms, respectively. \par
Now, let $\eps>0$. Then, we find a number $D\in\N$ such that
\begin{equation}\label{eq:D}
\sum_{k=D+1}^{\infty} A_k' <\frac{\eps}{3} \qquad \mbox{and} \qquad
\sum_{k=D+1}^{\infty} \delta_k <\frac{\eps}{3}.
\end{equation}
By the rearrangement theorem and the construction of the set $\Upsilon$, we obtain that, for $N\in \N$,
$$
\Bigl| \, \frac{1}{N } \, \, \sum_{\begin{subarray}{c} n\in \N \setminus J \\ n\le N  \end{subarray}} B_p(n) - Z \Bigr|
\leq \sum_{k=1}^{\infty} \Bigl| \frac{1}{N } \, \, \sum_{\begin{subarray}{c} n\in  I_k \setminus J \\ n\le N  \end{subarray}} B_p(n) - A_k \Bigr|
$$
By means of \eqref{BpBq} and \eqref{eq:D}, we get that the inequality
$$
\Bigl| \frac{1}{N } \, \, \sum_{\begin{subarray}{c} n\in \N \setminus J \\ n\le N  \end{subarray}} B_p(n) - Z \Bigr| \leq \sum_{k=D+1}^{\infty} (A_k' + \delta_k)
+  \sum_{k=1}^{D} \Bigl| \frac{1}{N } \, \, \sum_{\begin{subarray}{c} n\in I_k \setminus J \\ n\le N  \end{subarray}} B_p(n) - A_k \Bigr|
+ \sum_{k=D+1} A_k' \leq
$$
$$
\leq \sum_{k=1}^{D} \Bigl| \frac{1}{N } \, \, \sum_{\begin{subarray}{c} n\in I_k \setminus J \\ n\le N  \end{subarray}} B_p(n) - A_k \Bigr| + \eps 
$$
holds for every $N\in\N$. From statement \eqref{spade2} of Lemma \ref{lem:probmom} (iii), we derive that
\begin{equation}\label{limessup}
\limsup_{N\rightarrow\infty} \Bigl| \frac{1}{N } \, \, \sum_{\begin{subarray}{c} n\in I_k \setminus J \\ n\le N  \end{subarray}} B_p(n) - Z \Bigr| 
\leq \eps.
\end{equation}
Since this is true for any $\eps>0$, we conclude that, for any $J\subset \N$ with $\Upsilon\subset J$ and $\dens \ J = 0$, as $N\rightarrow\infty$,
\begin{equation}\label{eqq}
\frac{1}{N l} \sum_{\begin{subarray}{c}n\in\N\setminus J \\ n\leq N \end{subarray}} \int_{n l}^{(n+1) l}\int_{\alpha}^{2} p\bigl(\L(\sigma+it)\bigr) \d\sigma \d t = \int_{\alpha}^{2}\int_{K}  p\bigl(L(\sigma,x) \bigr)  \d \pmb{\sigma} \d\sigma + o(1).
\end{equation}
To establish statement (ii) of Theorem \ref{th:probmom} for the couple $(\L,p)$, we proceed in an analogous manner: for $\sigma\in[\alpha,2]$, we set
$$
\dot{A}_{k}(\sigma) = \int_{G_k} p(L(\sigma,x) \d\pmb{\sigma} ,
\qquad \qquad
\dot{A}'_{k}(\sigma):= \int_{G_k}  q(L(\sigma,x))  \d\pmb{\sigma} ,
$$
$$
\dot{Z}(\sigma):= \int_{K} p(L(\sigma,x)) \d\pmb{\sigma} ,
\qquad  \qquad
\dot{Z}'(\sigma) := \mathcal{Z}(\sigma) = \int_{K} q(L(\sigma,x)) \d\pmb{\sigma} ,
$$
$$
\dot{B}_{p}(\sigma,n):= \frac{1}{l}\int_{n l}^{(n+1) l}  p\bigl(\L(\sigma+it)\bigr) \d t
$$
and
$$
\dot{B}_{q}(\sigma,n):= \frac{1}{l}\int_{n l}^{(n+1) l}  q\bigl(\L(\sigma+it)\bigr) \d t .
$$
By replacing the quantities $A_k$, $A_k'$, $Z$, $Z'$, $B_p(n)$, $B_q(n)$ by $\dot{A}_{k}(\sigma)$, $\dot{A}'_{k}(\sigma)$, $\dot{Z}_{k}(\sigma)$, $\dot{Z}'_{k}(\sigma)$, $\dot{B}_{p}(n,\sigma)$, $\dot{B}_{q}(n,\sigma)$, respectively, and by using \eqref{diamond1} and \eqref{diamond2} of Lemma \ref{lem:probmom} instead of \eqref{spade1} and \eqref{spade2}, we can follow basically step by step the argumentation above in order to prove that, for any $J\subset \N$ with $\Upsilon\subset J$ and $\dens \ J = 0$, uniformly for $\sigma\in[\alpha,2]$, as $N\rightarrow\infty$,
\begin{equation}\label{eqqq}
\frac{1}{N l} \sum_{\begin{subarray}{c}n\in\N\setminus J \\ n\leq N \end{subarray}} \int_{n l}^{(n+1) l} p\bigl(\L(\sigma+it)\bigr) \d t = \int_{K}  p\bigl(L(\sigma,x) \bigr)  \d \pmb{\sigma} + o(1).
\end{equation}
Additional arguments are necessary in order to establish the uniformity in $\sigma$:\par
First, we consider the choice of $D$ in \eqref{eq:D}. We regard the sequence $(f_K)_K$ of functions $f_K:[\alpha,2]\rightarrow \R_0^+$ defined by
$$
f_K(\sigma)=\sum_{k=1}^{K} \dot{A}_k'(\sigma).
$$
Analogously to \eqref{AAA}, we have, for $\sigma\in[\alpha,2]$,
$$
\lim_{K\rightarrow\infty} f_K(\sigma) = \sum_{k=1}^{\infty}\dot{A}_k'(\sigma)
= \dot{Z}' (\sigma) =\mathcal{Z}(\sigma)<\infty.
$$
The function $\mathcal{Z}(\sigma)$ is continuous on $[\alpha,2]$ and, for every $\sigma\in[\alpha,2]$, the sequence $(f_K(\sigma))_K$ is monotonically increasing, since the function $q$ is real-valued and non-negative. This implies that the sequence $(f_K)_K$ converges uniformly on $[\alpha,2]$ to $\mathcal{Z}$. Hence, we find, for any given $\eps>0$, a positive integer $D$
such that
$$
\sum_{k=D+1}^{\infty} \dot{A}_k'(\sigma) <\frac{\eps}{3}
$$
holds for every $\sigma\in[\alpha,2]$.\par
Secondly, we consider the limit superior in \eqref{limessup}. Here, it is immediately clear from the statement \eqref{diamond2} of Lemma \ref{lem:probmom} (iii) that
$$
\limsup_{N\rightarrow\infty} \Bigl| \frac{1}{N } \, \, \sum_{\begin{subarray}{c} n\in \N \setminus J \\ n\le N  \end{subarray}} \dot{B}_p(n,\sigma) - \dot{Z}(\sigma) \Bigr| \leq \limsup_{N\rightarrow\infty} \sum_{k=1}^D  \Bigl| \frac{1}{N }  \sum_{\begin{subarray}{c} n\in I_k \setminus J \\ n\le N  \end{subarray}} \dot{B}_p(n,\sigma) - \dot{A_k}(\sigma) \Bigr| + \eps
\leq \eps
$$
holds uniformly for $\sigma\in[\alpha,2]$.\par

Statement (iii) of Theorem \ref{th:probmom} follows also along the lines of our considerations above: this time, we set for $\sigma\in[\alpha,2]$ and $\lambda\in[0,l]$,
$$
\ddot{A}_k(\sigma) := \int_{E(I_k)} p(L(\sigma,x) \d\pmb{\tau} ,
\qquad \qquad
\ddot{A}'_k(\sigma):= \int_{E(I_k)} q(L(\sigma,x))  \d\pmb{\tau} ,
$$
$$
\ddot{Z}_{k}(\sigma):=\int_{K_{2\pi/l}} p(L(\sigma,x)) \d\pmb{\tau} ,
\qquad  \qquad
\ddot{Z}'_k(\sigma) := \int_{K_{2\pi/l}} q(L(\sigma,x)) \d\pmb{\tau}, 
$$
$$
\ddot{B}_p(n,\sigma,\lambda):=   p\bigl(\L(\sigma+i\lambda + inl)\bigr) 
$$
and
$$
\ddot{B}_q(n,\sigma,\lambda):=   q\bigl(\L(\sigma+i\lambda + inl)\bigr).
$$

By replacing $A_k$, $A_k'$, $Z$, $Z'$, $B_p(n)$, $B_q(n)$ by $\ddot{A}_{k}(\sigma)$, $\ddot{A}'_{k}(\sigma)$, $\ddot{Z}_{k}(\sigma)$, $\ddot{Z}'_{k}(\sigma)$, $\ddot{B}_{p}(n,\sigma,\lambda)$, $\ddot{B}_{q}(n,\sigma,\lambda)$, respectively, in our argumentation above and by relying on \eqref{divide1} and \eqref{divide2} of Lemma \ref{lem:probmom}, we obtain that, for any $J\subset \N$ with $\Upsilon\subset J$ and $\dens \ J = 0$, uniformly for $\sigma\in[\alpha,2]$ and $\lambda\in[0,l]$, as $N\rightarrow\infty$,
\begin{equation}\label{eqqqq}
\frac{1}{N } \sum_{\begin{subarray}{c}n\in\N\setminus J \\ n\leq N \end{subarray}}  p\bigl(\L(\sigma+i\lambda + inl)\bigr)  = \int_{K_{2\pi/l}}  p\bigl(L(\sigma,x) \bigr)  \d \pmb{\sigma} + o(1).
\end{equation}
Additionally arguments are necessary in two steps of our argumentation in order to obtain uniformity in $\sigma$ and $\lambda$.\par

Firstly, we consider the choice of $D$ in \eqref{eq:D}. Since $l\in \Gamma_P$, we have
$$
\ddot{A}'_{k}(\sigma) = \dot{A}'_{k}(\sigma) \qquad \mbox{for }k\in\N. 
$$ 
Due to our analysis above, for given $\eps>0$, we find a positive integer $D$ such that, for every $\sigma\in[\alpha,2]$,
$$
\sum_{k=D+1}^{ \infty}\ddot{A}_{k}(\sigma)  < \eps.
$$

Secondly, it follows again from statement \eqref{diamond2} of Lemma \ref{lem:probmom} (iii) that the limit superior involved in \eqref{limessup} is uniform in $\sigma$ and $\lambda$. \par

In fact, a close look at the proof reveals that we can transfer our arguments to the case $l\in \Gamma_P$, if the function
$$
f(\sigma):= \int_{K_{2\pi/l}} q(L(\sigma,x)) \d\ptau
$$
is continuous for $\sigma>u$. \par
{\bf Theorem \ref{th:probmom} for countable many pairs $(\L_j,p_j)$.} For given $u\in[\frac{1}{2},1)$, let $(\L_j)_j$ be a sequence of functions $\L_j \in \No(u)$ and let $(p_j)_j$ be a sequence of continuous functions $p_j:\C\rightarrow \C$ that satisfy condition \eqref{condC} of Theorem \ref{th:probmom}. Then, we find for every pair $(p_j,\L_j)$ a set $\Upsilon_j \subset \N$ with $\dens \ \Upsilon_j = 0$ such that \eqref{eqq}, \eqref{eqqq} and \eqref{eqqqq} hold, respectively, for $p_j$ and $\L_j$ with $\Upsilon$ being replaced by $\Upsilon_j$.\par

We construct a common set $J$ with $\dens \ J = 0$ such that \eqref{eqq}, \eqref{eqqq} and \eqref{eqqqq} hold simultaneously for every pair $(p_j,L_j)$.\par

Let $(\eps_k)_k$ be a sequence of positive real numbers with $\lim_{k\rightarrow\infty} \eps_k = 0$. We set $N_1:=1$ and $J_1:= \Upsilon_1$.
Inductively for $k\in \N$ with $k\geq 2$, we fix a positive integer $N_k\geq N_{k-1}$ such that, for $N\geq N_k$, both
$$
\frac{\# \bigl( J_{k-1} \cap [1,N ] \bigr)}{N}  < \frac{\eps_k}{2} \qquad \mbox{and} \qquad
\frac{\# \bigl( \Upsilon_k \cap [1,N ] \bigr)}{N}  < \frac{\eps_k}{2} 
$$
and set
$$
J_{k} := \bigcup_{j=1}^{k} \bigl( \Upsilon_j\setminus[1,N_j) \bigr)
$$
By the construction of $J_k$, we have
$$
\frac{\#\bigl( J_k \cap [1,N ] \bigr)}{N}  < \eps_k \qquad \mbox{for }N\geq N_k. 
$$
This implies that the set
$$
J':= \bigcup_{j=1}^{\infty} \bigl( \Upsilon_j\setminus[1,N_j) \bigr)
$$
has density zero. Now, it is easy to see that for the special choice of $J=J'$, we obtain that for every $j\in \N$, as $N\rightarrow\infty$, 
$$
\frac{1}{N l} \sum_{\begin{subarray}{c}n\in\N\setminus J' \\ N_j\leq n\leq N \end{subarray}} \int_{n l}^{(n+1) l}\int_{\alpha}^{2} p_j\bigl(\L_j(\sigma+it)\bigr) \d\sigma \d t = \int_{\alpha}^{2}\int_{K}  p_j\bigl(L_j(\sigma,x) \bigr)  \d \pmb{\sigma} \d\sigma + o_j(1);
$$
for every $j\in \N$, uniformly for $\sigma\in[\alpha,2]$, as $N\rightarrow\infty$,
$$
\frac{1}{N l} \sum_{\begin{subarray}{c}n\in\N\setminus J' \\ N_j \leq n\leq N \end{subarray}} \int_{n l}^{(n+1) l} p_j\bigl(\L_j(\sigma+it)\bigr) \d t = \int_{K}  p_j\bigl(L_j(\sigma,x) \bigr)  \d \pmb{\sigma} + o_j(1);
$$
and for every $j\in \N$, uniformly for $\sigma\in[\alpha,2]$ and $\lambda\in[0,l]$, as $N\rightarrow\infty$,
$$
\frac{1}{N l} \sum_{\begin{subarray}{c}n\in\N\setminus J' \\ N_j\leq n\leq N \end{subarray}}  p_j\bigl(\L_j(\sigma+i\lambda + inl)\bigr)  = \int_{K_{2\pi/l}}  p_j\bigl(L_j(\sigma,x) \bigr)  \d \pmb{\tau} + o_j(1).
$$
If we set
$$
A:= \bigcup_{n\in J'} [nl, (n+1)l),
$$
then $A$ is an $l$-set of density zero and the statements above can be translated easily into the form given in Theorem \eqref{th:probmom} (i), (ii) and (iii). The theorem is proved.

\section{Applications to the value-distribution of the Riemann zeta-function}
The asymptotic expansions for the moments in Theorem \ref{th:probmom} yield information on the value-distribution of the functions $\L\in\No(u)$ under consideration. Exemplarily, we discuss the case of the Riemann zeta-function. We can retrieve some unconditional information on the growth behaviour of the Riemann zeta-function in $\sigma>\frac{1}{2}$.
\begin{corollary}\label{cor:app1}
Let $\alpha>\frac{1}{2}$ and $l>0$. Then, there exist an $l$-set $A\subset [1,\infty)$ of density zero such that, for any $\eps>0$, uniformly for $\sigma\geq \alpha$
\begin{equation}\label{app1}
\zeta(\sigma+it)\cdot \pmb{1}_{A^c}(t) \ll_{\eps} t^{\eps},  \qquad \mbox{as }t\rightarrow\infty
\end{equation}
and
\begin{equation}\label{app2}
\frac{1}{\zeta(\sigma+it)}\cdot \pmb{1}_{A^c}(t)  \ll_{\eps} t^{\eps},  \qquad \mbox{as }t\rightarrow\infty. 
\end{equation}
\end{corollary}
Under the assumption of the Lindel\"of hypothesis, it is known that \eqref{app1} holds for $A=\emptyset$; see Titchmarsh \cite[\S 13.2]{titchmarsh:1986}. If we assume the truth of the Riemann hypothesis, then \eqref{app2} holds for $A=\emptyset $; see Titchmarsh \cite[\S 14.2]{titchmarsh:1986}. In view of Lemma \ref{lem:JM} and Lemma \ref{lem:classNrelfct2}, however, the results of Corollary \ref{cor:app1} are not too surprising and can also be derived by other techniques.

\begin{proof}
We know that $\zeta^k\in\No(\frac{1}{2})$ for any $k\in\Z$. According to Theorem \ref{th:probmom}, there exist an $l$-set $A$ of density zero such that, uniformly for $\sigma\geq \alpha$, as $T\rightarrow\infty$,
$$
\lim_{T\rightarrow \infty}\frac{1}{T}\int_1^T \left|\zeta(\sigma+it) \right|^{2k} \pmb{1}_{A^c}(t) \d t =\sum_{n=1}^{\infty}\frac{d_k(n)^2}{n^{2\sigma}}.
$$
The assertions follows by standard techniques (see Titchmarsh \cite[\S 13.2]{titchmarsh:1986}).
\end{proof}

Corollary \ref{cor:app1} can be used to obtain a certain non-denseness result for the Riemann zeta-function in $\sigma<\frac{1}{2}$.

\begin{corollary}
Let $\alpha<\frac{1}{2}$ and $l>0$. Then, there exist an $l$-set $A\subset [1,\infty)$ of density zero such that, for every $\sigma\leq \alpha$,
\begin{equation}\label{app3}
\overline{\left\{\zeta(\sigma+it) \, : \, t\in A^c \right\}} \neq \C.
\end{equation}
\end{corollary}
Under the assumption of the Riemann hypothesis, Garunk\v{s}tis \& Steuding showed that \eqref{app3} holds for $A=\emptyset$.
\begin{proof}
This follows directly from \eqref{app2}, the functional equation of the Riemann zeta-function and the asymptotic expansion for $\Delta(s)$.
\end{proof}

It would be nice to find further applications of the moments in Theorem \ref{th:probmom} to the value-distribution of the Riemann zeta-function.

\setcounter{section}{0}
\renewcommand{\theequation}{A.\arabic{equation}}
\renewcommand{\thesection}{A.\arabic{section}}
\renewcommand{\thetheorem}{A.\arabic{theorem}}

\chapter*{Appendix}
\addcontentsline{toc}{chapter}{Appendix: Normal families of meromorphic functions}
\markboth{Normal families}{Normal families}

\section*{Normal families of meromorphic functions}
The theory of normal families provides a powerful tool to study the value-distribution of meromorphic functions in the neighbourhood of an essential singularity. We shall use this section to outline the basic ideas of this concept. For a thorough account of the theory of normal families and their connection to value-distribution theory, the reader is referred to a monography by Schiff \cite{schiff:1993} and a nice survey paper of Zalcman \cite{zalcman:1998} which we took as a guideline for most of the following introductory outline.\par

{\bf Sequences of analytic functions.} Let $\mathcal{H}(\Omega)$ denote the set of all analytic functions on a domain $\Omega\subset\C$. We start with some classical theorems on sequences $(f_n)_n$ of functions $f_n\in\mathcal{H}(\Omega)$. The locally uniform convergence of $(f_n)_n$ implies already that the limit function is also analytic and that the limiting process is stable with respect to complex differentiation.
\begin{theorem}[Theorem of Weierstrass]\label{th:weierstrass}
Let $(f_n)_n$ be a sequence of functions $f_n\in\mathcal{H}(\Omega)$ that converges locally uniformly on $\Omega$ to a function $f$. Then, $f$ is analytic in $\Omega$ and the sequence of derivatives $(f^{(k)}_n)_n$ converges locally uniformly on $\Omega$ to $f^{(k)}$, $k\in\N$.
\end{theorem}
For a proof, we refer to Schiff \cite[p. 9]{schiff:1993}.
It follows from the analyticity of the limit function $f$ and the theorem of Rouch\'{e} that, for sufficiently large $n$, the functions $f_n$ and $f$ have essentially the same number of zeros in $\Omega$.
\begin{theorem}[Theorem of Hurwitz]\label{th:hurwitz}
Let $(f_n)_n$ be a sequence of functions $f_n\in\mathcal{H}(\Omega)$ converge locally uniformly on $\Omega$ to a non-constant function $f$. If $f$ has a zero of order $m$ at $z_0$, then there exists an $r>0$ such that for sufficiently large $n\in\N$, $f_n$ has exactly $m$ zeros (counting multiplicities) in the disc $D_r(z_0)$.
\end{theorem}
A proof can be found in Schiff \cite[p. 9]{schiff:1993}.\par

{\bf Sequences of meromorphic functions.} Let $\mathcal{M}(\Omega)$ denote the set of all meromorphic functions on a domain $\Omega\subset\C$. To define convergence for sequences of meromorphic functions, one usually relies on the chordal metric on the Riemann sphere $\widehat{\C}$. Recall that the chordal metric $\chi(\cdot, \cdot)$ on $\widehat{\C}$ is defined by
\begin{align*}
\chi(z_1,z_2) & := \frac{|z_1-z_2|}{\sqrt{1+|z_1|^2} \sqrt{1+|z_2|^2}}, \qquad z_1,z_2\in\C,  \\
\chi(z_1,\infty) &= \chi(\infty,z_1) :=  \frac{1}{\sqrt{1+|z_1|^2}}, \quad \;\; z_1\in\C,\\
\chi(\infty, \infty) & := 0.
\end{align*} 
Suppose that a sequence $(f_n)_n$ of functions $f_n\in\mathcal{M}(\Omega)$ converges locally uniformly on $\Omega$ to a limit function $f$, with respect to the chordal metric. Then, it follows from the theorem of Weierstrass, that $f$ is either meromorphic on $\Omega$ or $f\equiv\infty$. Note that $(f_n)_n$ converges locally uniformly on $\Omega$ to a limit function $f\not\equiv\infty$ in the chordal metric if and only if $(f_n)_n$ converges locally uniformly on $\Omega$ to $f$ in the Euclidean one. Moreover, $(f_n)_n$ converges locally uniformly on $\Omega$ to $f\equiv \infty$ in the chordal metric if and only if $(1/f_n)_n$ converges locally uniformly to $f\equiv 0$ in the Euclidean one. 

{\bf Normal families.} The concept of normal families dates back to Montel \cite{montel:1911, montel:1946}. A family $\mathcal{F}\subset\mathcal{M}(\Omega)$ is called {\it normal in $\Omega$} if every sequence of functions in $\mathcal{F}$ contains a subsequence that converges locally uniformly with respect to the chordal metric on $\Omega$. The family $\mathcal{F}$ is said to be {\it normal in a point $z_0\in\Omega$} if $\mathcal{F}$ is normal in a neighbourhood of $z_0$. Normality is a local property: one can show that $\mathcal{F}$ is normal in $\Omega$ if and only if it is normal in every point $z_0\in\Omega$.\par
Normality yields a concept for compactness in the space $\widehat{\mathcal{M}}(\Omega)$ which we obtain by adding $f\equiv\infty$ to $\mathcal{M}(\Omega)$. The space $\widehat{\mathcal{M}}(\Omega)$ be endowed with the topology of uniform convergence on compact subsets of $\Omega$ and a metric generating the latter such that $\widehat{\mathcal{M}}(\Omega)$ becomes a complete metric space. Then, a family $\mathcal{F}$ of meromorphic functions is normal in $\Omega$ if and only if $\mathcal{F}$ is relatively compact in $\widehat{\mathcal{M}}(\Omega)$. \par

{\bf Characterizations of normality.} In terms of equicontinuity and local boundedness, the Theorem of Arzel\`{a}-Ascoli provides necessary and sufficient conditions for subsets of certain function spaces to be relatively compact. In the case of normal families of meromorphic functions this translates to the following.
\begin{theorem}[Theorem of Arzel\`{a}-Ascoli for families of meromorphic functions] \label{th:montel2}
A family $\mathcal{F}\subset\mathcal{M}(\Omega)$ is normal in $\Omega$ if and only if $\mathcal{F}$ is equicontinuous on $\Omega$ with respect to the chordal metric.
\end{theorem}
A proof can be found in Schiff \cite[p. 74]{schiff:1993}. Montel \cite{montel:1907} observed that for a family of analytic functions locally boundedness implies equicontinuity. We call a family $\mathcal{F}\subset\mathcal{H}(\Omega)$ {\it locally bounded on $\Omega$} if for every $z_0\in\Omega$ there exists a neighbourhood $U(z_0)$ and a positive real number $M$ such that $|f(z)|\leq M$ for all $z\in U(z_0)$ and all $f\in\mathcal{F}$.
\begin{theorem}[Montel's theorem] \label{th:montel1}
A family $\mathcal{F}\subset\mathcal{H}(\Omega)$ that is locally bounded on $\Omega$ is normal in $\Omega$.
\end{theorem}
For a proof we refer to Schiff \cite[p. 35]{schiff:1993}. The local boundedness of $\mathcal{F}$ in Montel's theorem implies that $\mathcal{F}$ has no subsequence that converges locally uniformly to $f\equiv\infty$. Thus, one can easily see that the converse of Montel's theorem is not true in general. However, if a family $\mathcal{F}\subset\mathcal{H}(\Omega)$ is normal in $\Omega$ and, additionally, there is a point $z_0\in\Omega$ and a positive real number $M$ such that
$$
|f(z_0)| \leq M \qquad \mbox{ for all }f\in\mathcal{F},
$$ 
then $\mathcal{F}$ is locally bounded.\par
A sufficient condition for the local boundedness of a family  $\mathcal{F}\subset \mathcal{H}(\Omega)$ can be formulated by means of an area integral.
\begin{theorem}\label{th:sufflocalbound}
Let $\mathcal{F}\subset\mathcal{H}(\Omega)$. Suppose that there exists a positive real number $M$ such that, for all $f\in\mathcal{F}$,
$$
\iint_{\Omega} \left|f(x+iy) \right|^2 \d x \d y \leq M.
$$ 
Then, $\mathcal{F}$ is locally bounded in $\Omega$. 
\end{theorem}
The proof of Theorem \ref{th:sufflocalbound} relies essentially on an integrated version of Cauchy's integral formula; see Titchmarsh \cite[\S 11.8]{titchmarsh:1986} or Schiff \cite[p. 39]{schiff:1993}.

Schiff \cite[p. 39]{schiff:1993} also mentions the following variation of Montel's theorem by Mandelbrojt \cite{mandelbrojt:1929}.
\begin{theorem}[Mandelbrojt's variation of Montel's theorem]
Let $\mathcal{F}\subset\mathcal{H}(\Omega)$ be a family of zero-free analytic functions. Then, $\mathcal{F}$ is normal in $\Omega$ if and only if the correspoonding family of functions given by 
$$
F(z,w) = \frac{f(z)}{f(w)}
$$
is locally bounded on $\Omega\times\Omega$.
\end{theorem}\par

Marty \cite{marty:1931} succeeded to connect the chordal equicontinuity of $\mathcal{F}$ in Theorem \ref{th:montel2} with the boundedness of a suitable derivative. For a meromorphic function $f$ on a domain $\Omega$, we define the {\it spherical derivative} by
$$
f^{\#} (z) := \lim_{h\rightarrow 0} \frac{\chi(f(z+h), f(z))}{|h|} = \frac{|f'(z)|}{1+|f(z)|^2}, \qquad z\in\Omega.
$$
\begin{theorem}[Marty's theorem]\label{th:marty} 
A family $\mathcal{F}\subset\mathcal{M}(\Omega)$ is normal in $\Omega$ if and only if the corresponding family 
$$
\mathcal{F}^{\#}:= \left\{f^{\#} \, : \, f\in\mathcal{F}\right\}
$$
of spherical derivatives is locally bounded on $\Omega$.
\end{theorem}
For a proof, we refer to Schiff \cite[p. 75]{schiff:1993}. Montel \cite{montel:1912} revealed a gainful connection of normality to value-distribution theory. We say that a family $\mathcal{F}\subset\mathcal{M}(\Omega)$ {\it omits a value $a\in\widehat{\C}$ on $\Omega$}, if no function $f\in\mathcal{F}$ assumes the value $a$ on $\Omega$. 
\begin{theorem}[Montel's fundamental normality test (FNT)]\label{th:FNT1} A family $\mathcal{F}\subset\mathcal{M}(\Omega)$ that omits three pairwise distinct values $a,b,c\in\widehat{\C}$ is normal in $\Omega$.
\end{theorem}
Schiff \cite{schiff:1993} presents five different proofs of Montel's FNT: via the elliptic modular function (which provides a mapping from the unit disc to the twice punctered plane), via Schottky's theorem, via Ahlfor's five islands theorem, via the rescaling lemma of Zalcman and via Nevanlinna theory. There are several extensions of Montel's FNT: assertion (a) of the next theorem deals with the case that all functions in $\mathcal{F}$ omit three values, but not necessarily the same. Assertion (b) treats families that omit just two value.
\begin{theorem}[Extensions of Montel's FNT]\label{th:FNTextension}
Let $\mathcal{F}\subset\mathcal{M}(\Omega)$.
\begin{itemize}
\item[(a)] Suppose that there exists a real number $\eps>0$ such that each $f\in\mathcal{F}$ omits three pairwise distinct values $a_f, b_f, c_f\in\widehat{\C}$ with
$$
\chi(a_f,b_f), \chi(a_f,c_f),\chi(b_f,c_f) \geq \eps.
$$
Then, $\mathcal{F}$ is normal in $\Omega$.
\item[(b)] 
Suppose that there are three pairwise distinct values $a,b,c\in\widehat{\C}$ such that $\mathcal{F}$ omits the values $a,b$ on $\Omega$ and that no function in $\mathcal{F}$ assumes the value $c$ at more than $m\in\N_0$ points. Then, $\mathcal{F}$ is normal in $\Omega$. 
\end{itemize}
\end{theorem}
Assertion (b) was proved by Carath\'{e}odory \cite[p. 202]{caratheodory:1960}. We mention here also Grahl \& Nevo \cite{grahlnevo:2010} who generalized assertion (b) to `omitted functions' instead of `omitted values'. Assertion (a) was basically known by Montel \cite{montel:1912}. For a proof, the reader is referred to Schiff \cite[p. 56]{schiff:1993}. \par

{\bf The rescaling lemma of Zalcman.} Bloch \cite{bloch:1926} observed an analogy between normal families of meromorphic functions and the value distribution of meromorphic functions in $\C$: certain properties which are responsible for a family $\mathcal{F}\subset \mathcal{M}(\Omega)$ to be normal, seem to force a function $f\in\mathcal{M}(\C)$ to be constant.\footnote{Compare for example Liouville's theorem with Montel's theorem; and Picard's great theorem with Montel's FNT.} Bloch summarized this observation in the words {\it ``Nihil est in infinito quod non prius fuerit in finito''}. Building on works of Lohwater \& Pomerenke \cite{lohwaterpommerenke:1973}, Bloch's heuristic principle was made rigorous by Zalcman \cite{zalcman:1975}. 
\begin{theorem}[Rescaling Lemma of Zalcman] \label{th:zalcman}
Let $\mathcal{F}\subset\mathcal{M}(\mathbb{D})$. Then, $\mathcal{F}$ is not normal in zero if and only if there exist
\begin{itemize}
 \item[(i)] a sequence $(f_n)_n $ of functions $f_n\in \mathcal{F}$,
 \item[(ii)] a sequence $(z_n)_n$ of numbers $z_n\in \D$ with $\lim_{n\rightarrow\infty} z_n = 0$,
 \item[(iii)] a sequence $(\rho_n)_n$ of numbers $\rho_n\in\R^+$ with $\lim_{n\rightarrow\infty} \rho_n = 0$,
\end{itemize}
such that 
$$
g_n(z):= f_n(z_n + \rho_n z)
$$ 
converges locally uniformly to a non-constant function $g\in\mathcal{M}(\C)$. The function $g$ may be taken to satisfy the normatlization
$$
g^{\#}(z) \leq g^{\#} (0)=1, \qquad z\in\C.
$$
\end{theorem}
For a proof of this version of Zalcman's lemma, we refer to Schwick \cite{schwick:1989} and also to Schiff \cite[p. 102]{schiff:1993}. There are several generalizations of Zalcman's lemma known, see for example Pang \cite{pang:1989}. For a given family $\mathcal{F}\subset\mathcal{M}(\D)$ which is not normal in zero, it seems in general very difficult to get more detailed information about the sequences $(f_n)_n$, $(z_n)_n$ and $(\rho_n)_n$ satisfying the assertion of the rescaling lemma. By using the rescaling lemma one can deduce, for example, Montel's FNT and its extensions from Picard-type theorems on the value distribution of functions in $\mathcal{M}(\C)$.

{\bf Schottky's Theorem.} By combining Montel's fundamental normality test with Montel's theorem, one can easily deduce a well-known theorem of Schottky \cite{schottky:1904}.
\begin{theorem}[Schottky's Theorem]
Let $f\in\mathcal{H}(\D)$. Suppose that $f$ omits the values $0$ and $1$ on $\D$ and that $|f(0)| \leq \alpha$. Then, for every $0<r<1$, there exists a positive constant $M(r,\alpha)$, which does only depend on $r$ and $\alpha$, such that
$$
\left|f(z)\right| \leq M(r,\alpha) \qquad \mbox{for all } z\in D_r(0).
$$
\end{theorem}
To obtain explicit expressions for the bounds $M(r,\alpha)$ in terms of $r$ and $\alpha$, further methods are necessary. We refer to Burckel \cite[Chapt. XII, \S 2]{burckel:1979}. Hempel \cite{hempel:1980} derived explicit bounds $M(r,\alpha)$  for Schottky's theorem which are, in a certain sense, best possible. According to his investigations, one can take
$$
M(r,\alpha) =  \frac{1}{16} \Bigl( \min\bigl\{ 16\alpha + 8 ; e^{\pi}\cdot \max\{\alpha;1\}\bigr\} \Bigr)^{\frac{1+r}{1-r}} .
$$
Moreover, if $\alpha<1$, the choice
$$
M(r,\alpha) =  \frac{1}{16} \exp\left( \frac{\pi^2}{\log(16/\alpha)}\cdot \frac{1+r}{1-r}\right)
$$
is admissible, too.

\end{document}